\documentclass[10pt, reqno]{amsart}
\numberwithin{equation}{section}
\newtheorem{thm}{Theorem}[section]
\newtheorem{cor}[thm]{Corollary}
\newtheorem{lem}[thm]{Lemma}
\newtheorem{prp}[thm]{Proposition}
\usepackage{amssymb,latexsym}
\usepackage{eucal}
\usepackage{tikz}
\newtheorem*{P1}{Proposition 1.1}
\newtheorem*{T1}{Theorem 1.2}
\newtheorem*{T2}{Theorem 1.4}
\newtheorem*{T3}{Theorem 1.5}
\newtheorem*{T4}{Theorem 1.7}
\newtheorem*{T5}{Theorem 1.8}
\newtheorem*{T6}{Theorem 1.9}

\newtheorem*{C1}{Corollary 1.3}
\newtheorem*{C2}{Corollary 1.6}
\newtheorem*{C3}{Corollary 1.10}

\newcommand{\al}{\alpha}

\newcommand{\D}{\Delta}

\newcommand{\p}{\partial}
\newcommand{\e}{\varepsilon}

\newcommand{\A}{\mathcal{A}}
\newcommand{\ph}{\varphi}

\newcommand{\GG}{\mathcal{G}}
\newcommand{\FF}{\mathcal{F}}
\newcommand{\R}{\mathcal{R}}

\newcommand{\TT}{\mathcal{T}}

\newcommand{\LL}{\mathcal{L}}
\newcommand{\N}{\mathcal{N}}

\newcommand{\wtl}{\widetilde}
\newcommand{\wht}{\widehat}

\newcommand{\er}{\eqref}
\newcommand{\BS}{bracket system}
\newcommand{\PBS}{pseudobracket system}
\newcommand{\EO}{elementary operation}
\newcommand{\EH}{elementary homotopy}
\newcommand{\EHs}{elementary homotopies}
\newcommand{\pbb}{pseudobracket}

\newcommand{\MF}{Magnus's Freiheitssatz}
\newcommand{\BWPP}{bounded word problem}
\newcommand{\PWPP}{precise word problem}
\newcommand{\ifff}{if and only if}
\newcommand{\ddd}{disk diagram}

\DeclareMathOperator{\ff}{\textsf{fact}}
\DeclareMathOperator{\sll}{\textsf{sl}}

\usepackage{lipsum}

\begin{document}
\title[Bounded and precise word problems for group presentations]
{The bounded and precise  word problems for presentations of groups}
\author{S.\ V.\ Ivanov}
\address{ Department of Mathematics\\
University of Illinois \\1409 West Green Street\\ Urbana\\   IL
61801\\ USA} \email{ivanov@illinois.edu}
\thanks{The work on this article was supported in smaller part by NSF grant  DMS 09-01782}


\begin{abstract}
We introduce and study the bounded word problem and the precise word problem  for groups given by means of generators and defining relations. For example,  for every finitely presented  group, the bounded word problem is in \textsf{NP}, i.e., it can be solved in nondeterministic polynomial time, and  the precise word problem is in \textsf{PSPACE}, i.e., it can be solved in polynomial space. The main technical result of the paper states that,  for certain finite presentations of groups, which include the Baumslag-Solitar one-relator groups and free products of cyclic groups,  the bounded word problem and  the precise word problem  can be solved   in polylogarithmic space. As consequences of developed techniques that can be described as calculus of brackets, we obtain polylogarithmic space bounds for the computational complexity of the diagram  problem for free groups, for the width problem for elements of  free groups, and for computation of the area defined by polygonal singular closed curves in the plane.
We also obtain polynomial time bounds for these problems.
\end{abstract}
\maketitle
\tableofcontents

\section{Introduction}

Suppose that a finitely generated group $\GG$ is defined by a presentation
by means of generators and defining relations
\begin{equation}\label{pr3}
\GG  = \langle \ \A \ \| \ R=1, \ R \in \R  \ \rangle \ ,
\end{equation}
where $\A = \{ a_1, \dots , a_m\}$ is a finite alphabet  and $\R$ is a
set of defining { relators} which are nonempty cyclically reduced words
over the alphabet $\A^{\pm 1} := \A \cup \A^{-1}$.
Let $\FF(\A)$ denote the
free group over $\A$, let $|W|$ be  the length of a word $W$ over
the alphabet $\A^{\pm 1}$, and let $\N( \R )$ denote
the normal closure of $\R$ in $\FF(\A)$. The notation \eqref{pr3}
implies that $\GG$ is the quotient group $\FF(\A)/\N( \R)$.
We will say that a word $W$ over $\A^{\pm 1}$  is equal to 1 in the group $\GG$, given by \eqref{pr3},
if  $W \in \N( \R)$ in which case we also write $W \overset \GG = 1$.
Recall that the  presentation \eqref{pr3} is called {\em finite}  if  both the sets $\A$ and $\R$ are finite, in which  case $\GG$ is called {\em finitely presented}. The presentation \eqref{pr3}  is called {\em decidable} if there is an algorithm that decides whether a given word over $\A^{\pm 1}$ belongs to $\R$.

The classical word problem for a finite group presentation \eqref{pr3},
put forward by Dehn \cite{Dehn} in 1911, asks whether, for a given word $W$ over $\A^{\pm 1}$, it is true that $W \overset \GG = 1$. The word problem is said to be {\em solvable} (or {\em decidable}) for a decidable presentation \eqref{pr3} if there exists an algorithm which, given a word $W$ over $\A^{\pm 1}$,  decides whether or not $W \overset \GG = 1$.

Analogously to Anisimov \cite{A}, one might consider the set  of {\em all} words $W$ over $\A^{\pm 1}$, not necessarily reduced, such that $W \overset \GG = 1$ as a language $\LL(\GG) := \{ W \mid W \overset \GG = 1 \}$  over $\A^{\pm 1}$  and inquire  about  a computational complexity class $\textsf K$  which would contain this language. If $\LL(\GG)$ is in $\textsf K$, we say that the word problem for $\GG$ is in $\textsf K$.
For example, it is well known  that the word problem is in \textsf{P}, i.e., solvable in deterministic polynomial time, for  surface groups,  for finitely presented    groups with  small cancellation condition $C'(\lambda)$, where $\lambda \le \tfrac 16$, for word hyperbolic groups, for finitely presented  groups given by Dehn presentations etc., see \cite{Gromov}, \cite{LS}. On the other hand, according to results of Novikov \cite{NovikovA}, \cite{Novikov} and Boone \cite{BooneA}, \cite{Boone}, based on earlier semigroup constructions of Turing \cite{Turing} and Post \cite{Post}, see also Markov's papers  \cite{M1}, \cite{M2}, there exists a finitely presented group $\GG$ for which the  word problem is unsolvable, i.e., there is no algorithm that decides whether $W \overset \GG = 1 $. The proof of this remarkable Novikov--Boone theorem
was later significantly simplified by Borisov \cite{Borisov}, see also \cite{Rt}.

In this paper,  we introduce and study two related problems which we call the bounded
word problem and the precise word problem  for a decidable  group presentation \eqref{pr3}.

The {\em bounded word problem} for a decidable presentation \eqref{pr3} inquires whether,  for given  a word $W$ over $\A^{\pm 1}$   and an integer $n \ge 0$ written in unary, denoted   $1^n$, one can represent  the word $W$ as a product in  $\FF(\A)$  of at most $n$ conjugates of some words in $\R^{\pm 1} := \R \cup \R^{-1}$, i.e., whether there are $R_1, \dots, R_k \in \R^{\pm 1}$ and   $S_1, \dots, S_k   \in \FF(\A)$  such that
$
W = S_1  R_1  S_1^{-1}   \dots S_k  R_k  S_k^{-1}
$
in  $\FF(\A)$ and $k \le n$.
Equivalently, for  given an input $(W, 1^n)$, the bounded
word  problem asks whether there exists a disk diagram over \eqref{pr3}, also called a van Kampen diagram, whose boundary label is $W$ and whose number of faces is at most $n$.

As above for the word problem, we say that the bounded
word problem for a decidable presentation \eqref{pr3} is solvable  if there exists an algorithm that, given an input $(W, 1^n)$, decides whether $W$ is a  product in $\FF(\A)$  of at most $n$ conjugates of words in $\R^{\pm 1}$.  Analogously, if the language of those pairs $(W, 1^n)$, for which the bounded word problem has a positive answer, belongs to a computational complexity class $\textsf K$, then we say that  the  bounded   word problem for $\GG$ is in $\textsf K$.

The {\em precise word problem} for a decidable presentation \eqref{pr3} asks whether,  for given  a word $W$ over $\A^{\pm 1}$   and a nonnegative integer  $1^n$, one can represent  the word $W$ as a product in  $\FF(\A)$  of  $n$ conjugates of words in $\R^{\pm 1}$ and $n$ is minimal with this property.
Equivalently, for  given an input $(W, 1^n)$, the precise
word  problem asks whether there exists a disk diagram over \eqref{pr3}  whose boundary label is $W$, whose number of faces is  $n$,  and there are no such diagrams with fewer number of faces.

The definitions for the  precise word problem for  \eqref{pr3} being     solvable and being in a  complexity class $\textsf K$ are similar to the corresponding definitions for the (bounded) word problem.

In the following proposition we list basic comparative  properties of the standard word problem and its bounded and precise versions.

\begin{prp} \label{prp1}
$(\mathrm{a})$ There exists a decidable group presentation
\eqref{pr3}  for which the word problem is solvable while the bounded  and precise  word problems are not solvable.

$(\mathrm{b})$ If the bounded  word problem is solvable for \eqref{pr3}, then the precise  word problem is also solvable.

$(\mathrm{c})$ For every finite  group  presentation
\eqref{pr3}, the bounded  word problem is in ${\textsf{NP}}$, i.e.,   it can be solved in nondeterministic polynomial time, and the precise  word problem is in  \textsf{PSPACE}, i.e., it can be solved in polynomial space.

$(\mathrm{d})$ There exists a  finite  group  presentation
\eqref{pr3} for which the bounded  and precise     word problems are solvable while the word problem is not solvable.

$(\mathrm{e})$ There exists a finitely  presented group
\eqref{pr3} for which the bounded  word problem is \textsf{NP}-complete
and the  precise word problem is \textsf{NP}-hard.
\end{prp}

Note that a number of interesting results on solvability of the word problem, solvability of the bounded word problem and computability of the Dehn function for decidable group presentations can be found in a preprint of Cummins \cite{DC} based on his PhD thesis written under the author's supervision.

It is of interest to look at the bounded and precise word problems for finitely presented groups for which  the word problem could be solved very easily.  Curiously, even for very simple presentations such as $ \langle \ a, b  \ \| \ a=1  \ \rangle $ and $ \langle \ a, b  \ \|  \ ab=ba  \ \rangle $, for which the word problem is obviously in $\textsf{L}$,   i.e., solvable in deterministic logarithmic space,
it does not seem to be possible to solve   the bounded and precise word problem in  logarithmic space.
In this article, we will show that  the bounded and precise word problems for these ``simple" presentations and their
generalizations can be solved in polynomial time. With much more effort,  we will also prove that  the bounded and precise word problems for such presentations can be solved in  polylogarithmic space.  Similarly to \cite{ATWZ}, we adopt
\textsf{NC}-style notation and denote $\textsf{L}^\alpha :=  \mbox{DSPACE}((\log s)^\alpha)$, i.e.,  $\textsf{L}^\alpha$ is the class of decision problems that can be solved in deterministic space $O((\log s)^\alpha)$, where $s$ is the size of input.

\begin{thm}\label{thm1}   Let  the group $\GG_2$   be defined by a presentation of the form
\begin{equation}\label{pr1}
    \GG_2 :=  \langle \,  a_1, \dots, a_m  \ \|  \   a_{i}^{k_i} =1, \ k_i \in E_i,   \ i = 1, \ldots,  m \,  \rangle ,
\end{equation}
where for  every $i$, one of the following holds: $E_i = \{ 0\}$ or, for some integer $n_i >0$,  $E_i = \{ n_i \}$ or $E_i = n_i \mathbb N = \{ n_i, 2n_i, 3n_i,\dots  \}$.
Then both the bounded and  precise  word  problems for  \eqref{pr1}  are in $\textsf{L}^3$ and in $\textsf{P}$. Specifically, the problems can be solved in  deterministic space $O((\log|W|)^3)$ or in deterministic time $O( |W|^4\log|W| )$.
\end{thm}

It is worth mentioning that, to prove  $\textsf{L}^3$ part of Theorem~\ref{thm1}, we will not devise a concrete algorithm that solves the bounded and   precise  word  problems for  presentation  \eqref{pr1} in   deterministic space $O((\log|W|)^3)$. Instead, we will develop a certain nondeterministic procedure that solves the bounded   word  problem for  presentation   \eqref{pr1}
nondeterministically in space $O((\log|W|)^2)$ and time  $O(|W|)$  and then use  Savitch's theorem \cite{Sav} on conversion of nondeterministic computations in  space $S$ and time $T$ into   deterministic computations in space $O(S \log T)$.

The proof of $\textsf{P}$ part of Theorem~\ref{thm1} is much easier. Here our arguments are analogous to  folklore arguments \cite{folk}, \cite{Ril16} that solve the  \PWPP\ for presentation
$\langle \, a, b \,  \|  \,  a=1, b =1  \,  \rangle$ in polynomial time and that utilize
the method of dynamic programming.  Interestingly, these folklore arguments are strikingly similar to
the arguments that have been used in computational biology to efficiently solve the problem of planar folding of long chains such as RNA and DNA biomolecules, see  \cite{CB1}, \cite{CB2}, \cite{CB3}, \cite{CB4}.


The techniques to prove   $\textsf{L}^3$ part of Theorem~\ref{thm1} and their generalizations,
that  occupy a significant part of
this article and that could be described as calculus of \BS s,  have applications
to other problems. For example, Grigorchuk and Kurchanov
\cite{GK} defined the {\em width} of an element $W$ of the free group $\FF(\A)$ over  $\A$ as the minimal number $h = h(W)$ so that
\begin{equation}\label{ff11}
W = S_1 a_{j_1}^{k_1} S_1^{-1} \dots S_h a_{j_h}^{k_h} S_h^{-1}
\end{equation}
in $\FF(\A)$, where $a_{j_1}, \dots, a_{j_h} \in \A$, $S_{1}, \dots, S_{h} \in \FF(\A)$ and  $k_{1}, \dots, k_{h}$ are some integers.
Alternatively,  the width $h(W)$ of $W$ can be defined as an integer such that
the \PWPP\ holds for the pair $(W, h(W))$ for the presentation \er{pr1} in which $E_i = \mathbb N$ for every $i$.
Grigorchuk and Kurchanov \cite{GK} found an algorithm that computes the width $h(W)$ for given $W \in \FF(\A)$ and inquired whether computation  of the  width  $h(W)$ can be done in deterministic
polynomial time. Ol'shanskii \cite{Ol} gave a different geometric proof to this  result of Grigorchuk and Kurchanov and suggested some generalizations.

Majumdar, Robbins, and Zyskin  \cite{SL1}, \cite{SL2}  introduced and investigated the {\em spelling length}  $h_1(W)$ of a word $W \in \FF(\A)$ defined by a similar to \er{ff11} formula in which  $k_j=\pm 1$ for every $j$. Alternatively,  the spelling length $h_1(W)$ is an integer such that
the \PWPP\ holds for the pair $(W, h_1(W))$ for the presentation \er{pr1} in which $E_i = \{ 1 \}$ for every $i$.

As a corollary of Theorem~\ref{thm1},  we obtain a positive solution to the problem of Grigorchuk and Kurchanov  and also compute both the width and the spelling length of $W$ in cubic logarithmic space. We remark that Riley \cite{Ril16} gave an independent solution to the Grigorchuk--Kurchanov problem.

\begin{cor}\label{cor1}  Let  $W$ be a word over $\A^{\pm 1}$ and $n \ge 0$ be an integer. Then the decision problems that inquire  whether the width $ h(W)$ or the spelling length $h_1(W)$  of $W$ is equal to $n$ belong to $\textsf{L}^3$ and $\textsf{P}$. Specifically, the problems can be solved in  deterministic space $O((\log|W|)^3)$ or in deterministic time $O( |W|^4\log|W| )$.
\end{cor}

Making many technical modifications but keeping the general strategy of arguments unchanged, we will  obtain similar to Theorem~\ref{thm1} results for Baumslag--Solitar one-relator groups.

\begin{thm}\label{thm2}   Let  the group $\GG_3$ be defined by a presentation of the form
\begin{equation}\label{pr3b}
    \GG_3 :=  \langle \  a_1, \dots, a_m  \ \|  \   a_2 a_1^{n_1} a_2^{-1} = a_1^{ n_2}  \,  \rangle ,
\end{equation}
where $n_1, n_2$ are some nonzero integers. Then both  the bounded  and  precise  word  problems for  \eqref{pr3b}  are in $\textsf{L}^3$ and in $\textsf{P}$. Specifically, the problems can be solved in  deterministic space
$O((\max(\log|W|,\log n) (\log|W|)^2)$ or in deterministic time $O(|W|^4)$.
\end{thm}

As another application of our techniques, we will obtain a  solution in polylogarithmic  space for the (minimal) diagram problem for presentation  \eqref{pr1} which includes the case of the free
group  $\FF(\A) = \langle \A \  \|  \  \varnothing \rangle$ over $\A$ without relations.
Recall that the {\em  diagram problem}    for a decidable presentation \eqref{pr3} is a search problem that, given a word $W$ over $\A^{\pm 1}$ with $W \overset \GG  = 1$, asks to algorithmically construct a disk diagram $\D$ over \eqref{pr3} whose boundary $\p \D$   is labeled by $W$, denoted $\ph(\p \D) \equiv W$, for the definitions see Sect.~2.
Analogously, the {\em   minimal diagram problem}    for a decidable presentation \eqref{pr3} is a search problem that, given a word $W$ over $\A^{\pm 1}$ with $W \overset \GG  = 1$, asks to algorithmically construct a disk diagram $\D$ over \eqref{pr3} such that $\ph(\p \D) \equiv W$ and $\D$ contains a minimal number of faces.

Recall that, according to Lipton and Zalcstein \cite{LZ}, the word problem for the free group $\FF(\A)$, given by presentation \eqref{pr3} with  $\R = \varnothing$, is in  \textsf{L}. However, construction of an actual diagram $\D$ over $\FF(\A)$ for a word $W$ over $\A^{\pm 1}$ such that $\ph(\p \D) \equiv W$, is a different matter and it is not known whether this construction could be done in polylogarithmic space (note that it is easy to construct such a diagram, which is a tree, in polynomial time). In fact, many results of this article grew out of the attempts to solve the diagram problem for  free groups with no relations in subpolynomial space.

\begin{thm}\label{thm3}   Both the diagram problem  and the minimal diagram problem
for  group presentation \eqref{pr1} can be solved in deterministic space  $O((\log|W|)^3)$ or in deterministic time $O(|W|^4\log|W|)$.

Furthermore,  let $W$ be a word such that $W \overset {\GG_2} =1$ and let
$$
\tau(\D) = (\tau_1(\D), \ldots, \tau_{s_\tau}(\D))
$$
be a tuple of integers, where the absolute value $| \tau_i(\D) |$ of each $\tau_i(\D)$  represents the number of certain vertices or faces in a disk diagram $\D$ over \eqref{pr1} such that $\ph(\p \D) \equiv W$. Then, in  deterministic  space $O( (\log |W|)^3 )$, one can algorithmically construct such a minimal diagram $\D$ which is also smallest relative to  the tuple $\tau(\D)$ (the tuples are ordered lexicographically).
\end{thm}

We point out that the case of the free group $\FF(\A) = \langle \A \  \|  \  \varnothing \rangle$ with no relations is covered by Theorem~\ref{thm3} and, since there are no relations, every diagram $\D$ over
$\FF(\A) = \langle \A \  \|  \  \varnothing \rangle$ is minimal.  Hence, for the free group  $\FF(\A)$, Theorem~\ref{thm3} implies the following.

\begin{cor}\label{cor2} There is a deterministic algorithm that, for given word $W$ over the alphabet $\A^{\pm 1}$ such that
$W \overset{\FF(\A)}{=} 1$, where  $\FF(\A) = \langle \A \  \|  \  \varnothing \rangle$ is the free group over $\A$, constructs a pattern of cancellations of letters in $W$ that result in the empty word and the  algorithm operates in space $O( (\log |W|)^3 )$.

Furthermore, let $\D$ be a disk diagram over $\FF(\A) $ that corresponds to a pattern of cancellations of letters in $W$, i.e., $\ph(\p \D) \equiv W$, and let
$$
\tau(\D) = (\tau_1(\D), \ldots, \tau_{s_\tau}(\D))
$$
be a tuple of integers, where the absolute value $| \tau_i(\D) |$ of each  $\tau_i(\D)$  represents the number  of vertices in $\D$ of certain degree. Then, also in  deterministic  space $O( (\log |W|)^3 )$, one can algorithmically construct such a diagram $\D$ which is smallest relative to the tuple $\tau(\D)$.
\end{cor}

Here is  the analogue of Theorem~\ref{thm3} for presentations \eqref{pr3b} of one-relator Baumslag--Solitar groups.

\begin{thm}\label{thm4}  Suppose that $W$ is a word  over the alphabet
$\A^{\pm 1}$ such that the bounded word problem for presentation \eqref{pr3b}  holds for the pair $(W, n)$.
Then a minimal diagram $\D$ over \eqref{pr3b} such that  $\ph( \p  \D) \equiv W$ can be algorithmically  constructed in deterministic space $O( \max(\log |W|, \log n)(\log |W|)^2)$ or in deterministic time $O( |W|^4)$.

In addition, if $|n_1 | = |n_2 |$ in \eqref{pr3b}, then  the minimal diagram problem for  presentation \eqref{pr3b}  can be solved in deterministic space $O( (\log |W|)^3 )$  or in deterministic time $O( |W|^3\log |W|)$.
\end{thm}

As further applications of the techniques of the proof of Theorems~\ref{thm2}, \ref{thm4},  we obtain computational
results on discrete homotopy of polygonal closed curves in the plane.

Let $\TT$  denote a tessellation of the plane $\mathbb  R^2$ into unit squares whose vertices are points with integer coordinates. Let $c$ be a finite closed path in $\TT$  so that edges of $c$ are edges of $\TT$. Consider the following two types of elementary operations over $c$. If $e$ is an oriented edge of $c$, $e^{-1}$ is the edge with an opposite to $e$ orientation, and $ee^{-1}$ is a subpath of $c$ so that  $c = c_1 ee^{-1} c_2$, where $c_1, c_2$ are subpaths of $c$, then the operation $c \to c_1 c_2$ over $c$ is called an {\em \EH\ of type 1}.  Suppose that $c =c_1 u c_2$, where
$c_1, u, c_2$ are subpaths of $c$, and a boundary path $\p s$ of a unit square $s$ of  $\TT$ is $\p s = uv$, where
$u, v$ are subpaths of $\p s$ and either of $u, v$  could be of zero length, i.e., either of $u, v$  could be a single vertex of $\p s$. Then the operation $c \to c_1 v^{-1} c_2$ over $c$ is called an {\em \EH\ of type 2}.

\begin{thm}\label{thm5} Let $c$ be a finite closed path in a tessellation $\TT$  of the plane $\mathbb  R^2$
into unit squares so that edges of $c$ are edges of $\TT$. Then a minimal number $m_2(c)$
such that there is a finite sequence of \EHs\ of type 1--2, which turns $c$ into a single point and which contains
$m_2(c)$ \EHs\ of type 2, can be computed in deterministic space $O((\log |c|)^3)$  or in deterministic time
$O( |c|^3\log |c|)$,
where $|c|$  is the length of $c$.

Furthermore, such a sequence of  \EHs\ of type 1--2,
which turns $c$ into a single point and which contains  $m_2(c)$ \EHs\ of type 2,  can also be computed in  deterministic space $O((\log |c|)^3)$ or in deterministic time $O( |c|^3\log |c|)$.
\end{thm}

We remark that this number $m_2(c)$ defined in Theorem~\ref{thm5} can be regarded as the  area ``bounded" by a closed path $c$  in $\TT$. Clearly, if $c$ is simple, i.e., $c$  has no self-intersections,
then  $m_2(c)$  is the area of the compact region  bounded by $c$.  If $c$ is simple then the area bounded by $c$ can be computed in logarithmic space $O(\log |c|)$ as follows from the ``shoelace"  formula for the area of a simple polygon.

More generally, assume that $c$ is a continuous closed curve in $\mathbb R^2$, i.e., $c$ is the image of a continuous map $\mathbb S^1 \to \mathbb R^2$, where $\mathbb S^1$ is a circle. Consider a homotopy
$H : \mathbb S^1 \times [0,1] \to \mathbb R^2$ that turns the curve $c = H(\mathbb S^1 \times \{ 0 \} )$
into a point  $H(\mathbb S^1 \times \{ 1 \} )$ so that for every $t$, $0\le t <1$,  $H(\mathbb S^1 \times \{ t \} )$ is a closed curve. Let $A(H)$ denote the area swept by the curves $H(\mathbb S^1 \times \{ t \} )$, $0\le t <1$, and let $A(c)$ denote the infimum $\inf_H A(H)$ over all such homotopies $H$. As above, we remark that this
number $A(c)$ can be regarded as the area defined (or ``bounded")  by  $c$.
Note that this number $A(c)$ is different from the signed area of $c$ defined by applying the ``shoelace"  formula to singular polygons.

Other applications that we discuss here involve  polygonal (or piecewise linear) closed curves in the plane and computation and approximation of the area  defined  by these curves in polylogarithmic space or in polynomial time.

We say that $c$ is a {\em polygonal} closed curve in the plane $\mathbb R^2$ with given tessellation $\TT$
into unit squares if $c$ consists of finitely many line segments $c_1, \dots, c_k$, $k >0$, whose endpoints
are vertices of $\TT$,  $c=  c_1 \dots c_k$, and $c$ is closed.  If $c_i \subset \TT$ then the $\TT$-length $|c_i|_{\TT}$ of $c_i$ is the number of edges of  $\TT$ in $c_i$. If
$c_i \not\subset \TT$ then the $\TT$-length $|c_i|_{\TT}$ of $c_i$
is the number of connected components in $c_i  \setminus \TT$. We assume that $|c_i|_{\TT} >0$ for every $i$ and set $|c|_{\TT} := \sum_{i=1}^k |c_i|_{\TT}$.

\begin{thm}\label{thm6} Suppose that $n \ge 1$ is a fixed integer and  $c$ is a polygonal closed curve in the plane $\mathbb R^2$ with given tessellation $\TT$
into unit squares. Then, in deterministic space $O( (\log |c|_{\TT} )^3)$  or in deterministic time $O( |c|_{\TT}^{n+3}\log |c|_{\TT} )$, one can compute a rational number $r_n$
such that $|A(c) - r_n | < \tfrac {1}{ |c|_{\TT}^n  }$.

In particular, if the area  $A(c)$ defined by $c$ is known to be an integer multiple of $\tfrac 1 L$, where $L>0$ is a given integer and $L< |c|_{\TT}^{n}/2$,  then
$A(c)$ can be computed in deterministic space $O( (\log |c|_{\TT} )^3)$
or in deterministic time $O( |c|_{\TT}^{n+3} \log |c|_{\TT})$.
\end{thm}

\begin{cor}\label{cor3}
Let $K \ge 1$ be a fixed integer and let $c$ be a  polygonal closed curve in the plane
$\mathbb R^2$ with given tessellation $\TT$ into unit squares such that $c$ has one
of the following two properties $\mathrm{(a)}$--$\mathrm{(b)}$.

$\mathrm{(a)}$ If $c_i, c_j$ are two nonparallel line segments of $c$ then their intersection point, if it exists,  has coordinates that are integer multiples of $\tfrac 1 K$.

$\mathrm{(b)}$ If $c_i$ is a line segment of $c$ and  $a_{i,x}, a_{i,y}$ are coprime integers such that
the line given by an equation $a_{i,x}x + a_{i,y}y = b_{i}$, where $b_i$ is an integer,
contains $c_i$, then $\max(|a_{i,x}|,|a_{i,y}|) \le K$.

 Then the area $A(c)$ defined by $c$ can be computed in deterministic space $O( (\log |c|_{\TT} )^3)$  or in deterministic time $O( |c|_{\TT}^{n+3}  \log |c|_{\TT} )$, where $n$ depends on $K$.

In particular, if $\TT_\ast$ is a tessellation  of the plane $\mathbb R^2$ into  equilateral  triangles of unit area,
or into regular hexagons of unit area, and $q$ is a finite closed path in $\TT_\ast$ whose edges are edges of $\TT_\ast$, then the area $A(q)$ defined by $q$ can be computed in deterministic space $O( (\log |q|)^3)$ or in deterministic time $O( |q|^5\log |q|)$.
\end{cor}

It is tempting to try to lift the restrictions of Corollary~\ref{cor3} to be able to compute, in
polylogarithmic space, the  area $A(c)$ defined by an arbitrary polygonal closed curve $c$ in the plane equipped with a tessellation $\TT$ into unit squares.
However, in the general situation, this idea would not work because the rational number $A(c)$ might have an exponentially large denominator, hence, $A(c)$ could take polynomial space just to store (let alone the computations), see
an example in the end of Sect.~10.

We remark in passing that there are decision problems in \textsf{NP}  that are not known to be
\textsf{NP}-complete or in
\textsf{P}, called \textsf{NP}-intermediate problems, that are solvable in polylogarithmic space.
For example, a restricted version of the \textsf{NP}-complete
clique problem asks whether a graph on $n$ vertices  contains a  clique with at most $[ \log n ]$ vertices,
where $[k]$ is the integer part of $k$, and this restriction is obviously a problem solvable in nondeterministic
space $O([ \log n ]^2)$. More natural examples of such \textsf{NP}-intermediate problems would be
a decision version of the problem on finding a minimum dominating set in a tournament,  \cite{MV}, \cite{PY},  and the problem on isomorphism of two finite groups given by their multiplication tables, \cite{BKLM}, \cite{LSZ}.

\section{Preliminaries}

If $U, V$ are words over an alphabet $\A^{\pm 1} := \A \cup \A^{-1}$,
then $U \overset 0 = V$ denotes the equality of $U, V$ as elements of the free group $\FF(\A)$ whose set of free generators is  $\A$.  The equality of natural images of words
$U, V$  in the group $\GG$, given by presentation \er{pr3}, is denoted $U \overset \GG = V$.

The letter-by-letter equality of words $U, V$ is denoted $U \equiv V$.
If $U \equiv a_{i_1}^{\e_1} \dots a_{i_\ell}^{\e_\ell}$, where  $a_{i_1}, \dots, a_{i_\ell} \in \A$ and $\e_1, \dots, \e_\ell \in \{ \pm 1\}$,  then the length of $U$ is $|U| =\ell$.
A nonempty word $U$ over  $\A^{\pm 1}$ is called {\em reduced} if $U$ contains no subwords of the form $a a^{-1}$, $a^{-1}a$, where $a \in \A$.

Let $\D$ be a 2-complex and let $\D(i)$ denote the set of nonoriented $i$-cells  of $\D$, $i=0,1,2$.
 We also consider the set $\vec \D(1)$ of oriented  1-cells of  $\D$.
If $e \in  \vec \D(1)$ then $e^{-1} $ denotes $e$ with the  opposite orientation,  note that $e \ne e^{-1}$.  For every $e \in \vec \D(1)$,  let $e_-$, $e_+$ denote the initial, terminal, resp., vertices of $e$. In particular, $(e^{-1})_- = e_+$ and $(e^{-1})_+ = e_-$. The closures of
$i$-cells  of  $\D$ are called {\em vertices, edges, faces} when  $i=0, 1, 2$, resp.

A path $p = e_1 \dots e_\ell$ in $\D$ is a sequence of oriented edges $e_1, \dots, e_\ell$ of $\D$
such that
$(e_i)_+ = (e_{i+1})_-$, $i =1,\dots, \ell-1$.
The length of a path $p= e_1 \dots e_\ell$ is $|p| = \ell$.
The initial vertex of $p$ is $p_- := (e_1)_-$ and the terminal vertex of $p$ is $p_+ := (e_\ell)_+$.
We allow the possibility that $|p| = 0$ and $p = p_-$.

A path $p$ is called {\em reduced} if $|p| > 0$ and
$p$ contains no subpath of the form $e e^{-1}$, where $e$ is an edge.
A path $p$ is called {\em closed} if $p_- = p_+$.

A {\em cyclic } path is a closed path with no distinguished initial vertex.
A path $p= e_1 \dots e_\ell$ is called {\em simple} if the  vertices
$(e_1)_-, \dots,   (e_\ell)_-, (e_\ell)_+$ are  all  distinct.
A closed path  $p = e_1 \dots e_\ell$    is  {\em simple} if the
vertices $(e_1)_-, \dots,   (e_\ell)_-$  are all distinct.

A {\em disk diagram}, also called a {\em van Kampen diagram},
$\D$ over a presentation \eqref{pr3} is a planar connected and
simply connected finite 2-complex which is equipped with a labeling function
$$
\ph : \vec \D(1) \to \A \cup \A^{-1} = \A^{\pm 1}
$$
such that, for every $e \in \vec \D(1)$,
$\ph(e^{-1}) = \ph(e)^{-1}$ and, for every face $\Pi$ of $\D$, if
$\p \Pi = e_1 \dots e_\ell$ is a boundary path of $\Pi$, where $e_1, \dots, e_\ell \in \vec \D(1)$,
then
$$
\ph(\p \Pi) :=   \ph(e_1) \dots \ph(e_\ell)
$$
is a cyclic permutation of one of the words in  $\R^{\pm 1} = \R \cup \R^{-1}$.

A disk diagram $\D$ over presentation \eqref{pr3} is always considered with an embedding
$\D \to \mathbb R^2$ into the plane $\mathbb R^2$. This  embedding makes it possible to define
positive (=counterclockwise) and  negative (=clockwise) orientations for  boundaries of
faces of $\D$, for the boundary path $\p \D$ of $\D$, and, more generally,
for  boundaries of   disk  subdiagrams of $\D$. It is convenient to assume, as in  \cite{Iv94}, \cite{Ol89}, that the boundary path $\p \Pi$ of every face $\Pi$ of a disk diagram $ \D$ has the positive orientation while the boundary path  $\p \D$ of  $\D$ has the negative orientation.

If $o \in \p \D$ is a vertex, let  $\p |_o \D$ denote a boundary path  of $\D$ (negatively oriented) starting (and ending) at $o$. Using this notation, we now state van Kampen lemma on geometric
interpretation of consequences of defining relations.

\begin{lem}\label{vk}  Let  $W$ be a nonempty word  over $\A^{\pm 1}$
and let a group $\GG$ be defined by presentation \er{pr3}.
Then $W \overset \GG = 1$ if and only if there is a disk diagram
over presentation \er{pr3} such that $\ph( \p |_o \D ) \equiv W$
for some vertex $o \in \p \D$.
\end{lem}

\begin{proof} The proof is straightforward, for details the reader is referred to \cite{Iv94}, \cite{Ol89}, see also \cite{LS}.
\end{proof}

A disk diagram $\D$ over presentation  \er{pr3} with    $\ph( \p  \D ) \equiv W$
is called {\em minimal } if $\D$ contains a minimal number of faces among all disk
diagrams $\D'$ such that  $\ph( \p  \D' ) \equiv W$.

A disk diagram $\D$ over \er{pr3} is called {\em reduced} if $\D$ contains
no two faces $\Pi_1$, $\Pi_2$ such that there is a vertex $v \in \p \Pi_1$, $v \in \p \Pi_2$
and  the boundary paths $\p |_v \Pi_1$, $\p |_v \Pi_2$ of the faces,
starting at $v$, satisfy  $\ph ( \p |_v \Pi_1 ) \equiv \ph (  \p |_v \Pi_2)^{-1}$.
If $\D$ is not reduced and $\Pi_1$, $\Pi_2$  is a pair of faces that violates
the definition of being reduced for $\D$, then these faces $\Pi_1$, $\Pi_2$  are
called  a {\em  reducible pair} (cf. similar definitions in  \cite{Iv94}, \cite{LS}, \cite{Ol89}).  A reducible pair of faces $\Pi_1$, $\Pi_2$ can be removed from $\D$ by a surgery that cuts through the vertex $v$ and identifies the boundary paths
$\p |_v \Pi_1$ and $(\p |_v \Pi_2)^{-1}$, see Fig.~2.1, more details can be found in
\cite{LS}, \cite{Ol89}.
As a result, one obtains a disk diagram $\D'$ such that
$\ph( \p  \D' ) \equiv \ph( \p  \D )$ and $|\D'(2) | = |\D(2) | -2$,
where $|\D(2)|$ is the number of faces in $\D$.
In particular, a minimal disk diagram is always reduced.

\begin{center}
\usetikzlibrary{arrows}
\begin{tikzpicture}[scale=.86]
\node at (-1,-.7) {$v$};
\node at (-2,0) {$\Pi_1$};
\node at (0,0) {$\Pi_2$};
\draw  (-1,0)[fill = black]circle (0.05);
\tikzstyle{myedgestyle} = [-open triangle 45]
\draw [-open triangle 45]  (1.5,0) -- (2.4,0);
\draw  (3,-1.2) -- (3,1.2);
\draw  (3,-1.2)[fill = black]circle (0.05);
\draw  (3,1.2)[fill = black]circle (0.05);
\node at (3.5,-1.5) {$v''$};
\node at (3.5,1.5) {$v'$};
\node at (1,-2) {Fig.~2.1};
\draw  (0,0) ellipse (1 and 0.7);
\draw  (-2,0) ellipse (1 and 0.7);
\end{tikzpicture}
\end{center}

Note that if   $o \in \p \D$ is a vertex and  $\p |_o \D = q_1 q_2$ is
a factorization of the boundary path $\p |_o \D$  of $\D$, $q_1$ is closed,
$0 < |q_1|, |q_2| <   | \p  \D |$,  then  the notation
$\p |_{o} \D $    is in fact
ambiguous,  because the path $q_2 q_1$ also has the form $\p |_o \D$. To avoid this (and other) type of ambiguity, for a given pair $(W, \D)$, where $W \overset \GG = 1$ and $\D$ is a disk diagram for $W$ as in
Lemma~\ref{vk},  we consider a ``model"  simple  path $P_W$ such that $| P_W | = |W|$,
$P_W$  is equipped with a labeling function $\ph : \vec P_W(1) \to \A^{\pm 1} $ on the
set $\vec P_W(1)$ of its oriented edges so that $\ph(e^{-1}) = \ph(e)^{-1}$ and $\ph( P_W) \equiv W$.

It will be convenient to identify  vertices of $P_W$ with integers $0, 1, \dots, |W|$
so that $(P_W)_- =0$ and the numbers corresponding to vertices increase
by one as one goes along $P_W$ from  $(P_W)_- =0$ to $(P_W)_+ =|W|$.
This makes it possible to compare vertices
$v_1, v_2 \in P_W(0)$ by the standard order $\le$ defined on the integers,
consider vertices $v_1\pm 1$ etc.

For a given pair $(W, \D)$, where $W \overset \GG = 1$, let
$$
\al : P_W \to \p \D
$$
be a continuous cellular map that preserves dimension of cells,
$\ph$-labels of edges, and has the property that $\al((P_W)_-) = \al(0) = o$,
where $o$ is a fixed vertex of $\p \D$ with    $\ph( \p |_o \D ) \equiv W$.

If $v$ is a vertex of $P_W$, let
$$
P_W(\ff, v)  = p_1p_2
$$
denote the factorization of   $P_W$ defined by $v$ so that $(p_1)_+ = v$.
Analogously, if $v_1, v_2$ are vertices of  $P_W$ with $v_1 \le v_2$,
we let
$$
P_W(\ff, v_1, v_2)  = p_1p_2p_3
$$
denote the factorization of   $P_W$ defined by $v_1, v_2$ so that $(p_2)_- = v_1$
and $(p_2)_+ = v_2$. Note that if $v_1 = v_2$, then $p_2 = \{ v_1\}$ and $|p_2 | = 0$.
Clearly, $|p_2| = v_2 - v_1$.

Making use of the introduced notation, consider a vertex $v$ of $P_W$   and let $P_W(\ff, v)  = p_1p_2$. Define $\p |_v \D : = \al(p_2) \al(p_1)$. This notation $\p |_v \D $, in place of   $\p |_{\al(v)} \D $, will help us to avoid the potential ambiguity when writing  $\p |_{\al(v)} \D $. In particular, if $\bar W$
is a  cyclic permutation of $W$ so that the first $k$, where $0 \le k \le |W|-1$, letters of $W$ are put at the end of $W$, then   $\ph(\p |_{k} \D) \equiv \bar W$. It is clear that  $\ph(\p |_{0} \D) \equiv W$.

Consider the following property.

\begin{enumerate}
\item[(A)] Suppose that $\D$ is a \ddd\ over \er{pr1}. If $\Pi$ is a face of $\D$ and
$e \in \p \Pi$ is an edge then $e^{-1} \in \p \D$.
 \end{enumerate}

We now state a lemma in which we record some simple properties of \ddd s over  \er{pr1}  related to property (A).

\begin{lem}\label{vk2} Let $\D$ be a disk diagram over presentation \er{pr1}. Then the following hold true.

$(\mathrm{a})$  If $\D$  has property (A), then the degree of every vertex of
$\D$ is at most $2|\p \D|$, the boundary path $\p \Pi$ of every face $\Pi$ of $\D$ is simple, and
$$
|\D(2)| \le  | \p \D | , \quad  \sum_{\Pi \in \D(2)} |\p \Pi | \le | \p \D | .
$$

$(\mathrm{b})$  There exists a  disk diagram $\D'$ over \er{pr1} such that $\ph( \p \D') \equiv \ph( \p \D)$, $| \D'(2) | \le |\D(2)|$  and $\D'$ has property (A).
\end{lem}

\begin{proof} (a) Let $v$ be a vertex of $\D$. By property (A), $v \in \p \D$ and if $e$ is an edge such that $e_- = v$, then either $e \in \p \D$ or $e \in \p \Pi$, where $\Pi$ is a face of $\D$,  and $e^{-1} \in \p \D$. This implies that $\deg v \le 2|\p \D|$.

If the  boundary path $\p \Pi$ of a face $\Pi$ of $\D$ is not simple, then there is a factorization $\p \Pi = u_1 u_2$, where $u_1, u_2$ are closed subpaths of $\p \Pi$ and
$0 < |u_1|, |u_2| < |\p \Pi|$, see Fig.~2.2.  Clearly, the edges of one of the paths $u_1, u_2$ do not belong to the boundary $\p \D$ of $\D$, contrary to property (A) of $\D$.

\begin{center}
\begin{tikzpicture}[scale=.44]
\draw  (0,-1) ellipse (2.8 and 3);
\draw  (0,-2) ellipse (1.5 and 2);
\draw [-latex]  (.1,2) -- (-.3,2);
\draw [-latex]  (-.1,0) -- (.2,0);
\draw  (0,-4) [fill = black] circle (.07);
\node at (0,1.4) {$u_1$};
\node at (0,-.7) {$u_2$};
\node at (1.6,0) {$\Pi$};
\node at (0,-5) {Fig. 2.2};
\node at (5,-3) {$\partial \Pi = u_1u_2$};
\node at (0,-2.2) {$\Delta_2$};
\end{tikzpicture}
\end{center}

The inequality $|\D(2)| \le  | \p \D |$ and its stronger version $\sum_{\Pi \in \D(2)} |\p \Pi | \le | \p \D |$ are immediate from property (A).

\smallskip
(b) Suppose that a disk diagram $\D$ over \er{pr1} does not have property (A).

First assume that $\D$ contains a face $\Pi$ whose  boundary path $\p \Pi$ is not simple.
Then, as in the proof of part (a),   there is a factorization $\p \Pi = u_1 u_2$, where $u_1, u_2$ are closed subpaths of $\p \Pi$ and
$0 < |u_1|, |u_2| < |\p \Pi|$.
Renaming $u_1$ and $u_2$ if necessary, we may assume that $u_2$ bounds  a disk subdiagram $\D_2$ of $\D$ such that $\D_2$ does not contain $\Pi$, see Fig.~2.2. If $u_2$ is not simple, then we can replace $u_2$ by its closed subpath $u'_2$ such that
$0 < |u'_2| < |u_2|$ and $u'_2$ bounds  a disk subdiagram $\D'_2$  that contains no $\Pi$.
Hence, choosing a shortest path $u_2$ as above, we may assume that $u_2$ is simple.
By Lemma~\ref{vk} applied to $\D_2$ with  $\p \D_2 = u_2$, we have
\begin{equation}\label{ud2}
   \ph(u_2) \equiv \ph(\p \D_2)  \overset{\GG_2} = 1.
\end{equation}

Denote $\ph(\p \Pi) \equiv a_i^{\e k}$, where $\e = \pm 1$, $k \in E_i$.

Note that the group ${\GG_2}$ is the free product of cyclic groups generated by the images of generators $a_1, \dots, a_m$ and the image of $a_j$ has order $n_j >0$ if $E_j \ne \{ 0\}$
or the image of $a_j$ has infinite order if $E_j = \{ 0\}$. Hence, an equality
$a_j^\ell \overset{\GG_2} = 1$, where $\ell \ne 0$, implies that $E_j \ne \{ 0\}$ and
$n_j$ divides $\ell$.

It follows from \er{ud2} and $0 < |u_2| < |\p \Pi|$ that $\ph(u_2) \equiv a_i^{\e k_2}$, where $k_2 \in E_i$, $k_2 < k$. Therefore, $E_i = \{ n_i, 2n_i, \dots \}$ and  $\ph(u_1) \equiv a_i^{\e (k-k_2)}$, where $k-k_2 \in E_i$. Hence, we can consider a face $\Pi'$ such that $\ph(\p \Pi') \equiv \ph(u_1) \equiv a_i^{\e (k-k_2)}$.
Now we take the subdiagram $\D_2$ out of $\D$ and replace the face $\Pi$ with $\ph(\p \Pi) \equiv a_i^{\e k}$ by the face  $\Pi'$ with $\ph(\p \Pi') \equiv a_i^{\e (k-k_2)}$.
Doing this results in a \ddd\ $\D'$ such that $\ph(\p \D')\equiv \ph(\p \D')$ and
$| \D'(2) | < | \D(2) |$ as $| \D_2(2) | >0$.

Assume that, for every face $\Pi$ in $\D$, the boundary path $\p \Pi$ is simple. Also, assume that the property (A) fails for $\D$. Then there are faces  $\Pi_1$, $\Pi_2$, $\Pi_1 \ne \Pi_2$, and an edge $e$ such that $e \in \p \Pi_1$ and $e^{-1} \in \p \Pi_2$.
Consider a disk subdiagram $\Gamma$ of $\D$ that contains $\Pi_1$, $\Pi_2$ and
$\Gamma$  is minimal with this property relative to  $|\Gamma(2)|+|\Gamma(1)|$.
Since  $\p \Pi_1$,  $\p \Pi_2$ are simple paths, it follows that $\p \Gamma = r_1 r_2$,
where $r_1^{-1}$ is a subpath of  $\p \Pi_1$ and $r_2^{-1}$ is a subpath of  $\p \Pi_2$.
Denote $\ph(\p \Pi_1) \equiv  a_i^{\e_1 k_1}$, where $\e_1 = \pm 1$ and $k_1 \in E_i$.
Clearly, $\ph(\p \Pi_2) \equiv  a_i^{-\e_1 k_2}$, where  $k_2 \in E_i$ and
$$
\ph(\p \Gamma) \equiv \ph(r_1) \ph(r_2) \overset{0} =   a_i^{\e k} ,
$$
where  $\e = \pm 1$ and $k \ge 0$. As above, we observe that $k \in E_i$ following from
$a_i^{\e k} \overset{\GG_2} = 1$. Hence, we may consider a disk diagram $\Gamma'$ such that
$\p \Gamma' = r'_1 r'_2$, where $\ph(r'_1)\equiv \ph(r_1)$,  $\ph(r'_2)\equiv \ph(r_2)$,
and $\Gamma'$ contains a single face $\Pi$  such that $\ph(\p \Pi) \equiv  a_i^{-\e k}$ if $k \ne 0$ or $\Gamma'$ contains no faces if $k = 0$. We take the subdiagram $\Gamma$ out of $\D$ and replace $\Gamma$ with $\Gamma'$, producing thereby a \ddd\  $\D'$ such that
$\ph(\p \D')\equiv \ph(\p \D')$ and $| \D'(2) | < | \D(2) |$.

We now observe that if  property (A) fails for $\D$ then there is a face $\Pi$ in $\D$ such that $\p \Pi$ is not simple or there are distinct faces $\Pi_1$, $\Pi_2$ and an edge $e$ such that $e \in \p \Pi_1$ and $e^{-1} \in \p \Pi_2$. In either case, as was shown above, we can find a \ddd\ $\D'$ such that $\ph(\p \D')\equiv \ph(\p \D')$ and $| \D'(2) | < | \D(2) |$.
Now obvious induction on $| \D(2) |$ completes the proof of part (b).
\end{proof}

In view of Lemma~\ref{vk2}(b), we will be assuming
in Sects. 4--5, 8 that if $\D$ is a \ddd\ over presentation \er{pr1}, then $\D$  has property  (A).

\section{Proof of Proposition \ref{prp1}}

\begin{P1}
$(\mathrm{a})$ There exists a decidable group presentation
\eqref{pr3}  for which the word problem is solvable while the bounded  and precise  word problems are not solvable.

$(\mathrm{b})$ If the bounded  word problem is solvable for \eqref{pr3}, then the precise  word problem is also solvable.

$(\mathrm{c})$ For every finite  group  presentation
\eqref{pr3}, the bounded  word problem is in ${\textsf{NP}}$, i.e.,   it can be solved in nondeterministic polynomial time, and the precise  word problem is in  \textsf{PSPACE}, i.e., it can be solved in polynomial space.

$(\mathrm{d})$ There exists a  finite  group  presentation
\eqref{pr3} for which the bounded  and precise     word problems are solvable while the word problem is not solvable.

$(\mathrm{e})$ There exists a finitely  presented group
\eqref{pr3} for which the bounded  word problem is \textsf{NP}-complete
and the  precise word problem is \textsf{NP}-hard.
\end{P1}

\begin{proof}
(a) We will use the construction of \cite[Example~3]{GI08} suggested by
 C.~Jockush and I.~Kapovich.  Consider the group presentation
\begin{equation}\label{jkp}
 \langle \ a, b \  \| \ a^i=1, \  a^i b^{k_i}=1,  \ \
 i \in \mathbb N \,   \rangle  ,
 \end{equation}
where $\mathbb K = \{ k_1, k_2, \dots \}$ is a recursively
enumerable but not recursive subset of the set of natural numbers $\mathbb N$ with the
indicated enumeration and $k_1 =1$. It is clear that the
set of relations is decidable  and this presentation defines the
trivial group, hence  the word problem is solvable for \er{jkp}. On the other
hand, it is easy to see that the bounded word problem for a pair  $(b^k, 2)$, where $k \in \mathbb N$, holds true
if and only if $k \in \mathbb K$. Analogously, the precise word problem for a pair  $(b^k, 2)$ holds true
if and only if $k \in \mathbb K$.  Since the set $\mathbb K$ is not
recursive, it follows that both the bounded word problem and  the precise word problem for presentation \er{jkp} are unsolvable.
\smallskip

(b) Note that the \PWPP\ holds true for a pair $(W, n)$  \ifff\
the \BWPP\ is true for $(W, n)$  and the \BWPP\ is false for $(W, n-1)$.  This remark means that the solvability
of the \BWPP\ for   \er{pr3} implies the solvability of the \PWPP\ for   \er{pr3}. On the other hand, the \BWPP\ holds for a pair $(W, n)$  \ifff\
the \PWPP\ holds  for $(W, k)$ with some $k \le n$. This remark  means that the solvability
of the \PWPP\ for  \er{pr3} implies the solvability of the \BWPP\ for   \er{pr3}, as required.
\smallskip

(c)  Suppose that  presentation  \er{pr3} is finite, i.e., both $\A$ and $\R$ are finite, and  we are given a pair
 $(W, 1^n)$. It follows from definitions and  Lemma~\ref{vk} that the \BWPP\ holds for the pair    $(W, 1^n)$ \ifff\ there is a disk diagram  $\D$ such that $\ph( \p \D) \equiv W$ and $ | \D(2) |  \le  n$.
 Observe that $ | \vec \D(1) |  \le  M n+ |W|$,  where $M = \max \{ |R| \, : \, R \in \R \}$ is a constant. Therefore, the size
 of a disk diagram  $\D$ with $\ph( \p \D) \equiv W$ is bounded by a linear function in $n +|W|$ and such a
 diagram  $\D$  can be used as a certificate to verify in polynomial time that the \BWPP\ holds for the pair   $(W, 1^n)$.
 Thus the \BWPP\ for   finite presentation  \er{pr3} is in \textsf{NP}.

Recall that the \PWPP\ holds  for a pair $(W, 1^n)$  \ifff\  the \BWPP\ is true for $(W, 1^n)$
and the \BWPP\ is false for $(W, 1^{n-1})$. As we saw above, the  \BWPP\  for
\er{pr3} is in \textsf{NP}, hence, the complement of the  \BWPP\  for
\er{pr3} is in \textsf{coNP}. Since both \textsf{coNP} and \textsf{NP} are subsets of \textsf{PSPACE},
it follows that the  \PWPP\  for   finite presentation  \er{pr3} is in \textsf{PSPACE}.
\smallskip

(d) According to Boone \cite{BooneA}, \cite{Boone}   and Novikov \cite{NovikovA}, \cite{Novikov}, see also  \cite{LS}, there exists a finite group presentation  \er{pr3} such that the word problem for this presentation is not solvable. In view of part (c) both the \BWPP\ and  \PWPP\ for this presentation are solvable.
\smallskip

(e) According to Birget, Sapir, Ol'shanskii and Rips \cite{BORS}, there exists a finite group presentation  \er{pr3}
whose isoperimetric function is bounded by a polynomial $p(x)$ and for which the word problem is \textsf{NP}-complete.
It follows from definitions that if $W \overset \GG = 1 $ and $\D$ is a minimal diagram over  presentation  \er{pr3}
such that  $\ph( \p \D) \equiv W$ then $ | \D(2) |  \le  p(|W|)$. Therefore, the \BWPP , whose input is  $(W, 1^n)$,
where $n \ge  p(|W|)$, is equivalent to the  word problem, whose input is $W$.
Since the latter problem is  \textsf{NP}-complete,  it follows that the  \BWPP\  for  \er{pr3} is \textsf{NP}-hard.
By part (c), the  \BWPP\  for  \er{pr3} is in \textsf{NP}, whence the \BWPP\  for  \er{pr3} is \textsf{NP}-complete.

Note that the word problem for given word $W$ is equivalent to the  disjunction of the claims that the \PWPP\  holds for the
pairs $(W, 1^1), (W, 1^2), \dots$, $(W, 1^{p(|W|)})$. Since $p(x)$ is a polynomial, it follows that the  \PWPP\ for
this presentation  \er{pr3} is \textsf{NP}-hard.
\end{proof}

\section{Calculus of Brackets for Group Presentation \eqref{pr1}  }

As in Theorem~\ref{thm1},  consider a group presentation of the form
\begin{equation*}\tag{1.2}
    \GG_2 =  \langle \  a_1, \dots, a_m  \ \|  \   a_{i}^{k_i} =1, \ k_i \in E_i,   \ i = 1, \ldots,  m  \,   \rangle ,
\end{equation*}
where, for  every $i$, one of the following holds: $E_i = \{ 0\}$,  or, for some
integer $n_i >0$,  $E_i = \{ n_i \}$, or $E_i = n_i \mathbb N = \{ n_i, 2n_i, 3n_i,\dots  \}$.

Suppose that $W$ is a nonempty word over $\A^{\pm 1}$, $W \overset {\GG_2} = 1$
and  $\D$ is a  disk diagram
over presentation  \eqref{pr1} such that $\ph(\p |_0 \D) \equiv W$ and $\D$ has property (A). Recall that the existence of such a diagram $\D$ follows from Lemmas~\ref{vk}--\ref{vk2}(b).

\begin{lem}\label{lem1} Suppose that $\D$ is a disk diagram over  presentation  \eqref{pr1} and
 $\D$ contains no faces, i.e.,  $\D$ is a tree or, equivalently,  $\D$ is a disk diagram
 over  presentation  $F(\A) = \langle \  a_1, \dots, a_m  \ \|  \  \varnothing  \   \rangle$ of the free group $F(\A)  $ with no relations,  and assume that $\ph(\p |_0 \D) \equiv W$, where $|W| > 2$.  Then there are vertices $v_1, v_2 \in P_W$  such that   $v_1< v_2$,
 $\al(v_1) = \al(v_2)$ and if $P_W(\ff,v_1, v_2) = p_1 p_2 p_3$
 is the factorization of $P_W$ defined by $v_1, v_2$, then
\begin{equation*}
 \min(|p_2|,  |p_1|+|p_3|) \ge \tfrac 13 | \p |_0 \D | =  \tfrac 13  |W| .
\end{equation*}
\end{lem}

\begin{proof} It is easy to  verify that if $|W| \le 6$, then Lemma~\ref{lem1} is true. Hence, we may assume that $|W| > 6$.

For every pair $v'_1, v'_2$ of vertices of $P_W$ such that $v'_1 < v'_2$ and
$\al(v'_1) = \al(v'_2)$, consider the factorization  $P_W(\ff,v'_1, v'_2) = p'_1 p'_2 p'_3$ and pick such a pair that maximizes
$\min(|p'_2|,  |p'_1|+|p'_3|)$. Let $v_1, v_2$ be such a  maximal pair and
denote  $P_W(\ff, v_1, v_2) = p_1 p_2 p_3$.
Arguing on the contrary, assume that
\begin{equation}\label{p13}
\min(|p_2|,  |p_1|+|p_3|) <  \tfrac 13 |W|  .
\end{equation}
Denote $q_i = \al(p_i)$, $i =1,2,3$. Let $e_1, \dots, e_k$, $f_1, \dots, f_\ell$, $k, \ell \ge 1$,   be all edges that start at the vertex $\al(v_1) = \al(v_2)$ so that
\begin{equation*}
q_2 = e_1 s_1 e_1^{-1} \dots e_k s_k e_k^{-1} , \quad
q_3q_1 = f_1 t_1 f_1^{-1} \dots f_\ell t_\ell f_\ell^{-1} ,
\end{equation*}
where  $s_1, \dots, s_k$ and $t_1, \dots, t_\ell$ are subpaths of $q_2$ and $q_3q_1 $, resp., see Fig.~4.1.

\begin{center}
\begin{tikzpicture}[scale=.31]
\draw (0,0) node (v1) {};
\draw [](-2,0) -- (v1);
\draw [](-1,0) node (v3) {} -- (v1);
\draw (0,2)  -- (v1) -- (v1) -- (-2,-2);
\draw [](0,1)  -- (v1) -- (v1) -- (-1,-1) node (v5) {};
\draw (8,0) -- (0,0) node (v4) {};
\draw  (0,4) circle (2);
\draw  (-4,0) node (v2) {} circle (0);
\draw  (v2) circle (2);
\draw  (-2,-4) circle (2);
\draw  (10,0) circle (2);
\draw (0,0) -- (2,-6);
\draw (0,0) -- (4,-4);
\draw  (2,-8) circle (2);
\draw  (6,-4) circle (2);
\draw [-latex](2,0) -- (2,0) -- (4,0) -- (0,0) -- (2,0) -- (4,0) -- (6,0);
\draw [-latex](0,0) -- (0,1.2);
\draw [-latex](0,0) -- (-1.2,0);
\draw [-latex](0,0) node (v6) {} -- (-1.2,-1.2);
\draw [-latex](v6) -- (1,-3);
\draw [-latex](v6) -- (2,-2);
\draw [-latex](-6,-0.3);
\draw [-latex](0,6)--(0.1,6);
\draw [-latex](12,0) -- (12,-0.1);
\draw [-latex](8,-4)--(8,-4.1);
\draw [-latex](4,-8)--(4,-8.1);
\draw [-latex](-4,-4.3);
\draw  (0,2) [fill = black] circle (.06);
\draw  (0,0) [fill = black] circle (.06);
\draw  (-2,0) [fill = black] circle (.06);
\draw  (-2,-2) [fill = black] circle (.06);
\draw  (2,-6) [fill = black] circle (.06);
\draw  (4,-4) [fill = black] circle (.06);
\draw  (8,0) [fill = black] circle (.06);
\node at (-0.9,0.7) {$e_1$};
\node at (-6.7,0) {$s_1$};
\node at (0,6.7) {$s_2$};
\node at (0.8,1) {$e_2$};
\node at (6,0.7) {$e_k$};
\node at (12.7,0) {$s_k$};
\node at (2.7,-1.6) {$f_1$};
\node at (8.7,-4) {$t_1$};
\node at (1.7,-2.8) {$f_2$};
\node at (2,-2) {};
\node at (4.7,-8) {$t_2$};
\node at (-0.3,-1.2) {$f_\ell$};
\node at (-4.7,-4) {$t_\ell$};
\node at (4,2) {$\ldots$};
\node at (0.7,-5) {$\ldots$};
\node at (10,-8) {Fig.~4.1};
\end{tikzpicture}
\end{center}

First we assume that  $|p_2|\ge  |p_1|+|p_3|$. Then, in view of inequality \eqref{p13},
\begin{equation}\label{gt0}
  |p_2| > \tfrac 23 |W| , \quad  |p_1|+|p_3| <  \tfrac 13 |W| .
\end{equation}
Suppose that for some $i$ we have
\begin{equation}\label{gt12}
| e_i s_i e_i^{-1}| \ge  \tfrac 12 |W| .
 \end{equation}

Pick vertices $v'_1, v'_2 \in P_W$  for  which if $P_W(\ff,v'_1, v'_2) = p'_1 p'_2 p'_3$ then $\al(p'_2) =  e_i s_i e_i^{-1}$.
If $k >1$,  then $| e_i s_i e_i^{-1}| <   |p_2|$  and we have a contradiction to the maximality of the pair  $v_1$, $v_2$ because  $\al(v'_1) =  \al(v'_2)$.  Hence, $k =1$ and $i=1$.

Now we pick vertices $v'_1, v'_2 \in P_W$  for  which    if $P_W(\ff,v'_1, v'_2) = p'_1 p'_2 p'_3$ then    $\al(p'_2) = s_1$.
Note $|p'_2|= |s_1| = |p_2|-2 >   \tfrac 23 |W| -2 \ge \tfrac 13 |W|  $ for $|W|>  6$ and  $|p'_1|+|p'_3|= |p_1|+|p_3|+2  < \tfrac 13 |W|  +2 \le    \tfrac 23 |W|$ for $|W|> 6$. Hence, either
\begin{equation}\label{ss11}
\min(|p'_2|,  |p'_1|+|p'_3|) \ge  \tfrac 13 |W|
\end{equation}
if $|p'_2| \le |p'_1|+|p'_3|$ or
\begin{equation}\label{ss12}
\min(|p'_2|,  |p'_1|+|p'_3|) > \min(|p_2|,  |p_1|+|p_3|)
\end{equation}
if $|p'_2| > |p'_1|+|p'_3|$. In either case, we obtain a contradiction to the maximality of the pair $v_1, v_2$  because  $\al(v'_1) =  \al(v'_2)$. Thus it is shown that  the inequality \eqref{gt12} is false, hence,  for every $i =1, \dots, k$,  we have $| e_i s_i e_i^{-1}| < \tfrac 12 |W|$.

Assume $| e_i s_i e_i^{-1}| \ge  \tfrac 13 |W|$ for some $i$.
Pick vertices $v'_1, v'_2 \in P_W$  for  which if  $P_W(\ff,v'_1, v'_2) = p'_1 p'_2 p'_3$ then   $\al(p'_2) =  e_i s_i e_i^{-1}$. Since
 $| e_i s_i e_i^{-1}| <  \tfrac 12 |W|$, it follows that
$ \min(|p'_2|,  |p'_1|+|p'_3|) = |p'_2| \ge  \tfrac 13 |W|$.
A contradiction to the  maximality of the pair  $v_1, v_2$  proves that
$| e_i s_i e_i^{-1}| < \tfrac 13 |W|$ for every $i =1, \dots, k$.
According to  \eqref{gt0}, $| p_2| >  \tfrac 23 |W|$, hence,  $k \ge 3$ and, for some $i \ge 2$, we obtain
\begin{equation*}
\tfrac 13 |W| \le     | e_1 s_1 e_1^{-1} \dots e_i s_i e_i^{-1} |
\le \tfrac 23 |W|  .
\end{equation*}
This means the existence of vertices  $v'_1, v'_2 \in P_W$  for  which,  if
$$
P_W(\ff,v'_1, v'_2) = p'_1 p'_2 p'_3 ,
$$
then the paths $p'_1, p'_2, p'_3$ have the properties that
$\al(p'_2) =  e_1 s_1 e_1^{-1} \dots e_i s_i e_i^{-1}$ and
\begin{equation*}
\min(|p'_2|,  |p'_1|+|p'_3|) \ge  \tfrac 13 |W|  .
\end{equation*}
This contradiction to the maximality of the pair  $v_1, v_2$ completes the first main case
 $|p_2| \ge  |p_1|+|p_3|$.
\smallskip

Now  assume that  $|p_2| <   |p_1|+|p_3|$. In this case, we repeat the above arguments with necessary changes.
By the  inequality \eqref{p13},  $|p_2| < \tfrac 13 |W|$ and $ |p_1|+|p_3| >  \tfrac 23 |W|$.
Suppose that  for some $j$ we have
 \begin{equation}\label{gt12b}
 | f_j t_j f_j^{-1}| \ge  \tfrac 12 |W|
 \end{equation}

Pick vertices $v'_1, v'_2 \in P_W$  so that $\al(v'_1) =  \al(v'_2)$, $v'_1 < v'_2$  and, if
$$
P_W(\ff,v'_1, v'_2) = p'_1 p'_2 p'_3 ,
$$
then  either   $\al(p'_2) = f_j t_j f_j^{-1}$ in case when $f_j t_j f_j^{-1}$ is a subpath of one of  $q_1, q_3$, or  $\al(p'_3) \al(p'_1)= f_j t_j f_j^{-1}$ in  case when $f_j t_j f_j^{-1}$
has common edges with both $q_1$ and $q_3$. In either case,
\begin{equation*}
\min(|p'_2|,  |p'_1|+|p'_3|) > \min(|p_2|,  |p_1|+|p_3|)
\end{equation*}
whenever  $\ell > 1$. By the maximality of the pair  $v_1$, $v_2$, we conclude that $\ell = 1$ and $j=1$.

 In the case $\ell=j=1$, we consider two subcases: $\min(|p_1|,  |p_3|) > 0$
 and  $\min(|p_1|,  |p_3|) = 0$.

Assume  that  $\min(|p_1|,  |p_3|) > 0$. Then
  we can pick vertices  $v'_1, v'_2 \in P_W$  for  which,   if $P_W(\ff,v'_1, v'_2) = p'_1 p'_2 p'_3$,   then
  the  subpaths $p'_1, p'_2, p'_3$ of $P_W$  have the properties that $\al(p'_2) = f_1^{-1} q_2 f_1$  and  $\al(p'_3) \al(p'_1) = t_1$.
 Similarly to the above arguments that led to inequalities \eqref{ss11}--\eqref{ss12}, it follows from the inequality $|W|>  6$ that either
$$
\min(|p'_2|,  |p'_1|+|p'_3|) \ge  \tfrac 13 |W|
$$
if $|p'_1|+|p'_3| < |p'_2|$ or
$$
\min(|p'_2|,  |p'_1|+|p'_3|) > \min(|p_2|,  |p_1|+|p_3|)
$$
if $|p'_1|+|p'_3| \ge |p'_2|$. In either case, we obtain a contradiction to the maximality of the pair $v_1, v_2$.

Now assume  that  $\min(|p_1|,  |p_3|) =0$. For definiteness, let  $|p_i|=0$, $i \in \{1,3\}$. Then
  we can pick vertices  $v'_1, v'_2 \in P_W$  for  which $\al(v'_1) =  \al(v'_2)$, $v'_1 < v'_2$ and,
     if $P_W(\ff,v'_1, v'_2) = p'_1 p'_2 p'_3$,  then
   the  subpaths $p'_1, p'_2, p'_3$ of $P_W$ have the properties that $p'_2 = p_{4-i}$, $p'_i = p_{2}$, $|p'_{4-i}| =| p_{i}| =0$.   Hence,  $|p'_2| > |p'_1|+|p'_3|$ and
  $$
  \min(|p'_2|,  |p'_1|+|p'_3|) = \min(|p_2|,  |p_1|+|p_3|) .
  $$
 This means that the  subcase   $\min(|p_1|,  |p_3|) =0$  is reduced to the case $|p'_2| \ge |p'_1|+|p'_3|$ which was considered above.

The case  $\ell=j=1$ is complete and  it is shown that the inequality  \eqref{gt12b} is false, hence, for every $j =1, \dots, \ell$,  we have $| f_j t_j f_j^{-1}  | < \tfrac 12 |W|$.

Suppose that $| f_j t_j f_j^{-1}  | \ge  \tfrac 13 |W|$ for some $j$.
Pick vertices $v'_1, v'_2 \in P_W$ so that if $P_W(\ff,v'_1, v'_2)= p'_1 p'_2 p'_3$, then  either
  $\al(p'_2) = f_j t_j f_j^{-1}$ in case when  $f_j t_j f_j^{-1}$ is a subpath of one of  $q_1, q_3$, or  $\al(p'_3) \al(p'_1)= f_j t_j f_j^{-1}$ in  case when $f_j t_j f_j^{-1}$
 has common edges with both $q_1$ and $q_3$. Since $| f_j t_j f_j^{-1}  | <  \tfrac 12 |W|$,  it follows  that
 $$
 \min(|p'_2|,  |p'_1|+|p'_3|) \ge  \tfrac 13 |W|  .
 $$
A contradiction to the maximality of the pair  $v_1$, $v_2$ proves that
$| f_j t_j f_j^{-1}  | <  \tfrac 13 |W|$ for every $j=1, \dots, \ell$.
Since $| p_1|+| p_3| >  \tfrac 23 |W|$, we get $\ell \ge 3$ and, for some $j \ge 2$, we obtain
\begin{equation*}
\tfrac 13 |W| \le     | f_1 t_1 f_1^{-1} \dots f_j t_j f_j^{-1}   |
\le \tfrac 23 |W| \ .
\end{equation*}
This means the existence of vertices  $v'_1, v'_2 \in P_W$  for  which $\al(v'_1)=\al(v'_2)$ and   if $P_W(\ff,v'_1, v'_2) = p'_1 p'_2 p'_3$  then the
 subpaths $p'_1$, $p'_2$, $p'_3$  have the following properties: Either
 $\al(p'_2) =  f_1 t_1 f_1^{-1} \dots f_j t_j f_j^{-1}$ or $\al(p'_3)\al(p'_1) =  f_1 t_1 f_1^{-1} \dots f_j t_j f_j^{-1}$. The it is clear that
$$
\min(|p'_2|,  |p'_1|+|p'_3|) \ge  \tfrac 13 |W| .
$$
This contradiction to the choice of the pair  $v_1, v_2$ completes the second main case when
 $|p_2| < |p_1|+|p_3|$.
\end{proof}

\begin{lem}\label{lem2}
Suppose $\D$ is a disk diagram with property (A) over  presentation
\eqref{pr1}   and $\ph(\p |_0 \D) \equiv W$ with $|W| > 2$. Then one of the following two claims holds.

$(\mathrm{a})$ There are vertices $v_1, v_2 \in P_W$  such that $\al(v_1) = \al(v_2)$, $v_1 < v_2$, and if
$P_W(\ff, v_1, v_2) = p_1 p_2 p_3$    is the factorization of $P_W$ defined by $v_1, v_2$, then
\begin{equation*}
 \min(|p_2|,  |p_1|+|p_3|) \ge  \tfrac 16 |W|  .
\end{equation*}

$(\mathrm{b})$  There exists a face $\Pi$ in $\D$ with  $| \p \Pi | \ge 2$  and there are vertices $v_1, v_2 \in P_W$  such that $\al(v_1), \al(v_2) \in \p \Pi$, $v_1 < v_2$, and if $P_W(\ff, v_1, v_2) = p_1 p_2 p_3$  then
\begin{equation} \label{in5}
 \min(|p_2|,  |p_1|+|p_3|) \ge
 \tfrac 16 |W|  .
\end{equation}
In addition, if $(\p \Pi )^{-1} = e_1  \dots e_{|\p \Pi |}$, where $e_i \in \vec \D(1)$,
and  $\p \D =   e_1 h_1 \dots e_{|\p \Pi |} h_{|\p \Pi |}$, where $h_i$ is a closed subpath of $\p \D$,  then, for every $i$, $h_i$ is a subpath of either $\al( p_2)$ or $\al( p_3) \al( p_1)$ and $|h_i | \le \tfrac 56 |W|$, see Fig.~4.2.
\end{lem}

\begin{center}
\begin{tikzpicture}[scale=.48]
\draw  (-2.1,-2.15) [fill = black] circle (.05);
\draw  (0,-3) [fill = black] circle (.05);
\draw  (-3,0) [fill = black] circle (.05);
\draw  (-2.1,2.1) [fill = black] circle (.05);
\draw  (0,2.98) [fill = black] circle (.056);
\draw  (2.1,-2.1) [fill = black] circle (.05);
\draw  (3,0) [fill = black] circle (.05);
\draw  (1.7,2.45) [fill = black] circle (.05);
\draw  (2.64,1.45) [fill = black] circle (.05);
\draw [-latex](.2,-4.9) --(.0,-4.95);
\draw [-latex](-2.8,1.1) --(-2.77,1.2);
\draw [-latex](-2.7,-1.3) --(-2.72,-1.228);
\draw [-latex](-5,-.1) --(-5,0.1);
\draw [-latex](-4.3,-3.2) --(-4.3,-3.);
\draw  (0,0) circle (3);
\draw  plot[smooth, tension=.7] coordinates {(2.64,1.45)  (3.2, 2.6)   (4, 2.6)  (4, 1.7) (2.64,1.45)};
\draw  plot[smooth, tension=.7] coordinates {(1.7,2.45)  (1.6,3.2)   (2.5,3.5)  (2.6,2.5) (1.7,2.45)};
\draw  plot[smooth, tension=.7] coordinates {(0,-3) (-1,-4.5) (0.4,-4.8) (0,-3)};
\draw  plot[smooth, tension=.7] coordinates {(0,-4)};
\draw  plot[smooth, tension=.7] coordinates {(-2.1,-2.1) (-2,-3) (-4,-4) (-4,-2) (-2.1,-2.1) };
\draw [-latex](-1.1,-2.79) -- (-1.2, -2.74);
\draw  plot[smooth, tension=.7] coordinates {(-3,0) (-4,-1) (-5,0) (-4,1) (-3,0)};
\draw  plot[smooth, tension=.7] coordinates {(-2.1,2.1) (-3.3,2) (-3.6,3) (-2.5,3) (-2.1,2.1)};
\draw  plot[smooth, tension=.7] coordinates {(-2,-2)};
\draw  plot[smooth cycle, tension=.5] coordinates {(0,3) (-1,4) (0,5) (1,4) (0,3)};
\draw  plot[smooth, tension=.7] coordinates {(3,0) (4,1) (5,0) (4,-1) (3,0)};
\draw  plot[smooth, tension=.7] coordinates {(2.1,-2.1) (4,-2) (4,-4) (2,-4) (2.1,-2.1)};
\draw  [-latex] (1.2,-2.73) -- (1.1,-2.78);
\draw  [-latex] (3.2,-4.27) -- (3.1,-4.28);
\draw  [-latex] (-3.2,3.14) -- (-3.1,3.14);
\node at (1.,-3.5) {$e_{| \partial \Pi  |}$};
\node at (0.2,-5.55) {$h_{| \partial \Pi  |}$};
\node at (-1.3,-3.4) {$e_{1}$};
\node at (-5,-3.1) {$h_{1}$};
\node at (-3.3,-1.4) {$e_{2}$};
\node at (-5.8,0) {$h_{2}$};
\node at (-3.4,1.4) {$e_{3}$};
\node at (0,0) {$\Pi$};
\node at (-5,-5.4) {Fig.~4.2};
\node at (3.4,-4.9) {$h_{| \partial \Pi  |-1}$};
\node at (-3.2,3.7) {$h_3$};
\end{tikzpicture}
\end{center}

\begin{proof} Since $\D$ has property (A), it follows that if $e \in \p \Pi $, where $\Pi$ is a face of $\D$,  then $e^{-1} \in \p \D$ and if $e\in \p \D$,
then $e^{-1}  \in \p \Pi' \cup  \p \D $, where $\Pi'$ is a face of $\D$.

Consider a planar graph $\Gamma_\D$ constructed from $\D$ as follows.
For every face $\Pi$ of $\D$, we pick a vertex $o_\Pi$ in the interior of $\Pi$. The vertex set of $\Gamma_\D$ is
\begin{equation*}
V(\Gamma_\D) := \D(0) \cup \{   o_\Pi \mid \Pi \in \D(2) \} \ ,
\end{equation*}
where $\D(i)$ is the set of $i$-cells of $\D$, $i =0,1,2$.
For every face $\Pi$ of $\D$, we delete   nonoriented edges of $\p \Pi$ and
draw $|\p \Pi |$ nonoriented edges that connect $o_\Pi$  to all vertices of $\p \Pi$, see Fig.~4.3.

\begin{center}
\usetikzlibrary{arrows}
\begin{tikzpicture}[scale=.37]
\draw  (-4,0) [fill = black] circle (.06);
\draw  (-6,2) [fill = black] circle (.06);
\draw  (-4,4) [fill = black] circle (.06);
\draw  (-2,2) [fill = black] circle (.06);
\draw  (4,0) [fill = black] circle (.06);
\draw  (2,2) [fill = black] circle (.06);
\draw  (4,4) [fill = black] circle (.06);
\draw  (6,2) [fill = black] circle (.06);
\draw  (4,2) [fill = black] circle (.06);
\draw  (-4,2) circle (2);
\tikzstyle{myedgestyle} = [-open triangle 45]
\draw[-latex] (-1,2) -- (1,2);
\draw (4,0) -- (4,4);
\draw (2,2) -- (6,2);
\node at (4.8,1.2) {$o_\Pi$};
\node at (2.7,0.7) {$E_\Pi$};
\node at (-4,2) {$\Pi$};
\node at (.3,-.8) {Fig.~4.3};
\end{tikzpicture}
\end{center}

\noindent
We draw these edges so that their interiors are disjoint and are contained in the interior of $\Pi$.  Let $E_\Pi$ denote the set of these $|\p \Pi |$ edges. The set $E(\Gamma_\D)$  of nonoriented edges of $\Gamma_\D$   is
$\D(1)$, without edges of faces of $\D$, combined with $\cup_{\Pi \in \D(2) } E_\Pi$, hence,
\begin{equation*}
E(\Gamma_\D) := \cup_{\Pi \in \D(2) } E_\Pi  \cup \left( \D(1) \setminus  \{ e \mid e \in \D(1),  e \in \p \Pi, \Pi  \in \D(2) \}    \right)   \ .
\end{equation*}
It follows from definitions that $| V(\Gamma_\D)| = |\D(0)| + |\D(2)|$,
$| E(\Gamma_\D)| = | \D(1)|$ and that $\Gamma_\D$ is a tree. Assigning labels to
oriented edges of $\Gamma_\D$, by using letters from $\A^{\pm 1}$, we can turn   $\Gamma_\D$  into a disk diagram over presentation     $\langle \  a_1, \dots, a_m  \ \|  \  \varnothing  \   \rangle$
of the free group $F(\A)$.

Denote   $\ph( \p |_{o'} \Gamma_\D    ) \equiv W'$ for some vertex $o' \in \D(0)$, where $| W'| =  |\vec \D (1)|$, and let $\al'(P_{W'}) =  \p |_{o'} \Gamma_\D    =  \p |_{0} \Gamma_\D $.
 Since  $|W' | \ge |W| > 2$,  Lemma~\ref{lem1} applies to   $\Gamma_\D$ and yields the existence of vertices  $u_1, u_2 \in P_{W'}$ such that $\al'(u_1) = \al'(u_2)$ in $\p \Gamma_\D$, $u_1 < u_2$,  and if
$P_{W'}(\ff, u_1, u_2) = r_1 r_2 r_3$, then
\begin{equation} \label{in5aa}
 \min(|r_2|,  |r_1|+|r_3|) \ge  \tfrac 13 |W'|  .
\end{equation}

First suppose that $\al'(u_1)$ is a vertex of $\D$. It follows from the definition of the tree $\Gamma_\D$ that there is a factorization
$P_{W} = p_1 p_2 p_3$ of the path $P_W$  such that the vertex
$\al((p_2)_-) = \al((p_2)_+)$ is $\al'(u_1) \in \D(0)$ and $|p_i| \le |r_i| \le 2|p_i|$, $i=1,2,3$.
Indeed, to get from $\p \D$ to $\p \Gamma_\D$ we replace every edge $e \in \p \Pi$, $\Pi \in \D(2)$, by
 two edges of $E_\Pi$, see Fig.~4.3.  Hence, if $r$ is a subpath of  $\p \Gamma_\D$ and $p$ is a corresponding to $r$ subpath of  $\p \D$    with $r_- = p_- \in \D(0)$, $r_+ = p_+ \in \D(0)$, then $|p | \le |r| \le 2 |p | $. Then it follows from \eqref{in5aa} that
  \begin{equation*}
 \min(|p_2|,  |p_1|+|p_3|) \ge \tfrac 12 \min(|r_2|,  |r_1|+|r_3|) \ge   \tfrac 16 |W'| \ge \tfrac 16 |W|  , \end{equation*}
 as required.

Now assume that  $\al'(u_1) = \al'(u_2) =o_\Pi$ for some face $\Pi \in \D(2)$.
 Let $e_1, \dots, e_k$, $f_1, \dots, f_\ell$, $k, \ell \ge 0$,   be all oriented edges of $\Gamma_\D$
  that start at the vertex $\al'(u_1) = \al'(u_2) = o_\Pi$ so that
\begin{equation*}
\al'(r_2) = e_1 s_1 e_1^{-1} \dots e_k s_k e_k^{-1} , \quad
\al'(r_3) \al'(r_1) = f_1 t_1 f_1^{-1} \dots f_\ell t_\ell f_\ell^{-1} ,
\end{equation*}
where  $s_1, \dots, s_k$ and $t_1, \dots, t_\ell$ are closed subpaths of $\al'(r_2) $
and $\al'(r_3) \al'(r_1)$, resp.
Since  $\al'(u_1) = \al'(u_2) =o_\Pi$, it follows that $k+\ell = |\p \Pi |$.
Since $\min(|r_2|,  |r_1|+|r_3|) \ge   \tfrac 13 |W'| >0$ is an integer, we also have that
$k,\ell \ge 1$ and $|\p \Pi | >1$.
If $|r_3| >0$, we consider vertices $u_1' := u_1+1$, $u_2' := u_2+1$. On the other hand, if
 $|r_3| =0$, then $|r_1| >0$ and we consider vertices  $u_1' := u_1-1$, $u_2' := u_2-1$. In either case, denote $P_{W'}(u_1', u_2' ) =  r'_1 r'_2 r'_3$. Then $|r'_2| = |r_2|$ and $|r'_1| +  |r'_3|= |r_1| +  |r_3|$, hence, by virtue of   \er{in5aa},
  $$
 \min(|r'_2|,  |r'_1|+|r'_3|) \ge   \tfrac 13 |W'| .
 $$
Note that the vertices
$\al'(  (r'_2)_- )$, $\al'(  (r'_2)_+ )$ belong to the boundary $\p \Pi$. Hence, as above, there is also a factorization
$P_{W}(\ff, v_1, v_2) = p_1 p_2 p_3$  such that $\al(v_1) =  \al'(u'_1)$, $\al(v_2) =  \al'(u'_2)$
and $|p_i| \le |r'_i| \le 2|p_i|$, $i=1,2,3$.  Therefore,
\begin{equation*}
 \min(|p_2|,  |p_1|+|p_3|) \ge \tfrac 12 \min(|r'_2|,  |r'_1|+|r'_3|) \ge   \tfrac 16 |W'| \ge \tfrac 16 |W| , \end{equation*}
 as required.

It remains to observe that it follows from the definition of vertices $v_1, v_2$ that every $h_i$ is a subpath of one of the paths $\al( p_2)$, $\al( p_3) \al( p_1)$. In particular, $|h_i| \le  \tfrac 56 |W|$, as desired.
\end{proof}

In the definitions below, we assume that $\D$ is a disk diagram over presentation \er{pr1} such that $\D$ has property (A), $\ph(\p |_0 \D ) \equiv W$, $|W| >0$,  and that the pair $(W, \D)$ is fixed.
\smallskip

A 6-tuple
$$
 b = (b(1), b(2), b(3), b(4), b(5), b(6))
$$
of integers   $b(1), b(2), b(3), b(4), b(5), b(6)$
is called a {\em bracket } for the pair $(W, \D)$ if $b(1), b(2)$ satisfy the inequalities
$0 \le b(1) \le b(2) \le |W|$ and, in the notation $P_W(\ff, b(1), b(2)   ) = p_1 p_2 p_3$,
one of the following two properties (B1)--(B2) holds true.
\begin{enumerate}
\item[(B1)]  $b(3)= b(4)=b(5)=0$,  $\al(b(1))=\al(b(2))$,   and the disk subdiagram $\D_b$ of $\D$, defined by $\p |_{b(1)} \D_b   = \al(p_2)$, contains  $b(6)$ faces, see Fig.~4.4(B1).

\item[(B2)]  $b(3) >0$ and    $\al(b(1)), \al(b(2)) \in \p \Pi$,  where $\Pi$ is a face of $\D$ such that  $\ph(\p \Pi) \equiv a_{b(3)}^{\e b(4)}$, $b(4) >0$, $\e = \pm 1$, and if  $\D_b $ is the disk subdiagram  of $\D$, defined by  $\p |_{b(1)} \D_b   = \al(p_2) u$, where $u$ is the subpath of $\p \Pi$ with
 $u_-=  \al(b(2))$  and $u_+=  \al(b(1))$, then $\ph(u) \equiv a_{b(3)}^{-b(5)}$ and $|\D_b(2)| = b(6)$, see Fig.~4.4(B2).
 \end{enumerate}

\begin{center}
\begin{tikzpicture}[scale=.5] 
\draw  (5,4) [fill = black] circle (.074);
\draw  (5,-1.5) [fill = black] circle (.074);
\draw  plot[smooth, tension=1.2] coordinates {(5,4) (2,6.5) (5,9) (8,6.5) (5,4)};
\draw  plot[smooth, tension=1.5] coordinates { (5,4) (2,1.5) (5,-1.5) (8,1.5) (5,4)};
\draw[-latex](2,1.5) -- (2,1.55);
\draw[-latex] (8,1.5) -- (8,1.45) ;
\draw[-latex](5,9) -- (5.1,9);
\node at (5,10) {$\alpha(p_2)= a(b)$};
\node at (0.96,4.1) {$\alpha(b(1))= \alpha(b(2))$};
\node at (9.2,1.5) {$\alpha(p_3)$};
\node at (0.8,1.5) {$\alpha(p_1)$};
\node at (5,-0.8) {$\alpha(0)$};
\node at (5,6.2) {$\Delta_b$};
\node at (5,-3) {Fig.~4.4(B1)};

\draw  (16.9,6) [fill = black] circle (.074);
\draw  (21,6) [fill = black] circle (.074);
\draw  (19,-2) [fill = black] circle (.074);
\draw[-latex](18.9,10.3) -- (19.1,10.3);

\draw[-latex](19.1,6.2) -- (18.9,6.2);
\draw[-latex](17,6) -- (17.1,6);
\draw[-latex](16.2,-1.1) -- (16.1,-1);
\draw[-latex](21.7,-1.13) -- (21.6,-1.21);
\draw[-latex](22.6,3.9) -- (22.6,4.1);
\draw[-latex](14.68,4.7) -- (14.8,4.87);
\draw  plot[smooth, tension=.9] coordinates {(17,6) (21,6) (23,4) (21,2) (17,2) (15,4) (17,6)};
\draw  plot[smooth, tension=.7] coordinates {(17,2) (15,0) (19,-2) (23,0) (21,2)};
\draw  plot[smooth, tension=.7] coordinates {(17,6) (15,8) (17,10) (21,10) (23,8) (21,6)};
\draw  plot[smooth, tension=.7] coordinates {(22.4,3.3) (22.6,4) (22.4,4.7)};
\node at (21.8,4) {$\partial \Pi$};
\node at (19,4) {$\Pi$};
\node at (15.2,-1.4) {$\alpha(p_1)$};
\node at (22.3,-1.6) {$\alpha(p_3)$};
\node at (19,-1.35) {$\alpha(0)$};
\node at (19,5.5) {$u$};
\node at (19,7.9) {$\Delta_b$};
\node at (18.9,9.2) {$\alpha(p_2)=a(b)$};
\node at (19,-3) {Fig.~4.4(B2)};
\draw  plot[smooth, tension=.7] coordinates {(14.8,2.8) (14.4,4) (14.9,5)};
\node at (12.2,5.7) {$ \alpha(p_1)\alpha(p_2)\alpha(p_3) = \partial \Delta$};
\end{tikzpicture}
\end{center}

A bracket $b$ is said to have {\em type B1} or {\em type B2} if the property (B1) or (B2), resp., holds for $b$. Note that
the equality $b(4) = 0$ in property (B1) and the inequality $b(4) >0$ in property (B2) imply that the type of a bracket is unique.

The boundary subpath $\al(p_2)$ of the disk subdiagram $\D_b$ associated with a bracket $b$ is denoted $a(b)$ and is called the {\em arc} of $b$, see Figs.~4.4(B1)--4.4(B2).

For example, $b_v = (v, v, 0, 0, 0,0)$ is a bracket of type B1 for every vertex $v$ of $P_W$, called a {\em  starting} bracket at $v$.    Note that $a(b_v) =  \al( v )= \al( b_v(1) )$.

The {\em final}  bracket for $(W, \D)$ is
$b_F = (0, |W|, 0, 0, 0, | \D(2) | )$, it has type B1  and   $a(b_F) =  \p |_{0} \D$.

Let $B$ be a set of brackets for the pair $(W, \D)$, perhaps, $B$ is empty, $B = \varnothing$.
We say that $B$ is a {\em bracket system} if,  for every pair $b, c \in B$ of distinct brackets,  either $b(2) \le c(1)$  or $c(2) \le b(1)$. In particular, the arcs of distinct brackets in $B$ have no edges in common.  A \BS\   consisting of a single final bracket  is called a {\em final \BS }.

Now we describe four kinds of  operations over brackets and over \BS s: additions, extensions, turns, and mergers. Let $B$  be a \BS .
\medskip

{\em Additions.}

Suppose $b$ is a starting bracket,  $b \not\in B$, and $B \cup \{ b\}$
is a \BS    . Then we may add $b$ to $B$ thus making an {addition} operation over $B$.
\medskip

{\em Extensions.}

Suppose $b \in B$, $b = (b(1), b(2), b(3), b(4), b(5), b(6))$,   and $e_1 a(b) e_2$ is a subpath of $\p |_{0} \D$, where $a(b)$ is the arc of $b$ and $e_1, e_2$ are edges one of which could be missing.

Assume that $b$ is of type B2, in particular, $b(3), b(4) >0$. Using the notation of the condition (B2),
suppose  $e_1^{-1}$ is an edge of $\p \Pi$, and $\ph(e_1) = a_{b(3)}^{\e_1}$, where $\e_1 = \pm 1$.

If $| b(5) | \le b(4)-2$ and $\e_1 b(5)  \ge 0$, then we consider a bracket $b'$ such that
\begin{align*}
&  b'(1) = b(1)-1, \  b'(2) = b(2) , \  b'(3) = b(3) ,  \\
& b'(4) = b(4) ,  \ b'(5) = b(5)+\e_1 ,  b'(6) =   b(6) .
\end{align*}
Note that $a(b') = e_1 a(b)$.
We say that $b'$ is obtained from $b$ by an extension of type 1 (on the left).
If $(B \setminus \{ b \}) \cup \{ b' \}$ is a \BS , then
replacement of $b \in B$ with $b'$ in $B$ is called an {\em extension } operation over $B$ of type~1.

On the other hand, if $| b(5) | = b(4)-1$   and $\e_1 b(5)  \ge 0$, then we  consider a  bracket $b'$ such that
$$
b'(1) = b(1)-1 , \ b'(2) = b(2) ,   b'(3) = b'(4) = b'(5) = 0 ,  b'(6) = b(6)+1 .
$$
In this case, we say that $b'$ is obtained from $b$ by an extension of type 2 (on the left).
Note that $a(b') = e_1 a(b)$ and $b'$ has type B1. If $(B \setminus \{ b \}) \cup \{ b' \}$ is a \BS , then
replacement of $b \in B$ with $b'$ in $B$ is called an {\em extension } operation over $B$ of type~2.

Analogously, assume that  $b$ has type B2,
$e_2^{-1}$ is an edge of $\p \Pi$, and $\ph(e_2) = a_{b(3)}^{\e_2}$, where $\e_2 = \pm 1$.

If $| b(5) | \le b(4)-2$ and $\e_2 b(5)  \ge 0$, then we consider a  bracket $b'$ such that
\begin{align*}
& b'(1) = b(1) , \ b'(2) = b(2)+1,  \ b'(3) = b(3) , \\
& b'(4) = b(4) , \ b'(5) = b(5)+\e_2 , \ b'(6) = b(6) .
\end{align*}
Note that $a(b') =  a(b)e_2$ and $b'$ has type B2. We say that $b'$ is obtained from $b$ by an extension of type 1 (on the right).
If $(B \setminus \{ b \}) \cup \{ b' \}$ is a \BS , then
replacement of $b \in B$ with $b'$ in $B$ is called an {\em extension } operation over $B$ of type~1.

On the other hand, if $| b(5) | = b(4)-1$, then we may consider a bracket $b'$ such that
$$
b'(1) = b(1) , \ b'(2) = b(2)+1 , \  b'(3) =b'(4) = b'(5) =  0 , \ b'(6) = b(6)+1 .
$$
Note that $a(b') = a(b) e_2$ and $b'$ has type B1.
We say that $b'$ is obtained from $b$ by an extension of type 2 (on the right).
If $(B \setminus \{ b \}) \cup \{ b' \}$ is a \BS , then
replacement of $b \in B$ with $b'$ in $B$ is called an {\em extension } operation over $B$ of type~2.

Assume that $b \in B$ is a bracket of type B1, $e_1 a(b) e_2$  is a subpath of  $\p |_{0} \D$,
both $e_1, e_2$ exist, and $e_1 = e_2^{-1}$.
Consider a  bracket $b'$ of type B1 such that
$$
b'(1) = b(1)-1 , \ b'(2) = b(2)+1 , \ b'(3) = b'(4) =b'(5) =0 , b'(6) = b(6) .
$$
Note that $a(b') = e_1 a(b)e_2$.
We say that $b'$ is obtained from $b$ by an extension of type 3.
If $(B \setminus \{ b \}) \cup \{ b' \}$ is a \BS , then
replacement of $b \in B$ with $b'$ in $B$ is called an {\em extension } operation over $B$ of type~3.
\medskip

{\em Turns.}

Let $b \in B$ be a bracket of type  B1.  Then $b(3) = b(4) =b(5) =0$.
Suppose $\Pi$  is  a  face in $\D$ such that $\Pi$ is not in the disk subdiagram $\D_b$,
associated with $b$ and bounded by the arc $a(b)$ of $b$, and $\al(b(1)) \in \p \Pi$. If $\ph(\p \Pi) = a_{j}^{\e n_{\Pi}}$, where $\e = \pm 1$, $n_{\Pi} \in E_j$, then we consider a bracket $b'$ with $b'(i) = b(i)$ for $i=1,2,5,6$, and  $b'(3) = j$,  $b'(4) = n_\Pi$. Note that $b'$ has type B2.
We say that $b'$ is obtained from $b$ by a {\em turn} operation.
Replacement of $b \in B$ with $b'$ in $B$ is also called a { turn } operation over $B$.
Note that $(B \setminus \{ b \}) \cup \{ b' \}$ is automatically a \BS\ (because so is $B$).
\medskip

{\em Mergers.}

Now suppose that $b, c \in B$ are distinct brackets such that  $b(2) = c(1)$ and one of $b(3), c(3)$ is 0. Then one of
 $b, c$ is of type B1  and the other has type  B1 or B2.
Consider a  bracket $b'$ such that $b'(1) = b(1)$, $b'(2) = c(2)$, and $b'(i) =  b(i) + c(i)$ for $i =3,4,5,6$. Note that $a(b') = a(b_1)a(b_2)$ and $b'$ is of type  B1 if both  $b_i, b_j$ have type B1 or $b'$ is of type B2 if one of  $b, c$ has type B2.
We say that $b'$ is obtained from $b, c$ by a merger operation.
Taking both $b, c$ out of $B$ and putting  $b'$ in $B$ is a {\em merger } operation over $B$.
Note that $(B \setminus \{ b, c \}) \cup \{ b' \}$ is automatically a \BS .
\medskip

We will say that additions, extensions, turns and mergers, as defined above, are {\em  \EO s} over brackets and \BS  s for the pair $(W, \D)$.

Assume that one \BS\   $B_\ell$ is obtained from another \BS\    $B_0$
by a finite sequence $\Omega$ of \EO s  and $B_0, B_1, \dots, B_\ell$ is the corresponding to $\Omega$ sequence of \BS  s. Such a sequence $B_0, B_1,  \dots, B_\ell$ of \BS  s will be called {\em operational}.

We will say that the sequence     $B_0,  B_1,  \dots$, $B_\ell$ has {\em size bounded by}
$ (k_1, k_2)$  if $\ell \le k_1$ and, for every $i$, the number of brackets in $B_i$ is at most $k_2$.
Whenever it is not ambiguous, we will also say that $\Omega$ has size bounded by $(k_1, k_2)$ if so does the corresponding to $\Omega$ sequence  $B_0,  B_1,  \dots, B_\ell$ of \BS s.

\begin{lem}\label{lem3}  There exists a sequence of \EO s  that converts the empty
\BS\   for $(W, \D)$ into the final \BS\    and has size  bounded by $(8|W|, |W|+1)$.
\end{lem}

\begin{proof} For every $k$ with $0 \le k \le |W|$, consider a starting bracket $(k,k,0,0,0,0)$ for
 $(W, \D)$. Making $|W|+1$ additions, we get a \BS\
 $$
 B_{W} = \{ (k,k,0,0,0,0) \mid  0 \le k \le |W| \}
 $$
of  $|W|+1$ starting brackets. Now, looking at the disk diagram $\D$, we can easily find a sequence of extensions, turns and mergers that  converts  $B_{W}$ into the final \BS , denoted $B_{\ell}$. Note that the inequality $b(4) \le |W|$ for every bracket
$b$ of intermediate systems follows from definitions and Lemma~\ref{vk2}.  To estimate the total number  of
\EO s, we note that the number of additions is $|W|+1$.  The  number  of extensions is at most $|W|$ because every extension applied to a \BS\ $B$ increases the number
$$
\eta(B) := \sum_{b \in B} (b(2)-b(1) )
$$
by 1 or 2 and $\eta( B_{W}) = 0$, whereas $\eta( B_{\ell}) = |W|$.
The  number  of mergers is $ |W|$ because the number of brackets $|B|$ in a \BS\  $B$ decreases by 1 if $B$ is obtained by a merger and
$|  B_{W}  | =  |W|+1$, $|B_{\ell} | =  1$. The number of turns does not exceed the total number  of additions, extensions, and mergers, because a turn operation is applied to a bracket of type B1 and results in a bracket of type B2 to which a turn operation may not be applied, whence a turn operation succeeds an addition, or an extension, or a merger. Therefore, the number of turns is at most $3|W|+1$ and so  $\ell \le 6|W|+2 \le 8|W|$.
\end{proof}

\begin{lem}\label{lem4}  Suppose there is  a sequence $\Omega$  of \EO s  that converts the empty
\BS\  $E$  for $(W, \D)$ into the final \BS\  $F$  and has size  bounded by $(k_1, k_2)$. Then there is also a sequence  of \EO s  that  transforms $E$ into $F$ and has size bounded by $(11|W|, k_2)$.
\end{lem}
\begin{proof} Assume that the sequence $\Omega$ has an addition operation which introduces a starting bracket $c = (k,k,0,0,0,0)$ with $0 \le k \le |W|$. Since the final \BS\ contains no starting brackets, $c$ must disappear and an \EO\ is applied to $c$. If a merger is applied to $c$ and another bracket $b$ and the merger yields $\wht c$, then $\wht c = b$. This means that the addition of $c$ and the merger could be skipped without affecting the sequence otherwise. Note that the size of the new sequence $\Omega'$ is  bounded by $(k_1-2, k_2)$. Therefore, we may assume that no merger is applied to a starting bracket in $\Omega$.

Now suppose that a turn is applied to $c$, yields $c'$ and then a merger is applied to $c'$,  $b$ and
the merger produces $\wht c$. Note that  $c'$ has type B2 and $b$ has type B1. Then it is possible to apply a turn to $b$ and get $b'$ with $b' = \wht c$. Hence, we can skip the addition of $c$, the turn of $c$,  the merger and use, in their place, a turn of $b$. Clearly, the size of the new sequence $\Omega'$ is  bounded by $(k_1-2, k_2)$.

Thus, by induction on $k_1$,  we may assume that, for every starting bracket $c$ which is added by $\Omega$, an extension is applied to $c$ or an  extension is applied to $c'$ and $c'$ is obtained from $c$ by a turn.

Now we will show that, for every starting bracket $c$
which is added by $\Omega$, there are at most 2 operations of additions of $c$ in $\Omega$. Arguing on the contrary, assume that  $c_1, c_2, c_3$ are the brackets equal to $c$ whose additions are done in $\Omega$. By the above remark,
for every $i=1,2,3$, either an extension is applied to $c_i$, resulting in a bracket $\wht c_i$, or a turn
is applied to $c_i$, resulting in  $c'_i$, and then an extension is applied to $c'_i$, resulting in  a bracket $\wht c_i$.

Let  $c_1, c_2, c_3$  be listed in the order in which the brackets $\wht c_1, \wht c_2, \wht c_3$ are created by $\Omega$. Note that if $\wht c_1$ is obtained from $c_1$ by an extension (with no turn), then
the extension has type 3 and
$$
\wht c_1(1) = c_1(1)-1 , \ \wht c_1(2) = c_1(2)+1 .
$$
This means that brackets
$\wht c_2$, $\wht c_3$ could not be created by an extension after  $\wht c_1$ appears, as  $d(2) \le d'(1) $
or $d(1) \ge d'(2) $ for distinct brackets $d, d' \in B$ of any \BS\ $B$. This contradiction proves that $\wht c_1$ is obtained from $c_1'$ by an extension. Then $c_1'$ has type B2, the extension has type 1 or 2 (and is either on the left or on the  right). Similarly to the forgoing argument, we can see that if
$\wht c_1$ is obtained by an  extension on the left/right, then $\wht c_2$ must be obtained by an  extension on the right/left, resp., and that $\wht c_3$ cannot be obtained by any  extension. This contradiction proves that it is not possible to have in $\Omega$ more than two additions of any starting bracket $c$.
Thus, the number of additions in   $\Omega$  is at most $2|W|+2$.

As in the proof of Lemma~\ref{lem3}, the number of extensions is at most $|W|$ and the number of mergers is    $\le  2|W|+1$. Hence, the number of turns is $\le 5|W|+3$ and the total number of \EO s is  at most  $10|W|+6 \le 11|W|$ as desired.
\end{proof}

\begin{lem}\label{lem5} Let there be a  sequence  $\Omega$  of \EO s    that transforms  the empty
\BS\   for the pair $(W, \D)$ into the final \BS\  and has size  bounded by $(k_1, k_2)$ and let $c$ be a starting bracket for $(W, \D)$. Then there is also a  sequence   of \EO s    that  converts  the
\BS\  $\{ c \}$   into the final \BS\  and has size  bounded by $(k_1+1, k_2+1)$.
\end{lem}
\begin{proof} Let
\begin{equation} \label{bso}
B_0,  B_1,  \dots, B_\ell
\end{equation}
be the corresponding to $\Omega$ sequence of \BS s, where $B_0$ is empty and $B_\ell$ is final.

First  suppose that $c(1) = 0$ or $c(1) = |W|$.

Assume that no addition of $c$ is used in  $\Omega$. Setting
$B'_i := B_i \cup \{ c \}$, $i=0,1, \dots, \ell$, we obtain an operational sequence of \BS s $B'_0, \dots,  B'_\ell$ that starts with $\{ c \}$ and ends in the \BS\  $ B'_\ell = \{ c, d_F \}$, where $d_F$ is the final bracket for $(W, \D)$. Note that $B_{i+1}'$ can be  obtained from $B_{i}'$  by the same elementary operation that was used to get  $B_{i+1}$  from $B_{i}$. A merger operation applied to $B'_\ell$ yields the final \BS\ $ B'_{\ell+1} = \{ d_F \}$ and $B'_0, \dots,  B'_\ell,  B'_{\ell+1}$ is a desired sequence of \BS s of size bounded by $(k_1+1, k_2+1)$.

Now suppose that an  addition of $c$ is used in  $\Omega$,   $B_{i^*+1} = B_{i^*} \cup \{ c \}$ is obtained from $B_{i'}$ by addition of $c$, and $i^*$ is minimal with this property. Define $B'_j := B_j \cup \{ c \}$ for $j=0, \dots, i^*$ and
 $B'_{i^*+1} := B_{i^*+2}, \dots, B'_{\ell -1} := B_{\ell}$.  Then  $B'_0, \dots,  B'_{\ell -1}$  is a desired operational sequence of
 \BS s that starts with $\{ c \}$,  ends in the final \BS , and has size bounded by $(k_1-1, k_2+1)$.

We may now assume that $0 < c(1) < |W|$.
 Let $B_{i^*}$ be the first \BS\ of the sequence \eqref{bso} such that $B_{i^*} \cup \{ c \}$ is not a \BS . The existence of such $B_{i^*}$ follows from the facts   that    $B_0 \cup \{ c \}$ is a \BS\ and     $B_\ell \cup \{ c \}$ is not.
Since
$B_0 \cup \{ c \}$ is a \BS , it follows that $i^* \ge 1$ and $B_{i^*-1} \cup \{ c \}$ is a \BS .  Since $B_{i^*-1}  \cup \{ c \}$ is a  \BS\ and  $B_{i^*} \cup \{ c \}$ is not, there is a bracket $b \in B_{i^*}$ such that
$b(1) < c(1) < b(2)$ and $b$ is obtained from a bracket  $d_1 \in B_{i^*-1}$ by an extension or
  $b$ is obtained from    brackets $d_1, d_2 \in B_{i^*-1}$ by a merger. In either case, it follows from
  definitions of \EO s that $d_i(j) = c(1)$ for some $i, j \in \{ 1,2\}$. Hence, we can use a merger applied to  $d_i(j)$ and $c$ which would result in $d_i(j)$, i.e., in elimination of $c$ from $B_{i^*-1} \cup \{ c \}$ and in getting thereby $B_{i^*-1}$ from $B_{i^*-1} \cup \{ c \}$.  Now we can see that the original sequence of  \EO s, together with the merger $B_{i^*-1} \cup \{ c \} \to  B_{i^*-1}$ can be used to produce the following operational sequence of \BS s
$$
B_{0} \cup \{ c \}, \dots,  B_{i^*-1} \cup \{ c \} ,   B_{i^*-1},  B_{i^*}, \dots, B_\ell .
$$
Clearly, the size of this new sequence is bounded by  $(k_1+1, k_2+1)$, as required.
\end{proof}

\begin{lem}\label{lem6} There exists a sequence of \EO s  that converts the empty
\BS\   for the pair $(W, \D)$ into the final \BS\    and has size bounded by
\begin{equation*}
  ( 11|W| , \  C(\log |W| +1) )  ,
\end{equation*}
where $C = (\log \tfrac 65)^{-1}$.
\end{lem}

\begin{proof}  First suppose that  $|W| \le 2$. Then $\D$ consists of a single edge, or of a single face $\Pi$ with $|\p \Pi | \le 2$, or of two faces $\Pi_1, \Pi_2$ with $|\p \Pi_1 | =  |\p \Pi_2 | = 1$. In each of these three cases, we can easily find a sequence of \EO s that transforms the empty \BS\   for $(W, \D)$ into the final \BS\   by using a single addition, at most two turns, and at most two extensions.  Hence, the size of the sequence is bounded by
$$
(5,1) \le  ( 11|W| , \  C(\log |W| +1) $$
as $C > 1$.

Assuming  $|W| > 2$,  we proceed by induction on the length $|W|$. By Lemma~\ref{lem2} applied to $(W, \D)$, we obtain either the existence of vertices $v_1, v_2 \in P_W$ such that  $v_1 < v_2$, $\al(v_1) = \al(v_2)$ and, if $P_W(\ff,v_1, v_2 ) = p_1 p_2 p_3$,  then
\begin{equation}\label{in10}
 \min(|p_2|,  |p_1|+|p_3|) \ge  \tfrac 16|W|
\end{equation}
or we get the existence of a face $\Pi$ in $\D$ and  vertices $v_1, v_2 \in P_W$ with the properties stated in part (b) of Lemma~\ref{lem2}.

First assume that   Lemma~\ref{lem2}(a) holds true  and $\p |_0 \D = q_1 q_2 q_3$, where
$q_i = \al(p_i) $, $i =1,2,3$. Consider disk subdiagrams $\D_1, \D_2$ of $\D$ given by $\p |_{v_2} \D_2 = q_2$,  $\p |_{0} \D_1 = q_1q_3$. Denote $W_2 \equiv \ph(q_2)$, $W_1 \equiv \ph(q_1)\ph(q_3)$ and let $P_{W_i} = P_{W_i}(W_i, \D_i)$, $i=1,2$,  denote the corresponding paths such that  $\al_1(P_{W_1}) = q_1q_3$ and $\al_2(P_{W_2}) = q_2$.

Since $|W_1|, |W_2| < |W|$, it follows from the induction hypothesis that there is a sequence $\Omega_2$ of \EO s  for $(W_2, \D_2)$ that transforms the empty \BS\   into the final system  and has size bounded by
\begin{equation}\label{in11}
(11 |W_2|, C (\log |W_2| +1))  .
\end{equation}
Let $B_{2,0}, \dots, B_{2, \ell_2}$ denote the corresponding to $\Omega_2$  sequence of \BS s, where $B_{2,0}$ is empty and $B_{2, \ell_2}$ is final.

We also consider a sequence  $\Omega_1$ of \EO  s for $(W_1, \D_1)$  that transforms the \BS\  $\{ c_0 \}$, where
\begin{equation}\label{c0h}
c_0 := (|p_1|, |p_1|, 0,0,0,0)
\end{equation}
is a starting bracket,  into  the final bracket system. It follows from the  induction hypothesis and Lemma~\ref{lem5}   that there is such a sequence  $\Omega_1$ of size bounded by
\begin{equation}\label{in12}
(11 |W_1|+1, C (\log |W_1| +1)+1)  .
\end{equation}
Let  $B_{1,0}, \dots, B_{1, \ell_1}$ denote the corresponding to $\Omega_1$  sequence of \BS s, where $B_{1,0} =  \{ c_0 \} $ and $B_{1, \ell_1}$ is final.

We will show below that these two sequences  $\Omega_2$,  $\Omega_1$  of \EO  s, the first one for $(W_2, \D_2)$ and the second one for  $(W_1, \D_1)$, could be modified and
combined into a single sequence of \EO  s for  $(W, \D)$ that transforms the empty \BS\
into the final system and has size with the desired upper bound.

First we observe that every bracket $b = (b(1), \dots, b(6))$ for   $(W_2, \D_2)$  naturally gives rise to a bracket $\wht b = (\wht b(1), \dots, \wht b(6))$ for  $(W, \D)$. Specifically, define
$$
\wht b := (b(1)+|p_1|, b(2)+|p_1|, b(3), b(4), b(5), b(6)) .
$$
Let $\wht B_{2,j}$ denote the \BS\    obtained from $B_{2,j}$ by replacing every bracket $b \in B_{2,j}$ with $\wht b$. Then
 $\wht B_{2,0}, \dots, \wht B_{2, \ell_2}$ is a sequence of \BS s for    $(W, \D)$ that changes the empty \BS\ into
 $$
 \wht B_{2, \ell_2} = \{  (|p_1|, |p_1|+ |p_2|, 0, 0, |\D_2(2)|) \} .
 $$

Define a relation $\succeq$ on the set of pairs $(b,i)$, where $b \in   B_{1,i}$,  that is the reflexive and transitive closure of the relation $(c, i+1) \succ (b,i)$,  where  $c \in   B_{1,i+1} $ is obtained from $b, b' \in   B_{1,i}$
by an \EO\ $\sigma$, where $b'$ could be missing. It follows from the definitions that  $\succeq$ is a partial order on the set of such  pairs $(b,i)$   and that if $(b_2, i_2) \succeq (b_1, i_1)$ then $i_2 \ge i_1$ and
$$
b_2(1) \le b_1(1) \le b_1(2) \le b_2(2) .
$$
Note that the converse need not hold, e.g., if $b_1$ is a starting bracket and  $i_1 = i_2$,
$b_1(1) = b_1(2) = b_2(1)$, $b_1 \ne b_2$, then the above inequalities hold but $(b_2, i_2) \not\succeq (b_1, i_1)$.

Now we observe that every bracket $d = (d(1), \dots, d(6))$, $d \in  B_{1,i}$, for   $(W_1, \D_1)$  naturally gives rise to a bracket
$$
\wht d = (\wht d(1), \dots, \wht d(6))
$$
for  $(W, \D)$ in the following fashion.

If $(d, i)$ is not comparable with $(c_0, 0)$, where $c_0$ is defined in \er{c0h}, by the relation $\succeq$ and    $d(1) \le d(2) \le |p_1|$, then
$$
\wht d :=d .
$$

If $(d, i) $ is not comparable with $(c_0,0)$    by the relation $\succeq$ and    $|p_1| \le d(1) \le d(2) $, then
\begin{equation*}
\wht d :=  (d(1)+ |p_2|,  d(2)+ |p_2|, d(3), d(4),d(5),d(6)) .
\end{equation*}

If $(d, i)   \succeq (c_0,0)$, then
\begin{equation*}
\wht d :=  (d(1),  d(2)+ |p_2|, d(3), d(4),d(5),d(6)  +|\D_2(2)|) .
\end{equation*}

Note that the above three cases cover all possible situations because if $(d,i)$ is not comparable with $(c_0,0)$    by the relation $\succeq$, then $d(2) \le |p_1| = c_0(1)$ or $d(1) \ge |p_1| = c_0(2)$.

As above, let $\wht B_{1,i} := \{ \wht d \mid d \in B_{1,i} \}$. Then  $\wht B_{1,0}, \dots, \wht B_{1, \ell_1}$ is a sequence of \BS s for    $(W, \D)$ that changes the  \BS\
\begin{equation*}
\wht B_{1,0} = \wht B_{2,\ell_2} = \{ (|p_1|,|p_1p_2|, 0, 0, 0, |\D_2(2)| ) \}
\end{equation*}
into the final \BS\
$$
\wht B_{1, \ell_1} = (0, |p_1|+ |p_2|+ |p_3|, 0, 0, 0, |\D_1(2)|+ |\D_2(2)|) .
$$
More specifically, it is straightforward to verify that  $\wht B_{1,0}, \dots, \wht B_{1, \ell_1}$ is an operational sequence of \BS s for    $(W, \D)$ which corresponds to an analogue
$\wht \Omega_1$ of the sequence of \EO s $\Omega_1$ for  $(W_1, \D_1)$ so that if a bracket $b \in B_{1,i}$, $i \ge 1$, is obtained  from  brackets $d_1, d_2 \in B_{1,i-1}$, where
one of $d_1, d_2$ could be missing, by an \EO\ $\sigma$ of $\Omega_1$, then $\wht b \in \wht B_{1,i}$ is obtained from $\wht d_1, \wht d_2 \in \wht B_{1,i-1}$  by an \EO\ $\wht \sigma$ of $\wht  \Omega_1$ and $\wht \sigma$ has the same type as $\sigma$.

Thus, with the indicated changes, we can now combine the foregoing sequences of \BS  s for $(W_2, \D_2)$ and for $(W_1, \D_1)$ into a single sequence
of \BS s for $(W, \D)$ that transforms the empty \BS\       into the \BS\
 $\{ (|p_1|,|p_1p_2|, 0, 0, 0, |\D_2(2)| ) \}$ and then continues to turn the latter
  into the final \BS . It follows from definitions and bounds \er{in11}--\er{in12} that the size of    thus constructed sequence is bounded by
 \begin{gather*}
 ( 11|W_1|+11|W_2|+1,\   \max( C(\log |W_1|+ 1)+1, \ C( \log |W_2|+1 ) )  ) \
 \end{gather*}
 Therefore, in view of Lemma~\ref{lem4},    it remains to show that
 \begin{gather*}
  \max( C(\log |W_1| +1)+ 1, C(\log |W_2|+1 ) ) \le
 C (\log |W|+1)  .
\end{gather*}
In view of inequality \eqref{in10},
$$
\max( C(\log |W_1| +1)+ 1, C(\log |W_2|+1 ) )  \le  C (\log (\tfrac 56|W|)+1) +1 ,
$$
and   $C (\log (\tfrac 56|W|)+1) +1  \le  C (\log |W|+1)  $ if $C \ge (\log \tfrac 65)^{-1}$.
Thus the first main case, when Lemma~\ref{lem2}(a) holds for the pair $(W, \D)$, is complete.
\medskip

Now assume that  Lemma~\ref{lem2}(b) holds true for  the pair $(W, \D)$   and  $\Pi$  is the face in $\D$ with  $| \p \Pi | \ge 2$,  $v_1, v_2 \in P_W$ are    the vertices of  $P_W$ such that  $v_1 < v_2$, $\al(v_1), \al(v_2) \in \p \Pi$, and if $P_W(\ff,v_1, v_2) = p_1 p_2 p_3$  then
\begin{equation*}
 \min(|p_2|,  |p_1|+|p_3|) \ge
 \tfrac 16 |W|   .
\end{equation*}
Furthermore, if  $(\p \Pi )^{-1} = e_1  \dots e_{|\p \Pi |}$, where $e_i \in \vec \D(1)$,
and the cyclic boundary $\p \D$ of $\D$ is $\p \D =   e_1 h_1 \dots$ $ e_{|\p \Pi |} h_{|\p \Pi |}$, where $h_i$ is a closed subpath of $\p \D$,  then, for every $i$, $h_i$ is a subpath of either $\al( p_2)$ or $\al( p_3) \al( p_1)$ and $|h_i | \le \tfrac 56 |W|$.

For $i = 2, \dots,|\p \Pi |$, denote $W_i := \ph(h_i)$ and let $\D_i $ be the disk subdiagram of $\D$ with $\p|_{(h_i)_-} \D_i = h_i$. We also consider a path $P_{W_i}$ with $\ph(P_{W_i}) \equiv W_i$ and let
$\al_i : P_{W_i} \to \D_i$ denote the map whose definition is analogous to that of
$\al : P_{W} \to \D$ (note that  $|W_i| = 0$ is now possible).

By the induction hypothesis on $|W|$ applied to $(W_i, \D_i)$ (the case $|W_i| = 0$ is vacuous), there is a
sequence  $\Omega_i$  of \EO s for  $(W_i, \D_i)$  that transforms the empty \BS\ into the final one and has size bounded by $( 10|W_i|, C(\log |W_i|+ 1) )$.

Making a cyclic permutation of indices of $e_i, h_i$ in $\p \D =  e_1 h_1 \dots e_{|\p \Pi |} h_{|\p \Pi |}$ if necessary, we may assume that
\begin{equation*}
P_W = s_2 f_2 q_2 f_3 q_3 \dots f_{|\p \Pi |} q_{|\p \Pi |} f_1 s_1 ,
\end{equation*}
where $\al(f_i) = e_i$,  $i = 1, \dots,|\p \Pi |$, $\al(q_j) = h_j$,  $j = 2, \dots,|\p \Pi |$, and
 $\al(s_1)\al(s_2) = h_1$, see Fig.~4.2.  Note that $|s_1| =0$ or $|s_2| =0$ is possible.

Let $B_{i,0}, \dots, B_{i, \ell_i}$ denote the corresponding to $\Omega_i$  sequence of \BS s, where $B_{i,0}$ is empty and $B_{i, \ell_i}$ is final. As in the above arguments,  we can easily convert every bracket $b \in \cup_j B_{i,j}$, where $i >1$, into a bracket $\wht b$ for  $(W, \D)$  by using the rule
\begin{equation*}
  \wht b := (b(1) + | s_2 f_2 \dots f_i |, b(2)+ | s_2 f_2 \dots f_i |  ,  b(3), b(4), b(5), b(6)) .
\end{equation*}
Then the sequence $\wht B_{i,0}, \dots, \wht B_{i, \ell_i}$, where  $\wht B_{i,j} := \{ \wht b \mid b \in B_{i,j} \}$, becomes an operational  sequence of \BS s for  $(W, \D)$  that transforms the empty \BS\ into
$\wht B_{i,\ell_i} = \{  \wht d_i  \}$, where
\begin{equation*}
    \wht d_i = (| s_2 f_2 \dots f_i |,  | s_2 f_2 \dots f_i q_i|,  0, 0,0, |\D_i(2)| )  \} , \ i >1 .
\end{equation*}
We also remark that the  sequence  of \BS s    $\wht B_{i,0}, \dots, \wht B_{i, \ell_i}$
 corresponds to an analogue
$\wht \Omega_i$ of the sequence of \EO s $\Omega_i$ for  $(W_i, \D_i)$ so that if a bracket $b \in B_{i,j}$, $j \ge 1$, is obtained  from  brackets $g_1, g_2 \in B_{i,j-1}$, where
one of $g_1, g_2$ could be missing, by an \EO\ $\sigma$ of $\Omega_i$, then $\wht b \in \wht B_{i, j}$ is obtained from $\wht g_1, \wht g_2 \in \wht B_{i,j-1}$  by an \EO\ $\wht \sigma$ of $\wht  \Omega_i$ and $\wht \sigma$ has the same type as $\sigma$.

Denote $\ph( (\p \Pi)^{-1}) = a_{j_\Pi}^{\e n_\Pi}$. Using the operational sequence  $\wht B_{2,0}, \dots, \wht B_{2, \ell_2}$, we convert the empty \BS\ into     $\{  \wht d_2  \}$. Applying a turn operation to $\wht d_2$, we change  $\wht d_2(3)$ from 0 to   $j_\Pi$ and $d_2(4)$ from 0 to $n_\Pi$. Note that  $n_\Pi \le |W|$ by  Lemma~\ref{vk2}, hence this is correct to do so. Then we apply two extensions of type 1 on the left and on the right to increase  $\wht d_2(2)$ by 1 and to decrease  $\wht d_2(1)$ by 1, see Fig.~4.2.  Let
\begin{equation*}
\wtl d_2 = (|s_2|,  |s_2 f_2 q_2 f_3|, j_\Pi, n_\Pi, 2\e, |\D_2(2) |  )
\end{equation*}
denote the bracket of type B2 obtained this way. Now, starting with the \BS\  $\{ \wtl d_2 \}$, we apply those \EO s that are used to create
the sequence  $\wht B_{3,0}, \dots, \wht B_{3, \ell_3}$, and obtain the \BS\  $\{ \wtl d_2, \wht d_3 \}$. Applying a merger to $\wtl d_2, \wht d_3$, we get
\begin{equation*}
     \wht d_3' = (| s_2 |,  | s_2 f_2 q_2f_3 q_3|, j_\Pi,  n_\Pi, 2\e, |\D_2(2) |  + |\D_3(2)| )  .
\end{equation*}
Let  $\wtl d_3$ be obtained from $  \wht d_3'$ by extension of type 1 on the right, so
\begin{equation*}
\wtl d_3 = (| s_2 |,  | s_2 f_2 q_2f_3 q_3 f_4|, j_\Pi,  n_\Pi, 3\e, |\D_2(2) |  + |\D_3(2)| )   .
\end{equation*}
Iterating in this manner, we will arrive at a \BS\ consisting of the single bracket
\begin{equation*}
    \wht d_{| \p \Pi |}' = (| s_2 |,  | s_2 f_2 \dots  q_{| \p \Pi |}  |, j_\Pi,   n_\Pi, \e(| \p \Pi |-1), \sum_{i\ge 2}|\D_i(2) | )   .
\end{equation*}
Applying to $ \wht d_{| \p \Pi |}' $ an extension of type 2 on the right along the edge $f_1 = \al(e_1)$, see Fig.~4.2,  we obtain the bracket
\begin{equation*}
    \wtl d_{| \p \Pi |} = ( | s_2 |,  |W| - |s_1|, 0,0,0, 1+ \sum_{i\ge 2}|\D_i(2) | )
\end{equation*}
of type B1.

For $i =1$, we let  $W_1 := \ph ( s_2) \ph (s_1) $ and let $\D_1$ be the disk subdiagram of $\D$ with
$\p |_{(s_2)_-} \D_1 =  s_2 s_1$.
Referring to the induction hypothesis on $|W|$ for $(W_1, \D_1)$ and to Lemma~\ref{lem5}, we conclude
that there is a sequence $\Omega'_1$  of \EO s that changes the starting bracket
\begin{gather}\label{c2h}
c_2 := (|s_2|, |s_2|, 0,0,0,0)
\end{gather}
into the final \BS\ and has size bounded by
$$
( 11|W_1| +1, \ C(\log |W_1|+ 1) +1).
$$

Let $B'_{1,0}, \dots, B'_{1, \ell'_1}$ denote the corresponding to $\Omega_1'$  sequence of \BS s, where $B'_{1,0} = \{  c_2 \}$ and $B'_{1, \ell'_1}$ is final. Similarly to  the above construction of  $\wht B_{1,i}$ from  $B_{1,i}$,  we will make changes over brackets  $b \in \cup_j B'_{1,j}$ so that every  $b$ becomes a bracket $\wht b$ for $(W, \D)$ and if $\wht B'_{1,i} := \{ \wht b \mid b \in  B'_{1,i} \} $, then the sequence $\wht B'_{1,0}, \dots, \wht B'_{1, \ell'_1}$  transforms the \BS\
$\wht B'_{1, 0} =  \{   \wht c_2 \} =   \{    \wtl d_{| \p \Pi |}  \} $ into  the final one for $(W, \D)$.

Specifically, define a relation   $\succeq'$  on the set of all pairs $(b,i)$, where $b \in  B'_{1,i}$, that is reflexive and transitive closure of the relation $(c, i+1) \succ' (b,i)$,  where a bracket  $c \in   B_{1,i+1} $ is obtained from brackets $b, b' \in   B_{1,i}$ by an \EO\ $\sigma$, where $b'$ could be missing. As before, it follows from the definitions that  $\succeq'$ is a partial order on the set of all such  pairs $(b,i)$ and that if $(b_2, i_2) \succeq' (b_1, i_1)$ then $i_2 \ge i_1$ and
$$
b_2(1) \le b_1(1) \le b_1(2) \le b_2(2) .
$$

Furthermore, every bracket $b= (b(1), \dots, b(6))$, where $b \in B'_{1,i}$, for $(W_1, \D_1)$, naturally gives rise to a bracket
$$
\wht b= (\wht b(1), \dots, \wht b(6))
$$
for $(W, \D)$ in the following fashion.

If the pair $(b, i)$ is not comparable with $(c_2, 0)$, where $c_2$ is defined in \er{c2h},  by the relation $\succeq'$ and  $ b(2) \le |s_2|$, then
$$
\wht b :=b .
$$

If the $(b, i) $ is not comparable with $(c_2,0)$    by the relation $\succeq'$ and    $|s_1| \le b(1)$, then
\begin{equation*}
\wht b :=  (b(1)+ |W|- |h_1|,  b(2)+ |W|- |h_1|, b(3), b(4),b(5),b(6)) .
\end{equation*}

If $(d, i)   \succeq' (c_2,0)$, then
\begin{equation*}
\wht b :=  (b(1),  b(2)+ |W|- |h_1|, b(3), b(4), b(5), b(6)  +1+\sum_{i\ge 2} |\D_i(2)| ) .
\end{equation*}

As before, we note that these three cases cover all possible situations because if $(b,i)$ is not comparable with $(c_2,0)$    by the relation $\succeq'$, then  either
$$
b(2) \le |W|- |h_1| = c_2(1)
$$
or $b(1) \ge |W|- |h_1| = c_2(2)$.

Similarly to the foregoing arguments, we check that  $\wht B'_{1,0}, \dots, \wht B'_{1, \ell_1}$ is an operational sequence of \BS s for    $(W, \D)$ which corresponds to an analogue $\wht \Omega_1'$ of the sequence of \EO s $\Omega_1'$ for  $(W_1, \D_1)$ so that  if a bracket $b \in B'_{1,i}$, $i \ge 1$, is obtained  from  brackets $g_1, g_2 \in B'_{1,i-1}$, where
one of $g_1, g_2$ could be missing, by an \EO\ $\sigma$ of $\Omega'_1$, then $\wht b \in \wht B'_{1,i}$ is obtained from $\wht g_1, \wht g_2 \in \wht B'_{1,i-1}$  by an \EO\ $\wht \sigma$ of $\wht  \Omega'_1$ and $\wht \sigma$ has the same type as $\sigma$.

Summarizing, we conclude that there exists a sequence of \EO s $\Omega$, containing
subsequences $\wht  \Omega_2, \dots$, $\wht  \Omega_{|\p \Pi | }$, that transforms the empty
 bracket system for  $(W, \D)$  into the  \BS\  $B = \{  \wtl d_{|\p \Pi |} \}$ and then, via
 subsequence  $\wht \Omega'_1$, continues to transform
 $$B = \{  \wtl d_{|\p \Pi |} \} = \{   \wht c_2 \} $$
 into the final  \BS\   for  $(W, \D)$. It follows from the induction hypothesis for the pairs $(W_i, \D_i)$ and definitions that the size of thus constructed  sequence $\Omega$ is bounded  by
\begin{equation*}
( (11 \!  \cdot \! \! \!  \! \! \sum_{1 \le i \le |\p \Pi |} |W_i| ) + 2 |\p \Pi | +2,   \max_{1 \le i \le |\p \Pi |} \{ C(\log |W_i| +1 ) +1  \}    )   .
\end{equation*}
In view of  Lemma~\ref{lem4}, it suffices to show that
\begin{equation*}
\max_{1 \le i \le |\p \Pi |} \{ C(\log |W_i| +1 ) +1  \} \le C(\log |W| +1 ) .
\end{equation*}
Since $|W_i| \le \tfrac 56|W|$, the latter inequality holds true, as in the above case, for  $C = (\log \tfrac 56)^{-1}$.
\end{proof}

Let $W$ be an arbitrary nonempty word over the alphabet $\A^{\pm 1}$, not necessarily representing the trivial element of the group given by presentation \er{pr1} (and $W$ is not necessarily reduced).  As before, let $P_W$ be a labeled simple
path with $\ph( P_W ) \equiv W$ and let vertices of $P_W$ be identified along $P_W$ with integers $0, \dots, |W|$ so that $(P_W)_- = 0, \dots$, $(P_W)_+ = |W|$.
\medskip

A 6-tuple
$$
b = (b(1), b(2), b(3), b(4), b(5), b(6))
$$
of integers $b(1), b(2), b(3), b(4), b(5), b(6)$ is called a {\em pseudobracket}  for the word $W$ if $b(1), b(2)$ satisfy the inequalities
$0 \le b(1) \le b(2) \le |W|$ and one of the following two properties (PB1), (PB2) holds for $b$.

\begin{enumerate}
\item[(PB1)] \ $b(3)= b(4)=b(5)=0$ and $0 \le b(6) \le |W|$.
\item[(PB2)]  \  $b(3) \in \{ 1, \dots, m\}$,  $b(4) \in E_{b(3)}$, where the set $E_{b(3)}$ is defined in \er{pr1},  $0 < b(4) \le |W|$,  $-b(4) < b(5) < b(4)$, and  $0 \le b(6) \le |W|$.
 \end{enumerate}

We say that a \pbb\  $b$ has {\em type PB1}  if the property
(PB1) holds for $b$. We also say that a \pbb\  $b$ has {\em type PB2}  if the property
(PB2) holds for $b$. Clearly, these two types are mutually exclusive.
\medskip

Let $p$ denote the subpath of $P_W$ such that $p_- = b(1)$ and $p_+ = b(2)$, perhaps, $p_- = p_+$ and $|p| = 0$. The  subpath $p$ of $P_W$ is denoted $a(b)$ and is called the {\em arc} of the pseudobracket $b$.

For example, $b_v = (v, v, 0, 0, 0,0)$ is a pseudobracket of type PB1 for every vertex $v \in P_W$ and such $b_v$ is  called a {\em  starting} pseudobracket.  Note that $a(b) = \{ v \}$.
A {\em final}  pseudobracket for $W$ is  $c = (0, |W|, 0, 0,0, k)$, where $k \ge 0$ is an integer. Note that   $a(c) =  P_W$.

Observe that if $b$ is a bracket for the pair $(W, \D)$ of type B1 or B2, then $b$ is also
a \pbb\ of type PB1 or PB2, resp., for the word $W$.

Let $B$ be a finite set of \pbb s for  $W$, perhaps, $B$ is empty.   We say that $B$ is a {\em pseudobracket system} for $W$ if, for every pair $b, c \in B$ of distinct \pbb s, either $b(2) \le c(1)$  or $c(2) \le b(1)$.  It follows from the definitions that every \BS\ for $(W, \D)$ is also a \PBS\ for the word $W$.  $B$  is called a {\em final \PBS\ } for $W$ if $B$ consists of a single final \pbb\ for $W$.
\medskip

Now we describe four kinds of \EO s over \pbb s and over \PBS s: additions, extensions, turns, and mergers,  which are analogous to those definitions for brackets and \BS s, except there are no any diagrams and no faces involved.

Let $B$  be a \PBS\ for a nonempty word $W$ over $\A^{\pm 1}$.
\medskip

{\em Additions.}

Suppose $b$ is a starting \pbb ,  $b \not\in B$, and $B \cup \{ b\}$
is a \PBS  . Then we may add $b$ to $B$ thus making an {addition} operation over $B$.
\medskip

{\em Extensions.}

Suppose $b \in B$, $b = (b(1), b(2), b(3), b(4), b(5), b(6))$,   and $e_1 a(b) e_2$ is a subpath of $P_W$, where $a(b)$ is the arc of $b$ and $e_1, e_2$ are edges one of which could be missing.

Assume that $b$ is of type PB2. Suppose  $\ph(e_1) = a_{b(3)}^{\e_1}$, where $\e_1 = \pm 1$.

If $| b(5) | \le b(4)-2$ and $\e_1 b(5)  \ge 0$, then we consider a \pbb\ $b'$ of type PB2 such that
\begin{align*}
&  b'(1) = b(1)-1, \  b'(2) = b(2) , \  b'(3) = b(3) ,  \\
& b'(4) = b(4) ,  \ b'(5) = b(5)+\e_1 ,  b'(6) =   b(6) .
\end{align*}
Note that $a(b') = e_1 a(b)$.
We say that $b'$ is obtained from $b$ by an extension of type 1 (on the left).
If $(B \setminus \{ b \}) \cup \{ b' \}$ is a \PBS , then
replacement of $b \in B$ with $b'$ in $B$ is called an {\em extension }
operation over $B$ of type 1.

On the other hand, if $| b(5) | = b(4)-1$   and $\e_1 b(5)  \ge 0$, then we  consider a  \pbb\ $b'$ of type PB1 such that
$$
b'(1) = b(1)-1 , \
b'(2) = b(2) , \  b'(3) = b'(4) = b'(5) = 0, \ b'(6) = b(6)+1 .
$$
In this case, we say that $b'$ is obtained from $b$ by an extension of type 2 (on the left).
Note that $a(b') = e_1 a(b)$ and $b'$ has type PB1. If $(B \setminus \{ b \}) \cup \{ b' \}$ is a \PBS , then
replacement of $b \in B$ with $b'$ in $B$ is called an {\em extension } operation over $B$ of type 2.

Analogously, assume that  $b$ has type PB2 and $\ph(e_2) = a_{b(3)}^{\e_2}$, where $\e_2 = \pm 1$.

If $| b(5) | \le b(4)-2$ and $\e_2 b(5)  \ge 0$, then we consider a  \pbb\ $b'$  such that
\begin{align*}
&  b'(1) = b(1), \  b'(2) = b(2)+1 , \  b'(3) = b(3) ,  \\
& b'(4) = b(4) ,  \ b'(5) = b(5)+\e_2 ,  b'(6) =   b(6) .
\end{align*}
Note that $a(b') =  a(b)e_2$ and $b'$ has type PB2. We say that $b'$ is obtained from $b$ by an extension of type 1 (on the right).
If $(B \setminus \{ b \}) \cup \{ b' \}$ is a \PBS , then
replacement of $b \in B$ with $b'$ in $B$ is called an {\em extension } operation over $B$ of type 1.

On the other hand, if $| b(5) | = b(4)-1$ and    $\e_2 b(5)  \ge 0$, then we consider a \pbb\  $b'$ such that
$$
b'(1) = b(1) , \ b'(2) = b(2)+1 ,  \ b'(3) =b'(4) = b'(5) =0 ,  \ b'(6) = b(6)+1 .
$$
Note that $a(b') = a(b) e_2$ and $b'$ has type PB1.
We say that $b'$ is obtained from $b$ by an extension of type 2 (on the right).
If $(B \setminus \{ b \}) \cup \{ b' \}$ is a \PBS , then
replacement of $b \in B$ with $b'$ in $B$ is called an {\em extension } operation over $B$ of type 2.

Assume that $b \in B$ is a \pbb\  of type PB1, $e_1 a(b) e_2$  is a subpath of   $P_W$, where $a(b)$ is the arc of $b$ and $e_1, e_2$ are edges  with  $\ph(e_1) = \ph(e_2)^{-1}$.
Consider a \pbb\  $b'$ of type PB1 such that
$$
b'(1) = b(1)-1 , \ b'(2) = b(2)+1 , \ b'(3) = b'(4) =b'(5) =0 , \ b'(6) = b(6) .
$$
Note that $a(b') = e_1 a(b)e_2$.
We say that $b'$ is obtained from $b$ by an extension of type 3.
If $(B \setminus \{ b \}) \cup \{ b' \}$ is a \PBS , then
replacement of $b \in B$ with $b'$ in $B$ is called an {\em extension } operation over $B$ of type~3.
\medskip

{\em Turns.}

Let $b \in B$ be a \pbb\  of type  PB1.  Then, by the definition,  $b(3) = b(4) =b(5) =0$.
Pick $j \in \{ 1, \dots, m\}$ such that $E_j \ne \{ 0 \}$. Consider a \pbb\ $b'$ with $b'(i) = b(i)$ for $i=1,2,4,5,6$, and $b'(3) = j$. Note that $b'$ is of type PB2 and $a(b') = a(b)$.
We say that $b'$ is obtained from $b$ by a {\em turn} operation.
Replacement of $b \in B$ with $b'$ in $B$ is also called a {\em turn } operation over $B$.
Note that $(B \setminus \{ b \}) \cup \{ b' \}$ is a \PBS\ because so is $B$.
\medskip

{\em Mergers.}

Now suppose that $b, c \in B$ are distinct \pbb s such that  $b(2) = c(1)$ and one of $b(3), c(3)$ is 0. Then one of
 $b, c$ is of type PB1  and the other has type  PB1 or PB2.
Consider a  \pbb\  $b'$ such that $b'(1) = b(1)$, $b'(2) = c(2)$, and $b'(i) =  b(i) + c(i)$  for $i=3,4,5,6$. Note that $a(b') = a(b)a(c)$ and $b'$ is of type  PB1 if both  $b, c$ have type PB1 or $b'$ is of type PB2 if one of  $b, c$ has type PB2.  We say that $b'$ is obtained from $b, c$ by a {\em merger} operation.
Taking both $b, c$ out of $B$ and putting  $b'$ in $B$ is a {\em merger } operation over $B$.
\medskip

As before, we will say that additions, extensions, turns and mergers, as defined above, are {\em  \EO s} over \pbb s and \PBS  s for $W$.

Assume that one \PBS\   $B_\ell$ is obtained from another \PBS\    $B_0$
by a finite sequence $\Omega$ of \EO s  and $B_0,  B_1,  \dots, B_\ell$ is the corresponding to $\Omega$  sequence of \PBS  s. Such a sequence  $B_0,  B_1,  \dots, B_\ell$ of \PBS s  is called {\em operational}.  We say that a  sequence $B_0,  B_1,  \dots, B_\ell$ of \PBS s has {\em size bounded by} $ (k_1, k_2)$  if $\ell \le k_1$ and, for every $i$, the number of \pbb s in $B_i$ is at most $k_2$.   Whenever it is not ambiguous, we also say that $\Omega$ has size bounded by $(k_1, k_2)$ if so does the corresponding to $\Omega$ sequence  $B_0,  B_1,  \dots, B_\ell$ of \PBS s.
\medskip

The significance of \PBS s and \EO s over \PBS s introduced above is revealed in the following lemma.

\begin{lem}\label{lem7}  Suppose that the empty  \PBS\    $B_0$ for $W$ can be transformed by a finite sequence $\Omega$ of \EO  s  into a final \PBS\  $\{ b_F \}$. Then  $W \overset{\GG_2} = 1$, i.e.,  $W$ represents the trivial element of the group  defined by  presentation \er{pr1}, and there is a disk diagram $\D$  over \er{pr1}  such that $\D$ has property (A), $\ph(\p |_0 \D) \equiv W$ and $|\D(2) | = b_F(6)$.
\end{lem}

\begin{proof}  Let $B_0, B_1, \dots, B_\ell$ be the corresponding to $\Omega$  sequence of \PBS s, where $B_0$ is empty and $B_\ell$ is final.  Consider the following Claims~(C1)--(C2) for a \pbb\ $b \in B_i$, $i = 1, \dots, \ell$, in which $a(b)$ denotes the arc of~$b$.

\begin{enumerate}
\item[(C1)] If a \pbb\  $b \in B_i$ has type PB1 then
 $\ph(a(b)) \overset {\GG_2} = 1$ in the group $\GG_2$ given by presentation \er{pr1} and there is a disk diagram $\D_b$ over \eqref{pr1}  such that $\D_b$ has property (A),  $\ph( \p \D_b ) \equiv  \ph(a(b))$ and  $| \D_b(2) | = b(6)$.

\item[(C2)] If  a \pbb\   $b \in B_i$  has type PB2, then  $\ph(a(b))  \overset {\GG_2} = a_{b(3)}^{b(5)}$   and there is a disk diagram $\D_b$ over \eqref{pr1}  such that  $\D_b$ has property (A), $\p \D_b = p q^{-1}$, where $p$, $q^{-1}$ are subpaths of $\p \D_b$, $\ph( p ) \equiv \ph(a(b))$,  $\ph( q ) \equiv a_{b(3)}^{b(5)}$, $| \D_b(2) | = b(6)$, and if $e$ is an edge of  $q^{-1}$ then $e^{-1} \in p$.
 \end{enumerate}

By induction on $i \ge 1$, we will prove that Claims (C1)--(C2) hold true for every \pbb\  $b' \in B_i$. Note that if  $b'$ is a starting \pbb , then $b'$ is of type  PB1 and Claim (C1) is obviously true for $b'$ (Claim (C2) is vacuously true for $b'$).  Since $B_1$ consists of a starting \pbb , the base of induction is established.

To make the induction step from $i-1$ to $i$, $i \ge 2$, we consider the cases corresponding to the type of the \EO\ that is used to get $B_i$ from $B_{i-1}$.

Suppose that $B_i$ is obtained from $B_{i-1}$ by an \EO\ $\sigma$ and $b' \in B_i$ is the \pbb\ obtained from $b, c \in B_{i-1}$ by application of $\sigma$, here one of $b, c$ or both, depending on  type of  $\sigma$, could be missing. By the induction hypothesis,  Claims (C1)--(C2) hold for every \pbb\ of $B_i$  different from $b'$ and it suffices to show that the  suitable Claim (C1)--(C2) holds for $b'$.

If $B_i$ is obtained from $B_{i-1}$ by an addition, then it suffices to refer to  the above remark that Claims (C1)--(C2) hold for a starting \pbb .

Suppose that  $B_i$ is obtained from $B_{i-1}$ by an extension of type 1 and let  $b' \in B_i$ be the \pbb\ created  from $b \in B_{i-1}$ by an extension of type 1 on the left (the ``right" subcase is symmetric).

Note that both $b$ and $b'$ have type PB2. By the induction hypothesis, which is Claim (C2) for $b$, there exists a disk diagram $\D_b$ such that $\D_b$ has property (A),
$$
\p \D_b = p q^{-1}, \ \ \ph( p ) \equiv \ph(a(b)),  \ \ \ph( q ) \equiv a_{b(3)}^{b(5)}, \ \
| \D_b(2) | = b(6),
$$
and if $e$ is an edge of  $q^{-1}$ then $e^{-1} \in p$. Here and below we use the notation of the definition of an extension of type 1.

Let $e_1$ denote the edge of $P_W$ such that $a(b') = e_1 a(b)$ and $\ph(e_1) = a_{b(3)}^{\e_1}$, $\e_1 = \pm 1$.
Consider a ``loose" edge $f$ with $\ph(f) = a_{b(3)}^{\e_1}$. We attach the vertex $f_+$ to the vertex $p_- = q_-$ of $\D_b$ to get a new disk  diagram $\D'_b$ such that  $\p \D'_b = f p q^{-1}f^{-1}$, see Fig.~4.5. Note that property (A) holds for $\D'_b$,
$$
\ph(fp) \equiv a_{b(3)}^{\e_1} \ph(a(b)) \equiv \ph(e_1 a(b)) \equiv \ph(a(b')) , \quad
\ph(q^{-1}f^{-1}) \equiv  a_{b(3)}^{b(5)+\e_1} \equiv a_{b(3)}^{b'(5)} ,
$$
 $| \D'_b(2) | = | \D_b(2) | = b(6)  = b'(6)$, and if $e$ is an edge of  $q^{-1}$ then $e^{-1} \in p$. Thus,  Claim (C2) holds for the \pbb\ $b'$ with $\D_{b'} := \D'_b$.

\begin{center}
\begin{tikzpicture}[scale=.64]
\draw (-26.1,-2)  arc (0:180:2);
\draw(-31.1,-2) -- (-26.1,-2);
\draw [-latex](-31.1,-2) --(-28,-2);
\draw [-latex](-31.1,-2) --(-30.5,-2);
\draw [-latex](-28.1,0) --(-28.05,0);
\draw  (-31.1,-2) [fill = black] circle (.05);
\draw  (-30.1,-2) [fill = black] circle (.05);
\draw  (-26.1,-2) [fill = black] circle (.05);
\node at (-28,-1) {$\Delta_b$};
\node at (-30.6,-2.7) {$f$};
\node at (-28,-2.7) {$q$};
\node at (-28.1,.6) {$p$};
\node at (-30.5,0) {$\Delta'_b$};
\node at (-28,-3.8) {Fig.~4.5};
\end{tikzpicture}
\end{center}

Assume that  $B_i$ is obtained from $B_{i-1}$ by an extension of type 2 and let  $b' \in B_i$ be the \pbb\ obtained  from $b \in B_{i-1}$ by an extension of type 2 on the left (the ``right" subcase is symmetric).

Note that $b$ has type PB2, while $b'$ has type PB1.  By the induction hypothesis, which is Claim (C2) for $b$, there exists a disk diagram $\D_b$ such that  $\D_b$ has property (A),
$$
\p \D_b = p q^{-1}, \ \ \ph( p ) \equiv \ph(a(b)),  \ \ \ph( q ) \equiv a_{b(3)}^{b(5)}, \ \
| \D_b(2) | = b(6),
$$
and if $e$ is an edge of  $q^{-1}$ then $e^{-1} \in p$. Here and below we use the notation of the definition of an extension of type 2.

Let $e_1$ denote the edge of $P_W$ such that $a(b') = e_1 a(b)$ and $\ph(e_1) = a_{b(3)}^{\e_1}$, $\e_1 = \pm 1$.
Consider a ``loose" edge $f$ with $\ph(f) = a_{b(3)}^{\e_1}$. We attach the vertex $f_+$ to  $p_- = q_-$ and   the vertex   $f_-$ to  $p_+ = q_+$ of $\D_b$. Since $\e_1 b(5) \ge 0$  and
$| b(5)| = b(4) -1$, it follows that $\e_1 + b(5) = \e_1 |b(4)|$. Therefore, $\ph(qf) \equiv  a_{b(3)}^{\e_1b(4)}$ and we can attach a new face $\Pi$ with $\ph(\p \Pi) \equiv a_{b(3)}^{-\e_1b(4)}$ so that the boundary path $\p \Pi$ is identified with the path $q^{-1}f^{-1}$, see Fig.~4.6.  This way we obtain a new disk  diagram $\D'_b$ such that
$\ph( \p|_{f_-} \D'_b ) \equiv  \ph( f ) \ph( p)$ and $| \D'_b(2) | = | \D_b(2) | +1$, see Fig.~4.6.  It follows from construction of $\D'_b$ that property (A) holds for $\D'_b$
$$
\ph(fp) \equiv a_{b(3)}^{\e_1} \ph(a(b)) \equiv \ph(e_1 a(b)) \equiv \ph(a(b')) \overset{\GG_2} = 1 , \ \
\ph(q^{-1}f^{-1}) \equiv  a_{b(3)}^{b(5)+\e_1} \equiv a_{b'(3)}^{b'(5)} ,
$$
and  $| \D'_b(2) | = | \D_b(2) |+1 = b(6)+1  = b'(6)$. Therefore,  Claim (C1) holds for the \pbb\ $b'$ if
we set $\D_{b'} := \D'_b$.

\begin{center}
\begin{tikzpicture}[scale=.64]
\draw (-19,-2)  arc (0:360:2);
\draw(-23,-2) -- (-19,-2);
\draw [-latex](-23,-2) --(-21,-2);
\draw [-latex](-21,0) --(-20.95,0);
\draw [-latex](-20.95,-4) --(-21,-4);
\draw  (-23,-2) [fill = black] circle (.05);
\draw  (-19,-2) [fill = black] circle (.05);
\node at (-21,-5) {Fig.~4.6};
\node at (-20.5,-1) {$\Delta_b$};
\node at (-21,-1.6) {$q$};
\node at (-21.,-.5) {$p$};
\node at (-23.5,-.5) {$\Delta'_b$};
\node at (-21.5,-2.7) {$\Pi$};
\node at (-21,-3.5) {$f$};
\end{tikzpicture}
\end{center}

Suppose that an extension of type 3 was applied to get $B_i$ from $B_{i-1}$ and  $b' \in B_i$ is the \pbb\
created from $b \in B_{i-1}$ by  the operation.

Note that both $b$ and $b'$ have type PB1. By the induction hypothesis, which is Claim (C1) for $b$, there exists a disk diagram $\D_b$ such that  $\D_b$ has property (A) and
$$
\ph(\p \D_b) \equiv \ph(a(b))   a_{b(3)}^{-b(5)}  , \quad | \D_b(2) | = b(6) .
$$
Here and below we use the notation of the definition of an extension of type 3.

Denote   $\p \D_b = p$, where
$\ph(p) \equiv \ph(a(b)) $ and let $e_1, e_2$ be the edges of $P_W$ such that $a(b') = e_1 a(b)e_2$ and $\ph(e_1) =\ph(e_2)^{-1}$.
Consider a ``loose" edge $f$ with $\ph(f) = \ph(e_1)$. We attach the vertex $f_+$ to the vertex $p_-$ of $\D_b$ to get a new disk  diagram $\D'_b$ such that  $\p \D'_b = f p f^{-1}$, see Fig.~4.7.  Since  $\D_b$ has property (A),
$$
\ph(   \p \D'_b ) \equiv \ph(fpf^{-1}) \equiv \ph(e_1)  \ph(a(b)) \ph(e_1)^{-1} \equiv  \ph(a(b'))
$$
and  $| \D'_b(2) | = | \D_b(2) | = b(6)  = b'(6)$, it follows that Claim (C1) holds for the
\pbb\ $b'$ with $\D_{b'} := \D'_b$.

\begin{center}
\begin{tikzpicture}[scale=.67]
\draw  (0,0) circle (1);
\draw (0,-2) --(0,-1);
\draw [-latex](0,-2) --(0,-1.4);
\draw [-latex](-.1,1) --(.1,1);
\draw  (0,-2) [fill = black] circle (.05);
\draw  (0,-1) [fill = black] circle (.05);
\node at (0,-2.7) {Fig.~4.7};
\node at (0,-.4) {$\Delta_b$};
\node at (-.4,-1.5) {$f$};
\node at (0,.5) {$p$};
\node at (1.2,-1.3) {$\Delta'_b$};
\end{tikzpicture}
\end{center}

Suppose that  $B_i$ is obtained from $B_{i-1}$ by a turn operation and let  $b' \in B_i$ be the \pbb\ created  from $b \in B_{i-1}$ by the turn operation. By the definition of a turn, $b$ has type PB1 and  $b'$ has type PB2. By the induction hypothesis, which is Claim (C1) for $b$, there exists a disk diagram $\D_b$ such that  $\D_b$ has property (A)
$$
\ph(\p \D_b) \equiv \ph(a(b))   , \quad | \D_b(2) | = b(6) .
$$
Here and below we use the notation of the definition of a turn operation.

Denote $\D'_b := \D_b$ and let $\p \D'_b := p q^{-1}$, where $p, q^{-1}$ are subpaths of $\p \D'_b$ such that
$\ph(p) \equiv \ph(a(b)) $ and  $|q| = 0$. Clearly,  $\D'_b$ has property (A). Since $a(b) = a(b')$, $b'(5) = b(5)=0$, and
$$
| \D'_b(2) | = | \D_b(2) | = b(6)  = b'(6) ,
$$
it follows that
$\ph(p) \equiv  \ph(a(b')) , \ \ph(q) \equiv  a_{b'(3)}^{b'(5)}$. Hence,  Claim (C2) holds for the \pbb\ $b'$ with $\D_{b'} := \D'_b$.
\medskip

Finally, assume that $B_i$  results from $B_{i-1}$ by a merger operation and let $b' \in B_i$ be  the \pbb\
created from \pbb s $b, c \in B_{i-1}$, where $b(2) = c(1)$, by the operation. By the definition of a merger operation, one of  the \pbb s  $b, c$ must have type PB1. By the induction hypothesis, there are disk diagrams $\D_{b},  \D_{c}$ for $b, c$, resp., as stated in Claims (C1)--(C2). Denote  $\p \D_{b} = p_b q_b^{-1}$, $\p \D_{c} = p_c q_c^{-1}$, where $\ph(p_b) \equiv \ph(a(b))$,  $\ph(p_c) \equiv \ph(a(c))$,  and, for $x \in \{ b,c \}$, $| q_x| =0$ if $x$ has type PB1 or $\ph(q_x) \equiv a_{x(3)}^{ x(5)}$ if $x$ has type PB2. Note that  $| \D_x(2) |  = x(6)$.

Consider a disk diagram $\D'$ obtained from $\D_{b},  \D_{c}$ by identification of the vertices
$(p_b)_+$ and $(p_c)_-$, see Fig.~4.8. Then
$$| \D'(2) | =  | \D_b(2) | +| \D_c(2) | , \qquad \p \D' = p_bp_c q_c^{-1} q_b^{-1} .
$$
Note that
$$
\ph(p_b p_c) \equiv \ph(a(b)) \ph(a(c)) \equiv \ph(a(b')) ,  \quad | \D'(2) | = b(6)+c(6)
$$
and  $|q_c^{-1} q_b^{-1} | = 0$ if both $b, c$ have type PB1, or $\ph(q_b q_c) \equiv a_{b(3)}^{b(5)}$ if     $b$ has type PB2 and   $c$ has type PB1, or
$\ph(q_b q_c) \equiv a_{c(3)}^{c(5)}$  if     $b$ has type PB1 and   $c$ has type PB2.
Therefore, it follows that Claim (C1) holds true for $b'$ with  $\D_{b'} = \D'$ if both $b, c$ have type PB1 or  Claim (C2) holds  for $b'$  with  $\D_{b'} = \D'$  if one of  $b, c$ has type PB2.

\begin{center}
\begin{tikzpicture}[scale=.64]
\draw (-20,-2)  arc (0:180:2);
\draw(-24,-2) node (v1) {}  -- (-20,-2);
\draw [-latex](-23,-2) --(-22,-2);
\draw [-latex](-22,0) --(-21.95,0);
\draw  (-24,-2) [fill = black] circle (.05);
\draw  (-20,-2) [fill = black] circle (.05);
\node at (-22,-1) {$\Delta_{c}$};
\node at (-22,0.5) {$p_c$};
\node at (-22,-2.7) {$q_c$};
\draw (-24,-2)  arc (0:180:2);
\draw(-28,-2) node (v1) {}  -- (-24,-2);
\draw [-latex](-27,-2) --(-26,-2);
\draw [-latex](-26,0) --(-25.95,0);
\draw  (-28,-2) [fill = black] circle (.05);
\draw  (-24,-2) [fill = black] circle (.05);
\node at (-26,-1) {$\Delta_{b}$};
\node at (-26,0.5) {$p_b$};
\node at (-24,0) {$\Delta'$};
\node at (-26,-2.7) {$q_b$};
\node at (-24,-3.5) {Fig.~4.8};
\end{tikzpicture}
\end{center}

All possible cases are  discussed and the induction step is complete. Claims (C1)--(C2) are proven.
\medskip

We can now finish the proof of Lemma~\ref{lem7}.  By Claim (C1) applied
to the \pbb\ $b_F = (0, |W|, 0, 0,0, b_F(6) )$ of the final \PBS\ $B_\ell = \{ b_F \}$, there is a disk diagram $\D_{b_F}$ over \er{pr1} such that   $\D_{b_F}$ has property (A),
$\ph( \p |_{0} \D_{b_F} ) \equiv W$ and
$| \D_{b_F}(2) | = b_F(6)$.  Thus, $\D_{b_F}$ is a desired \ddd\ and Lemma~\ref{lem7} is proven.
\end{proof}

\begin{lem}\label{lem8}  Suppose $W$ is a nonempty word over $\A^{\pm 1}$ and $n \ge 0$ is an integer. Then $W$ is a product of at most $n$ conjugates of words $R^{\pm 1}$, where $R=1$ is a relation of  presentation \eqref{pr1}, if and only if there is a sequence $\Omega$ of \EO  s such that $\Omega$   transforms the empty \PBS\    for $W$ into a final \PBS\
$\{ b_F \}$, where $b_F(6) \le n$, and $\Omega$ has size bounded by $ (11|W|,  C(\log|W| +1))$, where $C = (\log \tfrac 65)^{-1}$.
\end{lem}

\begin{proof}   Assume that  $W$ is a product of at most $n$ conjugates of words $R^{\pm 1}$, where $R=1$ is a relation of presentation  \eqref{pr1}. By Lemmas~\ref{vk}--\ref{vk2}(b), there is a  disk diagram $\D$ over \eqref{pr1} such that  $\D$ has property (A),   $\ph(\p|_0 \D) \equiv W$ and $|\D(2)| \le \min( n, |W|)$. Then Lemma~\ref{lem6} applies  and yields  a sequence $\Omega$ of \EO  s over \BS  s for $(W, \D)$ that converts the empty \BS\    for $(W, \D)$ into the final one and has  size bounded by $ (11|W|,  C(\log|W| +1))$. It follows from
arguments of Lemma~\ref{lem6} and Lemma \ref{vk2}(a) that if $b$
is a bracket of one of the intermediate \BS s, associated with $\Omega$, then $b(4) \le |W|$.
Observe that every bracket $b \in B$  and every intermediate  \BS\  $B$ for $(W, \D)$, associated with $\Omega$,  could be considered as a \pbb\ and a \PBS\ for $W$, resp., we automatically have a desired sequence of \PBS s.

Conversely, the existence of a  sequence  $\Omega$
of \EO  s over \PBS  s, as specified in  Lemma~\ref{lem8}, implies, by Lemma~\ref{lem7}, that there is a disk diagram $\D$ over  \eqref{pr1}   such that  $\D$ has property (A),    $\ph(\p|_0 \D) \equiv W$ and $|\D(2)| =b_F(6) \le n$. The existence of such a \ddd\ $\D$
means that  $W$ is a product of $|\D(2)| =b_F(6) \le n$ conjugates  of words $R^{\pm 1}$, where $R=1$ is a relation of presentation  \eqref{pr1}.
\end{proof}

\section{Proofs of Theorem~\ref{thm1} and Corollary~\ref{cor1}}

\begin{T1}   Let  the group $\GG_2$   be defined by a presentation of the form
\begin{equation*}\tag{1.2}
    \GG_2 :=  \langle \,  a_1, \dots, a_m  \ \|  \   a_{i}^{k_i} =1, \ k_i \in E_i,   \ i = 1, \ldots,  m \,  \rangle ,
\end{equation*}
where for  every $i$, one of the following holds: $E_i = \{ 0\}$ or, for some integer $n_i >0$,  $E_i = \{ n_i \}$ or $E_i = n_i \mathbb N = \{ n_i, 2n_i, 3n_i,\dots  \}$.
Then both the bounded and  precise  word  problems for \eqref{pr1} are in $\textsf{L}^3$ and in $\textsf{P}$. Specifically, the problems can be solved in deterministic space $O((\log|W|)^3)$ or in deterministic time $O( |W|^4\log|W|)$.
\end{T1}

\begin{proof} We start with $\textsf{L}^3$ part of Theorem~\ref{thm1}. First we discuss a nondeterministic algorithm  which  solves the bounded word problem for presentation  \eqref{pr1}  and which is based on Lemma~\ref{lem8}.

Given an input $(W, 1^n)$, where $W$ is a nonempty word (not necessarily reduced) over the alphabet $\A^{\pm 1}$ and $n \ge 0$ is an integer, written in unary notation as $1^n$, we begin with the empty \PBS\   and nondeterministically  apply a sequence of \EO  s  of size bounded by $( 11|W|,  C(\log|W| +1) )$, where $C = (\log \tfrac 65)^{-1}$.  If such a  sequence of \EO  s  results in a final \PBS\   $\{ (0,|W|,0,0,0,n') \}$, where $n' \le n$, then our algorithm accepts and, in view of Lemma~\ref{lem8}, we may conclude that  $W$ is a product of $n' \le n$ conjugates of words $R^{\pm 1}$, where $R=1$ is a relation of \eqref{pr1}.
 It follows from the definitions and  Lemma~\ref{lem8}  that the  number of \EO  s needed for this algorithm to accept is at most $11|W|$.
 Hence, it follows from the definition of \EO s over \PBS s for   $W$ that  the time needed to run this nondeterministic algorithm is $O(|W|)$.
 To estimate  the space requirements of this algorithm, we note that if $b$ is a \pbb\ for $W$, then $b(1)$, $b(2)$ are  integers in the range from 0 to $|W|$, hence, when written in binary  notation, will take at most $C'(\log|W| +1)$ space, where $C'$ is a constant. Since  $b(3), b(4), b(5), b(6)$ are also integers  that satisfy inequalities
 \begin{equation*}
 0 \le b(3) \le m , \ 0 \le b(4) \le |W| , \ |b(5)| \le b(4),  \  0 \le b(6) \le |W| ,
\end{equation*}
 and $m$ is a constant,  it follows  that the total space required to run this algorithm is at most
$$
 5C'(\log|W|+1 +\log m)C(\log|W| +1) =  O((\log |W|)^2) .
 $$
Note that this bound is independent of $n$ because it follows from
Lemma~\ref{vk2}(b) that if $W$ is a product of $n'$ conjugates of words $R^{\pm 1}$, where $R=1$ is a relation of \eqref{pr1}, then it is possible to assume that $n' \le |W|$.

Furthermore, according to Savitch's theorem \cite{Sav}, see also \cite{AB}, \cite{PCC}, the existence of a nondeterministic algorithm that recognizes a language in space $S$ and time $T$ implies the  existence of a deterministic algorithm that recognizes the language in space $O(S \log T)$.
Therefore, by Savitch's theorem \cite{Sav}, there is a deterministic  algorithm that solves the bounded word  problem for presentation  \eqref{pr1} in space $O((\log |W|)^3)$.

To solve the precise word problem for presentation  \eqref{pr1},  suppose that we are given a  pair $(W, 1^n)$ and wish to find out if  $W$ is a product of $n$ conjugates of words $R^{\pm 1}$, where $R=1$ is a relation of \eqref{pr1}, and $n$ is minimal with this property. By Lemma~\ref{vk2}, we may assume that $n \le |W|$. Using the foregoing deterministic algorithm, we consecutively  check whether  the bounded word problem is solvable for the two pairs   $(W, 1^{n-1})$ and  $(W, 1^n)$. It is clear that   the precise word problem for the pair $(W, 1^n)$ has a positive solution if and only if  the bounded word problem has a negative solution for the pair $(W, 1^{n-1})$ and has a positive solution for the  pair $(W, 1^n)$ and that these two facts can be verified in deterministic space $O((\log |W|)^3)$.
\medskip

Now we describe an algorithm that solves the \PWPP\ for presentation \eqref{pr1} in polynomial time.
Our arguments are analogous to folklore arguments \cite{folk}, \cite{Ril16} that solve the precise word problem for presentation
$\langle \, a, b \,  \|  \,  a=1,  b=1 \,  \rangle$
in polynomial time and that are based on the method of dynamic programming.

For a word $U$ over $\A^{\pm 1}$ consider the following property.
\begin{enumerate}
\item[(E)] $U$ is nonempty and if $b_1, b_2 \in \A^{\pm 1}$ are the first, last, resp., letters of $U$ then
$b_1 \ne b_2^{-1}$ and $b_1 \ne b_2$.
\end{enumerate}

\begin{lem}\label{lema1}  Let $\D$ be a \ddd\ over presentation \er{pr1} such that $\D$ has property (A) and
the word $W \equiv \ph(\p |_0 \D)$ has property (E). Then there exists a factorization $\p |_0 \D = q_1 q_2$ such that
$|q_1|, |q_2| >0$ and $(q_1)_+ = (q_2)_- = \al(0)$. In particular,   $\ph(q_1)  \overset  {\GG_2} = 1$,  $\ph(q_2)  \overset  {\GG_2} = 1$.
\end{lem}

\begin{proof} Let $\p |_0 \D = e_1 q e_2$, where $e_1, e_2$ are edges, $q$ is a subpath of $\p |_0 \D$.

Suppose $e_1^{-1}$ is an edge of $\p |_0 \D$. Since $e_1^{-1} \ne e_2$ by property (E) of $W$, it follows that
$e_1^{-1}$ is an edge of $q$ and $q = r_1 e_1^{-1} r_2 $, where $r_1, r_2$ are subpaths of $q$.
Hence, $q_1 = e_1 r_1 e_1^{-1} $ and $q_2 = r_2 e_2$ are desired paths.

Hence, we may assume that $e_1^{-1}$ is an edge of the boundary path $\p \Pi_1$ of  a face $\Pi_1$.
Arguing in a similar fashion, we may assume that
$e_2^{-1}$ is an edge of the boundary path $\p \Pi_2$ of  a face $\Pi_2$.  If $\Pi_1 = \Pi_2$ then, in view of relations of the presentation \eqref{pr1}, we have $\ph(e_1) = \ph(e_2)$ contrary to property (E) of $W$.  Hence,  $\Pi_1 \ne \Pi_2$.

Since for every edge $f$ of $\p \Pi_1$ the edge $f^{-1}$ belongs to $\p  \D$ by property (A), we obtain the existence of a desired path $q_1$ of the form $q_1 = e_1$ if $|\p \Pi_1| = 1$ or $q_1 = e_1 r f^{-1}$ if $|\p \Pi_1| > 1$, where $r$ is a subpath of
$\p |_0 \D$, $f$ is the edge of $\p \Pi_1$ such that $e_1^{-1} f$ is a subpath of $\p \Pi_1$, $f_- = \al(0)$, and $|q_1| < |\p  \D |$.
\end{proof}

If $U  \overset  {\GG_2} = 1$, let $\mu_2(U)$ denote the integer such that the \PWPP\  for presentation \eqref{pr1}  holds for the pair  $(U, \mu_2(U))$. If $U  \overset  {\GG_2} \ne 1$, we set  $\mu_2(U) := \infty$.

\begin{lem}\label{lema2}  Let $U$ be a word such that  $U  \overset  {\GG_2} = 1$ and $U$ has property (E).
Then there exists a factorization $U \equiv U_1 U_2$ such that
$|U_1|, |U_2| >0$ and
$$
\mu_2(U)  = \mu_2(U_1) +\mu_2(U_2) .
$$
\end{lem}

\begin{proof} Consider a \ddd\  $\D$ over \er{pr1} such that $\ph(\p |_0 \D) \equiv U$ and $|  \D(2)| = \mu_2(U)$.
By Lemma~\ref{vk2}(b), we may assume that $\D$ has property (A). Hence, it follows from  Lemma~\ref{lema1} applied to $\D$ that
there is a  factorization $\p |_0 \D = q_1 q_2$ such that
$|q_1|, |q_2| >0$ and $(q_1)_+ = (q_2)_- = \al(0)$. Denote $U_i := \ph(q_i)$ and let $\D_i$ be the subdiagram of $\D$ bounded by $q_i$, $i=1,2$. Since $\D$ is a minimal \ddd\ for $U$, it follows that  $\D_i$ is a minimal \ddd\ for $U_i$, $i=1,2$.
Hence, $|\D_i(2) | = \mu_2(U_i)$ and
$$
\mu_2(U) = |\D(2) | = |\D_1(2) |+|\D_2(2) | = \mu_2(U_1)+\mu_2(U_2) ,
$$
as required.
\end{proof}

Let $U$ be a nonempty word over $\A^{\pm 1}$ and let $U(i,j)$, where $1 \le i \le |U|$ and $0 \le j \le |U|$, denote the subword of $U$ that starts
with the $i$th letter of $U$ and has length $j$. For example, $U = U(1, |U|)$ and  $U(i, 0)$ is the empty subword.

If $a_j \in \A$, let $|U|_{a_j}$ denote the total number of occurrences of letters $a_j, a_j^{-1}$ in $U$.
Note that we can decide whether $U  \overset  {\GG_2} = 1$ in time $O(|U|^2)$ by cancelling subwords $a_j^{\pm n_j}$,
$a_j^{-1} a_j$, $a_j a_j^{-1}$, and checking whether the word obtained by a process of such cancellations is empty.
\smallskip

Let  $W$ be a nonempty word over $\A^{\pm 1}$.  Define a parameterized word $W[i, j, k,\ell]$, where  $1 \le i \le |W|$, $0 \le j \le |W|$,
$1\le k \le m$, and $\ell$ is an integer that satisfies $|\ell | \le |W|_{a_k} -  |W(i,j)|_{a_k}$ so that
\begin{gather}\label{eeqq1}
 W[i, j, k, \ell] :=       W(i,j)a_k^\ell
\end{gather}
and  the word $W(i,j)a_k^\ell $ is not empty, i.e., $j+|\ell| \ge 1$.

Note that the total number of such parameterized words  $W[i, j, k,\ell]$ is bounded by $O(|W|^3)$, here $m = |\A |$ is a constant as the presentation \eqref{pr1} is fixed. Let  $\mathcal S_2(W)$ denote the set   of all parameterized words  $W[i, j, k,\ell]$. Elements $W[i, j, k,\ell]$ and $W[i', j', k',\ell']$ of $\mathcal S_2(W)$  are defined to be equal  \ifff\ the quadruples $(i, j, k,\ell)$, $(i', j', k',\ell')$ are equal. Hence, we wish to distinguish between elements of $\mathcal S_2(W)$ and actual words represented by elements of $\mathcal S_2(W)$.
It is clear that we can have $W(i,j)a_k^\ell \equiv W(i',j')a_{k'}^{\ell'}$ when  $W[i, j, k,\ell] \ne W[i', j', k',\ell']$.

If $U$ is the word represented by $W[i, j, k,\ell]$, i.e., $U \equiv W(i,j)a_k^\ell$, then we denote this by writing
$$
U  \overset \star = W[i, j, k,\ell] .
$$

We introduce a partial order on the set $\mathcal S_2(W)$ by setting
$$
W[i', j', k',\ell']  \prec W[i, j, k,\ell]
$$
if $j' < j$ or $j' = j$ and $|\ell'| < |\ell|$. In other words, we partially order the set $\mathcal S_2(W)$ by using the lexicographical order on pairs $(j, |\ell|)$ associated with elements $W[i, j, k,\ell]$ of $\mathcal S_2(W)$.

Define
$$
\mu_2(W[i, j, k,\ell]) := \mu_2(W(i,j)a_k^\ell) .
$$

To compute the number $\mu_2(W) = \mu_2(W[1, |W|, 1,0])$ in polynomial time, we use the method of dynamic programming
in which the parameter is  $(j, |\ell|)$. In other words, we compute the number
$\mu_2(W[i, j, k,\ell])$ by induction on parameter $(j, |\ell|)$.
The base of induction (or initialization) for $j + |\ell| = 1$ is obvious as  $\mu_2(W[i, j, k,\ell])$ is $0$ or $1$ depending on the presentation \er{pr1}.

To make the induction step, we assume that the numbers $\mu_2(W[i', j', k',\ell'])$ are already computed for all $W[i', j', k',\ell']$ whenever $(j', |\ell'|)   \prec (j, |\ell|)$ and we compute the number $\mu_2(W[i, j, k,\ell])$.

If  $W(i,j)a_k^\ell \overset  {\GG_2} \ne 1$, we set  $\mu_2( W[i', j', k',\ell'] ) := \infty$.

Assume that $W(i,j)a_k^\ell \overset  {\GG_2} = 1$.

Suppose that the word $W(i,j)a_k^\ell$ has property (E). Consider all possible factorizations for  $W(i,j)a_k^\ell$ of the form
\begin{gather}\label{eeqq2}
 W(i,j)a_k^\ell \equiv U_1 U_2 ,
\end{gather}
where $|U_1|, |U_2| \ge 1$. Let us show  that either of the words $U_1, U_2$ can be represented in a form
$W[i', j', k',\ell']$ so that  $W[i', j', k',\ell'] \prec W[i, j, k,\ell]$.

First suppose $|U_1|  \le |W(i,j)| =j$. If $i +|U_1| \le |W|$, we have
$$
U_1 \overset \star = W[i, |U_1|, 1,0] ,  \quad   U_2 \overset \star  =  W[i +|U_1| , j-|U_1|,k,\ell] .
$$
On the other hand, if $i +|U_1| = |W|+1$,  we have
$$
U_1 \overset \star  = W[i, |U_1|, 1,0] ,  \quad     U_2\overset \star   =  W[|W| , 0,k,\ell] .
$$

Now suppose  $|U_1|  > |W(i,j)| =j$. Let $\ell_1, \ell_2$ be integers such that $\ell_1, \ell_2, \ell$ have the same sign,
$ \ell = \ell_1+\ell_2$, and $j + |\ell| = |U_1| +\ell_2$. Then we have
$$
U_1 \overset \star  = W[i, |U_1|, k,\ell_1] , \quad   U_2 \overset \star   =  W[ 1,0 , k, \ell_2] .
$$

Note that the induction parameter $(j', |\ell'|)$ for indicated representations of $U_1, U_2$ is smaller than
the parameter $(j, |\ell|)$  for
the original parameterized word $W[i, j, k,\ell]$.    It follows from Lemma~\ref{lema2} applied to the word
$W(i,j)a_k^\ell$ that there is a factorization
$$
W(i,j)a_k^\ell \equiv U'_1 U'_2
$$
such that $|U'_1|, |U'_2| \ge 1$ and
$\mu_2(W(i,j)a_k^\ell) = \mu_2(U'_1) +  \mu_2(U'_2)$.
Hence, taking the minimum
\begin{gather}\label{minUU}
\min (\mu_2(U_1) +  \mu_2(U_2))
\end{gather}
over all factorizations \eqref{eeqq2}, we obtain the number $\mu_2(W(i,j)a_k^\ell)$.
\medskip

Assume that the word $W(i,j)a_k^\ell$ has no property (E), i.e.,   $W(i,j)a_k^\ell \equiv bUb^{-1}$ or $W(i,j)a_k^\ell \equiv b$ or $W(i,j)a_k^\ell \equiv bUb$, where $b \in \A^{\pm 1}$ and $U$ is a word.
\smallskip

First suppose that $W(i,j)a_k^\ell \equiv bUb^{-1}$. Note that the word $U$ can be represented in a form
$W[i', j', k',\ell']$ so that  $W[i', j', k',\ell'] \prec W[i, j, k,\ell]$. Indeed, if $|\ell | = 0$ then
$$
U \overset \star = W[i+1, j-2, k, 0] .
$$
On the other hand, if  $|\ell | > 0$ then
$$
U \overset \star = W[i+1, j-1, k,\ell_1] ,
$$
where $|\ell_1| = |\ell|-1$ and  $\ell_1 \ell \ge 0$.
Hence, the number  $\mu_2(U)$ is available by induction hypothesis. Since $\mu_2(W[i, j, k,\ell])  =  \mu_2(U)$, we obtain the required number $\mu_2(W[i, j, k,\ell])$.
\smallskip

In case  $W(i,j)a_k^\ell \equiv b$, the number  $\mu_2(W[i, j, k,\ell])$ is either one or
$\infty$ depending on the presentation \eqref{pr1}.
\smallskip

Finally, consider the case when  $W(i,j)a_k^\ell \equiv bUb$. Denote $b = a_{k_1}^\delta$, where $a_{k_1} \in \A$ and
$\delta = \pm 1$.

If $j=0$, i.e., $ W(i,j)a_k^\ell \equiv a_k^\ell$, then $k = k_1$ and
the number  $\mu_2(W[i, j, k,\ell])$  can be easily computed in time $O(\log |W|)$ as $|\ell | \le |W|$ and we can use binary representation for $\ell$.

Suppose  $j>0$. If $\ell  =0$ then $Ub^2  \overset \star = W[i+1, j-1, k_1,\delta]$.
If $|\ell | >0$ then $k=k_1$, $\ell $ and $\delta$ have the same sign, and  $Ub^2 \overset \star = W[i+1, j-1, k,\ell + \delta]$.

In either subcase $j=0$ or $j>0$, we obtain
\begin{align*}
 & Ub^2 \overset \star =  W[i', j', k',\ell'] ,  \quad  W[i', j', k',\ell'] \prec W[i, j, k,\ell] , \\
 &  \mu_2(W[i, j, k,\ell])  = \mu_2(W[i', j', k',\ell']) .
\end{align*}

This completes our inductive procedure of computation of numbers $\mu_2(W[i, j, k,\ell])$ for all $W[i, j, k,\ell] \in \mathcal S_2(W)$.

\smallskip

Since the length of every word $ W(i,j)a_k^\ell$ is at most $|W|$, it follows that our  computation of the number
$\mu_2(W[i, j, k,\ell])$ can be done in time $O(|W|\log|W| )$ including additions of binary representations of numbers $\mu_2(U_1), \mu_2(U_2)$ to compute the minimum \eqref{minUU}.
Since the cardinality  $| \mathcal S_2(W) | $ of the set $\mathcal S_2(W)$
is bounded by $O(|W|^3)$, we conclude that computation of numbers $\mu_2(W[i, j, k,\ell])$ for all words
$W[i, j, k,\ell] \in \mathcal S_2(W)$  can be done in deterministic time  $O(|W|^4\log|W|)$. This means that both the \BWPP\ and the \PWPP\ for presentation \eqref{pr1} can be solved in time $O(|W|^4\log|W|)$.

The proof of Theorem~\ref{thm1} is complete. \end{proof}

\begin{C1}   Let  $W$ be a word over $\A^{\pm 1}$ and $n \ge 0$ be an integer. Then the decision problems that inquire  whether the width $ h(W)$ or the spelling length $h_1(W)$  of $W$ is equal to $n$ belong to $\textsf{L}^3$ and $\textsf{P}$.  Specifically, the problems can be solved in  deterministic space $O((\log|W|)^3)$ or in deterministic time $O( |W|^4\log|W| )$.
\end{C1}

\begin{proof}  Observe that the decision problem asking whether the width (resp. the spelling length)  of a word $W$  is $n$ is equivalent to the \PWPP\  whose input is  $(W, n)$ for presentation
\eqref{pr1} in which $E_i = \mathbb N$   (resp. $E_i = \{ 1\}$) for every $i =1, \dots, m$. Therefore, a reference to  proven Theorem~\ref{thm1} shows that the  problem asking whether the  width (resp. the spelling length)  of $W$ is $n$  belongs to both  ${\textsf L}^3$ and ${\textsf P}$ and yields the required space and time bounds for this problem.
\end{proof}

\section{Calculus of Brackets for Group Presentation \eqref{pr3b}}

As in Theorem~\ref{thm2},  consider a group presentation
\begin{equation*}\tag{1.4}
\GG_3 = \langle \, a_1  , a_2  , \dots,   a_m \, \| \, a_2 a_1^{n_1} a_2^{-1} = a_1^{n_2}  \rangle ,
\end{equation*}
where $n_1, n_2$ are nonzero integers.
Suppose $W$ is a nonempty word  over $\A^{\pm 1}$ such that  $W \overset {\GG_3 } = 1$, $n \ge 0$ is an integer and   $W$ is a product of at most $n$ conjugates of the words $(a_2 a_1^{n_1} a_2^{-1} a_1^{-n_2})^{\pm 1}$. Then,
as follows from arguments of the proof  of Lemma~\ref{vk}, there exists
a disk diagram $\D$ over presentation  \eqref{pr3b} such that $\ph(\p |_0 \D) \equiv W$  and   $| \D(2) | \le n$. Without loss of generality, we may assume that $\D$ is reduced.  Consider the word
$\rho_{a_1}(W)$ over  $\A^{\pm 1}  \setminus \{ a_1, a_1^{-1} \}$ which is obtained from $W$ by erasing all occurrences of $a_1^{\pm 1}$ letters.

An oriented edge $e$ of $\D$ is called an {\em $a_i$-edge} if $\ph(e) = a_i^{\pm 1}$.
We also consider a tree $\rho_{a_1}(\D)$  obtained from $\D$ by  contraction of  all $a_1$-edges of $\D$  into points and subsequent identification of all pairs of edges $e, f$ such that  $e_- = f_-$ and $e_+ = f_+$. If an edge $e'$ of $\rho_{a_1}(\D)$  is obtained from an edge
$e$ of $\D$ this way,  we set $\ph(e') := \ph(e)$.  It is easy to check that this definition is correct.  This labeling function turns the tree $\rho_{a_1}(\D)$ into a disk diagram over presentation   $\langle \,  a_2 , \dots , a_m   \, \| \, \varnothing  \, \rangle $.

Two faces $\Pi_1, \Pi_2$ in a disk diagram $\D$ over  \eqref{pr3b} are termed {\em related}, denoted  $\Pi_1 \leftrightarrow  \Pi_2$, if there is an  $a_2$-edge $e$ such that $e \in \p \Pi_1$ and  $e^{- 1} \in \p \Pi_2$.  Consider the minimal equivalence relation $\sim_2$ on the set of faces of $\D$ generated by the relation $\leftrightarrow$. An {\em $a_i$-band}, where $i >1$,  is a minimal subcomplex $\Gamma$ of  $\D$ ($\Gamma$ is not necessarily simply connected) that contains an $a_i$-edge $f$, and, if there is a face $\Pi$ in $\D$ with $f \in (\p \Pi)^{\pm 1}$, then $\Gamma$ must contain all faces of the equivalence class
$[\Pi]_{\sim_2}$ of $\Pi$. Hence, an $a_i$-band  $\Gamma$  is either a subcomplex   consisting of a single
nonoriented edge, denoted $\{ f, f^{- 1} \}$, where $f$ is an $a_i$-edge, $i >1$,  and   $f, f^{- 1} \in \p \D$, or $\Gamma$  consists of all faces of an equivalence class $[\Pi]_{\sim_2}$ when $i=2$. Clearly,  the latter case is possible only if $i =2$. In this latter case, if $\Gamma$ contains faces but $\Gamma$ has no face in $[\Pi]_{\sim_2}$ whose boundary contains an $a_2$-edge $f$ with $f^{-1} \in \p \D$,  $\Gamma$ is called a  {\em closed} $a_2$-band.

\begin{lem}\label{b1}   Suppose that $\D$ is a reduced disk diagram over presentation  \er{pr3b}. Then there are no closed $a_2$-bands in $\D$ and every $a_i$-band $\Gamma$ of $\D$, $i >1$, is a disk subdiagram of $\D$ such that
$$
\p |_{(f_1)_-} \Gamma = f_1 s_1 f_2 s_2 ,
$$
where $f_1, f_2$ are edges of $\p \D$ with $\ph(f_1) = \ph(f_2)^{-1} = a_i$,
$s_1, s_2$ are simple paths with $\ph(s_1) \equiv a_1^{k n_1}$, $\ph(s_2) \equiv a_1^{-k n_2}$ for some integer $k \ge 0$, $\Gamma $ contains $k$ faces, and $\p \Gamma$ is a simple closed path when $k >0$.
\end{lem}

\begin{proof}  Since  $\D$ is reduced, it follows that if two faces $\Pi_1, \Pi_2$ are related by the relation  $\sim$, then $\Pi_1$ is not a mirror copy of $\Pi_2$, i.e.,  $\ph (\Pi_1 ) \not\equiv  \ph (\Pi_2 )^{-1}$.

Assume that there is a closed $a_2$-band  $\Gamma_0$ in $\D$.  Then it follows from the above remark that there is a disk subdiagram  $\D_0$ of $\D$, surrounded by $\Gamma_0$,  such that $\ph(\D_0) \equiv a^{k_0 n_i}_1$, where $i =1$ or $i=2$ and $k_0 = | \Gamma_0(2) | >0$. This, however, is a contradiction to the \MF , see \cite{LS}, \cite{MKS}. Recall that the \MF\   for a one-relator group presentation
$$
\GG_0 =  \langle \, \A \, \| \,  R=1  \, \rangle
$$
claims that every letter $a \in \A^{\pm 1}$, that appears in a cyclically reduced word $R$ over $\A^{\pm 1}$, will also appear in $W^{\pm 1}$ whenever $W$ is reduced and $W \overset {\GG_0} = 1$.

Hence, no $a_2$-band $\Gamma$ in $\D$ is closed and,  therefore, $\p \Gamma = f_1 s_1 f_2 s_2$ as described in Lemma's statement.    If $k >0$ and $\p \Gamma$ is not a simple closed path, then a proper subpath of $\p \Gamma$ bounds a disk subdiagram    $\D_1$ such that $\ph(\D_1)$ is a proper subword of $\ph(  \p \Gamma ) \equiv a_2 a_1^{k n_1}  a_2^{-1}  a_1^{-k n_2}$, where $k >0$.  Then it follows from Lemma~\ref{vk} that either $a_1^{\ell}  \overset {\GG_3} = 1 $, where $\ell \ne 0$, or
$a_1^{\ell'} a_2^\e  \overset {\GG_3} = 1 $, where $\e = \pm 1$. The first equality contradicts \MF\ and the second one is impossible in the abelianization of $\GG_3$.
\end{proof}

The factorization  $\p |_{(f_1)_-} \Gamma = f_1 s_1 f_2 s_2$ of Lemma~\ref{b1} will be called the {\em standard boundary} of an $a_i$-band $\Gamma$, where $i >1$. Note that $|s_1| =|s_2| =0$ is the case when  $k =0$, i.e.,  $\Gamma$ contains no faces and   $\p  \Gamma = f_1 f_2$ with $f_2^{-1} = f_1$.

An alternative way of construction  of the tree  $\rho_{a_1}(\D)$ from $\D$  can now be described as  contraction of the paths $s_1, s_2$ of the standard boundary of every  $a_2$-band $\Gamma$ of $\D$ into points, identification of the edges $f_1, f_2^{-1}$ of $(\p \Gamma)^{\pm 1}$, and contraction into points of all $a_1$-edges $e$ with $e \in \p\D$.

Let $s$ be a path in a disk diagram $\D_0$ over presentation \er{pr3b}. We say that $s$ is a {\em special arc} of $\D_0$ if either $|s| = 0$ or $|s| > 0$, $s$ is a reduced simple path, and $s$ consists entirely of $a_1$-edges. In particular, if $s$ is a special arc then  every subpath of $s, s^{-1}$ is also a special arc, and $s_- \ne s_+$ whenever $|s| > 0$.

\begin{lem}\label{arc}  Suppose that $\D_0$ is a reduced disk diagram  over presentation \er{pr3b} and $s, t$ are
special arcs in $\D_0$ such that $s_+ = t_-$.  Then there are factorizations $s = s_1 s_2$ and $t = t_1 t_2$ such that $s_2 = t_1^{-1}$ and $s_1 t_2$ is a special arc, see Fig.~6.1.
\end{lem}

\begin{center}
\begin{tikzpicture}[scale=.57]
\draw [-latex](-3,-3) --(-1.4,-3);
\draw (-3,-3) --(3,-3);
\draw (1,-3) --(1,-1);
\draw  [-latex](1,-3) --(1,-1.7);
\draw  [-latex](1,-3) --(2,-3);
\node at (-0.25,-1.95) {$s_2 = t_1^{-1}$};
\draw  (-3,-3) [fill = black] circle (.06);
\draw  (1,-3) [fill = black] circle (.06);
\draw  (3,-3) [fill = black] circle (.06);
\draw  (1,-1) [fill = black] circle (.06);
\node at (-1.4,-3.6) {$s_1$};
\node at (1.9,-3.6) {$t_2$};
\node at (0.2,-4.8) {Fig.~6.1};
\end{tikzpicture}
\end{center}

\begin{proof}  First we observe that if $\D_1$ is a disk subdiagram of $\D_0$  such that every edge of $\p \D_1$ is an $a_1$-edge, then it follows from Lemma~\ref{b1}  that $\D_1$ contains no faces, i.e., $ | \D_1(2) | = 0$, and so $\D_1$ is a tree.

We now prove this Lemma by induction on the length $| t| \ge 0$.
If $| t| = 0$ then our claim is true with  $s_1 = s$ and $t_2 = t$. Assume that  $| t| >  0$ and $t = re$, where $e$ is an edge. Since $r$ is also a special arc with $|r| < |t|$,  the induction hypothesis applies to  the pair $s, r$ and yields factorizations $s = s'_1 s'_2$ and $r= r_1 r_2$ such that $s'_2 = (r_1)^{-1}$ and $s'_1 r_2$ is a special arc.
 If $s'_1 r_2 e$ is also  a special arc, then the factorizations $s = s'_1 s'_2$ and  $t = t_1 t_2$, where $t_1 = r_1$ and  $t_2 = r_2 e$, have the desired property and the induction step is complete.

Hence, we may assume that  $s'_1 r_2 e$ is not a special arc. Note that if $|r_2| > 0$, then the path $s'_1 r_2 e$ is reduced because  $s'_1 r_2 $ and $r_2 e$
 are special arcs.  Assume that  $s'_1 r_2 e$ is not reduced. Then $|r_2| = 0$ and the last edge of $s_1'$ is $e^{-1}$. Denote
 $s'_1 =s'_{11} e^{-1}$. Then, letting   $s_1 := s'_{11}$,
 $s_2 := t^{-1}$ and $t = t_1 t_2$, where $t_1 := t$ and  $| t_2 | = 0$, we have $s_2 = t^{-1}$ and $s_1 t_2$ is a desired special arc. Thus, we may assume that the path  $s'_1 r_2 e$ is reduced but not simple. Note that the path  $s'_1 r_2 e$
 consists entirely of $a_1$-edges. Since
 $s'_1 r_2$ and $ r_2e$ are simple, it follows that $e_+$ belongs to $s_1'$ and defines a factorization $s'_1 = s_{11} s_{12}$, where $| s_{12}| >0$. Hence, the path $s_{12} r_2 e$ is a  simple closed path  which bounds a disk subdiagram  $\D_1$ whose boundary $\p \D_1$ consists of $a_1$-edges.  By the observation made above,
 $\D_1$ is a tree, which, in view of the fact that $\p \D_1$ is a  simple closed path, implies that
 $\p \D_1 = e^{-1} e$. Therefore, $| r_2 | = 0$ and  $ s_{12} =  e^{-1}$. This means that the path $s'_1 r_2 e$ is not reduced, contrary to our assumption. This contradiction completes the proof.
\end{proof}

Let $U$ be a word over $\A^{\pm 1}$.
If $a \in \A$ is a letter, then the {\em $a$-length } $|U|_a$ of  $U$ is the number of occurrences of $a$ and $a^{- 1}$ in $U$. We also define  the {\em complementary  $a$-length } by  $|U|_{\bar a} := |U| - |U|_{a}$.

From now on we assume, unless stated otherwise,  that $W$ is a nonempty word over $\A^{\pm 1}$ such that $W \overset{\GG_3} = 1$ and $\D$ is a reduced disk diagram over presentation \er{pr3b}
such that $\ph( \p |_0 \D) \equiv W$.

The definitions of the path $P_W$, of the map $\al : P_W \to \D$ and those of related notions, given in Sect.~2, are retained  and are analogous to those used in Sects.~4--5.

\begin{lem}\label{b2} Suppose that  $|W|_{\bar a_1} >2$. Then there exist
vertices $v_1, v_2 \in P_W$  such that $v_1 < v_2$ and if $P_W(\ff, v_1, v_2) = p_1 p_2 p_3$, then the following hold true.
There is a special arc $r$ in $\D$ such that $r_- = \al(v_1)$, $r_+ = \al(v_2)$,  $\ph(r) \equiv a_1^{\ell}$ for some $\ell$,  and
 \begin{equation*}
    \min(|\ph(p_2)|_{\bar a_1},  |\ph(p_1)|_{\bar a_1}+|\ph(p_3)|_{\bar a_1}) \ge \tfrac 13 |\ph(\p |_0 \D)|_{\bar a_1} =  \tfrac 13  |W|_{\bar a_1}  .
 \end{equation*}
In addition, if   $| \vec \D(1) |_{a_1} :=  |W|_{a_1} + (|n_1| + |n_2|)|\D(2)| $ denotes the number of $a_1$-edges of $\D$, then    $|r|  \le  \tfrac 12  | \vec \D(1) |_{a_1}$.
\end{lem}

\begin{proof} Consider the tree $\rho_{a_1}(\D)$ constructed from $\D$ as above by collapsing $a_1$-edges $e$ into points and subsequent identification of edges $f, g$ with $f_- = g_-$, $f_+ = g_+$. Then
 $\rho_{a_1}(\D)$ is a disk diagram over \eqref{pr3b} with no faces and there is a surjective  continuous cellular map
 $$
 \beta_1 : \D \to \rho_{a_1}(\D)
 $$
 which preserves labels of edges $e$ with $\ph(e) \ne a_1^{\pm 1}$ and sends $a_1$-edges into vertices.
 It is also clear that if
 $$
 \al_{1} : P_{\rho_{a_1}(W)} \to \rho_{a_1}(\D)
 $$
 is the corresponding to the pair  $(\rho_{a_1}(W), \rho_{a_1}(\D))$ map, then there is a cellular continuous surjective map
 $$
 \beta : P_W \to    P_{\rho_{a_1}(W)}
 $$
which preserves labels of edges $e$ with $\ph(e) \ne a_1^{\pm 1}$, sends
$a_1$-edges into vertices, and makes the following diagram commutative.

$$
\renewcommand{\arraystretch}{1.5}
\begin{array}{ccccc}
P_W  &\xrightarrow{\hskip0.8cm {\al} \hskip0.6cm} & \D \\
\Big\downarrow\rlap{$\beta$} &&\Big\downarrow\rlap{$\beta_1$}&\\
 P_{\rho_{a_1}(W)}   &\xrightarrow{\hskip0.8cm \al_1 \hskip0.6cm}&\rho_{a_1}(W)\\
\end{array}
$$

By Lemma~\ref{lem1} applied to $\rho_{a_1}(\D)$, there are vertices
$u_1, u_2 \in    P_{\rho_{a_1}(W)}$  such that  $\al_1(u_1) = \al_1(u_2)$ and if $P_{\rho_{a_1}(W)}(\ff, u_1, u_2)  = s_1 s_2 s_3$, then
\begin{equation}\label{in7}
\min(|s_2|,  |s_1|+|s_3|) \ge \tfrac 13 | \p \rho_{a_1}(\D)| =  \tfrac 13  |W|_{\bar a_1} \ .
\end{equation}
It follows from definitions and Lemma~\ref{b1} that there are some vertices $v_1, v_2 \in P_W$ such that $\beta(v_j) = u_j$, $j=1,2$,  the vertices  $\al(v_1), \al(v_2)$ belong to $\p\Gamma$, where $\Gamma$ is an $a_i$-band in $\D$, $i >1$, and the vertices  $\al(v_1), \al(v_2)$  can be connected along $\p \Gamma$  by a simple reduced  path $r$ such that $\ph(r) \equiv a_1^{\ell}$ for some integer $\ell$, where $\ell = 0$ if $i >2$. Clearly, $r$ is a special arc of $\D$.
 Denote
 $P_{W}(\ff, v_1, v_2)  = p_1 p_2 p_3$. Then it follows from definitions that $| \ph(p_i) |_{\bar a_1} = |s_i|$, $i=1,2,3$, hence the inequality \eqref{in7} yields that
 \begin{equation*}
\min(  | \ph(p_2) |_{\bar a_1},   | \ph(p_1) |_{\bar a_1} +  | \ph(p_3) |_{\bar a_1} )   \ge \tfrac 13 | \p \rho_{a_1}(\D)| =  \tfrac 13  |W|_{\bar a_1}  ,
\end{equation*}
as required.

Furthermore,  since $r$ is simple, reduced  and $r_- \ne r_+$ unless $|r| =0$,  the path $r$ contains every $a_1$-edge
$e$ of $\D$ at most once and if $r$ contains $e$ then $r$ contains no $e^{-1}$.
Since the total number $| \vec \D(1) |_{a_1}$ of  $a_1$-edges of $\D$  is equal to
$$
| \vec \D(1) |_{a_1} = |\ph(\p \D)|_{a_1} + (|n_1| + |n_2|)|\D(2)| ,
$$
we obtain  the desired inequality
$$
|r| = |\ph(r)|_{a_1}   \le \tfrac 12  | \vec \D(1) |_{a_1} =  \tfrac 12 |W|_{a_1} + \tfrac 12(|n_1| + |n_2|)|\D(2)| .
$$
\end{proof}

A quadruple
 $$
 b = (b(1), b(2), b(3), b(4))  ,
 $$
of  integers  $b(1), b(2), b(3), b(4)$ is called a   {\em bracket }  for the pair $(W, \D)$
if  $b(1), b(2)$ satisfy the inequalities $0 \le b(1)\le b(2) \le |W|$, and, in the notation
$$
P_W(\ff, b(1), b(2)) = p_1p_2p_3 ,
$$
the following holds true. There exists a special arc $r(b)$ in $\D$ such that $r(b)_- = \al(b(1))$,  $r(b)_+ = \al(b(2))$, and $\ph(r(b)) \overset 0 = a_1^{b(3)}$.
Furthermore,  if $\D_b$ is the disk subdiagram of $\D$  defined by $\p |_{b(1)} \D_b =   \al(p_2) r(b)^{-1}$
(such $\D_b$ is well defined as $r(b)$ is a special arc in $\D$),  then $|\D_b(2)| = b(4)$, see Fig.~6.2.
\vskip 2.8mm

\begin{center}
\begin{tikzpicture}[scale=.72]
\draw (-19,-2)  arc (0:360:2);
\draw(-23,-2) -- (-19,-2);
\draw [-latex](-23,-2) --(-21,-2);
\draw [-latex](-21,0) --(-20.95,0);
\draw [-latex](-22.5,-3.3) --(-22.6,-3.2);
\draw [-latex] (-19.5,-3.3) -- (-19.6,-3.4);
\draw  (-23,-2) [fill = black] circle (.05);
\draw  (-19,-2) [fill = black] circle (.05);
\draw  (-21,-4) [fill = black] circle (.05);
\node at (-21,-4.8) {Fig.~6.2};
\node at (-21,-1.4) {$\Delta_b$};
\node at (-21,-2.5) {$r(b)$};
\node at (-21.,-.5) {$\alpha(p_2)$};
\node at (-23.4,-3.3) {$\alpha(p_1)$};
\node at (-18.6,-3.4) {$\alpha(p_3)$};
\node at (-21,-3.5) {$\alpha(0)$};
\node at (-24,-2) {$\alpha(b(1))$};
\node at (-18,-2) {$\alpha(b(2))$};
\end{tikzpicture}
\end{center}

This disk  subdiagram  $\D_b$  of $\D$, defined by $\p |_{b(1)} \D_b =   \al(p_2) r(b)^{-1}$,  is associated with the bracket $b$. The boundary subpath $\al(p_2)$ of $\D_b$ is denoted $a(b)$ and called the {\em arc} of the bracket $b$.
The path  $r(b)$ is  termed the {\em complementary arc} of $b$.

For example, $b_v = (v, v, 0, 0)$ is a bracket  for every vertex $v$ of $P_W$, called a {\em  starting} bracket at $v = b(1)$. Note that $a(b) = \al(v)$,   $r(b) = \al(v)$, and  $\D_b= \{ \al(v)  \}$.

The {\em final}  bracket for   $(W, \D)$ is
$c = (0, |W|, 0, | \D(2) | )$. Observe that  $a(c) =  \p |_{0} \D$,    $r(c) =  \{ \al (0) \}$ and  $ \D_c = \D$.
\medskip

Let $B$ be a finite set of brackets  for the pair   $(W, \D)$, perhaps, $B$ is empty.   We say that $B$ is a {\em bracket system}   for $(W, \D)$ if,  for every pair $b, c \in B$ of distinct brackets,  either $b(2) \le c(1)$  or $c(2) \le b(1)$.
\medskip

Now we describe three  kinds of  elementary operations over a \BS\   $B$ for the pair   $(W, \D)$: additions, extensions,  and mergers.
\medskip

\textit{Additions}.

Suppose $b$ is a starting bracket, $b \not\in B$, and $B \cup \{ b\}$
is a \BS  .
Then we may add $b$ to $B$ thus making an {\em addition} operation over $B$.
\medskip

\textit{Extensions}.

Suppose $B$ is a \BS ,
$b \in B$ is a bracket  and $e_1 a(b) e_2$ is a subpath of the boundary path $\p |_{0} \D$, where $a(b)$ is the arc of $b$ and $e_1, e_2$ are edges one of which could be missing.

Assume  that $\ph(e_1)= a_1^\e$, where $\e = \pm 1$. Since $e_1$ and $r(b)$ are special arcs of $\D$ and $e_1 r(b)$ is a path, Lemma~\ref{arc} applies and yields that either   $e_1 r(b)$ is a special arc or
$r(b) = e^{-1}_1 r_1$, where $r_1$ is a subpath of $r(b)$ and $r_1$ is a special arc. In the first case,  define $r(b') := e_1 r(b)$.   In the second case, we set $r(b') := r_1$. Note that, in either case,  $r(b')$ is a special arc  and  $\ph( r(b')  ) \overset 0  = a_1^{b(3)+\e}$.   In either case, we
consider a  bracket $b'$ such that
$$
b'(1) = b(1)-1 , \ b'(2) = b(2) ,  \ b'(3) = b(3)+\e ,  b'(4) = b(4) .
$$
In either case, we have  that $a(b') = e_1 a(b)$, $r(b')$ is defined as above,  and
$\D_{b'} $ is the disk subdiagram whose boundary is $\p \D_{b'} = a(b') r(b')^{-1}$.
We say that $b'$ is obtained from $b$ by an extension of type 1 (on the left).
If $(B \setminus \{ b \}) \cup \{ b' \}$ is a \BS , then
replacement of $b \in B$ with $b'$ in $B$ is called an {\em extension } operation over $B$ of type 1.

Similarly, assume  that $\ph(e_2)= a_1^\e$, where $\e = \pm 1$.
 Since  $r(b)$  is a simple path, it follows as above from Lemma~\ref{arc} that either  $r(b)e_2 $ is a special arc  or
$r(b) = r_2 e^{-1}_2$, where $r_2$ is a subpath of $r(b)$ and $r_2$ is a special arc.
In the first case, define $r(b') := r(b)e_2$.  In the second case, we set $r(b') := r_2$. Note that, in either case,   $r(b')$ is a special arc and  $\ph( r(b')  ) \overset 0  = a_1^{b(3)+\e}$.      In either case, we
consider a  bracket $b'$ such that
$$b'(1) = b(1) , \ b'(2) = b(2)+1 , \ b'(3) = b(3)+\e, \ b'(4) = b(4) .
$$
In either case, we have  that $a(b') = a(b)e_2 $, $r(b')$ is defined as above,  and
$\D_{b'} $ is the disk subdiagram whose boundary is $\p \D_{b'} = a(b') r(b')^{-1}$.
We say that $b'$ is obtained from $b$ by an extension of type 1 (on the right).
If $(B \setminus \{ b \}) \cup \{ b' \}$ is a \BS , then
replacement of $b \in B$ with $b'$ in $B$ is called an {\em extension } operation over $B$ of type 1.

Now suppose that  $\ph(e_1) = \ph(e_2)^{-1}=a_2^{\e}$,   $\e = \pm 1$,   and there is an $a_2$-bond $\Gamma$ whose standard boundary is $\p |_{(e_i)_-}   \Gamma = e_i  s_i  e_{3-i}  s_{3-i}$, where $i =1$ if $\e =1$ and $i =2$ if $\e =-1$. Recall that the standard boundary of an $a_2$-bond starts with
an edge $f$ with $\ph(f) = a_2$.

First we assume that  $|b(3)|\ne 0$, i.e., $\ph(s_i) \overset 0 \neq 1$.
By Lemma~\ref{b1}, the paths $s_1, s_2$ are special arcs.  Moreover, $\Gamma$ consists of
$|b(3)|/|n_{i}|$ faces. Recall that $\D$ is  reduced.  It follows from Lemma~\ref{arc} applied to special arcs $r(b)$ and $s_i^{-1}$ that $r(b) = s_i$. Consider a bracket $b'$ such that $b'(1) = b(1)-1$, $b'(2) = b(2)+1$,  $b'(3) = (|n_1^{-1} n_2|)^{\e} b(3)$, and $b'(4) = b(4) + |b(3)|/|n_{i}|$.  Note that  $a(b') = e_1 a(b)e_2$ and   $r(b) = s_{i+1}^{-1}$.
We say that $b'$ is obtained from $b$ by an extension of type 2.
If $(B \setminus \{ b \}) \cup \{ b' \}$ is a \BS , then
replacement of $b \in B$ with $b'$ in $B$ is called an {\em extension } operation over $B$ of type 2.

Assume that $e_1 = e_2^{-1}$ and    $\ph(e_1) \ne a_1^{\pm 1}$.
Then we may consider a  bracket $b'$ such that $b'(1) = b(1)-1$,
$b'(2) = b(2)+1$,  $b'(3) = b(3)$, and $b'(4) = b(4)$.  Since
the path $a(b)$ is closed, it follows from Lemma~\ref{vk} that  $\ph(a(b))   \overset {\GG_3} = 1$  and so  $b'(3) = b(3)=0$ by \MF . Note that  $a(b') = e_1 a(b)e_2$ and  $| r(b')| =0$.
We say that $b'$ is obtained from $b$ by an extension of type 3.
If $(B \setminus \{ b \}) \cup \{ b' \}$ is a \BS , then
replacement of $b \in B$ with $b'$ in $B$ is called an {\em extension } operation over $B$ of type 3.
\medskip

\textit{Mergers}.

Suppose that $b_1, b_2$ are distinct brackets in $B$ such that  $b_1(2) = b_2(1)$. Consider the disk diagram $\D_{b_i}$, associated with the bracket $b_i$,  $i=1,2$,  and let $\p |_{b_i(1)} \D_{b_i} =   a(b_i) r(b_i)^{-1}$ be the boundary of $\D_{b_i}$, where $a(b_i), r(b_i)$ are the arcs of $b_i$, $i=1,2$.

Since $r(b_1), r(b_2)$  are special arcs with $r(b_1)_- =  r(b_2)_+$,  it follows from Lemma~\ref{arc}  that
there are factorizations $r(b_1) = r_1 r_0$,  $r(b_2) = r_0^{-1} r_2$ such that the path $r(b') := r_1 r_2$ is a special arc of $\D$. Note that $r(b')_- = \al(b_1(1))$,  $r(b')_+ = \al(b_2(2))$,  and the disk subdiagram $\D_{0}$ of $\D$ defined by $\p \D_{0} = a(b_1)a(b_2) r(b')^{-1}$  contains $b_1(4) + b_2(4)$ faces.     Therefore, we may consider a bracket $b'$ such that $b'(1) = b_1(1)$, $b'(2) = b_2(2)$, $b'(3) =  b_1(3) + b_2(3)$, and  $b'(4) =  b_1(4) + b_2(4)$. Note that $a(b') = a(b_1)a(b_2)$,  $r(b') = r_1 r_2$,  and $\D_{b'} = \D_0$.
We say that $b'$ is obtained from $b_1, b_2$ by a merger operation.
If $(B \setminus \{ b_1, b_2 \}) \cup \{ b' \}$ is a \BS , then taking both $b_1$, $b_2$ out of $B$ and putting  $b'$ in $B$ is a {\em merger } operation over $B$.
\medskip

We will say that additions, extensions,  and mergers, as defined above, are {\em  \EO s} over brackets and \BS  s for $(W, \D)$.

Assume that one \BS\   $B_\ell$ is obtained from another \BS\    $B_0$
by a finite sequence $\Omega$ of \EO s  and $B_0,  B_1, \dots, B_\ell$ is the corresponding to $\Omega$ sequence of \BS  s. As before,  such a sequence $B_0,  B_1,  \dots, B_\ell$ of \BS  s is called {\em operational}.   We  say that the sequence     $B_0,  B_1,  \dots, B_\ell$ has {\em size bounded by}
$(k_1, k_2)$  if $\ell \le k_1$ and, for every $i$,  the number of brackets in $B_i$ is at most $k_2$.
Whenever it is not ambiguous, we also say that $\Omega$ has size bounded by $(k_1, k_2)$ if so does the corresponding to $\Omega$ sequence  $B_0,  B_1, \dots, B_\ell$ of \BS s.
\medskip

We now study properties of brackets and \BS s.

\begin{lem}\label{ca} If $b$ is a bracket for the pair $(W, \D)$, then  $| b(3) |  \le   \tfrac 12  | \vec \D(1) |_{a_1}$.
\end{lem}

\begin{proof} By the definition of a bracket,  the complementary arc $r(b)$ of $b$   is a special arc in $\D$, hence,
either $|r(b)| = 0$ or, otherwise, $|r(b)| > 0$ and $r(b)$ is
a simple, reduced path consisting entirely of $a_1$-edges. Since the total number of $a_1$-edges in $\D$ is  $| \vec \D(1) |_{a_1}$, it follows   that $| b(3) |  \le | r(b) | \le \tfrac 12   | \vec \D(1) |_{a_1}$.
\end{proof}

\begin{lem}\label{NUP}  Suppose that $b$, $c$ are two brackets for the pair $(W, \D)$ and $b(1) = c(1)$, $b(2) = c(2)$. Then $b =c$.
\end{lem}

\begin{proof}  Consider the complementary arcs $r(b)$, $r(c)$ of brackets $b$, $c$, resp. Since
$r(b)_- = r(c)_-$, $r(b)_+ = r(c)_+$, and $r(b)$, $r(c)^{-1}$ are special arcs of $\D$, it follows from Lemma~\ref{arc} applied to special arcs $r(b)$, $r(c)^{-1}$ that
$r(b) = r(c)$. This means that $b(3) = c(3)$, $\D_b = \D_c$ and $b(4) = c(4)$.  Therefore, $b =c$, as desired.
\end{proof}

\begin{lem}\label{lem3b} There exists a sequence of \EO s  that converts the empty
\BS\   for $(W, \D)$ into the final \BS\    and has size  bounded by $(4|W|, |W|)$.
\end{lem}

\begin{proof} For every $k$ with $0 \le k \le |W|$, consider a starting bracket $(k,k,0,0)$ for
 $(W, \D)$. Making $|W|+1$ additions, we get a \BS\
\begin{equation*}
 B_{W} = \{ (k,k,0,0) \mid  0 \le k \le |W| \}
\end{equation*}
of $|W|+1$ starting brackets. Now, looking at $\D$, we can easily find  a sequence of extensions and mergers that  converts  $B_{W}$ into the final \BS\  $B_{\ell}$.   To estimate the total number  of
\EO s, we note that the number of additions is $|W|+1$.  The number of extensions is at most $|W|$ because every extension applied to a \BS\ $B$ increases the number
$$
\eta(B) := \sum_{b \in B} (b(2)-b(1) )
$$
by 1 or 2 and $\eta( B_{W}) = 0$, while $\eta( B_{\ell}) = |W|$. The  number  of mergers is $ |W|$ because the number of brackets $|B|$ in a \BS\  $B$ decreases by 1 if $B$ is obtained by a merger and
$|  B_{W}  | =  |W|+1$, $|B_{\ell} | =  1$.  Hence,
$  \ell \le 3|W|+1 \le 4|W|$, as required.
\end{proof}

\begin{lem}\label{lem4b}  Suppose there is  a sequence $\Omega$  of \EO s  that converts the empty
\BS\  $E$  for $(W, \D)$ into the final \BS\  $F$  and has size  bounded by $(k_1, k_2)$. Then there is also a sequence  of \EO s  that  changes $E$ into $F$ and has size bounded by $(7|W|, k_2)$.
\end{lem}
\begin{proof} Assume that the sequence $\Omega$ has an addition operation which introduces a starting bracket $c = (k,k,0,0)$ with $0 \le k \le |W|$. Since the final \BS\ contains no starting bracket, $c$ must disappear and an \EO\ is applied to $c$. If a merger is applied to $c$ and another bracket $b$ and the merger yields $\wht c$, then $\wht c = b$. This means that the addition of $c$ and the merger could be skipped without affecting the sequence otherwise. Note that the size of the new sequence $\Omega'$ is  bounded by $(k_1-2, k_2)$. Therefore, we may assume that no merger is applied to a starting bracket in $\Omega$. We also observe  that, by the definition, an extension of type 2 is not applicable to a starting bracket (because of the condition $b(3) \ne 0$).

Thus, by an obvious induction on $k_1$,   we may assume that, for every starting bracket $c$ which is added by $\Omega$, an extension of type 1 or 3 is applied to $c$.

Now we will show that, for every starting bracket $c$
 which is added by $\Omega$, there are at most 2 operations of addition of $c$ in $\Omega$. Arguing on the contrary, assume that  $c_1, c_2, c_3$ are brackets equal to $c$ whose additions are used in $\Omega$. By the above remark,
for every $i=1,2,3$, an extension of type 1 or 3 is applied to $c_i$, resulting in a bracket $\wht c_i$.

Let  $c_1, c_2, c_3$  be listed in the order in which the brackets $\wht c_1, \wht c_2, \wht c_3$
are created by $\Omega$. Note that if $\wht c_1$ is obtained from $c_1$ by an extension of type 3, then
$$
\wht c_1(1) = c_1(1)-1, \quad  \wht c_1(2) = c_1(2)+1 .
$$
This means that brackets
$\wht c_2$, $\wht c_3$ could not be created by extensions after  $\wht c_1$ appears, as  $b(2) \le c(1) $
or $b(1) \ge c(2) $ for distinct brackets $b, c \in B$ of any \BS\ $B$. This observation proves that $\wht c_1$ is obtained from $c_1$ by an extension of type 1.  Similarly to the forgoing argument, we can see that if
$\wht c_1$ is obtained by an  extension of type 1 on the left/right, then $\wht c_2$ must be obtained by an  extension on the right/left, resp., and that $\wht c_3$ cannot be obtained by any  extension. This contradiction proves that it is not possible to have in $\Omega$ more than two additions of any starting bracket $c$.
Thus, the number of additions in $\Omega$  is at most $ 2|W|+2$. As in the proof of Lemma~\ref{lem3b}, the number of extensions is $\le |W|$ and the number of mergers is at most $2|W|+1$. Hence, the total number of \EO s is at most $ 5 |W|+3 \le 7|W|$ as desired.
\end{proof}

\begin{lem}\label{lem5b}  Let there be a  sequence  $\Omega$  of \EO s    that transforms  the empty
\BS\   for $(W, \D)$ into the final \BS\  and has size  bounded by $(k_1, k_2)$ and let $c$ be a starting bracket for $(W, \D)$. Then there is also a  sequence   of \EO s    that  converts  the
\BS\  $\{ c \}$   into the final \BS\  and has size  bounded by $(k_1+1, k_2+1)$.
\end{lem}

\begin{proof} The proof is analogous to the proof of Lemma~\ref{lem5}.  The arguments in the cases when $c(1) = 0$ or $c(1) = |W|$ are retained verbatim.

Assume that $0 < c(1) < |W|$. As before, let $B_0,  B_1,  \dots, B_\ell$ be the corresponding to $\Omega$ operational sequence of \BS s, where $B_0$ is empty and $B_\ell$ is final. Let $B_{i^*}$ be the first \BS\ of the sequence such that $B_{i^*} \cup \{ c \}$ is not a \BS . The existence of  $B_{i^*}$ follows from the facts that $B_0 \cup \{ c \}$ is a \BS\ and $B_\ell \cup \{ c \}$ is not.  Since
$B_0 \cup \{ c \}$ is a \BS , it follows that ${i^*} \ge 1$ and $B_{{i^*}-1} \cup \{ c \}$ is a \BS .  Since $B_{{i^*}-1} \cup \{ c \}$ is a  \BS\ and  $B_{{i^*}} \cup \{ c \}$ is not, there is a bracket $b \in B_{{i^*}}$ such that
$b(1) < c(1) < b(2)$ and $b$ is obtained from a bracket  $d_1 \in B_{{i^*}-1}$ by an extension or
  $b$ is obtained from    brackets $d_1, d_2 \in B_{{i^*}-1}$ by a merger. In either case, it follows from
  definitions of \EO s that $d_j(k) = c(1)$ for some $j,k \in \{ 1,2\}$. Hence, we can use a merger applied to  $d_j$ and $c$ which would result in $d_j$, i.e., in elimination of $c$ from $B_{{i^*}-1} \cup \{ c \}$ and in getting thereby $B_{{i^*}-1}$ from $B_{{i^*}-1} \cup \{ c \}$.  Now we can see that the original sequence of  \EO s, together with the merger $  B_{{i^*}-1} \cup \{ c \} \to  B_{{i^*}-1}$ can be used to produce the following operational sequence of \BS s
$$
B_{0} \cup \{ c \}, \dots,    B_{{i^*}-1} \cup \{ c \},  B_{{i^*}-1}, \dots,
  B_\ell .
$$
Clearly, the size of this new sequence is bounded by  $(k_1+1, k_2+1)$, as desired.
\end{proof}

\begin{lem}\label{2bb}  Let  $P_W  = p_1 p_2$, let the path $\al(p_2)$ be a  special arc of $\D$  and let there exist a sequence  $\Omega$   of \EO  s  that transforms the empty \BS\    for $( W, \D)$ into the final \BS\    and has size bounded by $ (\ell_1, \ell_2)$. Then there is also a sequence  $\Omega_{p_1}$   of \EO  s such that $\Omega_{p_1}$  transforms the empty \BS\    for $( W, \D)$ into the \BS\   $\{ (0, |p_1|, -k,  |\D(2)|) \}$, where $\ph(p_2) \overset 0 = a_1^k$,  and $\Omega_{p_1}$ has size  bounded by $(\ell_1, \ell_2)$. In addition,  for every bracket  $b \in B$, where $B$ is an intermediate  \BS\ of the corresponding to $\Omega_{p_1}$ operational sequence of \BS s,  it is true that  $b(2) \le |p_1|$.
\end{lem}

\begin{proof}  Let $B_0, B_1,  \dots, B_\ell$ be
the corresponding to  $\Omega$  operational  sequence of \BS  s, where $B_0$ is empty and  $B_\ell = \{ (0, |W|, 0, |\D(2)| ) \}$ is the final \BS\   for  $( W, \D)$.

For every \BS\   $B_i$, we will construct an associated \BS\   $B'_i$ as follows. Let $b \in B_i$ be  a bracket.
If $b(1) > |p_1|$, then we disregard $b$.

Suppose  that $b(1) \le |p_1|$ and $b(2) > |p_1|$.  Let $p_b$ be a subpath of $p_2$  such that
$(p_b)_- = (p_2)_-$, $(p_b)_+ = b(2)$ and let  $\ph(p_b) \overset 0 = a_1^{k_p}$, see Fig.~6.3.  Since $r(b)$ and $\al(p_b)$ are special arcs in $\D$,  it follows from   Lemma~\ref{arc} that the reduced path $r_p$,  obtained   from  $r(b) \al(p_b)^{-1}$   by canceling  pairs of edges of the form $e e^{-1}$, is a special arc such that
$$
(r_p)_- = \al(b(1)) ,   \quad   (r_p)_+ = \al((p_b)_+) = \al(|p_1|) .
$$
Hence,  we can define a bracket
$$
b' := (b(1), |p_1|, b(3)-k_p, b(4))
$$
whose complementary arc $r(b')$ is $r_p$, $r(b') := r_p$.
For every such $b \in B_i$, we include  $b'$  in $B'_i$.
\vskip3mm

\begin{center}
\begin{tikzpicture}[scale=.72]
\draw (-19,-2)  arc (0:360:2);
\draw(-23,-2) -- (-19,-2);
\draw [-latex](-23,-2) --(-21,-2);
\draw [-latex](-22.5,-3.3) --(-22.6,-3.2);
\draw [-latex] (-22.5,-.7) -- (-22.4,-.6);
\draw [-latex] (-19.5,-3.3) -- (-19.6,-3.4);
\draw [-latex] (-19.7,-.5) -- (-19.6,-.6);
\draw  (-21,0) [fill = black] circle (.05);
\draw  (-23,-2) [fill = black] circle (.05);
\draw  (-19,-2) [fill = black] circle (.05);
\draw  (-21,-4) [fill = black] circle (.05);
\node at (-21,-4.8) {Fig.~6.3};
\node at (-21,-2.5) {$r(b)$};
\node at (-21.,-.5) {$\alpha(|p_1|)$};
\node at (-23.4,-3.3) {$\alpha(p_1)$};
\node at (-23.2,-.5) {$\alpha(p_1)$};
\node at (-19,-.3) {$\alpha(p_b)$};
\node at (-21,-3.5) {$\alpha(0)$};
\node at (-24,-2) {$\alpha(b(1))$};
\node at (-18,-2) {$\alpha(b(2))$};
\end{tikzpicture}
\end{center}

If $b \in B_i$ satisfies $b(2) \le |p_1|$, then  we add $b$ to $B'_i$.

Going over all brackets $b \in B_i$ as described above, we obtain a new  \BS\   $B'_i$ associated with  $B_i$. Observe that  $B'_{i+1}$ is either identical to $B'_i$ or
$B'_{i+1}$ is obtained from $B'_i$ by a single \EO\  which is identical to the \EO\   that is used to get $B_{i+1}$ from $B_i$.

Moreover, if  $B'_{i+1} \ne  B'_i$ and $B_{i+1}$ is obtained from $B_i$ by application of an \EO\ $\sigma$ to a bracket $b_1 \in B_i$ or to  brackets $b_1, b_2 \in B_i$  (in the case when $\sigma$ is a merger) and this application results in $c$, written $c= \sigma(b_1)$  or $c= \sigma(b_1, b_2)$,   then $B'_{i+1}$ is obtained from $B'_i$ by application of  $\sigma'$, where $\sigma'$ has the same type as $\sigma$,
 to the bracket $b'_1 \in B'_i$ or to the  brackets $b'_1, b'_2 \in B'_i$ (in the case when $\sigma$ is a merger) and $c' =  \sigma'(b'_1) $ or $c'= \sigma'(b'_1, b'_2)$.

Indeed, this claim is immediate for additions, extensions of type 1 and mergers. On the other hand, it follows from the definitions of extensions of type 2 and 3 that an extension of type 2 or 3 can be applied only to a bracket $b_1$ with $b_1(2) \le |p_1|$
and results in a bracket $c = \sigma(b_1)$ with $c(2) \le |p_1|$.  Hence, in this case,
$b'_1 = b_1$,   $c' = c$, and our claim holds true again.

It follows from foregoing observations that a  new sequence  $B'_0, B'_1, \dots, B'_\ell$, when repetitions are disregarded, is operational,  has  size  bounded by  $(\ell_1, \ell_2)$, $B'_0$   is empty  and  $B'_\ell$ consists of the single bracket $(0, |p_1|, -k,  |\D(2)|)$,  as desired. The last  inequality  $b(2) \le |p_1|$ in lemma's statement is immediate from the definitions.
\end{proof}

\begin{lem}\label{3bb}  Let $P_W = p_1 p_2 p_3$ be a factorization of the path  $P_W$ such that $\al(p_2)$ is a special arc of $\D$ and let  there exist a sequence of \EO  s of size  bounded by  $(\ell_1, \ell_2)$  that transforms the empty \BS\    for $( W, \D)$ into the final \BS . Then there is also a sequence of \EO  s    that transforms a \BS\    consisting of the single bracket $c_0 := (|p_1|, |p_1p_2|, k_2, 0)$, where $\ph(p_2) \overset 0 = a_1^{k_2}$, into the final \BS\   for $( W, \D)$ and has size  bounded by  $(\ell_1+1, \ell_2+1)$.
\end{lem}

\begin{proof}  Consider a sequence $\Omega$ of \EO  s that transforms the empty \BS\    for $( W, \D)$ into the final \BS\    and has size  bounded by  $(\ell_1, \ell_2)$. Let
$
 B_0, B_1, \dots, B_\ell
$
be the corresponding  to $\Omega$   sequence of \BS  s, where $B_0$ is empty and  $B_\ell = \{ (0, |W|, 0, |\D(2)| ) \} $ is the final \BS\   for  $( W, \D)$.

Define $i^*$ to be the minimal index so that $B_{i^*}$ contains a bracket $b$ such that  $b(1) \le |p_1|$ and $b(2) \ge |p_1p_2|$, i.e.,  the arc $a(b)$ of $b$  contains the path $\al(p_2)$. Since  $B_\ell$ has this property and $B_0$ does not, such an $i^*$ exists and $0< i^* \le \ell$.

First suppose that $|p_2| =0$. Then $c_0 = (|p_1|, |p_1|, 0, 0)$ and we can define the following operational sequence of \BS s $B'_i := B_i \cup \{ c_0 \}$ for $i=0, \dots, i^*-1$. By the minimality of $i^*$, there is a unique bracket $b^* \in B_{i^*}$  such that $b^*(1) = |p_1|$  or  $b^*(2) = |p_1|$.

Suppose  $b^*(1) =b^*(2) =|p_1|$. Then $ B_{i^*}$ is obtained from $B_{i^*-1}$ by addition of $c_0$,  $ B_{i^*}= B_{i^*-1} \cup \{ c_0 \}$, and we can consider the following operational sequence
of \BS s
\begin{equation*}
    B'_0 = \{ c_0 \}, B'_1,    \dots,   B'_{i^*-2},   B'_{i^*-1}= B_{i^*},  B_{i^*+1}, \dots,  B_{\ell}
\end{equation*}
that transforms $\{ c_0 \}$ into the final \BS\ and has size bounded by $(\ell_1-1, \ell_2+1)$.

Suppose  $b^*(1) < b^*(2)$. Then a merger of $b^*$ and $c_0$ yields $b^*$, hence this merger turns $B'_{i^*} = B_{i^*}\cup \{ c_0\}$ into $B_{i^*}$. Now we can continue the original sequence $\Omega$ of \EO s to get  $B_{\ell}$. Thus the following
\begin{equation*}
    B'_0 = \{ c_0 \}, B'_1,    \dots,   B'_{i^*}, B_{i^*},  B_{i^*+1}, \dots,  B_{\ell}
\end{equation*}
is an  operational sequence
of \BS s that  transforms $\{ c_0 \}$ into the final \BS\ and has size bounded by $(\ell_1+1, \ell_2+1)$.

Now assume that $|p_2| >0$. For every \BS\   $B_i$, where $i < i^*$,  we construct an associated \BS\   $B'_i$ in the following manner.

Let $b \in B_i$ be a bracket.
If $b(1) \ge |p_1|$  and  $b(2) \le  |p_1p_2|$ and  $b(1) - |p_1| + |p_1p_2|- b(2) >0  $, i.e.,  the arc $a(b)$ of $b$ is a proper subpath of the path $\al(p_2)$, then  we disregard $b$.

Suppose that $b(1) \le |p_1|$ and  $ |p_1|  < b(2) \le |p_1p_2|$, i.e.,  the arc $a(b)$ overlaps with a prefix subpath of the path $\al(p_2)$ of positive length. Let $p_{21}$ denote a subpath of $p_2$ given by
$(p_{21})_- = (p_2)_-$ and $(p_{21})_+ = b(2)$, see Fig.~6.4.
Then it follows from Lemma~\ref{arc} applied to  special arcs  $r(b)$, $\al(p_{21})^{-1}$ of $\D$ that  the reduced
path, obtained from  $r(b) \al(p_{21})^{-1}$ by canceling pairs of edges of the form $ee^{-1}$, is a  special arcs  of $\D$. Therefore, we may consider a bracket
$$
b' = (b(1), |p_1|, b(3)-k_{21}, b(4)) ,
$$
where  $\ph(p_{21}) \overset 0 = a_1^{k_{21}}$.   For every such a bracket $b \in B_i$, we include  $b'$  into  $B'_i$.

\begin{center}
\begin{tikzpicture}[scale=.72]
\draw (-19,-2)  arc (0:360:2);
\draw(-23,-2) -- (-19,-2);
\draw [-latex](-23,-2) --(-21,-2);
\draw [-latex](-22.5,-3.3) --(-22.6,-3.2);
\draw [-latex] (-22.5,-.7) -- (-22.4,-.6);
\draw [-latex] (-19.18,-2.8) -- (-19.28,-3.04);
\draw [-latex] (-19.7,-.5) -- (-19.6,-.6);
\draw [-latex] (-20,-3.74) -- (-20.2,-3.83);
\draw  (-21,0) [fill = black] circle (.05);
\draw  (-23,-2) [fill = black] circle (.05);
\draw  (-19,-2) [fill = black] circle (.05);
\draw  (-21,-4) [fill = black] circle (.05);
\draw  (-19.6,-3.43) [fill = black] circle (.05);
\node at (-21,-4.8) {Fig.~6.4};
\node at (-21,-1.4) {$\Delta_b$};
\node at (-21,-2.6) {$r(b)$};
\node at (-21.,-.5) {$\alpha(|p_1|)$};
\node at (-18.5,-3.6) {$\alpha(|p_1p_2|)$};
\node at (-23.4,-3.3) {$\alpha(p_1)$};
\node at (-23.2,-.5) {$\alpha(p_1)$};
\node at (-19.7,-4.2) {$\alpha(p_3)$};
\node at (-18.8,-.5) {$\alpha(p_{21})$};
\node at (-18.3,-2.9) {$\alpha(p_{22})$};
\node at (-21,-3.5) {$\alpha(0)$};
\node at (-24,-2) {$\alpha(b(1))$};
\node at (-18,-2) {$\alpha(b(2))$};
\end{tikzpicture}
\end{center}

Suppose $ |p_1|  \le b(1) < |p_1p_2|$ and  $ |p_1p_2| \le  b(2)$, i.e., the arc $a(b)$ overlaps with a suffix of the path $\al(p_2)$  of positive length. Let $p_{22}$ denote a subpath of $p_2$ given by  $(p_{22})_- = b(1)$ and $(p_{22})_+ = (p_2)_+$, see Fig.~6.5. Then  it follows from Lemma~\ref{arc} applied to  special arcs
$\al(p_{22})^{-1}$, $r(b)$ of $\D$ that the reduced
path, obtained from  $\al(p_{22})^{-1}r(b) $ by canceling pairs of edges of the form $ee^{-1}$, is a  special arcs  of $\D$. Hence, we may consider a bracket
$$
b' = (|p_1p_2|, b(2), b(3)-k_{22}, b(4)) ,
$$
where $\ph(p_{22})  \overset 0   = a_1^{k_{22}}$.  For every such a bracket $b \in B_i$, we include  $b'$  into  $B'_i$.

\begin{center}
\begin{tikzpicture}[scale=.72]
\draw (-19,-2)  arc (0:360:2);
\draw(-23,-2) -- (-19,-2);
\draw [-latex](-23,-2) --(-21,-2);
\draw [-latex] (-22.5,-.7) -- (-22.4,-.6);
\draw [-latex] (-19.4,-3.2) -- (-19.5,-3.32);
\draw [-latex] (-19.7,-.5) -- (-19.6,-.6);
\draw [-latex] (-22.79,-2.9) -- (-22.84,-2.8) ;
\draw [-latex] (-21.6,-3.9) -- (-21.8,-3.83);
\draw  (-22.44,-3.4) [fill = black] circle (.05);
\draw  (-21,0) [fill = black] circle (.05);
\draw  (-23,-2) [fill = black] circle (.05);
\draw  (-19,-2) [fill = black] circle (.05);
\draw  (-21,-4) [fill = black] circle (.05);
\node at (-20.6,-4.8) {Fig.~6.5};
\node at (-21,-1.4) {$\Delta_b$};
\node at (-21,-2.6) {$r(b)$};
\node at (-21.,-.5) {$\alpha(|p_1p_2|)$};
\node at (-23.64,-2.9) {$\alpha(p_{21})$};
\node at (-23.3,-.5) {$\alpha(p_{22})$};
\node at (-22.3,-4.2) {$\alpha(p_1)$};
\node at (-18.8,-.5) {$\alpha(p_{3})$};
\node at (-18.55,-3.4) {$\alpha(p_{3})$};
\node at (-21,-3.5) {$\alpha(0)$};
\node at (-24,-2) {$\alpha(b(1))$};
\node at (-18,-2) {$\alpha(b(2))$};
\node at (-23.3,-3.57) {$\alpha(|p_1|)$};
\end{tikzpicture}
\end{center}

If $b \in B_i$ satisfies either $b(2) \le |p_1|$ or $b(1) \ge |p_1 p_2|$, then  we just add $b$ to $B'_i$  without  alterations.

We also put the bracket $c_0 =(|p_1|, |p_1p_2|, k_2, 0)$ in every
$B'_i$, $i =0,\dots, i^*-1$.

Going over all brackets $b \in B_i$, as described above, and adding $c_0$,  we obtain a new  \BS\   $B'_i$ associated with  $B_i$.

Observe that for every $i < i^* -1$ the following holds true.
Either $B'_{i+1}$ is identical to $B'_{i}$ or  $B'_{i+1}$ is obtained from $B'_{i}$ by a single \EO\
which is identical to the \EO\   that is used to get $B_{i+1}$ from $B_i$. Moreover, if $B'_{i+1} \ne  B'_i$ and $B_{i+1}$ is obtained
from $B_i$ by application of an \EO\ $\sigma$ to a bracket $b_1 \in B_i$ or to brackets $b_1, b_2 \in B_i$ (in  the case when $\sigma$ is a merger) and this application of  $\sigma$   results in $c$, written $c=\sigma(b_1)$ or $c=\sigma(b_1, b_2)$,  then $B'_{i+1} $ is  obtained from $B'_i$ by application of  $\sigma'$, where $\sigma'$ has the same type as $\sigma$,  to the bracket $b'_1 \in B'_i$ or to the brackets $b'_1, b'_2 \in B'_i$ (in the  case when $\sigma$ is a merger) and $c'=\sigma'(b'_1)$ or $c'=\sigma'(b'_1, b'_2)$, resp. Indeed, this claim is immediate for additions, extensions of type 1 and mergers. On the other hand, it follows from the definitions of extensions of types 2--3 and from the definition of $i^*$ that if $\sigma$ is an extension of type 2 or 3 then  $\sigma$ can only be applied to a bracket $b_1$ such that $b_1(2) \le |p_1|$ or $ |p_1p_2| \le b_1(1)$ and this application results in a bracket
$c = \sigma(b_1)$ with $c(2) \le |p_1|$ or   $c(1) \le |p_1p_2|$.
Hence, in this case,  $b'_1 = b_1$, $c' = c$ and our claim holds true.

Since the \BS\    $B_{i^*}$ contains a bracket $d$ such that $\al(p_2)$ is a subpath of  $a(d)$ and $i^*$ is the minimal index with this property, it follows from the definition of \EO  s and from the facts that
$\al(p_2)$ consists of $a_1$-edges and $|p_2| >0$ that either $d$ is obtained from a bracket
$b_1 \in B_{i^*-1}$ by an extension of type 1 or  $d$ is obtained from  brackets $b_2, b_3 \in B_{i^*-1}$ by a merger.

First suppose that $d$ is obtained from a bracket $b_1 \in B_{i^*-1}$ by an extension of type 1. In this case, we pick the image  $b'_1 \in B'_{i^*-1}$ of $b_1$ and use a merger operation over $b'_1, c_0 \in B'_{i^*-1}$ to get a bracket $c_3$. Let $B'_{i^*}$ denote the new \BS\     obtained from  $B'_{i^*-1}$ by the merger of  $b'_1, c_0 \in B'_{i^*-1}$, hence,
$$
B'_{i^*} :=
(B'_{i^*-1} \setminus \{ c_0, b'_1 \}) \cup \{ c_3 \} .
$$

Since $d$ is obtained from $b_1$ by an extension of type 1, it follows that, for every  $b \in  B_{i^*-1}$, $b \ne b_1$, we have either $b(1) \ge d(2) \ge |p_1p_2|$   or   $b(2) \le d(1) \le |p_1|$. Therefore, every bracket $b \in B_{i^*-1}$, $b \ne b_1$,
has an image $b' \in   B'_{i^*-1}$ and $b' = b$. This, together with equalities
$c_3(1) = d(1)$, $c_3(2) = d(2)$ and Lemma~\ref{NUP}, implies that $c_3 = d$ and $B'_{i^*} =  B_{i^*}$. Thus we can consider the following sequence of \BS  s
\begin{gather}\label{newC}
 B'_0,  \dots, B'_{i^*-1},  B'_{i^*} = B_{i^*},  B_{i^*+1},   \dots, B_\ell
\end{gather}
 that starts at $B'_0 = \{ c_0 \}$ and ends in the final \BS\    $B_\ell$ for $(W, \D)$.  It is clear that the size of this sequence is  bounded by  $(\ell_1, \ell_2+1)$. Hence, after deletion of    possible repetitions in  \er{newC2}, we obtain a desired operational sequence.

Now assume that $d$ is obtained from  brackets $b_2, b_3 \in B_{i^*-1}$ by a merger and let
$b_2(2) = b_3(1)$.  To define  $B'_{i^*-\tfrac 12} $, we apply a merger operation to $b'_2$ and $c_0$,
which results in a bracket $c_2$. To define  $B'_{i^*} $, we apply a merger operation to $c_2$  and $b'_3$ which results in a bracket $c_3$.

As above, we observe that since $d$ is obtained from $b_2, b_3$ by a merger, it follows that, for every bracket $b \in B_{i^*-1}$, $b \not\in \{ b_1, b_2 \}$, we have either $b(1) \ge d(2) \ge |p_1p_2|$   or   $b(2) \le d(1) \le |p_1|$. Therefore, every bracket $b \in B_{i^*-1}$,  $b \not\in \{ b_1, b_2 \}$,
has an image $b' \in   B'_{i^*-1}$ and $b' = b$. This, together with equalities
$c_3(1) = d(1)$, $c_3(2) = d(2)$ and Lemma~\ref{NUP},  implies that $c_3 = d$ and   $B'_{i^*} =  B_{i^*}$. Thus we can consider the following sequence of \BS  s
 \begin{gather}\label{newC2}
 B'_0,  \dots, B'_{i^*-1}, B'_{i^*-\tfrac 12},  B'_{i^*} =B_{i^*},  B_{i^*+1},   \dots, B_\ell
\end{gather}
 that starts at $B'_0 = \{ c_0 \}$ and ends in the final \BS\   $B_\ell$.
 Clearly, the  size of the sequence \eqref{newC2} is  bounded by  $(\ell_1+1, \ell_2+1)$. Hence,  after deletion of  possible repetitions in  \er{newC2}, we obtain a desired operational sequence.
\end{proof}

\begin{lem}\label{4bb}  There exists a sequence of \EO s  that converts the empty
\BS\   for $(W, \D)$ into the final \BS\    and has size bounded by
\begin{equation} \label{esma}
  ( 7|W| , \  C(\log |W|_{\bar a_1}  +1)  ) ,
\end{equation}
where $C = (\log \tfrac 65)^{-1}$  and $\log |W|_{\bar a_1} := 0$ if $|W|_{\bar a_1}=0$.
\end{lem}

\begin{proof}
The arguments of this proof are similar to those of the  proof of Lemma~\ref{lem6}.
Now we proceed by induction on $|W|_{\bar a_1} \ge 0$.  To establish the basis for induction, we first check our claim in the case when $|W|_{\bar a_1} \le 2$.

Let  $|W|_{\bar a_1} \le 2$. Then, looking at the tree
 $\rho_{a_1}(\D)$, we see that $|W|_{\bar a_1} = 0$ or  $|W|_{\bar a_1} = 2$.
If $|W|_{\bar a_1} = 0$, then it follows from Lemma~\ref{b1} applied to $\D$  that $\D$ is a tree consisting of $a_1$-edges and we can easily convert
the empty \BS\   for $(W, \D)$ into the final \BS\   by using a single addition and $|W|$ extensions of type 1. The size of the corresponding  sequence is bounded by  $  ( |W| +1 , \ 1   ) $ and  the bound \er{esma} holds as $C \ge 1$.

If now $|W|_{\bar a_1} = 2$, then it follows from Lemma~\ref{b1} applied to $\D$ that either $\D$ is a tree consisting of $a_1$-edges and two (oriented) $a_i$-edges, $i >1$, that form an  $a_i$-band $\Gamma_0$ with no faces  or  $\D$ contains a single $a_2$-band $\Gamma$ that has faces and all edges of $\D$, that are not in  $\Gamma$, are $a_1$-edges that form a forest whose trees are attached to $\Gamma$. In either case, we can convert the empty \BS\   for $(W, \D)$ into the final \BS\   by using a single addition and   $ \le |W|$ extensions of type 1 and 3 if  $\D$ is a tree or of type 1 and  2  if  $\D$ is not a tree.
The size of the corresponding  sequence of  \EO s is bounded by  $  ( |W| +1 , \ 1   ) $ and                   the bound \er{esma} holds as $C \ge 1$. The base step is complete.

Making the induction step, assume $|W|_{\bar a_1} > 2$.
By Lemma~\ref{b2} applied to $(W, \D)$, we obtain vertices $v_1, v_2 \in P_W$ such that if $P_W(\ff, v_1, v_2)  = p_1 p_2 p_3$ then there is a special arc $r$ in $\D$ such that $r_- = \al(v_1)$, $r_+ = \al(v_2)$ and
\begin{align}\label{inq10b}
   \min(|\ph(p_2)|_{\bar a_1},  |\ph(p_1)|_{\bar a_1}+|\ph(p_3)|_{\bar a_1})  \ge  \tfrac 13  |W|_{\bar a_1}  ,
\end{align}

Let $\p |_0 \D = q_1 q_2 q_3$, where
$q_i = \al(p_i) $, $i =1,2,3$. Consider disk subdiagrams $\D_1, \D_2$ of $\D$ given by $\p |_{v_1} \D_2 = q_2 r^{-1}$ and by  $\p |_{0} \D_1 = q_1 r q_3$, see Fig.~6.6. Denote
$$
W_2 \equiv \ph(q_2r^{-1}) ,  \quad W_1 \equiv \ph(q_1 r q_3)
$$
and let $P_{W_i} = P_{W_i}(W_i, \D_i)$, $i=1,2$,  denote the corresponding paths such that  $\al_1(P_{W_1}) = q_1 r q_3$ and $\al_2(P_{W_2}) = q_2r^{-1}$. We also denote  $\ph(r ) \overset 0 = a_1^{k_r}$.

\begin{center}
\begin{tikzpicture}[scale=.72]
\draw (-19,-2)  arc (0:360:2);
\draw(-23,-2) -- (-19,-2);
\draw [-latex](-23,-2) --(-21,-2);
\draw [-latex] (-19.4,-3.2) -- (-19.5,-3.32);
\draw [-latex] (-21.1,0) -- (-20.9,0) ;
\draw [-latex]  (-22.44,-3.4) -- (-22.53,-3.3);
\draw  (-23,-2) [fill = black] circle (.05);
\draw  (-19,-2) [fill = black] circle (.05);
\draw  (-21,-4) [fill = black] circle (.05);
\node at (-21,-5.4) {Fig.~6.6};
\node at (-21,-1.3) {$\Delta_2$};
\node at (-21,-2.5) {$r$};
\node at (-21.,-.64) {$\alpha(p_2)=q_2$};
\node at (-21,-3.2) {$\Delta_1$};
\node at (-18,-3.33) {$\alpha(p_{3})=q_3$};
\node at (-21,-4.5) {$\alpha(0)$};
\node at (-24,-2) {$\alpha(v_1)$};
\node at (-18,-2) {$\alpha(v_2)$};
\node at (-23.8,-3.33) {$\alpha(p_1)=q_1$};
\end{tikzpicture}
\end{center}

Since $|W_1|_{\bar a_1}, |W_2|_{\bar a_1} < |W|_{\bar a_1}$  by \eqref{inq10b}, it follows from the induction hypothesis that there is a sequence $\wht \Omega_i$  of \EO  s  for $(W_i, \D_i)$ that transforms the empty \BS\  into the final system and
has size bounded by
\begin{gather}\label{esmib}
( 7|W_i| , \  C(\log |W_i|_{\bar a_1} +1 )  )    ,
\end{gather}
where $i=1,2$.

Furthermore, it follows from  the bound  \eqref{esmib} with $i=2$ and from  Lemma~\ref{2bb}  applied to $\D_2$ whose boundary label has factorization
$\ph(\p|_{0} \D_2 ) \equiv  W_2 \equiv  W_{21} W_{22}$, where  $W_{21} \equiv  \ph(p_2)$ and $W_{21} \equiv  \ph(r)^{-1}$,  that there exists a sequence $\wht \Omega_2$
of \EO  s  for the pair $(W_2, \D_2)$ that transforms the empty \BS\ $B_{2, 0}$  into the \BS\
$B_{2, \ell_2} = \{ (0,  |p_2|,  k_r,  | \D_2(2) |  ) \}$  and has size bounded by
\begin{align}\label{esm2b}
 \left(7|W_2 |  , C(\log |W_2|_{\bar a_1} +1 )   \right)  .
\end{align}
Let   $B_{2,0}, B_{2,1}, \dots, B_{2, \ell_2}$ denote the associated with $\wht \Omega_2$ sequence of \BS s.

It follows from   the bound  \eqref{esmib} with $i=1$ and from     Lemma~\ref{3bb} applied to $\D_1$ whose boundary label has  factorization  $\ph(\p|_{0}  \D_1 ) \equiv  W_1 \equiv W_{11} W_{12} W_{13}$, where  $W_{11} \equiv  \ph(p_1)$, $W_{12} \equiv  \ph(r)$, $W_{13} \equiv  \ph(p_3)$,   that there exists a sequence  $\Omega_1$
of \EO  s  for $(W_1, \D_1)$ that transforms the  \BS\    $B_{1, 0} = \{ c_r \}$, where
\begin{equation*}
c_r := ( |p_1|,  |p_1| +|r|,  k_r,  0 ) ,
\end{equation*}
into the final \BS\   $B_{1, \ell_1}$    and has size bounded by
\begin{align}\label{esm3b}
 \left( 7 |W_1 | +1 , C(\log |W_1|_{\bar a_1}  +1 )    \right)    .
\end{align}
Let   $B_{1,0}, B_{1,1}, \dots, B_{1, \ell_1}$ be the associated with $\Omega_1$ sequence of \BS s.

Now we will show that these two sequences $\Omega_2$, $\Omega_1$ of \EO  s, the first one for $(W_2, \D_2)$, and the second one for  $(W_1, \D_1)$, could be modified and combined into a single
sequence of \EO s  for  $(W, \D)$ that  converts  the empty \BS\  into the final \BS\   and has size with a desired bound.

Note that every bracket $b = (b(1),b(2), b(3), b(4))$,  $b \in \cup_j B_{2,j}$,  for   $(W_2, \D_2)$, by Lemma~\ref{2bb},  has the property  that $b(2) \le |p_2 |$ and $b$   naturally gives rise to the bracket
$$
\wht b := (b(1)+|p_1|, b(2)+|p_1|, b(3), b(4))
$$
for  $(W, \D)$. Let $\wht B_{2,j}$ denote the \BS\   for  $(W, \D)$    obtained from $B_{2,j}$ by replacing every bracket $b \in B_{2,j}$ with $\wht b$. Then it is easy to verify that
$\wht B_{2,0}, \dots, \wht B_{2, \ell_2}$ is an operational sequence of \BS s for $(W, \D)$ that changes the empty \BS\ into
$$
 \wht B_{2, \ell_2} = \{  (|p_1|, |p_1|+ |p_2|, k_r,  |\D_2(2)|) \}
$$
and has size bounded by \er{esm2b}.

Consider a relation $\succ_1$ on the set of all pairs $(b,i)$,
where $b \in   B_{1,i}$, $i=0, \dots, \ell_1$, defined so that $(c, i+1) \succ_1 (b,i)$ if and only if
$c \in   B_{1,i+1} $ is obtained from brackets $b, b' \in   B_{1,i}$ by an \EO\ $\sigma$, here
$b'$ is missing if  $\sigma$ is an extension.

Define a relation $\succeq$ on the set of all pairs $(b,i)$,
where $b \in   B_{1,i}$, $i=0, \dots, \ell_1$,  that is the reflexive and transitive closure of the relation $\succ_1$.  Clearly, $\succeq$ is a partial order on the set of such  pairs $(b,i)$   and if $(b_2, i_2) \succeq (b_1, i_1)$ then $i_2 \ge i_1$ and
$$
b_2(1) \le b_1(1) \le b_1(2) \le b_2(2) .
$$

Now we observe that every bracket $d = (d(1), d(2), d(3), d(4))$, $d \in  B_{1,i}$, for   $(W_1, \D_1)$  naturally gives rise to a bracket $\wht d = (\wht d(1), \wht d(2), \wht d(3),  \wht d(4))$ for  $(W, \D)$ in the following manner.

If $(d, i)$ is not comparable with $(c_r, 0)$  by the relation $\succeq$ and  $d(1) \le |p_1|$,
 then
 $$
 \wht d := d .
 $$

If $(d, i) $ is not comparable with $(c_r,0)$  and    $|p_1| + |r| \le d(2)$, then
$$
\wht d :=  (d(1)+ |p_2|-|r|,  d(2)+ |p_2|-|r|, d(3), d(4) ) .
$$

If $(d, i) \succeq (c_r,0)$, then
$$
\wht d :=(d(1),  d(2)+ |p_2|-|r|,  d(3), d(4)+ |\D_2(2)| ) .
$$

Note that the above three cases cover all possible situations.

As before, let $\wht B_{1,i} := \{ \wht d \mid d \in B_{1,i} \}$. Then  it is easy to verify that    $\wht B_{1,0}, \dots, \wht B_{1, \ell_1}$ is an operational  sequence of \BS s for    $(W, \D)$ that changes the  \BS\
$$
\wht B_{1,0} = \wht B_{2,\ell_2} = \{ (|p_1|,|p_1| + |p_2|, k_r,  |\D_2(2)| ) \}
$$
into the final \BS\  $\wht B_{1, \ell_1} = \{  (0, |p_1|+ |p_2|+ |p_3|, 0,  |\D_1(2)|+ |\D_2(2)|) \}$.

Thus, with the indicated changes, we can now combine the foregoing modified sequences of \BS  s for $(W_2, \D_2)$ and for $(W_1, \D_1)$ into a single operational   sequence
\begin{equation*}
  \wht B_{2,0}, \dots, \wht B_{2, \ell_2} = \wht B_{1,0}, \dots, \wht B_{1, \ell_1}
\end{equation*}
of \BS s for $(W, \D)$
that transforms the empty \BS\       into the \BS\
 $\{  (|p_1|,|p_1| + |p_2|, k_r,  |\D_2(2)| ) \}$ and then continues to turn the latter
 into the final \BS . It follows from definitions and bounds \er{esm2b}--\er{esm3b} that the size of    thus constructed sequence is bounded by
 \begin{gather*}
 ( 7|W_1|+7|W_2|+1,\    \max( C(\log |W_1|_{\bar a_1}+ 1)+1, \ C( \log |W_2|_{\bar a_1}+1 ) )  ) \
 \end{gather*}
 Therefore, in view of Lemma~\ref{lem4b},    it remains to show that
 \begin{gather*}
  \max( C (\log |W_1|_{\bar a_1} +1)+ 1, C (\log |W_2|_{\bar a_1} +1 ) ) \le
 C (\log |W|_{\bar a_1}+1)  .
\end{gather*}
In view of the inequality \eqref{inq10b},
$$
\max( C(\log |W_1|_{\bar a_1} +1)+ 1, C(\log |W_2|_{\bar a_1}+1 ) )  \le  C (\log (\tfrac 56|W|_{\bar a_1})+1) +1 ,
$$
and   $C (\log (\tfrac 56|W|_{\bar a_1})+1) +1  \le  C (\log |W|_{\bar a_1}+1)  $ for $C = (\log \tfrac 65)^{-1}$.
\end{proof}

Let $W$ be an arbitrary nonempty word over the alphabet $\A^{\pm 1}$, not necessarily representing the trivial element of the group given by presentation \er{pr3b}.
As before, let $P_W$ be a simple labeled path with $\ph( P_W ) \equiv W$ and let vertices of $P_W$ be identified along $P_W$ with integers $0,1, \dots, |W|$ so that $(P_W)_- = 0, \dots$, $(P_W)_+ = |W|$.

A quadruple
 $$
 b = (b(1), b(2), b(3), b(4))
 $$
of  integers  $b(1), b(2), b(3), b(4)$ is called a {\em pseudobracket}  for the word $W$
if  $b(1), b(2)$ satisfy the inequalities $0 \le b(1)\le b(2) \le |W|$ and  $b(4) \ge 0$.

Let $p$ denote  the subpath
 of $P_W$ with $p_- = b(1)$, $p_+ = b(2)$, perhaps, $p_- = p_+$ and $|p| = 0$. The  subpath $p$ of $P_W$ is denoted $a(b)$ and called the {\em arc} of the pseudobracket $b$.

For example, $b_v = (v, v, 0, 0)$ is a pseudobracket for every vertex $v \in P_W$ and such $b_v$ is called a {\em  starting} pseudobracket. Note that $a(b) =  v= b(1)$. A {\em final}  pseudobracket for $W$ is
 $c = (0, |W|, 0,  k)$, where $k \ge 0$ is an integer. Note that  $a(c) =  P_W$.
\smallskip

Observe that if $b$ is a bracket for the pair $(W, \D)$, then $b$ is also
a \pbb\ for the word $W$.

Let $B$ be a finite set of \pbb s for  $W$, perhaps, $B$ is empty.
We say that $B$ is a {\em pseudobracket system}  if,  for every pair $b, c \in B$ of distinct \pbb s,
either $b(2) \le c(1)$  or $c(2) \le b(1)$. As before, $B$ is called  a {\em final } \PBS\ if $B$ contains a single
\pbb\ which is final.
Clearly, every \BS\ for $(W, \D)$ is also a \PBS\ for the word $W$.
\medskip

Now we describe three kinds of \EO s over \pbb s and over \PBS s: additions, extensions, and mergers,  which are analogous to those definitions for brackets and \BS s, except there are no any diagrams and faces involved.

Let $B$  be a \PBS\ for a word $W$.
\medskip

\textit{Additions}.

Suppose $b$ is a starting \pbb , $b \not\in B$, and $B \cup \{ b\}$
is a \PBS  .
Then we can add $b$ to $B$ thus making an {\em addition} operation over $B$.
\medskip

\textit{Extensions}.

Suppose $B$ is a \PBS ,
$b \in B$ is a \pbb\   and $e_1 a(b) e_2$ is a subpath of $P_W$, where
$a(b)$ is the arc of $b$ and $e_1, e_2$ are edges one of which could be missing.

Assume  that $\ph(e_1)= a_1^{\e_1}$, where $\e_1 = \pm 1$. Then we
consider a \pbb\  $b'$ such that
$$
b'(1) = b(1)-1 , \ b'(2) = b(2) ,  \ b'(3) = b(3)+\e_1, \ b'(4) = b(4) .
$$
Note that $a(b') = e_1 a(b)$. We say that $b'$ is obtained from $b$ by an extension of type 1 (on the left).
If $(B \setminus \{ b \}) \cup \{ b' \}$ is a \PBS ,
then replacement of $b \in B$ with $b'$ in $B$ is called an {\em extension } operation over $B$ of type~1.

Similarly, assume  that $\ph(e_2)= a_1^{\e_2}$, where $\e_2 = \pm 1$.
Then we consider a  \pbb\  $b'$ such that
$$
b'(1) = b(1) , \ b'(2) = b(2)+1 ,  \ b'(3) = b(3)+\e_2 , \ b'(4) = b(4) .
$$
Note that $a(b') = a(b)e_2 $.
We say that $b'$ is obtained from $b$ by an extension of type 1 (on the right).
If $(B \setminus \{ b \}) \cup \{ b' \}$ is a \PBS , then
replacement of $b \in B$ with $b'$ in $B$ is called an {\em extension } operation over $B$ of type~1.

Now suppose that both edges $e_1, e_2$ do exist and  $\ph(e_1) = \ph(e_2)^{-1}=a_2^{\e}$,   $\e = \pm 1$. We also
assume that $b(3)\ne 0$ and that  $b(3)$ is a multiple of $n_1$ if $\e =1$ or  $b(3)$ is a multiple of $n_2$ if $\e =-1$.
Consider a \pbb\  $b'$ such that
$$
b'(1) = b(1)-1 , \ b'(2) = b(2)+1 ,  \ b'(3) = (|n_1^{-1} n_2|)^{\e} b(3) ,  \  b'(4) = b(4) + |b(3)|/|n_{i(\e)}| ,
$$
where $i(\e) =1$ if $\e =1$ and $i(\e) =2$ if $\e =-1$.  Note that $a(b') = e_1 a(b)e_2$.
We say that $b'$ is obtained from $b$ by an extension of type 2.
If $(B \setminus \{ b \}) \cup \{ b' \}$ is a \PBS , then
replacement of $b \in B$ with $b'$ in $B$ is called an {\em extension } operation over $B$ of type~2.

Assume that both $e_1, e_2$ do exist,  $\ph(e_1) = \ph(e_2)^{-1}$,    $\ph(e_1) \ne a_1^{\pm 1}$, and $b(3)=0$.
Consider a  \pbb\  $b'$ such that
$$
b'(1) = b(1)-1 , \  b'(2) = b(2)+1 , \  b'(3) = b(3) , \ b'(4) = b(4) .
$$
Note that  $a(b') = e_1 a(b)e_2$. We say that $b'$ is obtained from $b$ by an extension of type 3.
If $(B \setminus \{ b \}) \cup \{ b' \}$ is a \PBS , then
replacement of $b \in B$ with $b'$ in $B$ is called an {\em extension } operation over $B$ of type~3.
\medskip

\textit{Mergers}.

Suppose that $b_1, b_2$ are distinct \pbb s in $B$ such that  $b_1(2) = b_2(1)$. Consider a \pbb\  $b'$ such that
$$
b'(1) = b_1(1) , \ b'(2) = b_2(2) ,  \ b'(3) =  b_1(3) + b_2(3) , \  b'(4) =  b_1(4) + b_2(4).
$$
Note that $a(b') = a(b_1)a(b_2)$. We say that $b'$ is obtained from $b_1, b_2$ by a merger operation.
Taking both $b_1$, $b_2$ out of $B$ and putting  $b'$ in $B$ is a {\em merger } operation over $B$.
\medskip

We will say that additions, extensions,  and mergers, as defined above, are {\em  \EO s} over \pbb s and \PBS  s for $W$.

Assume that one \PBS\   $B_\ell$ is obtained from another \PBS\    $B_0$ for $W$
by a finite sequence $\Omega$ of \EO s  and $B_0,  B_1, \dots, B_\ell$ is the corresponding to $\Omega$ sequence of \PBS  s.
As above, we say that the sequence     $B_0,  B_1, \dots, B_\ell$ is {\em operational} and that $B_0,  B_1,  \dots, B_\ell$  has {\em size bounded by}
$ (k_1, k_2)$  if $\ell \le k_1$ and, for every $i$,  the number of  \pbb s in $B_i$ is at most $k_2$.
Whenever it is not ambiguous, we also say that $\Omega$ has size bounded by $(k_1, k_2)$ if so does the corresponding to $\Omega$ sequence  $B_0,  B_1,  \dots, B_\ell$ of \PBS s.

\begin{lem}\label{7bb} Suppose that the empty  \PBS\    $B_0$ for $W$ can be transformed by a finite sequence $\Omega$ of \EO  s  into a  final \PBS\ $B_\ell = \{  (0, |W|, 0, k) \}$. Then $W \overset  {\GG_3} = 1$  and there is a reduced disk diagram $\D$ over presentation \er{pr3b} such that $\ph(\p|_0 \D) \equiv W$ and $| \D(2) | \le k$.
\end{lem}

\begin{proof} Let $B_0, B_1, \dots, B_\ell$ be the corresponding to $\Omega$  sequence of \PBS  s, where $B_0$ is empty and $B_\ell$ is final. Consider the following claim.

\begin{enumerate}
\item[(D1)]  If $c$ is a \pbb\  of $B_i$, $i =1,\dots,\ell$,    then
 $\ph(a(c))  \overset  {\GG_3} = a_1^{c(3)}$,  where $a(c)$ is the arc of $c$, and there is a disk diagram $\D_c$ over presentation \eqref{pr3b}  such that  $ \p|_0 \D_c = s r^{-1}$, where $\ph(s) \equiv  \ph(a(c))$,
  $r$ is simple, $ \ph( r ) \equiv   a_1^{c(3)}$ and  $| \D_c(2) | = c(4)$.
\end{enumerate}

{\em Proof of Claim (D1).}  By induction on $i \ge 1$, we will prove that Claim (D1) holds for every \pbb\ $c \in B_i$, $i =1,\dots,\ell$.

The base step of this induction is obvious since $B_1$ consists of a starting \pbb\  $c$ for which we have a disk diagram  $\D_c$ consisting of a single vertex.

To make the induction step from $i$ to $i+1$, $i \ge 1$,
we consider the cases corresponding to the type of the \EO\ that is used to get $B_{i+1}$ from $B_i$.

Suppose that $B_{i+1}$ is obtained from $B_{i}$ by an \EO\ $\sigma$ and $c \in B_{i+1}$ is the \pbb\
obtained from $b_1, b_2 \in B_{i}$ by application of $\sigma$, denoted   $c =\sigma(b_1, b_2)$. Here one of $b_1, b_2$ or both, depending on  type of  $\sigma$, could be missing.
By the induction hypothesis,  Claim (D1) holds for every \pbb\ of $B_{i+1}$
different from $c$ and it suffices to show that  Claim (D1) holds for $c$.

If $c \in B_{i+1}$ is obtained  by an addition, then
Claim (D1) holds for $c$ because it holds for any starting \pbb .

Let $c \in B_{i+1}$ be obtained from a \pbb\
$b \in  B_{i}$ by an extension of  type 1 on the left and $a(c) = e_1 a(b)$, where $\ph(e_1) = a_1^{\e_1}$, $\e_1 = \pm 1$, and $a(c), a(b)$ are the arcs of $c, b$, resp.
By the induction hypothesis applied to $b$,  there is a disk diagram $\D_b$
over presentation \eqref{pr3b}  such that  $ \p|_0 \D_b = s r^{-1}$, where $\ph(s) \equiv  \ph(a(b))$,
$r$ is simple, $ \ph( r )  \equiv  a_1^{b(3)}$,  and  $| \D_b(2) | = b(4)$.
First we assume that the integers $\e_1, b(3)$ satisfy  $\e_1 b(3) \ge 0$. Then we consider a new
``loose" edge $f$ such that  $\ph(f) = a_1^{\e_1}$.
We attach the vertex $f_+$ to the vertex $s_-$ of $\D_b$ and obtain thereby a disk diagram
$\D'_b$ such that $\p \D'_b = fs r^{-1} f^{-1}$, see Fig.~6.7(a).
Since $\ph(a(c)) \equiv \ph(fs) \equiv a_1^{\e_1}  \ph(a(b))$, $fr$ is simple and
$$
 \ph(fr) \equiv  a_1^{\e_1+b(3)} , \ \   | \D'_b(2) | = | \D_b(2) | = b(4) ,
$$
it follows that we can use $\D'_b$ as a required disk diagram $\D_c$ for $c$.

\begin{center}
\begin{tikzpicture}[scale=.72]
\draw (-26.1,-2)  arc (0:180:2);
\draw(-31.1,-2) -- (-26.1,-2);
\draw [-latex](-31.1,-2) --(-28,-2);
\draw [-latex](-31.1,-2) --(-30.5,-2);
\draw [-latex](-28.1,0) --(-28.05,0);
\draw  (-31.1,-2) [fill = black] circle (.04);
\draw  (-30.1,-2) [fill = black] circle (.04);
\draw  (-26.1,-2) [fill = black] circle (.04);
\draw (-19,-2)  arc (0:180:2);
\draw(-23,-2) -- (-19,-2);
\draw [-latex](-23,-2) --(-22.4,-2);
\draw [-latex](-22,-2) --(-20.4,-2);
\draw [-latex](-21,0) --(-20.95,0);
\draw  (-22,-2) [fill = black] circle (.04);
\draw  (-23,-2) [fill = black] circle (.04);
\draw  (-19,-2) [fill = black] circle (.04);
\node at (-21,-1) {$\Delta_b$};
\node at (-22.4,-2.7) {$g$};
\node at (-20.5,-2.7) {$r'$};
\node at (-21.,.6) {$s$};
\node at (-23.,0) {$\Delta'_b$};
\node at (-28,-1) {$\Delta_b$};
\node at (-30.6,-2.7) {$f$};
\node at (-28,-2.7) {$r$};
\node at (-28.1,.6) {$s$};
\node at (-30.5,0) {$\Delta'_b$};
\node at (-28,-3.8) {Fig.~6.7(a)};
\node at (-21,-3.8) {Fig.~6.7(b)};
\node at (-19,-3) {$r = gr'$};
\end{tikzpicture}
\end{center}
\vskip2mm

Now suppose that  the integers $\e_1, b(3)$ satisfy  $\e_1 b(3) < 0$, i.e., $b(3) \ne 0$ and $\e_1$, $b(3)$ have different signs. Then we can write  $r = gr'$, where $g$ is an edge of $r$ and $\ph(g) = a_1^{-\e_1}$. Let $\D'_b$
denote the diagram  $\D'_b$ whose boundary has factorization $\p |_0 \D'_b = (g^{-1} s) (r')^{-1}$, see Fig.~6.7(b).
Note that    $\ph(a(c)) \equiv \ph(g^{-1} s)$, $r'$ is simple,  and
$$
    \ph(r') \equiv  a_1^{\e_1+b(3)} , \ \   | \D'_b(2) | = | \D_b(2) | = b(4) .
$$
Hence, it follows that we can use $\D'_b$ as a desired  diagram $\D_c$ for $c$.
The case of an    extension of  type 1 on the right is similar.
\medskip

 Let $c \in B_{i+1}$ be obtained from a \pbb\
 $b \in  B_{i}$ by an extension of  type 2  and assume that $a(c) = e_1 a(b) e_2$, where $\ph(e_1) =\ph(e_2)^{-1} = a_2$ (the subcase when $\ph(e_1) =\ph(e_2)^{-1} = a_2^{-1}$ is similar). Let $a(c), a(b)$ denote the arcs of \pbb s $c, b$, resp.  According to the definition of an extension of  type  2,  $b(3)$ is a multiple of $n_1$, say, $b(3) = n_1 \ell$ and $\ell \ne 0$.
By the induction hypothesis, there is a disk diagram $\D_b$ over presentation
\eqref{pr3b}  such that  $ \p|_0 \D_b = s r^{-1}$, where $\ph(s) \equiv  \ph(a(b))$, $r$ is simple,  $\ph( r )  \equiv   a_1^{b(3)}$,  and  $| \D_b(2) | = b(4)$.
Since $b(3) = n_1 \ell$, there is a disk diagram $\Gamma$ over \eqref{pr3b} such that
$$
\p \Gamma = f_1 s_1 f_2 s_2 ,   \ \ph(f_1) = a_2 ,  \ \ph(f_2) = a_2^{-1} ,
\ \ph( s_1 )  \equiv  a_1^{b(3)} , \  \ph( s_2^{-1} ) \equiv  a_1^{b(3) n_2/n_1} ,
$$
and $\Gamma$ consists of $| \ell | >0$ faces. Note that  $\Gamma$ itself is an $a_2$-band with $|\ell|$ faces.  Attaching  $\Gamma$ to $\D_b$ by identification of the paths $s_1$ and $r$ that have identical labels, we obtain a disk diagram $\D'_b$ such that $ \p|_0 \D'_b = f_1 s  f_2 s_2$ and $|\D'_b(2) | =  |\D_b(2) | + |\ell |$, see Fig.~6.8. Then  $\ph( a(c) ) \equiv \ph(   f_1 s  f_2 )$, $s_2^{-1}$ is simple, $\ph(s_2^{-1}) \equiv  a_1^{b(3) n_2/n_1}$  and we see that $\D'_b$ satisfies all desired conditions for  $\D_{c}$.

\begin{center}
\begin{tikzpicture}[scale=.72]
\draw (-19,-2)  arc (0:180:2);
\draw(-23,-2) node (v1) {}  -- (-19,-2);
\draw [-latex](-22.,-2) --(-21,-2);
\draw [-latex](-21,0) --(-20.95,0);
\draw [-latex](-21,-3.4) --(-21.1,-3.4);
\draw [-latex](-23,-2.7) --(-23,-2.6);
\draw [-latex](-19,-2.6) --(-19,-2.75);
\node at (-21,-1.4) {$r=s_1$};
\node at (-21,-.7) {$\Delta_b$};
\node at (-21.,.5) {$s$};
\node at (-23.,0) {$\Delta'_b$};
\node at (-21,-4.6) {Fig.~6.8};
\draw  (v1) rectangle (-19,-3.4) node (v2) {};
\node at (-21,-3.9) {$s_2$};
\node at (-21,-2.5) {$s_1$};
\node at (-23.8,-2.7) {$f_1$};
\node at (-18.4,-2.7) {$f_2$};
\node at (-20,-2.6) {$\Gamma$};
\end{tikzpicture}
\end{center}
\vskip 2mm

Now let $c \in B_{i+1}$ be obtained from a \pbb\
 $b \in  B_{i}$ by an extension of  type 3  and let $a(c) = e_1 a(b) e_2$, where $\ph(e_1) = \ph(e_2)^{-1}$,
 $\ph(e_1) = a_j^\e$,  $\e = \pm 1$ and $j >1$. By the definition of an extension of type 3, we have $b(3) = 0$. Hence, by the induction hypothesis,  there is a \ddd\ $\D_b$ over
\eqref{pr3b}  such that  $ \p|_0 \D_b = s r^{-1}$, where $\ph(s) \equiv  \ph(a(b))$, $|r|=0$  and  $| \D_b(2) | = b(4)$.
Consider a new ``loose" edge $f$ such that $\ph(f) = a_j^{\e}$.
We attach the vertex $f_+$ to the vertex $s_-=s_+$ of $\D_b$ and obtain thereby a disk diagram
$\D'_b$ such that $\p \D'_b = fs f^{-1}r'$, where $|r'| = 0$, see Fig.~6.9. Since $\ph(a(c)) \equiv \ph(fsf^{-1})$ and   $ | \D'_b(2) | = | \D_b(2) | = b(4)$, it follows
that we can use $\D'_b$ as a desired diagram $\D_c$ for $c$.

\begin{center}
\begin{tikzpicture}[scale=.69]
\draw  (0,0) circle (1);
\draw (0,-2) --(0,-1);
\draw [-latex](0,-2) --(0,-1.4);
\draw [-latex](-.1,1) --(.1,1);
\draw  (0,-2) [fill = black] circle (.05);
\draw  (0,-1) [fill = black] circle (.05);
\node at (0,-2.7) {Fig.~6.9};
\node at (0,-.4) {$\Delta_b$};
\node at (-.4,-1.5) {$f$};
\node at (0,.5) {$s$};
\node at (1.2,-1.3) {$\Delta'_b$};
\end{tikzpicture}
\end{center}

Let $c \in B_{i+1}$ be obtained from  \pbb s
 $b_1, b_2 \in  B_{i}$ by a merger  and $a(c) = a(b_1) a(b_2)$, where  $a(c), a(b_1)$, $ a(b_2)$ are the arcs of the \pbb s $c, b_1, b_2$, resp.

By the induction hypothesis, there are disk diagrams $\D_{b_1}$, $\D_{b_2}$  over presentation \eqref{pr3b}  such that
$$
\p|_0 \D_{b_j} = s_{j} r_j^{-1} , \  \ph(s_j) \equiv  \ph(a(b_j)) , \ \ph( r_j )  \equiv   a_1^{b_j(3)} , \ | \D_{b_j}(2) | = b_j(4) ,
$$
and $r_j$ is a simple path for  $j=1,2$.
Assume that the numbers $ b_1(3), b_2(3)$ satisfy the condition $b_1(3) b_2(3) \ge 0$. We
attach $\D_{b_1}$ to $\D_{b_2}$   by identification of the vertices $(s_1)_+$ and $(s_2)_-$ and obtain thereby a disk diagram $\D'_b$ such that $ \p \D'_b = s_1 s_2  (r_1 r_2)^{-1}$ and $|\D'_b(2) | =  |\D_{b_1}(2) | + |\D_{b_2}(2) |$, see Fig.~6.10(a).  Note that  $\ph( a(c) ) \equiv \ph(   s_1 s_2 )$, $r_1 r_2$ is a simple path and
 $\ph(r_1 r_2 ) \equiv  a_1^{b_1(3) +b_2(3)}$. Hence, $\D'_b$ is a desired disk diagram   $\D_{c}$ for $c$.

\begin{center}
\begin{tikzpicture}[scale=.72]
\draw (-20,-2)  arc (0:180:2);
\draw(-24,-2) node (v1) {}  -- (-20,-2);
\draw [-latex](-23,-2) --(-22,-2);
\draw [-latex](-22,0) --(-21.95,0);
\draw  (-24,-2) [fill = black] circle (.05);
\draw  (-20,-2) [fill = black] circle (.05);
\node at (-22,-1) {$\Delta_{b_2}$};
\node at (-22,0.5) {$s_2$};
\node at (-22,-2.7) {$r_2$};
\draw (-24,-2)  arc (0:180:2);
\draw(-28,-2) node (v1) {}  -- (-24,-2);
\draw [-latex](-27,-2) --(-26,-2);
\draw [-latex](-26,0) --(-25.95,0);
\draw  (-28,-2) [fill = black] circle (.05);
\draw  (-24,-2) [fill = black] circle (.05);
\node at (-26,-1) {$\Delta_{b_1}$};
\node at (-26,0.5) {$s_1$};
\node at (-24,0) {$\Delta'_b$};
\node at (-26,-2.7) {$r_1$};
\node at (-24,-4.8) {Fig.~6.10(a)};
\draw (-12,-2)  arc (0:180:2);
\draw (-15,-2)  arc (180:360:1.5);
\draw(-16,-2) node (v1) {}  -- (-12,-2);
\draw [-latex](-15,-2) --(-13.5,-2);
\draw [-latex](-15.6,-2) --(-15.5,-2);
\draw [-latex](-14,0) --(-13.95,0);
\draw [-latex](-13.4,-3.5) --(-13.5,-3.5);
\draw  (-16,-2) [fill = black] circle (.05);
\draw  (-12,-2) [fill = black] circle (.05);
\draw  (-15,-2) [fill = black] circle (.05);
\node at (-14,-0.7) {$\Delta_{b_1}$};
\node at (-14,0.5) {$s_1$};
\node at (-12,0) {$\Delta'_b$};
\node at (-13.5,-1.4) {$r_{12}= r_2^{-1}$};
\node at (-15.5,-2.6) {$r_{11}$};
\node at (-13.4,-4) {$s_2$};
\node at (-13.6,-2.6) {$\Delta_{b_2}$};
\node at (-15,-4.8) {Fig.~6.10(b)};
\end{tikzpicture}
\end{center}

Now suppose that $ b_1(3), b_2(3)$ satisfy the condition $b_1(3) b_2(3) < 0$.
For definiteness, assume that $ |b_1(3)| \ge | b_2(3)|$ (the case $ |b_1(3)| \le | b_2(3)|$ is analogous).
Denote $r_1 = r_{11}r_{12}$, where $| r_{12}| = | r_{2}| = |b_2(3)|$. Then we attach  $\D_{b_2}$  to
$\D_{b_1}$ by identification of the paths $r_{12}$ and $r_{2}^{-1}$  that have identical labels equal to
$a_1^{-b_2(3)}$, see Fig.~6.10(b). Doing this produces a \ddd\ $\D'_b$  such that
$\p \D'_b = s_1 s_2 r_{11}^{-1}$ and $|\D'_b(2) | =  |\D_{b_1}(2) | + |\D_{b_2}(2) |$.  Note that
$\ph( a(c) ) \equiv \ph(   s_1 s_2 )$,  $r_{11}$ is a simple path and
$\ph(r_{11} ) \equiv  a_1^{b_1(3) +b_2(3)}$. Hence $\D'_b$ is a desired disk diagram   $\D_{c}$ for $c$.
The induction step is complete and Claim (D1) is proven. \qed

\medskip
To finish the proof of Lemma~\ref{7bb}, we note that it follows from Claim (D1) applied to the \pbb\  $b_F = (0, |W|, 0, k)$ of the final system
$B_\ell = \{  b_F  \}$ that there is a disk diagram $\D_{b_F}$ such that  $ \p|_0 \D_{b_F} = s r^{-1}$, where $\ph(s) \equiv  W$,  $| r | =0$ and  $| \D_{{b_F}}(2) | \le k$.  It remains to pick a reduced disk diagram $\D$ with these properties of $ \D_{{b_F}}$.
\end{proof}

\begin{lem}\label{8bb} Suppose $W$ is a nonempty word over $\A^{\pm 1}$ and $n \ge 0$ is an integer. Then $W$ is a product of at most $n$ conjugates of the words  $(a_2 a_1^{n_1} a_2^{-1}  a_1^{-n_2})^{\pm 1}$ if and only if there is a sequence $\Omega$ of \EO  s
such that $\Omega$ transforms the empty \PBS\    for $W$ into  a final \PBS\    $\{ (0,|W|,0,n') \}$,
where $n' \le n$, $\Omega$ has size bounded by
$$
(7|W|,  C(\log|W|_{\bar a_1} +1)) ,
$$
where $C= (\log \tfrac 65)^{-1}$, and,  if $b$ is a \pbb\ of an intermediate \PBS\
of the  corresponding to $\Omega$ sequence of \PBS s,  then
\begin{equation}\label{d1d}
    |b(3)| \le \tfrac 12 ( |W|_{a_1} + (|n_1| + |n_2|)n) .
\end{equation}
\end{lem}

\begin{proof}   Assume that  $W$ is a product of at most $n$ conjugates of the words
$$
(a_2 a_1^{n_1} a_2^{-1}  a_1^{-n_2})^{\pm 1} .
$$
It follows from  Lemma~\ref{vk} that there is a reduced disk diagram $\D$ over presentation  \eqref{pr3b} such that $\ph(\p|_0 \D) \equiv W$ and $|\D(2)| \le n$. Applying Lemma~\ref{4bb} to the pair $(W, \D)$, we obtain a sequence
of \EO  s over \BS  s for $(W, \D)$ that converts the empty \BS\    for $(W, \D)$ into the final \BS\ $\{ (0, |W|, 0, |\D(2)| ) \}$ and has  size bounded by  $(7|W|,  C(\log|W|_{\bar a_1} +1))$.
Since every bracket $b$  and every \BS\ $B$ for $(W, \D)$  could be considered as a \pbb\ and a \PBS\ for $W$, resp., we automatically have a desired sequence of \PBS s. The inequality \er{d1d} follows from the definition of a bracket for
 $( W, \D)$ and the inequalities
\begin{equation*}
 |b(3) | \le \tfrac 12   | \vec \D(1) |_{a_1} =  \tfrac 12  (|W|_{a_1} + (|n_1| + |n_2|) |\D(2)| ) , \quad |\D(2)| \le n .
\end{equation*}

Conversely, the existence of a  sequence of \EO  s over \PBS  s for $W$, as specified in  Lemma~\ref{8bb}, implies, by virtue of Lemma~\ref{7bb}, that $W \overset{\GG_3} = 1$ and that there exists a reduced disk diagram $\D$ such that   $\ph(\p|_0 \D) \equiv W$ and $|\D(2)| =n' \le n$. Hence,  $W$ is a product of $n' \le n$ conjugates of words $(a_2 a_1^{n_1} a_2^{-1}  a_1^{-n_2})^{\pm 1}$, as required.
\end{proof}

\section{Proof of Theorem~\ref{thm2}}

\begin{T2}  Let  the group $\GG_3$ be defined by a presentation of the form
\begin{equation*}\tag{1.4}
    \GG_3 :=  \langle \  a_1, \dots, a_m  \ \|  \   a_2 a_1^{n_1} a_2^{-1} = a_1^{ n_2}  \,  \rangle ,
\end{equation*}
where $n_1, n_2$ are some nonzero integers. Then both  the bounded  and  precise  word  problems for  \eqref{pr3b}  are in $\textsf{L}^3$ and in $\textsf{P}$. Specifically, the problems can be solved in  deterministic space
$O((\max(\log|W|,\log n) (\log|W|)^2)$ or in deterministic time $O(|W|^4)$.
\end{T2}

\begin{proof} As in the proof of Theorem~\ref{thm1},  we start with $\textsf{L}^3$ part of Theorem~\ref{thm2}.
First we discuss a nondeterministic algorithm  which  solves the bounded word problem for presentation  \eqref{pr3b} and which is now based on Lemma~\ref{8bb}.

Given an input $(W, 1^n)$, where $W$ is a nonempty word (not necessarily reduced) over the alphabet $\A^{\pm 1}$ and $n \ge 0$ is an integer, written in unary notation as $1^n$, we begin with the empty \PBS\   and
nondeterministically   apply a sequence of \EO  s  of size $\le( 7|W|,  C(\log|W|_{\bar a_1} +1) )$, where $C = (\log \tfrac 65)^{-1}$.  If such a  sequence of \EO  s  results in a final \PBS , consisting of a single \pbb\   $\{ (0,|W|,0,n') \}$, where $n' \le n$, then our algorithm accepts and, in view of Lemma~\ref{8bb}, we may conclude that  $W$ is a product of $n' \le n$ conjugates of the words $(a_2 a_1^{n_1} a_2^{-1}  a_1^{-n_2})^{\pm 1}$.
 It follows from definitions and  Lemma~\ref{8bb}  that the  number of \EO  s needed for this algorithm to accept is at most $7|W|$.
 Hence, it follows from the definition of \EO s over \PBS s for $W$ that
 the time needed to run this nondeterministic algorithm is bounded by $O(|W|)$.
  To estimate  the space requirements of this algorithm, we note that if $b$ is a \pbb\ for $W$, then $b(1), b(2)$ are  integers in the range from 0 to $|W|$, hence, when written in binary notation, will take at most $C'(\log|W| +1)$ space, where $C'$ is a constant. Since  $b(3), b(4)$ are also  integers  that satisfy inequalities
\begin{align}\label{ineqn1}
 & |b(3)|  \le \tfrac 12 ( |W|_{a_1} + (|n_1| + |n_2|)n) \le  |W|(|n_1| + |n_2|)(n+1) ,  \\ \label{ineqn2}
 & 0  \le b(4) \le n ,
\end{align}
 and $|n_1| + |n_2|$ is a constant,  it follows  that the total space required to run this algorithm is at most
\begin{align*}
& ( 2C'(\log|W|+1)   +C' \log ( |W|(|n_1| + |n_2|)(n+1) ) ) \cdot C(\log|W|_{\bar a_1} +1) + \\
& + ( C' \log ( n+1) ) \cdot C(\log|W|_{\bar a_1} +1) =     O( (\max (\log |W| , \log n ) \log |W| ) .
\end{align*}

As in the proof of Theorem~\ref{thm1}, it now follows from  Savitch's theorem \cite{Sav}, see also \cite{AB}, \cite{PCC}, that there is a deterministic  algorithm that solves the bounded word  problem for presentation  \eqref{pr3b} in space $ O( \max (\log |W|, \log n)(\log |W|)^2)$.

To solve the precise word problem for presentation  \eqref{pr3b},  assume that we are given a pair $(W, 1^n)$ and we wish to find out if  $W$ is a product of $n$ conjugates of the words $(a_2 a_1^{n_1} a_2^{-1}  a_1^{-n_2})^{\pm 1}$ and $n$ is minimal with this property.  Using the foregoing deterministic algorithm, we  check whether  the bounded word problem is solvable for the two pairs  $(W, 1^{n-1})$ and $(W, 1^n)$. It is clear that
the precise word problem for the pair $(W, 1^n)$ has a positive solution if and only if  the bounded word problem has a negative solution for the pair $(W, 1^{n-1})$ and   has a positive solution for the  pair $(W, 1^n)$. Since these facts can be verified in space $ O( \max (\log |W|, \log n)(\log |W|)^2)$, we obtain a solution for the
precise word problem in deterministic space $ O( \max (\log |W|, \log n)(\log |W|)^2)$, as desired.
\medskip

Now we describe an algorithm that solves the \PWPP\ for presentation \eqref{pr3b} in polynomial time.

Recall that if $a_i \in \A$ and $U$ is a word over $\A^{\pm 1}$, then $|U|_{a_i}$ denotes the number of all occurrences of
letters $a_i$, $a_i^{-1}$ in $U$ and $|U|_{\bar a_i} := |U| - |U|_{a_i}$.

\begin{lem}\label{lema1a}  Let $\D$ be a reduced \ddd\ over presentation \er{pr3b}, $W \equiv \ph(\p |_0 \D)$ and
$|W|_{\bar a_1} >0$. Then at least one of the following claims $\mathrm{(a)}$--$\mathrm{(b)}$ holds true.

$\mathrm{(a)}$ There is a vertex $v$ of the path $P_W$ such that  if $P_W(\ff, v)  = p_1p_2$
then $|\ph(p_1)|_{\bar a_1} >0$, $|\ph(p_2)|_{\bar a_1} >0$ and there is a special arc $s$ in $\D$ such that
$s_- = \al(0)$ and $s_+ = \al(v)$.

$\mathrm{(b)}$  Suppose $P_W(\ff, v_1, v_2)  = p_1p_2p_3$ is a factorization of $P_W$ such that  $|\ph(p_1)|_{\bar a_1} =0$, $|\ph(p_2)|_{\bar a_1} = 0$, and if $q_i = \al(p_i)$, $i=1,2,3$, $e_1, e_2$ are the first, last, resp., edges of $q_2$, then neither of
$e_1, e_2$ is an $a_1$-edge. Then there exists an $a_k$-band $\Gamma$ in $\D$ whose standard boundary is
$\p \Gamma = e_i r_i e_{3-i} r_{3-i}$, where  $i=1$ if $\ph(e_1) \in \A$ and $i=2$ if $\ph(e_1)^{-1} \in \A$.
In particular, the subdiagram  $\D_1$ of $\D$ defined by $\p \D_1 = q_1 r_2^{-1} q_3$, see Fig.~7.1,  contains no faces and  if
$\D_2$ is the subdiagram of $\D$ defined by $\p \D_2 = q_2 r_1^{-1}$, see Fig.~7.1, then $| \D(2) | = | \Gamma(2) | + | \D_2(2) |$ and $\ph(q_2) \overset  {\GG_3} = a_1^\ell$, where $\ell = 0$ if $k>2$, $\ell = n_1 \ell_0$ for some integer $\ell_0$ if
 $\ph(e_1) = a_2$,  and $\ell = n_2 \ell'_0$ for some integer $\ell'_0$ if
 $\ph(e_1) = a_2^{-1}$.
\end{lem}

\begin{center}
\begin{tikzpicture}[scale=.53]
\draw  plot[smooth, tension=.5] coordinates {(-2,-3.)(-3.93,-1.5) (-3.9,1.5) (-2,2.4) (-.1,1.5) (-0.07,-1.5)(-2,-3.)};
\draw(-4.1,.5) -- (0.1,.5);
\draw(-4.1,-.7) -- (0.1,-.7);
\draw  [-latex]  (-4.1,.5) -- (-2.,.5);
\draw [-latex]  (0.1,-.7)  --(-2.1,-.7) ;
\draw  [-latex]  (-3.5,-2) -- (-3.6,-1.9);
\draw  [-latex]  (-.5,-2) -- (-.6,-2.08);
\draw [-latex]  (-4.12,-.1)  --(-4.12,-.0) ;
\draw [-latex]  (.12,-.1)  --(.12,-.19) ;
\node at (-2,2.87) {$q_2$};
\draw [-latex](-2.1,2.4) --(-1.9,2.4);
\node at (-2,-.1) {$\Gamma$};
\node at (-4.6,-0.07) {$e_1$};
\node at (.6,-0.1) {$e_2$};
\node at (-2.,0.9) {$r_1$};
\node at (-2.,-1.2) {$r_2$};
\node at (-2,-2) {$ \Delta_{1}$};
\node at (-2,1.6) {$ \Delta_{2}$};
\draw  (-4.1,-.7) [fill = black] circle (.05);
\draw  (.1,-.7) [fill = black] circle (.05);
\node at (-2,-5.) {Fig.~7.1};
\node at (-2,-3.5) {$\alpha(0)$};

\draw  (-2,-3) [fill = black] circle (.05);
\node at (-3.9,-2.2) {$q_1$};
\node at (-.1,-2.2) {$q_3$};
\end{tikzpicture}
\end{center}

\begin{proof} Since $|W|_{\bar a_1} >0$, it follows that there is a unique factorization $P_W  = p_1p_2p_3$ such that
$|\ph(p_1)|_{\bar a_1} =0$, $|\ph(p_2)|_{\bar a_1} = 0$, and the first and the last edges of $p_2$ are  not
 $a_1$-edges. Denote $q_i = \al(p_i)$, $i=1,2,3$, and let $e_1, f$ be the first, last, resp., edges of $q_2$.

Let $\ph(e_1)= a_k^\delta$, $\delta =\pm 1$. Since $k >1$, there is an $a_k$-band $\Gamma$ whose standard boundary is
$\p \Gamma = e_i r_i e_{3-i} r_{3-i}$, where  $i=1$ if $\ph(e_1) = a_k$ and $i=2$ if $\ph(e_1)= a_k^{-1}$.

First assume that $e_2 \ne f$. By Lemma~\ref{b1}, $r_2$ is a special arc in $\D$. Since
$|\ph(q_1)|_{\bar a_1} =0$, there is a special arc  $r$ in $\D$ such that $r_- = \al(0)$,  $r_+ = (e_1)_-$. Applying
Lemma~\ref{arc} to special arcs $r, r_2^{-1}$, we obtain a special arc $s$ such that $s_- = \al(0)$,  $s_+ = (r_2)_-$.
Now we can see from $e_2 \ne f$ that claim (a) of Lemma~\ref{lema1a} holds for $v = (p_2)_+$.
\smallskip

Suppose $e_2 = f$.  In view of the definition of an $a_k$-band and Lemma~\ref{b1}, we conclude that
claim (b) of Lemma~\ref{lema1a} holds true.
\end{proof}

If $U  \overset  {\GG_3} = a_1^\ell$ for some integer $\ell$, we let $\mu_3(U)$ denote the integer such that the \PWPP\  for presentation \eqref{pr3b}  holds for the pair  $(Ua_1^{-\ell}, \mu_3(U))$. If $U \overset  {\GG_3} \ne a_1^\ell$ for any $\ell$, we set  $\mu_3(U) := \infty$.

\begin{lem}\label{lema2a}  Let $U$ be a word such that  $U  \overset  {\GG_3} = a_1^\ell$ for some integer $\ell$ and
$| U |_{\bar a_1} >0$. Then at least one of the following claims $\mathrm{(a)}$--$\mathrm{(b)}$ holds true.

$\mathrm{(a)}$ There is a factorization  $U \equiv U_1 U_2$ such that $| U_1 |_{\bar a_1}, | U_2 |_{\bar a_1}  >0$,
$U_1  \overset  {\GG_3} = a_1^{\ell_1}$, $U_2 \overset  {\GG_3} = a_1^{\ell_2}$ for some integers $\ell_1, \ell_2$,
$\ell  = \ell_1 + \ell_2$, and $\mu_3(U) = \mu_3(U_1)+ \mu_3(U_2)$.

$\mathrm{(b)}$ There is a factorization
 $
 U \equiv U_1  a_k^{\delta} U_2 a_k^{-\delta} U_3 ,
 $
where $| U_1 |_{\bar a_1} = | U_2 |_{\bar a_1} = 0$, $k \ge 2$, $\delta = \pm 1$, such that $U_2 \overset  {\GG_3} = a_1^{\ell_2}$ for some integer $\ell_2$ and the following is true.  If $k >2$, then $\ell_2 = 0$ and $U_1 U_3 \overset  {0} = a_1^{\ell}$.
If $k =2$ and $\delta = 1$, then $\ell_2/n_1$ is an integer and $U_1 a_1^{\ell_2n_2/n_1}  U_3 \overset  {0} = a_1^{\ell}$.
 If $k =2$ and $\delta = -1$, then $\ell_2/n_2$ is an integer and $U_1 a_1^{\ell_2n_1/n_2}  U_3 \overset  {0} = a_1^{\ell}$.

 Then, in case $k >2$, we have $\mu_3(U) = \mu_3(U_2)$;  in case $k =2$ and $\delta = 1$, we have $\mu_3(U) = \mu_3(U_2) + |\ell_2/n_1|$; in case $k =2$ and $\delta = -1$, we have $\mu_3(U) = \mu_3(U_2) + |\ell_2/n_2|$.
 \end{lem}

\begin{proof} Consider a \ddd\ over \er{pr3b} such that $\ph(\p |_0 \D) \equiv U a_1^{-\ell}$ and $|  \D(2)| = \mu_3(U)$.
Clearly, $\D$ is reduced and  Lemma~\ref{lema1a} can be applied to $\D$.

Assume that Lemma~\ref{lema1a}(a) holds for $\D$.
Then there is a special arc $s$ in $\D$ such that $s_- = \al(0)$ and $s_+ = (q_1)_+$, $\p |_0 \D = q_1 q_2$, and
$|\ph(q_1)|_{\bar a_1},  |\ph(q_2)|_{\bar a_1} >0$. Cutting $\D$ along the  simple path $s$, we obtain two diagrams $\D_1, \D_2$ such that $\p \D_1 = q_1 s^{-1}$, $\p \D_2 = q_2 s$. Since $ |\D(2) | = \mu_3(U)$ and $\ph(s) \overset 0 = a_1^{\ell_0}$ for some $\ell_0$, it follows that if $U_1 := \ph(q_1)$, $U_2a_1^{-\ell} := \ph(q_2)$, $U \equiv U_1 U_2$, then
$ |\D_i(2) |= \mu_3(U_i)$, $i=1,2$. Since $|\D(2) | = |\D_1(2) |+|\D_2(2) | $, we obtain
$$
\mu_3(U) = |\D(2) | = |\D_1(2) |+|\D_2(2) | = \mu_3(U_1)+\mu_3(U_2) ,
$$
as required.

Assuming that Lemma~\ref{lema1a}(b) holds for $\D$, we derive claim (b) from the definitions and  Lemma~\ref{lema1a}(b).
\end{proof}

Let $W$ be a  word over $\A^{\pm 1}$ and let  $| W |_{\bar a_1} >0$.  As in the proof of Theorem~\ref{thm1},
let $W(i,j)$, where $1 \le i \le |W|$ and $0 \le j \le |W|$, denote the subword of $W$ that starts
with the $i$th letter of $W$ and has length $j$. If $W(i,j) \overset  {\GG_3}  = a_1^{\ell_{ij}}$ for some integer $\ell_{ij}$,
we set $\lambda(W(i,j)) := \ell_{ij}$. Otherwise, it is convenient to define $\lambda(W(i,j)) := \infty$.

We now compute the numbers   $\lambda(W(i,j))$, $\mu_3(W(i,j))$ for all $i, j$ by the method of dynamic programming
in which
the parameter is $| W(i,j) |_{\bar a_1} \ge 0$.
In other words, we compute the numbers  $\lambda(W(i,j))$, $\mu_3(W(i,j))$
by induction on parameter $| W(i,j) |_{\bar a_1} \ge 0$.

To initialize, or to make the base step, we note that if $| W(i,j) |_{\bar a_1} = 0$ then
$$
\lambda(W(i,j)) = \ell_{ij} ,  \quad  \mu_3(W(i,j)) = 0
$$
where  $W(i,j)\overset  {0} = a_1^{\ell_{ij}}$.
\smallskip

To make the induction step,
assume that the numbers  $\lambda(W(i',j'))$, $\mu_3(W(i',j'))$  are already computed for all $W(i',j')$ such that
$| W(i',j') |_{\bar a_1} <  | W(i,j) |_{\bar a_1}$, where $W(i,j)$ is fixed and  $| W(i,j) |_{\bar a_1} >0$.

Applying Lemma~\ref{lema2a} to the word $U \equiv  W(i,j)$, we consider all factorizations of the form
$$
U  \equiv U_1 U_2 ,
$$
where $| U_1 |_{\bar a_1}, | U_2 |_{\bar a_1}  >0$. Since $U_1 = W(i, j_1)$, where $j_1 = |U_1|$ and $U_2 = W(i+ j_1, j-j_1)$,
it follows from  $| U_1 |_{\bar a_1}, | U_2 |_{\bar a_1} < | U |_{\bar a_1}$   that the numbers
$\lambda(U_1)$, $\mu_3(U_1)$,  $\lambda(U_2)$, $\mu_3(U_2)$ are available by the induction hypothesis. Hence, we can compute the minimum
\begin{gather}\label{minM}
M_a = \min (  \mu_3(U_1) +  \mu_3(U_2) )
\end{gather}
over all such factorizations $U \equiv U_1 U_2$. If $M_a < \infty$ and $M_a =  \mu_3(U'_1) +  \mu_3(U'_2)$ for some factorization
$U \equiv U'_1 U'_2$, then we set $L_a := \lambda(U'_1) + \lambda(U'_2)$.

We also consider a factorization for $U = W(i,j)$ of the form
\begin{gather}\label{eeqq4}
 U \equiv U_1  a_k^{\delta} U_2 a_k^{-\delta} U_3 ,
\end{gather}
where $| U_1 |_{\bar a_1} = | U_2 |_{\bar a_1} = 0$, and $k \ge 2$, $\delta = \pm 1$.

If such a factorization \er{eeqq4} is impossible, we set  $M_b = L_b = \infty$.

Assume that a factorization  \er{eeqq4} exists (its uniqueness is obvious).  Since
$$
U_2 = W(i+|U_1|+1, |U_2|) ,  \quad | U_2 |_{\bar a_1} < | U |_{\bar a_1} ,
$$
it follows that the numbers  $\lambda(U_2)$, $\mu_3(U_2)$ are available.  Denote $U_1 U_3 \overset  {0} = a_1^{\ell}$.

If $k >2$ and $\lambda(U_2) =0$, we set
$$
L_b := \ell , \quad  M_b := \mu_3(U_2) .
$$

If $k >2$ and $\lambda(U_2) \ne 0$, we
set  $L_b = M_b := \infty$.

If $k =2$ and $\delta = 1$,  we check whether $\lambda(U_2)$ is divisible by $n_1$. If so, we set
$$
L_b := \ell + \lambda(U_2)n_2/n_1 , \quad M_b := \mu_3(U_2) + |\lambda(U_2)/n_1| .
$$

If $k =2$, $\delta = 1$,  and  $\lambda(U_2)$ is not divisible by $n_1$,  we set $L_b = M_b := \infty$.

If $k =2$ and $\delta = -1$,  we check whether $\lambda(U_2)$ is divisible by $n_2$. If so, we set
$$
L_b := \ell + \lambda(U_2)n_1/n_2, \quad M_b := \mu_3(U_2) + |\lambda(U_2)/n_2| .
$$

If $k =2$, $\delta = -1$,  and  $\lambda(U_2)$ is not divisible by $n_2$,  we set $L_b = M_b := \infty$.
\smallskip

It follows from Lemma~\ref{lema2a} applied to the word $U = W(i,j)$ that if $U \overset  {\GG_3} = a_1^{\ell'}$ for some integer $\ell'$, then $\lambda(U) = \min( L_a, L_b)$ and  $\mu_3(U) = \min( M_a, M_b)$. On the other hand, it is immediate
that if $\lambda(U) = \infty$, then $L_a = L_b = \infty$ and $M_a = M_b = \infty$. Therefore, in all cases, we obtain  that
$$
\lambda(U) = \min( L_a, L_b), \quad \mu_3(U) = \min( M_a, M_b) .
$$
This completes  our computation of the numbers $\lambda(U), \mu_3(U)$ for $U = W(i,j)$.
\smallskip

It follows by the above inductive argument that, for every word $W(i,j)$ with finite $\lambda(W(i,j) )$, we have
\begin{align}\label{bnnd}
\max (| \lambda(W(i,j) ) |,  \mu_3(W(i,j) ) ) \! \le \! | W(i,j) | \max(|n_1|, |n_2|)^{| W(i,j) |_{\bar a_1} } \! \le 2^{O(|W|)} .
\end{align}
Hence, using binary representation for numbers $\lambda(W(i',j') ), \mu_3(W(i',j') )$, we can run the foregoing computation of
numbers $\lambda(U), \mu_3(U)$ for $U = W(i,j)$  in time $O(| W(i,j)|^2)$, including division of
$\lambda (U_2)$ by $n_1$ and by $n_2$ to decide whether $\lambda (U_2)/ n_1$, $\lambda (U_2)/ n_2$ are integers and including additions to compute the minimum \eqref{minM}.

Since the number of subwords $W(i,j)$ of $W$ is bounded by $O(| W|^2)$, it follows that running time of the described algorithm that computes the numbers $\lambda(W ), \mu_3(W)$ is $O(| W|^4)$.  This implies that both the \BWPP\ and the \PWPP\ for presentation \eqref{pr3b} can be solved in deterministic time $O(|W|^4)$.

Theorem~\ref{thm2} is proved.
\end{proof}

\section{Minimizing Diagrams over \eqref{pr1} and Proofs of Theorem~\ref{thm3} and Corollary~\ref{cor2}}

\begin{T3}
Both the diagram problem  and the minimal diagram problem
for group presentation \eqref{pr1} can be solved in deterministic space  $O((\log|W|)^3)$ or in deterministic time $O(|W|^4\log|W|)$.

Furthermore,  let $W$ be a word such that $W \overset {\GG_2} =1$ and let
$$
\tau(\D) = (\tau_1(\D), \ldots, \tau_{s_\tau}(\D))
$$
be a tuple of integers, where the absolute value $| \tau_i(\D) |$ of each $\tau_i(\D)$  represents the number of certain vertices or faces in a disk  diagram $\D$ over \eqref{pr1} such that $\ph(\p \D) \equiv W$. Then, in  deterministic  space $O( (\log |W|)^3 )$, one can algorithmically construct such a minimal diagram $\D$ which is also smallest relative to  the tuple $\tau(\D)$ (the tuples are ordered lexicographically).
\end{T3}

\begin{proof} We start by proving the space part of Theorem~\ref{thm3}.

Let $W$ be a nonempty word over $\A^{\pm 1}$ and let $\Omega$ be a finite sequence of \EO s that transforms the empty \PBS\ for $W$ into a final  \PBS . By Lemma~\ref{lem7}, $W  \overset  {\GG_2} = 1$ and there is a diagram $\D$ over \eqref{pr1} such that $\ph(\p |_{0} \D) \equiv W$.  We now describe an explicit construction of such a diagram $\D = \D(\Omega)$ with property (A) based on the sequence $\Omega$.  Note that this construction is based on the proof of  Lemma~\ref{lem7} and essentially follows that proof.  The question on how to construct the sequence $\Omega$ will be addressed later.

Let $B_0, B_1, \dots,  B_\ell$ be the  sequence of \PBS s associated with  $\Omega$, where $B_0$ is empty and $ B_\ell$  is final. We will construct $\D(\Omega)$ by induction on $i$. Let $B_i$ be obtained from $B_{i-1}$, $i\ge 1$, by an \EO\ $\sigma$ and let $b' \in  B_i$ be the \pbb\ obtained from $b, c \in B_{i-1}$ by $\sigma$. Here we assume that one of
$b, c$ (or both) could be missing.

If $\sigma$ is an addition, then we do not do anything.

Assume that $\sigma$ is an extension of type 1 or type 2 on the left and $a(b') = e_1 a(b)$, where $e_1$ is an edge of $P_W$ and $\ph(e_1) = a_{b(3)}^{\e_1}$. Here we use the notation of the definition of an extension of type 1 or type 2 on the left. Then we construct $\D(\Omega)$ so that the edge $\al(e_1) = e$ belongs to the boundary $\p \D(\Omega)$ of $\D(\Omega)$ and $e^{-1}$ belongs to $\p \Pi$, where $\Pi$ is a face of $\D(\Omega)$ such that $\ph ( \p \Pi) \equiv a_{b(3)}^{-\e_1b(4)}$. Note that the label $\ph(\p \Pi)$ is uniquely  determined by either of the \pbb s $b, b'$.

The case of  an extension of type 1 or type 2 on the right is analogous.

Suppose that $\sigma$ is an extension of type 3 and $a(b') = e_1 a(b)e_2$, where $e_1, e_2$ are edges of $P_W$
with $\ph(e_1) = \ph(e_2)^{-1}$. Here we use the notation of the definition of an extension of type 3. Then we construct $\D(\Omega)$ so that  $\al(e_1) =  \al(e_2)^{-1}$ and both  $\al(e_1), \al(e_2)$ belong to the boundary $\p \D(\Omega)$ of $\D(\Omega)$.

If $\sigma$ is a turn or a merger, then we do not do anything to the diagram  $\D(\Omega)$ under construction.

Performing these steps during the process of \EO s  $\Omega$, we will construct, i.e., effectively describe,  a required diagram $\D(\Omega)$ such that $| \D(\Omega)(2) | =b_F(6)$, where $b_F$ is the final \pbb\ of $B_\ell$,  and   $\ph(\p |_{0} \D(\Omega)) \equiv W$. Note that this particular diagram $\D(\Omega)$, depending on the sequence $\Omega$,
may not be minimal or even reduced, however, as in the proof of  Lemma~\ref{lem7},
our construction guarantees that  $\D(\Omega)$ does have the property (A), i.e., the property that if $e$ is an edge of the boundary path $\p \Pi$ of a face $\Pi$ of $\D(\Omega)$, then $e^{-1} \in \p \D(\Omega)$.

We also observe that, in the proof of Lemma~\ref{lem6}, for a given pair $(W, \D_W)$, where $\D_W$ is a diagram over \eqref{pr1} such that   $\ph(\p |_{0} \D_W) \equiv W$, we constructed by induction an operational sequence $\Omega = \Omega(\D_W)$ of \EO s that convert the empty \BS\ into the final \BS\ and has size bounded by
\begin{equation}\label{estt0}
   (11|W|, C(\log |W| +1))) .
\end{equation}

Recall that every \BS\ is also a \PBS .
Looking at details of the construction of the sequence $\Omega(\D_W)$, we can see that if we reconstruct a diagram $\D(\Omega(\D_W))$, by utilizing our above algorithm based upon the sequence  $\Omega(\D_W)$, then the resulting diagram will be identical to the original diagram $\D_W$, hence, in this notation, we can write
\begin{equation}\label{dwd}
\D(\Omega(\D_W)) = \D_W .
\end{equation}

Recall that Savitch's conversion \cite{Sav}, see also \cite{AB}, \cite{PCC}, of nondeterministic computations in space $S$ and time $T$ into deterministic  computations in space $O(S\log T)$ simulates possible nondeterministic computations by using  tuples of configurations, also called instantaneous descriptions, of cardinality $O(\log T)$. Since a configuration takes space $S$
and the number of configurations that are kept in memory at every moment is at most  $O(\log T)$, the space bound  $O(S\log T)$  becomes evident. Furthermore, utilizing the same space   $O(S\log T)$, one can use  Savitch's  algorithm
to compute an actual sequence of configurations that transform an initial configuration into a final configuration.  To do this, we consider every configuration $\psi$ together with a number $k$, $0 \le k \le T$, thus replacing every configuration  $\psi$ by a pair $(\psi, k)$, where $k$ plays the role of counter. We apply Savitch's algorithm to  such pairs $(\psi, k)$, so that $(\psi_2, k+1)$ is obtained from $(\psi_1, k)$  by a single operation whenever $\psi_2$ is obtained from $\psi_1$ by a single command.  If $\psi$ is a final configuration, then we also assume that $(\psi, k+1)$ can be obtained from $(\psi, k)$, $0 \le k \le T-1$,  by a single operation.

Now we apply Savitch's  algorithm to two pairs $(\psi_0, 0)$,  $(\psi_F, T)$, where $\psi_0$ is an initial configuration and  $\psi_F$ is a final configuration. If
 Savitch's  algorithm  accepts these two pairs $(\psi_0, 0)$,  $(\psi_F, T)$, i.e., the algorithm turns the pair $(\psi_0, 0)$ into  $(\psi_F, T)$,  then there will be a corresponding sequence
 \begin{equation}\label{eqom1}
    (\psi_0, 0), (\psi_1, 1), \dots, (\psi_i, i), \dots, (\psi_F, T)
\end{equation}
of pairs such that   $(\psi_i, i)$ is obtained from $(\psi_{i-1}, i-1)$, $i \ge 1$, by a single operation.
Observe that we can retrieve  any term, say, the $k$th term of the sequence \er{eqom1}, by rerunning
Savitch's  algorithm and memorizing a current pair $(\psi, k)$ in a separate place in memory. Note that the current
value of $(\psi, k)$ may change many times during computations but in the end it will be the desired pair $(\psi_k, k)$. Clearly, the space needed to run this modified Savitch's  algorithm  is still $O(S\log T)$ and, consecutively
outputting these pairs $(\psi_0, 0), (\psi_1, 1), \dots$, we will construct the  required sequence  \er{eqom1} of configurations.
\medskip

Coming back to our specific  situation, we recall that a \PBS\ plays the role of a configuration, the empty \PBS\ is an initial configuration and a final \PBS\ is a final configuration. Therefore, it follows from Lemmas~\ref{lem6}--\ref{lem7},  and the foregoing equality  \er{dwd} that picking the empty
\PBS\  $B_0$ and a final \PBS\  $\{ b_F \}$, where $b_F = (0, |W|, 0,0,0, n')$, we can use Savitch's algorithm
to verify in deterministic space $O((\log |W|)^3)$ whether $W  \overset  {\GG_2} = 1$ and whether there is a \ddd\
$\D$ over \eqref{pr1} such that  $\ph(\p \D) \equiv W$ and   $| \D(2) | = n'$.
In addition, if the algorithm accepts, then the algorithm also constructs an operational sequence
$B_0, B_1, \dots, B_\ell = \{ b_F \}$ of \PBS s of size bounded by \eqref{estt0} which, as was discussed above, can be used to construct a \ddd\ $\D_1$ with  property (A) such that $\ph(\p \D_1) \equiv W$ and   $| \D_1(2) | = n'$. It is clear
that the entire construction of $\D_1$ can be done in deterministic space $O((\log |W|)^3)$, as desired.
\medskip

To construct a minimal disk diagram for $W$, we consecutively choose $n' =0,1,2, \dots$ and run  Savitch's  algorithm to determine whether the algorithm can turn the empty \PBS\ into a final \PBS\  $\{ b_F \}$, where $b_F(6) = n'$. If $n_0$ is the minimal such  $n'$, then we run  Savitch's  algorithm  again to construct an operational sequence
of \PBS s and a corresponding \ddd\ $\D_0$ such that  $\ph(\p \D_0) \equiv W$ and   $| \D_0(2) | = n_0$. Similarly to the above arguments, it follows from Lemmas~\ref{lem6}--\ref{lem7}, and from the equality  $\D(\Omega(\D_W)) = \D_W$, see \eqref{dwd}, that  $\D_0$ is a minimal diagram and the algorithm uses  space $O((\log |W|)^3)$. Thus, the minimal
diagram problem can be solved for  presentation \eqref{pr1} in  deterministic space $O((\log |W|)^3)$.
\medskip

Now we will discuss the additional statement of Theorem~\ref{thm3}.
For a diagram $\D$ over presentation \er{pr1}  such that    $\ph(\p |_{0} \D) \equiv W$ and $\D$ has  property (A),   we consider a parameter
\begin{equation}\label{tau}
    \tau(\D) = ( \tau_1(\D), \dots, \tau_{s_\tau}(\D)) ,
\end{equation}
where $ \tau_i(\D)$ are integers that  satisfy  $0 \le | \tau_i(\D)| \le C_\tau |W|$,  $C_\tau >0$ is a fixed constant, and that represent numbers of  certain vertices, edges, faces of $\D$ (we may also consider some functions of the numbers of  certain vertices, edges, faces of $\D$ that are computable in space $O((\log |W|)^3)$ and  whose absolute values do not exceed $C_\tau |W|$).

For example, if we  set $\tau_1(\D) := |\D(2)|$ then we would be constructing a minimal
 diagram  which is also smallest with respect to  $\tau_2(\D)$, then smallest with respect to  $\tau_3(\D)$ and so on.  Let us set  $\tau_2(\D) := |\D(2)_{a_1^{\pm n_1}}|$, where  $|\D(2)_{a_1^{\pm n_1}}|$ is the number of faces $\Pi_2$ in $\D$ such that  $\ph( \p \Pi_2) \equiv   a_1^{\pm n_1}$,  $\tau_3(\D) := |\D(2)_{a_2^{ n_2}}|$, where  $|\D(2)_{a_2^{ n_2}}|$ is the number of faces $\Pi_3$ in $\D$ such that  $\ph( \p \Pi_3) \equiv   a_2^{ n_2}$.

We may also consider functions such as $\tau_4(\D) :=  \sum_{ k_{4,1} \le  |\p \Pi | \le  k_{4,2} }  |\p \Pi |$, where the summation takes place  over all faces $\Pi$ in $\D$ such that  $ k_{4,1} \le  |\p \Pi | \le  k_{4,2}$, where $k_{4,1}, k_{4,2}$ are fixed integers with $0 \le k_{4,1} \le  k_{4,2} \le |W|$. Note that the latter bound follows from Lemma~\ref{vk2}(a).

If $v$ is a vertex of $\D$, let $\deg v$ denote the {\em degree} of $v$, i.e., the number of oriented edges $e$ of $\D$ such that  $e_+ = v$. Also, if $v$ is a vertex of $\D$, let $\deg^F v$ denote the {\em face  degree} of $v$, i.e.,
the number of faces  $\Pi$ in $\D$ such that $v \in \p \Pi$. Note that it follows from Lemma~\ref{vk2}(a) that  every face $\Pi$ of a \ddd\ $\D$ with  property (A) can contribute at most 1 to this sum.

We may also consider entries in $\tau(\D) $ that count
the number of vertices in $\D$ of certain degree. Note that doing this would be especially interesting and meaningful
when the presentation \eqref{pr1} contains no defining relations, hence, diagrams over \eqref{pr1}  are diagrams over the free
group $F(\A) = \langle \A \ \| \ \varnothing    \rangle $ with no relations. For instance, we may set
$\tau_5(\D) := |\D(0)_{\ge 3}|$, where $ |\D(0)_{\ge 3}|$ is the number of vertices in $\D$ of degree at least 3, and set $\tau_6(\D) := |\D(0)^F_{[k_{6,1}, k_{6,2} ]}|$, where $|\D(0)^F_{[k_{6,1}, k_{6,2} ]}|$ is the number of vertices $v$ in $\D$ such that    $ k_{6,1}\le \deg^F v \le  k_{6,2}$, where $k_{6,1}, k_{6,2}$  are fixed integers with $0 \le k_{6,1} \le k_{6,2} \le |W|$.

Our goal is to make some modification of the foregoing algorithm that would enable us to construct a \ddd\ $\D$  over \eqref{pr1} such that  $\ph(\p |_{0} \D_W) \equiv W$, $\D$ has  property (A), and $\D$ is smallest relative to the parameter $\tau(\D)$. Recall that the tuples $\tau(\D)$ are ordered in the standard lexicographical way.  Note that if $\tau_1(\D) = - |\D(2)|$ then $\D$ would be a  diagram for $W$ with property (A) which has the maximal number of faces.

Our basic strategy remains unchanged, although we will need more complicated bookkeeping (to be introduced below). We start by picking a value for the parameter $\wtl \tau = (\wtl \tau_1, \dots, \wtl \tau_{s_\tau} )$, where $\wtl \tau_i$ are fixed integers with  $|\wtl \tau_i| \le C_\tau |W|$.  As before, using Lemmas~\ref{lem6},~\ref{lem7},
and the identity $\D(\Omega(\D_W)) = \D_W$, see \er{dwd}, we can utilize
Savitch's algorithm  which verifies, in deterministic space $O( (\log |W|)^3 )$, whether the empty \PBS\ $B_0$ can be turned by \EO s into a final \PBS\ $B = \{ b_F \}$ which also changes the $\tau$-parameter from  $(0,0, \dots, 0)$ to $\wtl \tau$. If this is possible, then the algorithm also computes an operational sequence $B_0, B_1, \dots, B_\ell = B$ of \PBS s and, based on this sequence, constructs a \ddd\ $\D$ such that $\ph(\p \D) \equiv W$, $\D$ has property (A) and $\tau(\D) = \wtl \tau$. To make sure that $\D$ is smallest relative to the parameter $\tau(\D)$, we can use Savitch's algorithm  for every $\wtl \tau' <  \wtl \tau$ to verify that the empty \PBS\ $B_0$ may not be turned by \EO s into a final \PBS\ $B = \{ b_F \}$ so that  the $\tau$-parameter concurrently changes from  $(0,0, \dots, 0)$ to $\wtl \tau'$. This concludes the description of our algorithm
for construction of a  diagram  $\D$ over \eqref{pr1} such that  $\ph(\p |_{0} \D_W) \equiv W$, $\D$ has  property (A) and $\D$ is smallest relative to the parameter $\tau(\D)$.
\medskip

Let us describe necessary modifications in bookkeeping. First, we add six more integer entries to each \pbb\ $b$. Hence, we now have
\begin{equation*}
   b= (b(1), \dots,  b(6), b(7), \dots,   b(12) ) ,
\end{equation*}
where $b(1), \dots, b(6)$ are defined exactly as before, $b(7), \dots, b(12)$ are integers such that
\begin{align}\label{cond1}
b(7)+b(8)  = b(5) , \quad   0 \le b(7), b(8)\le b(5) \quad & \mbox{if}
\quad  b(5) \ge 0 ,  \\ \label{cond2}
b(7)+b(8)  = b(5) , \quad  b(5) \le b(7),  b(8) \le 0 \quad &  \mbox{ if} \quad   b(5) \le 0 .
\end{align}
The numbers $b(9), \dots, b(12)$ satisfy the inequalities
$$
0 \le  b(9), \dots, b(12) \le |W|
$$
and   the equalities $b(7)= b(8) = b(11)= b(12) =0$ whenever $b$ has type PB1.
The purpose of these new entries, to be specified below, is to keep track of the information on degrees and face degrees of the vertices of the diagram  $\D(\Omega)$ over \eqref{pr1} being constructed by means of an operational sequence $\Omega$ of \EO s.
\medskip

We now inductively describe the changes over the entries  $b(7), \dots, b(12)$ and over the parameter
$\tau(\D) = ( \tau_1(\D), \dots, \tau_{s_\tau}(\D))$ that are  done in the process of performance of an operational  sequence $\Omega$ of \EO s over \PBS s that change the empty \PBS\ into a final \PBS\  for $W$, where $W  \overset  {\GG_2} = 1$ and $|W| >0$.

To make the inductive definitions below easier to understand, we will first make informal remarks about meaning of the new entries
$b(7), \dots,   b(12)$.

Suppose that $b$ is a \pbb\ of type  PB1. Then $b(7)= b(8) =0$ and $b(11)=b(12)=0$. The entry $b(9)$ represents the current (or intermediate) degree $\deg v$ of the vertex $v = \al(b(1)) = \al(b(2))$ which is subject to change in process of construction of the diagram $\D(\Omega)$.  The entry  $b(10)$ represents the current  face degree $\deg^F v$ of the vertex $v = \al(b(1)) = \al(b(2))$ which is also subject to change. For example, if a \pbb\ $b'$ is obtained from \pbb s $b, c$ of type PB1 by a merger, then $b'(9) :=b(9)+c(9)$ and  $b'(10) :=b(10)+c(10)$. Note that  $b'(9), b'(10)$ are still intermediate degrees of the vertex $\al(b'(1)) = \al(b'(2))$, i.e., $b'(9), b'(10)$ are not actual degrees of $\al(b'(1))$,  because there could be more \EO s such as extension of type 3, mergers, and turns that could further increase $b'(9), b'(10)$.
\medskip

Assume that $b$ is a \pbb\ of type PB2. Then the entry $b(7)$ represents the exponent in the power $a_{b(3)}^{b(7)} \equiv \ph(u_1^{-1})$,  where $u_1$ is the arc of $\p \Pi$, where    $\Pi$ is a face   such that $(u_1)_- = \al(v)$, $(u_1)_+ = \al(b(1))$, see Fig.~8.1, and $v$ is the vertex of $P_W$ at which a turn operation was performed to get a \pbb\ $\bar b$ of type PB2 which, after a number of extensions of type 1 and mergers with \pbb s of type PB1, becomes $b$.

\begin{center}
\usetikzlibrary{arrows}
\begin{tikzpicture}[scale=.7]
\draw (7.75,-2.622) arc (20:155:4);
\draw  (1.1,-1.24) [fill = black] circle (.06);
\draw  (4,0) [fill = black] circle (.06);
\draw  (6.8,-1.14) [fill = black] circle (.06);
\draw  [-latex](2.7,-0.21) -- (2.4,-0.33);
\draw  [-latex](5.7,-0.39) -- (5.5,-0.3);
\node at (4,-2.) {$\Pi$};
\node at (4,-3.1) {Fig.~8.1};
\node at (2.4,0.2) {$u_1$};
\node at (5.65,0.2) {$u_2$};
\node at (4,0.44) {$\alpha(v)$};
\node at (0.12,-1.1) {$\alpha(b(1))$};
\node at (7.82,-1.) {$\alpha(b(2))$};
\node at (4,-1.) {$u = u_2 u_1$};
\node at (-1,0.2) {$\varphi(u_1^{-1}) \equiv a_{b(3)}^{b(7)}$};
\node at (8.4,0.2) {$\varphi(u_2^{-1}) \equiv a_{b(3)}^{b(8)}$};
\end{tikzpicture}
\end{center}

\noindent
Similarly, the entry $b(8)$ represents the exponent in the power $a_{b(3)}^{b(8)} \equiv \ph(u_2^{-1})$, where $u_2$ is the arc of $\p \Pi$ such that $(u_2)_- = \al(b(2))$, $(u_1)_+ = \al(v)$, and $v$ is the vertex of $P_W$ defined as above,  see Fig.~8.1. Since $\ph(u^{-1}) \equiv  a_{b(3)}^{b(5)}$, where $u = u_2u_1$  is the arc of $\p \Pi$ defined by
$u_- = \al(b(2))$, $u_+ = \al(b(1))$, see Fig.~8.1, it follows from the definitions that the conditions \er{cond1}--\er{cond2} hold true.

As in the above case when $b$ has type PB2, the entry $b(9)$ represents the current (or intermediate) degree $\deg v_1$ of the vertex $v_1 = \al(b(1))$ which is subject to change in the process of construction of $\D(\Omega)$.  As above, the entry  $b(10)$ represents the current face degree $\deg^F v_1$ of $v_1 = \al(b(1))$ which is also subject to change. The entry  $b(11)$ is the intermediate  degree $\deg v_2$ of the vertex $v_2 = \al(b(2))$ and  $b(12)$ represents the current face degree $\deg^F v_2$ of the vertex $v_2 = \al(b(2))$.
\medskip

Our description for meaning of the new  entries $b(7), \dots, b(12)$ in thus augmented \pbb\ $b$ is complete and we now inductively describe the changes over these new entries $b(7), \dots, b(12)$ and over the parameter $\tau(\D)$ in process of performance of a sequence $\Omega$  of \EO s over  \PBS s  that change the empty \PBS\ into a final \PBS\  for a word $W$ such that $W  \overset  {\GG_2} = 1$ and $|W| >0$.

As above, assume that $\Omega$ is a sequence  of \EO s that change the empty \PBS\ into a final \PBS\  for $W$.
Let $B_0,  B_1,  \dots, B_\ell$ be the corresponding to $\Omega$ sequence  of \PBS s, where
$B_0$ is empty and $B_\ell$ is final.

If $b \in B_{i+1}$, $i \ge 0$, is a starting \pbb , then we set
$$
b(7)= b(8)= b(9)= b(10)= b(11)= b(12)=0 .
$$
As was mentioned above, the entries  $b(1), \dots, b(6)$ in every \pbb\ $b \in \cup_{j=0}^\ell B_{j}$ are defined exactly as before.

Suppose that  a \pbb\ $b' \in B_{i+1}$, $i \ge 1$, is obtained from $b \in B_{i}$ by an extension of type 1 on the left and $a(b') = e_1 a(b)$, where $a(b'), a(b)$ are the arcs of   $b', b$, resp., and $e_1$ is an edge of $P_W$, $\ph(e_1) = a_{b(3)}^{\e_1}$, $\e_1 = \pm 1$. Recall that both  \pbb s  $b', b$ have type PB2.

First we assume that $b(5) \ne 0$, i.e., one of $b(7), b(8)$ is different from 0. Then we set
\begin{align*}\notag
  & b'(7)  :=  b(7) +  \e_1 , \quad  b'(8) := b(8) , \quad  b'(9) := 2 , \\
   &  b'(10)  := 1 , \quad  b'(11) := b(11) , \quad  b'(12) := b(12) .
\end{align*}
Note that it follows from the definitions that $|b'(7)| = |b(7)| +1$.
We also update those entries in the sequence $\tau(\D)$ that are affected by the fact that
$v_1 = \al(b(1))$ is now known to be a vertex  of the diagram $\D(\Omega)$ (which is still under construction)
such that $\deg v_1 = b(9)$ and $\deg^F v_1 = b(10)$.

For example, if  $\tau_5(\D) = |\D(0)_{\ge 3}|$, where $ |\D(0)_{\ge 3}|$ is the number of vertices in $\D$ of degree at least 3, and $b(9) \ge 3$, then $\tau_5(\D)$ is increased by 1. If
$\tau_6(\D) = |\D(0)^F_{[k_{6,1}, k_{6,2} ]}|$, where $|\D(0)^F_{[k_{6,1}, k_{6,2} ]}|$ is the number of vertices $u$ in $\D$ such that    $ k_{6,1}\le \deg^F u \le  k_{6,2}$, where $k_{6,1}, k_{6,2}$  are fixed integers with $0 \le k_{6,1} \le k_{6,2} \le |W|$, and if $\deg^F \al(b(1)) = b(10)$ satisfies $k_{6,1} \le   b(10) \le k_{6,2}$, then we increase $\tau_6(\D)$  by 1.

Suppose that $b(5) = 0$, i.e.,  $b(7)= b(8) =0$. Then we set
\begin{align*}\notag
  & b'(7)  :=    \e_1 , \quad  b'(8) := 0 , \quad  b'(9) := 2 , \\
   &  b'(10)  := 1 , \quad  b'(11) := b(9) , \quad  b'(12) := b(10) .
\end{align*}
These formulas, see also \er{eqq5}--\er{eqq6} and other formulas below, reflect the convention that, for a \pbb\ $b$ of type PB2, the information about edge and face degrees of the vertex $\al(b(2))$ is stored in components $b(9), b(10)$, resp., whenever $\al(b(1)) = \al(b(2))$, i.e., $b(5)=0$. However, when $\al(b(1)) \ne \al(b(2))$, i.e., $b(5) \ne 0$, this  information for the vertex $\al(b(2))$  is separately kept in components $b(11), b(12)$, whereas this information for the vertex  $\al(b(1))$  is stored in components $b(9), b(10)$. In the current case $b(5) = 0$, no change to  $\tau(\D)$ is done.

The case of an extension of type 1 on the left is complete.
\medskip

Assume that a \pbb\ $b' \in B_{i+1}$, $i \ge 1$, is obtained from $b \in B_{i}$ by an extension of type 1 on the right.
This case is similar to its ``left" analogue studied above but not quite symmetric, because of the way we keep information about degrees of vertices and we will write down all necessary formulas. As above,  both \pbb s $b', b$ have type PB2.

Let $a(b') = a(b)e_2$, where $a(b'), a(b)$ are the arcs of $b', b$, resp., and $e_2$ is an edge of $P_W$, $\ph(e_2) = a_{b(3)}^{\e_2}$, $\e_2 = \pm 1$.
If $b(5) \ne 0$, then we define
\begin{align}\label{eqq5}
  & b'(7)  :=  b(7) , \quad  b'(8) := b(8)  + \e_2 , \quad  b'(9) :=  b(9) , \\ \label{eqq6}
   &  b'(10)  := b(10) , \quad  b'(11) := 2 , \quad  b'(12) := 1 .
\end{align}
Note that it follows from the definitions that $|b'(8)| = |b(8)| +1$.
We also update those entries in the sequence $\tau(\D)$ that are affected by the fact that
$v_2 = \al(b(2))$ is a vertex  of $\D(\Omega)$ (which is still under construction) such that $\deg v_2 = b(11)$ and
$\deg^F v_2 = b(12)$.

For example, if  $\tau_5(\D) = |\D(0)_{\ge 3}|$ and $b(11) \ge 3$, then $\tau_5(\D)$ is increased by 1.
If $\tau_6(\D) = |\D(0)^F_{[k_{6,1}, k_{6,2} ]}|$ and if $\deg^F \al(b(2)) = b(12)$ satisfies the inequalities $k_{6,1} \le   b(12) \le k_{6,2}$, then we increase $\tau_6(\D)$  by 1.

If $b(5) = 0$, then, to define $b'$, we again use formulas \er{eqq5}--\er{eqq6} but
we do not make any changes to  $\tau(\D)$.

The case of  an extension of type 1 on the right is complete.
\medskip

Suppose that a \pbb\ $b' \in B_{i+1}$, $i \ge 1$, is obtained from $b \in B_{i}$ by an extension of type 2 on the left. Denote $a(b') = e_1 a(b)$, where $a(b'), a(b)$ are the arcs of $b', b$, resp., and $e_1$ is an edge of $P_W$, $\ph(e_1) = a_{b(3)}^{\e_1}$, $\e_1 = \pm 1$. Recall that $b$ has type PB2 and $b'$ has type PB1.

First we assume that $b(5) \ne 0$, i.e., one of $b(7), b(8)$ is different from 0. Then we set
\begin{align*}
  & b'(7)  :=  0 , \quad  b'(8) := 0 , \quad  b'(9) :=  b(11)  , \\
   &  b'(10)  :=  b(12) , \quad  b'(11) := 0 , \quad  b'(12) := 0 .
\end{align*}
We also update  the sequence $\tau(\D)$ according to the information  that
$v_1 = \al(b(1))$ is a vertex  of $\D(\Omega)$ (which is still under construction) such that $\deg v_1 = b(9)$,
$\deg^F v_1 = b(10)$, and that $\D(\Omega)$ contains a face $\Pi_b$ such that  $\ph(\p \Pi_b) \equiv  a_{b(3)}^{-\e_1 b(4)}$.

For example, if  $\tau_5(\D) = |\D(0)_{\ge 3}|$ and $b(9) \ge 3$, then $\tau_5(\D)$ is increased by 1. If
$\tau_6(\D) = |\D(0)^F_{[k_{6,1}, k_{6,2} ]}|$ and if $\deg^F \al(b(1)) = b(10)$ satisfies $k_{6,1} \le   b(10) \le k_{6,2}$, then we increase $\tau_6(\D)$  by 1.   If  $\tau_2(\D) = |\D(2)_{a_1^{\pm n_1}}|$, where  $|\D(2)_{a_1^{\pm n_1}}|$ is the number of faces $\Pi_2$ in $\D$ such that  $\ph( \p \Pi_2) \equiv   a_1^{\pm n_1}$, and $b(3) = 1$, $b(4) = n_1$, then
we increase $\tau_2(\D)$  by 1.  If  $\tau_3(\D) = |\D(2)_{a_2^{ n_2}}|$, where  $|\D(2)_{a_2^{ n_2}}|$ is the number of faces $\Pi_3$ in $\D$ such that  $\ph( \p \Pi_3) \equiv   a_2^{ n_2}$, and $b(3) = 2$, $-\e_1 b(4) =   n_2$, then
we increase $\tau_3(\D)$  by 1.  If $\tau_4(\D) =  \sum_{ k_{4,1} \le  |\p \Pi | \le  k_{4,2} }  |\p \Pi |$, where the summation takes place  over all faces $\Pi$ in $\D$ such that  $ k_{4,1} \le  |\p \Pi | \le  k_{4,2}$, where $k_{4,1}, k_{4,2}$ are fixed integers with $0 \le k_{4,1} \le  k_{4,2} \le |W|$, and  $k_{4,1} \le  b(4) \le  k_{4,2}$,
then we increase $\tau_4(\D)$  by 1.

Now assume that $b(5) = 0$,  i.e., $b(7)= b(8)=0$, and so $b(4) =1$.  Then
 we set
\begin{align*}
  & b'(7)  :=  0 , \quad  b'(8) := 0 , \quad  b'(9) := b(9)  , \\
  &  b'(10)  :=   b(10) , \quad  b'(11) := 0 , \quad  b'(12) := 0 .
\end{align*}
We also update  the tuple  $\tau(\D)$ according to the information  that  $\D(\Omega)$ contains a face $\Pi_b$ such that  $\ph(\p \Pi_b) \equiv  a_{b(3)}^{-\e_1}$.

The case of  an extension of type 2 on the left is complete.
\medskip

Suppose that a \pbb\ $b' \in B_{i+1}$, $i \ge 1$, is obtained from $b \in B_{i}$ by an extension of type 2 on the right.
This case is similar to its ``left" analogue discussed above but is not quite symmetric and we will write down all necessary formulas.

Denote $a(b') =  a(b)e_2$, where $a(b'), a(b)$ are the arcs of $b', b$, resp., and $e_2$ is an edge of $P_W$, $\ph(e_2) = a_{b(3)}^{\e_2}$, $\e_2 = \pm 1$. Recall that $b$ has type PB2 and $b'$ has type PB1.

First we assume that $b(5) \ne 0$.  Then we set
\begin{align}\label{eqq1}
  & b'(7)  :=  0 , \quad  b'(8) := 0 , \quad  b'(9) :=  b(9)  , \\ \label{eqq2}
   &  b'(10)  :=  b(10) , \quad  b'(11) := 0 , \quad  b'(12) := 0 .
\end{align}
We also update  the sequence $\tau(\D)$ according to the information  that
$v_2 = \al(b(2))$ is a vertex  of $\D(\Omega)$ (which is still under construction) such that $\deg v_2 = b(11)$,
$\deg^F v_2 = b(12)$, and that $\D(\Omega)$ contains a face $\Pi_b$ such that $\ph(\p \Pi_b) \equiv  a_{b(3)}^{-\e_2 b(4)}$.

For example, if  $\tau_5(\D) = |\D(0)_{\ge 3}|$ and $b(11) \ge 3$, then $\tau_5(\D)$ is increased by 1. If
$\tau_6(\D) = |\D(0)^F_{[k_{6,1}, k_{6,2} ]}|$ and if $\deg^F \al(b(2)) = b(12)$ satisfies $k_{6,1} \le   b(12) \le k_{6,2}$, then we increase $\tau_6(\D)$  by 1.   If  $\tau_2(\D) = |\D(2)_{a_1^{\pm n_1}}|$, as above, and $b(3) = 1$, $b(4) = n_1$, then
we increase $\tau_2(\D)$  by 1.  If  $\tau_3(\D) = |\D(2)_{a_2^{ n_2}}|$, as above, and $b(3) = 2$, $-\e_1 b(4) =   n_2$, then we increase $\tau_3(\D)$  by 1.  If $\tau_4(\D) :=  \sum_{ k_{4,1} \le  |\p \Pi | \le  k_{4,2} }  |\p \Pi |$, as above,  and  $k_{4,1} \le  b(4) \le  k_{4,2}$, then we increase $\tau_4(\D)$  by 1.

Now we assume that $b(5) = 0$,   i.e., $b(7)= b(8)=0$, and so $b(4) =1$.  Then, to define $b$, we again use formulas \er{eqq1}--\er{eqq2}.  We also update  the tuple  $\tau(\D)$ according to the information  that  $\D(\Omega)$ contains a face $\Pi_b$ with $\ph(\p \Pi_b) \equiv  a_{b(3)}^{-\e_2}$.

The case of  an extension of type 2 on the right is complete.
\medskip

Suppose that a \pbb\ $b' \in B_{i+1}$, $i \ge 1$, results from $b \in B_{i}$ by an extension of type 3.
Denote $a(b') =  e_1 a(b)e_2$, where $a(b'), a(b)$ are the arcs of $b', b$, resp., and $e_1, e_2$ are edges of $P_W$, $\ph(e_1) = \ph(e_2)^{-1} = a_{j}^{\e}$, $\e = \pm 1$. Recall that both $b$ and $b'$ have  type PB1.  Then we set
\begin{align*}
  & b'(7)  :=  0 , \quad  b'(8) := 0 , \quad  b'(9) :=  1  , \\
   &  b'(10)  :=  0 , \quad  b'(11) := 0 , \quad  b'(12) := 0 .
\end{align*}
We also update  the tuple $\tau(\D)$ according to the information  that
$v= \al(b(1)) = \al(b(2))$ is a vertex  of $\D(\Omega)$ (which is still under construction) such that $\deg v = b(9)+1$ and
$\deg^F v = b(10)$.

For example, if  $\tau_5(\D) = |\D(0)_{\ge 3}|$ and $b(9) \ge 3$, then $\tau_5(\D)$ is increased by 1. If
$\tau_6(\D) = |\D(0)^F_{[k_{6,1}, k_{6,2} ]}|$ and if $\deg^F v = b(10)$ satisfies $k_{6,1} \le   b(10) \le k_{6,2}$, then we increase $\tau_6(\D)$  by 1.

The case of  an extension of type 3 is complete.
\medskip

Assume that  a \pbb\ $b' \in B_{i+1}$, $i \ge 1$, is obtained from $b \in B_{i}$ by a turn operation.
 Recall that $b$ has  type PB1 and $b'$ has  type PB2.  Then we set
\begin{align*}
  & b'(7)  :=  0 , \quad  b'(8) := 0 , \quad  b'(9) := b(9)+2  , \\
   &  b'(10)  :=   b(10)+1 , \quad  b'(11) := 0 , \quad  b'(12) := 0 .
\end{align*}

No change over $\tau(\D)$ is necessary under  a turn operation.
The case of  a turn operation is complete.
\medskip

Suppose that  a \pbb\ $b' \in B_{i+1}$, $i \ge 1$, is obtained from $b, c \in B_{i}$ by a merger operation.
Without loss of generality, we assume that the \pbb s $b, c$ which satisfy the condition $b(2)=c(1)$, i.e., $b$ is on the left of $c$. Recall that one of $b, c$ must have type PB1 and the other one has  type  PB1 or PB2. Consider three cases corresponding to the types of the  \pbb s $b, c$.

First assume that both $b, c$ have type  PB1.  Then we set
\begin{align*}
  & b'(7)  :=  0 , \quad  b'(8) := 0 , \quad  b'(9) := b(9) + c(9)  , \\
   &  b'(10)  := b(10) + c(10)  , \quad  b'(11) := 0 , \quad  b'(12) := 0 .
\end{align*}
No change over $\tau(\D)$ is made.

Assume that $b$ has type PB1 and $c$ has  type  PB2.
Then, keeping in mind that $b(2)=c(1)$, we set
\begin{align*}
  & b'(7)  :=  c(7) , \quad  b'(8) := c(8) , \quad  b'(9) := b(9) + c(9)  , \\
   &  b'(10)  := b(10) + c(10)  , \quad  b'(11) := c(11)  , \quad  b'(12) :=c(12)  .
\end{align*}
As above, no change over $\tau(\D)$ is necessary.

Assume that $b$ has type PB2 and $c$ has  type  PB1. If $b(5) \ne 0$, then we set
\begin{align*}
  & b'(7)  :=  b(7) , \quad  b'(8) := b(8) , \quad  b'(9) := b(9)  , \\
  &  b'(10)  := b(10)  , \quad  b'(11) := b(11) + c(9)  , \quad  b'(12) := b(12)+ c(10) .
\end{align*}

On the other hand, if $b(5) = 0$, i.e.,   $b(7) = b(8) = 0$,   then we set
\begin{align*}
  & b'(7)  :=  0 , \quad  b'(8) := 0 , \quad  b'(9) := b(9)+ c(9)   , \\
   &  b'(10)  := b(10)+ c(10)   , \quad  b'(11) := 0 , \quad  b'(12) := 0 .
\end{align*}
As before, no change over $\tau(\D)$ is made under  a merger operation.

The case of  a merger operation is complete.
\medskip

Our inductive definitions of extended \pbb s and modifications of $\tau(\D)$ are complete. To summarize, we conclude that
the described changes to the extended  \pbb s and to the tuple $\tau(\D)$
will guarantee that, if our nondeterministic  algorithm, which is based on Lemma~\ref{lem8} and which follows a sequence $\Omega$
of \EO s as above, accepts a pair of configurations $\psi_0$, $\psi_F$, where $\psi_0$ is the empty \PBS\ and  $\psi_F$ is a final \PBS , then the final tuple $\tau(\D)$, i.e., the tuple associated with $\psi_F$, will be equal to the tuple $\tau(\D(\Omega))$ which represents the tuple of actual parameters $\tau_1(\D(\Omega)), \dots, \tau_{s}(\D(\Omega))$  of the diagram $\D(\Omega)$.

Consider augmented configurations $\bar \psi = (B, \tau(\D))$, corresponding to a sequence $\Omega$ of \EO s as above, where $B$ is a system of extended \pbb s and $\tau(\D)$ is the tuple associated with $B$. Recall that, by Lemma~\ref{lem8}, we may assume that  $\Omega$ has size bounded by
$(11 |W|, C (\log |W| +1) )$. As in the proof of Theorem~\ref{thm1}, we note that the entries
$b(1), \dots,  |b(5)|,  b(6)$ are nonnegative and bounded by $\max(|W|, m)$. It follows from the definitions and Lemma~\ref{vk2}(a) that for the new entries $b(7), \dots,  b(12)$, we have that
$|b(7)|,  |b(8)|  \le  |b(5)| \le  |W|$ and
\begin{align*}
  0 \le    b(9),  b(10),  b(11), b(12)  \le  2 |W| .
\end{align*}
By the definition, every entry $\tau_i(\D)$  in  $\tau(\D)$ satisfies $| \tau_i (\D) | \le C_\tau |W|$, $i = 1, \dots, s_\tau$. Since  $s_\tau$, $C_\tau$ are constants, we conclude that the space needed to store an augmented configuration $\bar \psi = (B, \tau(\D))$ is  $O((\log |W|)^2)$.
Therefore, utilizing Savitch's algorithm as before, which now applies to augmented configurations of the form $\bar \psi = (B, \tau(\D))$, we will be able to find out, in deterministic space $O((\log |W|)^3)$, whether the algorithm accepts a pair $(B_0, \tau^0(\D))$,  $(B_F, \tau^F(\D))$,  where
$B_0$ is empty, $\tau^0(\D)$ consists of all zeros, $B_F$ is a final \PBS , and $\tau^F(\D)$ is a final tuple corresponding to $B_F$. Since the space needed to store a final configuration $(B_F, \tau^F(\D))$ is $O(\log |W|)$, we will be able to compute, in deterministic space $O((\log |W|)^3)$,  a lexicographically smallest tuple $\wht \tau^F(\D))$ relative to the property that the pair of augmented  configurations
$(B_0, \tau^0(\D))$,  $(\wht B_F, \wht \tau^F(\D))$, where $\wht B_F$ is some final \PBS , is accepted by Savitch's algorithm. Now we can do the same counter trick as in the beginning of this proof, see \er{eqom1}, to compute, in deterministic space $O((\log |W|)^3)$,  a sequence $\Omega^*$ of \EO s and the corresponding to   $\Omega^*$ sequence of augmented  configurations which turns
$(B_0, \tau^0(\D))$ into $(\wht B_F, \wht \tau^F(\D))$ and which is constructed by Savitch's algorithm. Finally, as  in the beginning of the proof of Theorem~\ref{thm3}, using the sequence
$\Omega^*$, we can construct a desired diagram $\D^* = \D(\Omega^*)$  so that $\tau(\D^*) = \wht \tau^F(\D)$ and this construction can be done in deterministic space $O((\log |W|)^3)$. This completes the proof of the space part of Theorem~\ref{thm3}.
\medskip

To prove the time part  of Theorem~\ref{thm3}, we review the proof of $\mathsf P$ part of Theorem~\ref{thm1}. We observe that our arguments enable us, concurrently with computation of the  number $\mu_2(W[i,j,k,l])$ for every parameterized word $W[i,j,k,l] \in \mathcal S_2(W)$ such that $W(i,j)a_k^\ell \overset {\GG_2} = 1$,  to
inductively construct a minimal diagram  $\D[i,j,k,l]$ over \er{pr1}  such that
$\ph( \p |_0  \D  [i,j,k,l]) \equiv  W(i,j)a_k^\ell$.  Details of this construction are straightforward in each of the subcases considered in the proof of $\mathsf P$ part of Theorem~\ref{thm1}. Note that
the time needed to run this extended algorithm is still $O( |W|^4 \log |W| )$.

Theorem~\ref{thm3} is proven.
\end{proof}

\begin{C2} There is a deterministic algorithm that, for given word $W$ over the alphabet $\A^{\pm 1}$ such that
$W \overset{\FF(\A)}{=} 1$, where  $\FF(\A) = \langle \A \  \|  \  \varnothing \rangle$ is the free group over $\A$, constructs a pattern of cancellations of letters in $W$ that result in the empty word and the  algorithm operates in space $O( (\log |W|)^3 )$.

Furthermore, let $\D$ be a disk diagram over $\FF(\A) $ that corresponds to a pattern of cancellations of letters in $W$, i.e., $\ph(\p \D) \equiv W$, and let
$$
\tau(\D) = (\tau_1(\D), \ldots, \tau_{s_\tau}(\D))
$$
be a tuple of integers, where the absolute value $| \tau_i(\D) |$ of each  $\tau_i(\D)$  represents the number  of vertices in $\D$ of certain degree. Then, also in  deterministic  space $O( (\log |W|)^3 )$, one can algorithmically construct such a diagram $\D$ which is smallest relative to the tuple $\tau(\D)$.
\end{C2}

\begin{proof}   Corollary~\ref{cor2} is immediate from  Theorem~\ref{thm3} and does not need a separate proof. Nevertheless,
it is worth mentioning  that, in the case of presentation
$\FF(\A) = \langle \A \  \|  \  \varnothing \rangle$ of  the free group $\FF(\A)$ over $\A$, our definitions of brackets, \pbb s,  \EO s and subsequent arguments become significantly simpler.
Since there are no relations, we do not define brackets of type B2, nor we define \pbb s of type PB2. In particular, for every bracket or  \pbb\ $b$,  the entries $b(3), b(4), b(5)$ are always zeroes and could be disregarded. In the extended version of a  \pbb\ $b$, defined for minimization of $\D$ relative to $\tau(\D)$,  the entries $b(7), b(8)$, $b(10), b(12)$ are also always zeroes and could be disregarded.  Furthermore, there is no need to consider extensions of type 1, 2 and turn operations. Hence, in this case, we only need \EO s which are additions, extensions of type 3 and mergers over brackets of type B1, \pbb s  of type PB1, and over their systems.
\end{proof}

\section{Construction of Minimal Diagrams over \eqref{pr3b} and Proof of Theorem \ref{thm4} }

\begin{T4}  Suppose that $W$ is a word  over the alphabet
$\A^{\pm 1}$ such that the bounded word problem for presentation \eqref{pr3b}  holds for the pair $(W, n)$.
Then a minimal diagram $\D$ over \eqref{pr3b} such that  $\ph( \p  \D) \equiv W$ can be algorithmically  constructed in deterministic space $O( \max(\log |W|, \log n)(\log |W|)^2)$  or in deterministic time $O( |W|^4)$.

In addition, if $|n_1 | = |n_2 |$ in \eqref{pr3b}, then  the minimal diagram problem for  presentation \eqref{pr3b}  can be solved in deterministic space $O( (\log |W|)^3 )$   or in deterministic time $O( |W|^3\log |W|)$.
\end{T4}

\begin{proof} First we  prove the space part of Theorem~\ref{thm4}.

Let $W$ be a nonempty word over $\A^{\pm 1}$ such that $W \overset{\GG_3}{=} 1$, where $\GG_3$ is defined by  presentation \eqref{pr3b}, and there is a \ddd\  $\D$ over \eqref{pr3b} such that $\ph(\p \D) \equiv W$ and $|\D(2)| \le n$, i.e., the bounded word problem has a positive solution for the pair $(W, n)$.

It follows from Lemma~\ref{8bb} that there is a finite sequence $\Omega$
of \EO s such that  $\Omega$ converts the empty \PBS\ for $W$ into a final one and  $\Omega$ has other properties stated in Lemma~\ref{8bb}. As in the proof of  Theorem~\ref{thm2},  Lemma~\ref{8bb} gives us a
nondeterministic algorithm which runs in time $O(|W|)$ and space $O((\max(\log|W|,\log n) \log|W|)$ and which accepts a word $W$   over $\A^{\pm 1}$  \ifff\  the bounded word problem has a positive solution for the pair $(W, n)$.  Note that here the big-$O$ constants can be written down explicitly (see the proof of Theorem~\ref{thm2}).

Furthermore,  as in the proof of  Theorem~\ref{thm2},  using Savitch's theorem \cite{Sav}, see also \cite{AB}, \cite{PCC}, we obtain a deterministic algorithm which runs in space
\begin{equation}\label{sp}
O((\max(\log|W|,\log n) (\log|W|)^2) ,
 \end{equation}
and which computes a minimal integer $n(W)$, $0 \le n(W) \le n$, such that there is a \ddd\ $\D$ over \eqref{pr3b} so that $\ph(\p \D) \equiv W$ and $|\D(2)| = n(W)$. To do this, we can check by  Savitch's algorithm  whether the empty \PBS\ $B_0$ can be transformed  by \EO s
into a final \PBS\ $\{ b_F \}$, where $b_F(4) = n'$ for  $n' = 0,1,2,\dots, n$.

Without loss of generality, we may assume that $n(W) >0$ because if  $n(W)=0$ then  Corollary~\ref{cor2} yields the desired result.

Having found this number  $n(W) \ge 1$ in deterministic space \eqref{sp}, we will run
Savitch's algorithm again for the pair $B_0$ and $\{ b^*_F \}$, where $b^*_F = (0, |W|, 0, n(W))$, and use the counter trick, as in the proof of Theorem~\ref{thm3}, to compute  an instance of a sequence $\Omega$ of \EO s
and the corresponding to  $\Omega$ sequence $B_0, B_1, \dots, B_\ell = \{ b^*_F \}$   of \PBS s.
After computing these sequences   $\Omega$  and $B_0, B_1, \dots, B_\ell = \{ b^*_F \}$, our algorithm halts. Denote this modification of Savitch's algorithm by ${\mathfrak{A}}_{n}$.

Denote
$$
\Omega = (\omega_1, \dots, \omega_\ell),
$$
where $\omega_1, \dots, \omega_\ell$ are \EO s and, as above, let
$B_0,  B_1,  \dots, B_\ell$ be the corresponding to  $\Omega$  sequence of \PBS s so that $B_j$ is obtained by application of $\omega_j$ to $B_{j-1}$, so  $B_j = \omega_j(B_{j-1})$.
We also let
$$
(\chi_1, \dots, \chi_{\ell_2} )
$$
denote the subsequence of
$\Omega$ that consists of all extensions of type 2. Also, let $c_i \in B_{j_{i}-1}$ denote the \pbb\ to which the \EO\  $\chi_i$ applies and let $d_i\in B_{j_{i}}$ denote the \pbb\ obtained from $c_i$ by application of $\chi_i$, so $d_i = \chi_i(c_i)$.

According to the proof of Lemma~\ref{7bb}, the sequence $\Omega$ or, equivalently, the  sequence $B_0,  B_1,  \dots, B_\ell$,  defines a \ddd\ $\D(\Omega)$ which can be inductively constructed as in the proof of Claim (D1). Furthermore, according to the proof of  Claim (D1),   all faces of $\D(\Omega)$ are contained
in $\ell_2$  $a_2$-bands $\Gamma_1, \dots, \Gamma_{\ell_2}$ which are in bijective correspondence with
\EO s $\chi_1, \dots, \chi_{\ell_2}$ so that $\Gamma_i$ corresponds to $\chi_i$, $i =1, \dots, \ell_2$.
Denote
\begin{equation}\label{bgi}
\p \Gamma_i = f_i t_i g_i u_i ,
 \end{equation}
where $f_i, g_i$ are edges of $\p \D(\Omega)$, $\ph(f_i) = \ph(g_i)^{-1} \in \{  a_2^{\pm 1} \}$,  and $t_i, u_i$ are simple paths whose labels are powers of $a_1$. Hence,
\begin{equation*}
\sum_{i=1}^{\ell_2} | \Gamma_i(2) |  = | \D(\Omega)(2)  | = n(W) .
 \end{equation*}

It follows from the definitions that if $a(c_i)$, $a(d_i)$ are the arcs of the \pbb s $c_i, d_i$,
resp., and  $e_{1,i}  a(c_i) e_{2,i}$ is a subpath of the path $P_W$, where $e_{1,i}, e_{2,i}$ are edges of $P_W$,  then,  renaming  $f_i \leftrightarrows g_i$,  $t_i \leftrightarrows u_i$ if necessary,  we have the following equalities
\begin{gather*}
\al(e_{1,i}) = f_i, \ \ \al(e_{2,i}) = g_i, \ \ \ph(f_{i}) = \ph(g_i)^{-1} = a_2^{\e_i} , \ \ \e_i = \pm 1 , \\
\ph(t_{i}) \equiv a_1^{c_i(3)}, \quad \ph(u_{i}) \equiv a_1^{-d_i(3)} ,  \quad \al(a(c_i)) = t_i,
\quad \al(a(d_i)) = u_i^{-1}
\end{gather*}
for every $i = 1, \dots, \ell_2$, see Fig.~9.1.
\vskip 2mm

\begin{center}
\begin{tikzpicture}[scale=.57]
\draw [-latex](-4,-3) --(-4,-1.3);
\draw [-latex](4,0) --(4,-1.7);
\draw [-latex](4,-3) --(0,-3);
\draw [-latex](-4,0) --(-0,0);
\draw  (-4,0) rectangle (4,-3);
\node at (-5.9,-1.5) {$f_i = \alpha(e_{1,i}) $};
\node at (5.8,-1.5) {$ g_i= \alpha(e_{2,i})$};
\node at (0,.7) {$t_i = \alpha(a(c_i))$};
\node at (0,-2.3) {$u_i = \alpha(a(d_i))^{-1}$};
\node at (0,-1) {$\Gamma_i$};
\node at (0,-4) {Fig.~9.1};
\end{tikzpicture}
\end{center}

Note that each of these $a_2$-bands  $\Gamma_1, \dots, \Gamma_{\ell_2}$  can be constructed,
as a part of the sequence $\Omega$, in deterministic space \er{sp} by running the algorithm $\mathfrak{A}_{n}$. For instance, if we wish to retrieve information about the diagram  $\Gamma_i$, we would be looking for the $i$th
extension of type 2 in the sequence  $\Omega$, denoted above by
$\chi_i$. We also remark that the parameters $(c_i, \e_i)$, associated with the \EO\ $\chi_i$,
contain all the information about the diagram   $\Gamma_i$.

Observe that we are not able to keep all these pairs $(c_i, \e_i)$, $i =1, \dots, \ell_2$,
in our intermediate computations aimed to  construct $\D(\Omega)$ in polylogarithmic space because doing this would take polynomial space. However, we can reuse space, so we keep one pair $(c_i, \e_i)$, or a few pairs, in memory at any given time and, when we need a different pair  $(c_j, \e_j)$, we erase $(c_i, \e_i)$ and compute the new pair $(c_j, \e_j)$ by running the algorithm $\mathfrak{A}_{n}$ as discussed above.

We can also output all these pairs $(c_i, \e_i)$, $i =1, \dots, \ell_2$, as a part of our description of the
\ddd\  $\D(\Omega)$ still under construction.

Thus, in deterministic space \er{sp}, we have obtained the information about $a_2$-bands $\Gamma_1,
\dots, \Gamma_{\ell_2}$ of  $\D(\Omega)$ that contain all of the faces of  $\D(\Omega)$ and it remains
to describe, working in space \er{sp}, how the edges of $(\p \D(\Omega))^{\pm 1}$ and those of
$(\p \Gamma_1)^{\pm 1},
\dots, (\p \Gamma_{\ell_2})^{\pm 1}$ are attached to each other.
\medskip

Observe that by taking the subdiagrams $\Gamma_1, \dots, \Gamma_{\ell_2}$ out of  $\D(\Omega)$, we will produce  $\ell_2 +1$ connected components $\D_1, \dots, \D_{\ell_2+1}$  which are  \ddd s with no faces, i.e., $\D_1, \dots, \D_{\ell_2+1}$  are  \ddd s  over the free group $F(\A) = \langle \A \ \| \ \varnothing    \rangle $. Note that the boundary of each \ddd\ $\D_i$ has a natural factorization
\begin{equation}\label{Dbi}
\p \D_i = q_{1,i}  r_{1,i} \ldots  q_{k_i,i}  r_{k_i,i} ,
 \end{equation}
where every   $q_{j,i} $   is a subpath  of the cyclic path $\p \D(\Omega)$, perhaps,  $|q_{j,i}|=0$,
and  every   $r_{j,i} $   is one of the paths $t_1^{-1}, u_1^{-1}, \dots, t_{\ell_2}^{-1}, u_{\ell_2}^{-1}$,
 $|r_{j,i} | > 0$, see Fig.~9.2.

\begin{center}
\begin{tikzpicture}[scale=.79]
\draw  (0,0) circle (3);
\draw  plot[smooth, tension=.7] coordinates {(-2.65,-1.4) (-2,0) (-2.8,1.1)     };
\draw  plot[smooth, tension=.7] coordinates {(-2.15,-2.1) (-1,0)   (-2.3,1.9)};
\draw  plot[smooth, tension=.7] coordinates {(-1.1,2.8)(0,1.8) (1.1,2.8)};
\draw  plot[smooth, tension=.7] coordinates {(-1.8,2.4)(0,.9) (1.8,2.4)};
\draw  plot[smooth, tension=.7] coordinates {(2.65,-1.4) (2,0) (2.8,1.1)     };
\draw  plot[smooth, tension=.7] coordinates {(2.15,-2.1) (1,0)   (2.3,1.9)};
\draw [-latex](.1,-3) --(-.1,-3);
\draw [-latex](-.1,3) --(.1,3);
\draw [-latex](-.1,.9) --(.1,.9);
\draw [-latex](.06,1.8) --(-.06,1.8);
\draw [-latex](-3.,-.10) --(-3,.1);
\draw [-latex](-2,.10) --(-2,-.1);
\draw [-latex](2.03,2.2) --(2.13,2.1);
\draw [-latex](-1,-.1) --(-1,.1);
\draw [-latex](1,.1) --(1,-.1);
\draw [-latex](2,-.1) --(2,.1);
\draw [-latex](3,.1) --(3,-.1);
\draw [-latex](-2.12,2.1) --(-2.05,2.19);
\node at (3.6,0) {$q_{1,2}$};
\node at (2.5,0) {$\Delta_2$};
\node at (1.55,0) {$r_{1,2}$};
\node at (.54,0) {$r_{2,4}$};
\node at (-3.6,0) {$q_{1,3}$};
\node at (-2.5,0) {$\Delta_3$};
\node at (-1.55,0) {$r_{1,3}$};
\node at (-.54,0) {$r_{3,4}$};
\node at (-2.46,2.4) {$q_{1,4}$};
\node at (2.46,2.4) {$q_{2,4}$};
\node at (0,3.34) {$q_{1,1}$};
\node at (0,2.4) {$\Delta_1$};
\node at (0,1.5) {$r_{1,1}$};
\node at (0,.5) {$r_{1,4}$};
\node at (0,-2.6) {$q_{3,4}$};
\node at (-3.6,0) {$q_{1,3}$};
\node at (-2.5,0) {$\Delta_3$};
\node at (-1.55,0) {$r_{1,3}$};
\node at (0,-1.2) {$\Delta_4$};
\node at (0,-3.7) {Fig.~9.2};
\node at (-.9,2) {$\Gamma_2$};
\node at (-1.8,-0.9) {$\Gamma_1$};
\node at (2,0.9) {$\Gamma_3$};
\node at (-2.5,-2.7) {$\Delta$};
\end{tikzpicture}
\end{center}

It follows from the definitions and Lemma~\ref{8bb} that
 \begin{equation}\label{deli}
\sum_{i=1}^{\ell_2}  | \p \D_i | <  | \p \D(\Omega) |  +
| \D(\Omega)(2) |( |n_1| + |n_2|) = O(\max ( |W|, n(W)  )) ,
 \end{equation}
 as $| \D(\Omega)(2) | = n(W)$ and  $| \p \D(\Omega)| = |W|$.
 \medskip

Suppose that a  vertex $v \in P_W$ is given, $0 \le v \le |P_W| = |W|$.  Then there is a unique \ddd\
$\D_{i_v}$ among $\D_1, \dots, \D_{\ell_2+1}$ such that $\al(v) \in \p \D_{i_v}$.
We now describe an algorithm $\mathfrak{B}_{v}$ that, for an arbitrary edge $e$ of the boundary
$\p  \D_{i_v}$,  computes in deterministic space \er{sp} the label $\ph(e)$ of $e$ and the unique location
of $e$ in one of the paths $\p |_0 \D(\Omega)$, $t_1^{-1}, u_1^{-1}, \dots, t_{\ell_2}^{-1}, u_{\ell_2}^{-1}$. To do this, we will go around the boundary path $\p |_{\al(v)} \D_{i_v}$, starting at the vertex
$\al(v)$. Initializing the parameters $v^* , d^*$, we set
$$
v^* := v, \qquad  d^* := 1 .
$$

If $e$ is the $k$th edge of $\p |_{\al(v)} \D_{i_v}$, $1 \le k \le |\p \D_{i_v}|$, then we  also say that
$e$ is the edge of $\p |_{\al(v)} \D_{i_v}$ {\em with number} $k$.

Let $e_c$ denote the edge of $P_W$ such that $(e_c)_- = v^*$ if $v^* < |W|$ and $(e_c)_- = 0$ if $v^* = |W|$. We now consider three possible Cases 1--3.
 \medskip

Case 1.  \ Assume that $\al(e_c) = f_i$ for some $i = 1, \dots, \ell_2$, see the definition \er{bgi} of $\p \Gamma_i$.
Note that such an equality $\al(e_c) = f_i$ can be verified in space \er{sp} by checking, one by one,
all $a_2$-bands $\Gamma_1, \dots, \Gamma_{\ell_2}$.

If $v^* = v$, then the first $|t_i|$ edges of the boundary path
$$
\p |_{\al(v)} \D_{i_v} =  t_i^{-1} \ldots
$$
are the edges of the path $t_i^{-1}$, see Fig.~9.3(a).

In the general case when $v^*$ is arbitrary, we obtain that the edges of the boundary path
$\p |_{\al(v)} \D_{i_v}$ with numbers $d^*, \dots, d^*+ |t_i|-1$ are consecutive edges of the path $t_i^{-1}$ starting from
the first one.

Let $v' \in P_W$ denote the vertex such that $\al(v') = (t_i)_-$, see Fig.~9.3(a). Also, denote
$d' := d^* +  |t_i|$.
 \medskip

Case 2.  \ Assume that $\al(e_c) = g_i$ for some $i = 1, \dots, \ell_2$, see  \er{bgi}.
As in Case~1, we can verify this equality in space \er{sp}.

If $v^* = v$, then the first $|u_i|$ edges of
the  boundary  path
$$
\p |_{\al(v)} \D_{i_v} =  u_i^{-1} \ldots
$$
are the edges of the path $u_i^{-1}$, see Fig.~9.3(b).

In the general case when $v^*$ is arbitrary, we obtain that the edges of the  boundary  path
$\p |_{\al(v)} \D_{i_v}$ with numbers $d^*, \dots, d^*+ |u_i|-1$ are consecutive  edges of the path $u_i^{-1}$ starting from the first one.

Let $v' \in P_W$ denote the vertex such that $\al(v') = (u_i)_-$, see Fig.~9.3(b). Also, denote
$d' := d^* +  |u_i|$.

\begin{center}
\begin{tikzpicture}[scale=.7]
\draw  plot[smooth, tension=.5] coordinates {(-4,-3) (-4,2) (-2,2.8) (0,2) (0,-3)};
\draw(-4.1,.5) -- (0.1,.5);
\draw(-4.1,-.7) -- (0.1,-.7);
\draw  [-latex]  (-4.1,.5) -- (-2.,.5);
\draw [-latex]  (0.1,-.7)  --(-2.1,-.7) ;
\draw [-latex]  (-4.12,-.1)  --(-4.12,-.0) ;
\draw [-latex]  (.12,-.1)  --(.12,-.19) ;
\node at (-2,2.8) {};
\draw [-latex](-2.1,2.8) --(-2.,2.8);
\node at (-2,-.1) {$\Gamma_i$};
\node at (-4.6,-0.07) {$f_i$};
\node at (.6,-0.1) {$g_i$};
\node at (-2.1,0.9) {$u_i$};
\node at (-2.1,-1.2) {$t_i$};
\node at (-2,-2) {$ \Delta_{i_v}$};
\node at (-2,3.25) {$\partial \Delta$};
\draw  (-4.1,-.7) [fill = black] circle (.05);
\draw  (.1,-.7) [fill = black] circle (.05);
\node at (-4.9,-0.7) {$\alpha(v^*)$};
\node at (0.85,-.7) {$\alpha(v')$};
\node at (-2,-3.3) {Fig.~9.3(a)};
\draw  plot[smooth, tension=.5] coordinates {(5,-3) (5,2) (7,2.8) (9,2) (9,-3)};
\draw(4.9,.5) -- (9.1,.5);
\draw(4.9,-.7) -- (9.1,-.7);
\draw  [-latex]  (4.9,.5) -- (7.,.5);
\draw [-latex]  (9.1,-.7)  --(6.9,-.7) ;
\draw [-latex]  (4.88,-.1)  --(4.88,-.0) ;
\draw [-latex]  (9.12,-.1)  --(9.12,-.19) ;
\draw [-latex](6.9,2.8) --(7,2.8);
\node at (7,-.1) {$\Gamma_i$};
\node at (4.4,-0.07) {$f_i$};
\node at (9.6,-0.1) {$g_i$};
\node at (6.9,0.9) {$u_i$};
\node at (6.9,-1.2) {$t_i$};
\node at (7,1.7) {$\Delta_{i_v}$};
\node at (7,3.25) {$\partial \Delta$};
\draw  (4.88,.5) [fill = black] circle (.05);
\draw  (9.12,.5) [fill = black] circle (.05);
\node at (4.1,0.5) {$\alpha(v')$};
\node at (9.85,.5) {$\alpha(v^*)$};
\node at (7,-3.3) {Fig.~9.3(b)};
\end{tikzpicture}
\end{center}

Case 3. \ Suppose that $\al(e_c)$ is not one  of the edges $f_i$, $g_i$ of   $a_2$-bands $\Gamma_1,
\dots, \Gamma_{\ell_2}$. As above, we can verify this claim in space \er{sp}.

If $v^* = v$, then the first  edge of
the  boundary   path
$$
\p |_{\al(v)} \D_{i_v} =  \al(e_c) \ldots
$$
is $\al(e_c)$.

In the general case when $v^*$ is arbitrary, we have that the edge of the  boundary  path
$\p |_{\al(v)} \D_{i_v}$ with number $d^*$ is $\al(e_c)$, see Fig.~9.4.

Denote  $v':= (e_c)_-$ and let $d' := d^* + 1$, see Fig.~9.4.

\begin{center}
\begin{tikzpicture}[scale=.64]
\draw  plot[smooth, tension=.5] coordinates {(-4,-1.6) (-4,2) (-2,2.8) (0,2) (0,-1.6)};
\draw [-latex]  (-4.1,-.1)  --(-4.1,-.0) ;
\draw [-latex](-2.1,2.8) --(-2.,2.8);
\node at (-4.6,-0.07) {$e_c$};
\node at (-2.8,0) {$\Delta_{i_v}$};
\node at (-2,3.25) {$\partial \Delta$};
\draw  (-4.06,-.7) [fill = black] circle (.05);
\draw  (-4.1,.5) [fill = black] circle (.05);
\node at (-4.9,-0.7) {$\alpha(v^*)$};
\node at (-4.9,.5) {$\alpha(v')$};
\node at (-2,-2.2) {Fig.~9.4};
\draw  plot[smooth, tension=.7] coordinates {(-2.7,2.72)   (-1.8,1.4) (0.1,1.2)  };
\draw  plot[smooth, tension=.7] coordinates {(-3.5,2.45)  (-2.2,.8) (0.1,.5) };
\node at (-2.5,1.6) {$\Gamma_j$};
\end{tikzpicture}
\end{center}

The foregoing mutually exclusive Cases~1--3 describe a cycle, including the first one when $v^* = v$,
of the  algorithm $\mathfrak{B}_{v}$ which is finished by reassignment $v^* := v'$ and  $d^*  := d'$.

We now repeat the above cycle with new values of parameters  $v^*, d^*$, whose storage along with storage of the original vertex $v$ requires additional $$
O(\max(\log n, \log |W|))
$$ space, as follows from the definitions and inequality \er{deli}.
We keep on repeating this cycle until we obtain at certain point that the vertex $v^*$ is again $v$ which means that all the edges of the path $\p |_{\al(v)} \D_{i_v}$  have been considered,  we are back at $v$ and we stop.  Note that $ d^* = |\p |_{\al(v)} \D_{i_v}|$ when we stop.
\medskip

We remark that, when running the algorithm $\mathfrak{B}_{v}$ we can stop at once and abort computations whenever we encounter the value of parameter $v^*$ less than $v$, because, in this  case,  the \ddd\ $\D_{i_v}$ which contains the vertex $v$ is identical to  $\D_{i_{v^*}}$ and the information about the edges of $\D_{i_{v^*}}$ can be obtained by running the  algorithm $\mathfrak{B}_{v^*}$ instead. Thus, we now have that either the  algorithm $\mathfrak{B}_{v}$  aborts at some point or $\mathfrak{B}_{v}$ performs several cycles and stops when $v = v^*$ and $d^* >1$, in which case we say that the  algorithm $\mathfrak{B}_{v}$ {\em accepts} the diagram $\D_{i_v}$.

By running the  algorithms $\mathfrak{B}_{v}$ consecutively  for  $v=0, 1, \dots$, with possible stops as described above, we will get the information about the edges of all \ddd s   $\D_1, \dots, \D_{\ell_2+1}$ so that for every $j = 1, \dots, \ell_2+1$, there will be a unique vertex $v(j) \in P_W$ such that
$\D_j = \D_{i_{v(j)}}$  and the  algorithm $\mathfrak{B}_{v(j)}$  accepts $\D_j$.
Recall that, in the proof of Theorem~\ref{thm1}, see also Corollary~\ref{cor2}, we constructed a deterministic algorithm  $\mathfrak{C}$ such that, when given a word $U$ so that $U = 1$ in the free group $F(\A)$, the algorithm  $\mathfrak{C}$ outputs a diagram $\D_U$ over $F(\A) = \langle \A \ \| \ \varnothing    \rangle $ such that $\ph(\p |_{0} \D_U) \equiv U$ and the algorithm  $\mathfrak{C}$ operates in space $O((\log |U|)^3)$. Next, we observe that our ability  to deterministically compute, in  space \er{sp},  the label $\ph(e)$ of every edge $e$ of the boundary path $\p |_{\al(v)} \D_{i_v}$, as well as the location of $e$ in one of the paths $\p |_0 \D(\Omega)$, $t_1^{-1}, u_1^{-1}, \dots, t_{\ell_2}^{-1}, u_{\ell_2}^{-1}$, combined together with the algorithm  $\mathfrak{C}$,
means that we can construct a \ddd\ $\wtl \D_{i_v}$ over  $F(\A)$ such that $\ph (\wtl \D_{i_v} ) \equiv \ph (\D_{i_v} )$ in deterministic space
$$
O( |\p \D_{i_v}|^3 ) = O( (\max(\log |W|, \log n(W))^3)) ,
$$
see \er{deli}. Replacing the \ddd s   $\D_1, \dots, \D_{\ell_2+1}$ in  $\D(\Omega)$ with their substitutes
$\wtl \D_1, \dots, \wtl \D_{\ell_2+1}$, we will obtain a  disk diagram $\wtl \D(\Omega)$ over
\eqref{pr3b} such that
$$
\ph (\p \wtl \D(\Omega)) \equiv \ph (\p  \D(\Omega)) ,  \qquad  | \wtl \D(\Omega)(2)| = | \D(\Omega)(2)| = n(W) , $$
i.e., $\wtl \D(\Omega)$ is as good as $\D(\Omega)$.

Since the \ddd s   $ \wtl \D_1, \dots,  \wtl \D_{\ell_2+1}$ along with $a_2$-bands $\Gamma_1, \dots, \Gamma_{\ell_2}$ constitute the entire diagram $ \wtl \D(\Omega)$, our construction of $\wtl \D(\Omega)$, performed in space \er{sp}, is now complete.
\smallskip

It remains to prove the additional claim of Theorem~\ref{thm4}.  We start with the following.

\begin{lem}\label{91}   Suppose $\D_0$ is a reduced \ddd\ over \eqref{pr3b}, where $|n_1| = |n_2|$.
Then the number $| \D_0(2)|$ of faces in $\D_0$ satisfies  $| \D_0(2)| \le \tfrac {1}{4|n_1|} |\p \D_0|^2$.
\end{lem}

\begin{proof} Denote $n_0 := |n_1|$ and set $n_2 = \kappa n_1$, where $\kappa = \pm 1$.
Consider the presentation
\begin{equation}\label{pr3bac}
\GG_4 = \langle \, \A, b_1 , \dots,   b_{n_0-1}   \, \| \, a_2 a_1 b_1^{-1} = a_1^{\kappa},
   b_1 a_1 b_2^{-1} = a_1^{\kappa}, \ldots,   b_{n_0-1} a_1 a_2^{-1} = a_1^{\kappa}
  \rangle
\end{equation}
that has $n_0$ relations which are obtained by splitting the relation $ a_2 a_1^{n_1} a_2^{-1} = a_1^{\kappa n_1}$ of  \eqref{pr3b}  into $n_0$ ``square" relations, see Fig.~9.5  where the case $n_0 = 3$ is depicted.

\begin{center}
\begin{tikzpicture}[scale=.7]
\draw  (-2,2) rectangle (4,0);
\draw (0,0) -- (0,2);
\draw (2,0) -- (2,2);
\draw [-latex](-1.1,2) --(-.9,2);
\draw [-latex](2.9,2) --(3.1,2);
\draw [-latex](.9,2) --(1.1,2);
\draw [-latex](-1.1,0) --(-.9,0);
\draw [-latex](2.9,0) --(3.1,0);
\draw [-latex](.9,0) --(1.1,0);
\draw [-latex](-2,0.9) --(-2,1.1);
\draw [-latex](0,0.9) --(0,1.1);
\draw [-latex](2,0.9) --(2,1.1);
\draw [-latex](4,0.9) --(4,1.1);
\draw [-latex](2.9,0) --(3.1,0);
\draw [-latex](.9,0) --(1.1,0);
\node at (-1,2.4) {$a_1$};
\node at (1,2.4) {$a_1$};
\node at (3,2.4) {$a_1$};
\node at (-2.6,1) {$a_2$};
\node at (-.6,1) {$b_1$};
\node at (1.4,1) {$b_2$};
\node at (4.6,1) {$a_2$};
\node at (-1,-.6) {$a_1^{\kappa}$};
\node at (1,-.6) {$a_1^{\kappa}$};
\node at (3,-.6) {$a_1^{\kappa}$};
\node at (1,-1.5) {Fig.~9.5};
\end{tikzpicture}
\end{center}

Note that there is a natural isomorphism $\psi : \GG_4 \to \GG_3$ defined by the map
$\psi(a) = a$ for $a \in \A$, and $\psi(b_j) = a_1^{-\kappa j} a_2 a_1^j$ for $j =1, \dots, n_0-1$.

If $\D_0$ is a reduced \ddd\ over \eqref{pr3b}, then we can split faces of $\D_0$ into ``square" faces over \eqref{pr3bac}, as in Fig.~9.5, and obtain a reduced \ddd\  $\D_{0, b}$ over \eqref{pr3bac} such that
\begin{equation}\label{d0b}
\ph(\p \D_{0, b}  ) \equiv \ph(\p \D_0) , \ \  \  | \D_{0, b}(2)| = n_0 | \D_{0}(2)| .
\end{equation}

Let $\D$ be an arbitrary reduced \ddd\ over presentation \eqref{pr3bac}.

We  modify the definition of an $a_i$-band given in Sect.~6 for diagrams over presentation \eqref{pr3b}
so that this definition would be suitable for diagrams over \eqref{pr3bac}.

Two faces $\Pi_1, \Pi_2$ in a disk diagram $\D$ over  \eqref{pr3bac} are called {\em $j$-related}, where $j =1,2$, denoted  $\Pi_1 \leftrightarrow_j   \Pi_2$,  if there is an edge $e$ such that $e \in \p \Pi_1$,  $e^{- 1} \in \p \Pi_2$, and
$\ph(e) = a_1^{\pm 1}$ if $j =1$ or   $\ph(e) \in \{ a_2^{\pm 1}, b_1^{\pm 1}, \dots, b_{n_0-1}^{\pm 1} \}$ if $j =2$.
As before, we consider a minimal equivalence relation, denoted $\sim_j$, on the set of faces of $\D$ generated by the relation $\leftrightarrow_j$.

An {\em $a_i$-band}, where $i \ge 1$ is now arbitrary,  is a minimal subcomplex
$\Gamma$ of  $\D$  that contains an edge $e$ such that  $\ph(e)=a_i^{\pm 1}$ if $i \ne 2$ or
$\ph(e) \in \{ a_2^{\pm 1}, b_1^{\pm 1}, \dots, b_{n_0-1}^{\pm 1} \}$ if $i =2$ and $\Gamma$ has the following property.
If there is a face $\Pi$ in $\D$ such that $e \in (\p \Pi)^{\pm 1}$, then $\Gamma$ must contain all faces of the equivalence class
$[\Pi]_{\sim_i}$ of $\Pi$. As before, this definition implies that  an $a_i$-band  $\Gamma$  is either a subcomplex   consisting of a single nonoriented edge $\{ f, f^{- 1} \}$, where   $\ph(f)=a_i^{\pm 1}$ if $i \ne 2$ or
$\ph(f) \in \{ a_2^{\pm 1}, b_1^{\pm 1}, \dots, b_{n_0-1}^{\pm 1} \}$ if $i =2$  and   $f, f^{- 1} \in \p \D$,
or $\Gamma$ consists of all faces of an equivalence class $[\Pi]_{\sim_i}$, here $i =1,2$.

If an $a_i$-band $\Gamma$ contains faces, then $\Gamma$ is called {\em essential}.

If an $a_i$-band $\Gamma$ is essential but $\Gamma$ contains no face whose boundary contains an edge $f$ such that $f^{-1} \in \p \D$ and $\ph(f)=a_1^{\pm 1}$ if $i =1$ or
$\ph(f) \in \{ a_2^{\pm 1}, b_1^{\pm 1}, \dots, b_{n_0-1}^{\pm 1} \}$ if $i =2$, then $\Gamma$ is called a  {\em closed} $a_i$-band. It follows from the definitions that if
$\Gamma$ is an essential $a_i$-band, then $i=1,2$.

If  $\Gamma$ is an essential $a_i$-band and $\Pi_1, \dots, \Pi_k$ are all faces of $\Gamma$, we consider a simple arc or a simple curve $c(\Gamma)$ in the interior of $\Gamma$ such that the intersection $c(\Gamma) \cap \Pi_j$ for every $j = 1, \dots, k$, is a properly embedded arc in $\Pi_j$ whose boundary points belong to the interior of different edges of
$\p \Pi_j$ whose $\ph$-labels are in  $ \{ a_1^{\pm 1} \}$ if $i =1$ or in $ \{ a_2^{\pm 1}, b_1^{\pm 1}, \dots, b_{n_0-1}^{\pm 1} \}$ if $i =2$. This arc or curve $c(\Gamma)$ will be called a {\em connecting line} of $\Gamma$.

Note that if  $\Gamma$ contains a face $\Pi$ whose boundary has an edge $f$ such that $f^{-1} \in \p \D$ and   $\ph(f)=a_1^{\pm 1}$ if $i =1$ or
$\ph(f) \in \{ a_2^{\pm 1}, b_1^{\pm 1}, \dots, b_{n_0-1}^{\pm 1} \}$ if $i =2$,   then a connecting line $c(\Gamma)$ of  $\Gamma$   connects two points
$\p c(\Gamma)$ on the boundary $\p \D$. On the other hand, if  $\Gamma$ contains no face $\Pi$ as above, then    $c(\Gamma)$ is a closed simple curve, in which case $\Gamma$ is called a {\em closed} $a_i$-band.

If $\Pi$ is a face in $\D$ over \eqref{pr3bac}, then there are exactly two bands, $a_1$-band  and $a_2$-band, denoted $\Gamma_{\Pi, 1}$, $\Gamma_{\Pi, 2}$, resp., whose connecting lines $c(\Gamma_{\Pi, 1})$, $c(\Gamma_{\Pi, 2})$ pass through $\Pi$. Without loss of generality, we may assume that the intersection
$$
c(\Gamma_{\Pi, 1}) \cap c(\Gamma_{\Pi, 2}) \cap \Pi
$$
consists of a single (piercing) point. The following lemma has similarities with Lemma~\ref{b1}.

\begin{lem}\label{c1}  Suppose that $\D$ is a reduced disk diagram over presentation  \er{pr3bac}. Then there are no closed $a_i$-bands in $\D$  and every $a_i$-band $\Gamma$ of $\D$ is a disk subdiagram of $\D$ such that
$$
\p |_{(f_1)_-} \Gamma = f_1 s_1 f_2 s_2 ,
$$
where $f_1, f_2$ are edges of $\p \D$ such that
\begin{align*}
\ph(f_1), \ph(f_2) & \in \{ a_i^{\pm 1} \} \ \ \mbox{if} \ i \ne 2 , \\
\ph(f_1), \ph(f_2) & \in \{ a_2^{\pm 1}, b_1^{\pm 1}, \dots, b_{n_0-1}^{\pm 1} \} \ \  \mbox{if} \  i =2,
\end{align*}
$s_1, s_2$ are simple paths such that $|s_1| =  |s_2| = |\Gamma(2)|$, and $\p \Gamma$ is a reduced simple closed path when $|\Gamma(2)|>0$. In addition, if
$\Gamma_1$ and $\Gamma_2$ are essential $a_1$-band and $a_2$-band, resp., in $\D$,  then their connecting lines $c(\Gamma_{1}), c(\Gamma_{2})$ intersect in at most one point.
\end{lem}

\begin{proof}  Since $\D$ is reduced, it follows that if $\Pi_1 \leftrightarrow_j   \Pi_2$, $j=1,2$, then $\Pi_1$ is not a mirror copy of $\Pi_2$, i.e.,  $\ph (\Pi_1 ) \not\equiv  \ph (\Pi_2 )^{-1}$.
This remark, together with the definition of an $a_i$-band and defining relations of presentation \er{pr3bac}, implies that if $\Pi_1 \leftrightarrow_j  \Pi_2$ then the faces $\Pi_1, \Pi_2$ share exactly one nonoriented edge.
This, in particular, implies
$|s_1| =  |s_2| = |\Gamma(2)|$ if $\Gamma$ is an $a_i$-band such that $\p \Gamma = f_1 s_1 f_2 s_2$ as in the statement of Lemma~\ref{c1}.

Assume, on the contrary, that there is an essential $a_i$-band  $\Gamma_0$ in $\D$ such that either
$\Gamma_0$ is closed or  $\p \Gamma_0 = f_1 s_1 f_2 s_2$ as above, and one of the paths $s_1, s_2$ is not simple.
Then there is a disk subdiagram $\D_2$ of $\D$ bounded by  edges of $\p \Gamma_0$ such that $\ph( \p  \D_2)$ is a  nonempty reduced word over the alphabet
$\{ a_1^{\pm 1} \}$ if $i=2$ or  over the alphabet $\{ a_2^{\pm 1}, b_1^{\pm 1}, \dots, b_{n_0-1}^{\pm 1} \}$ if $i =1$.
Pick such diagrams $\Gamma_0$ and $\D_2$ so that  $|\D_2(2)|$ is minimal. Since $\D_2$ contains faces and  $\ph(\p \D_2)$ contains no letters either from $\{ a_1^{\pm 1} \}$ if $i=1$ or from $\{ a_2^{\pm 1}, b_1^{\pm 1}, \dots, b_{n_0-1}^{\pm 1} \}$ if $i=2$, it follows that every $a_{3-i}$-band in $\D_2$ is closed and bounds a similar to $\D_2$ diagram $\D'_2$ such that  $|\D'_2(2)|< |\D_2(2)|$. This contradiction proves that an $a_i$-band $\Gamma_0$ with the above properties does not exist.
\medskip

To prove the space part of the additional claim, suppose, on the contrary, that there are essential $a_1$- and $a_2$-bands $\Gamma_1$ and  $\Gamma_2$, resp., in $\D$ such that the intersection  $c(\Gamma_{1})\cap c(\Gamma_{2})$ of their connecting lines $c(\Gamma_{1}), c(\Gamma_{2})$ contains at least two points.  We  pick two such consecutive along
$c(\Gamma_{1}), c(\Gamma_{2})$  points  and let $\Pi_1, \Pi_2$ be the faces that contain these two points.   Then there exists a disk subdiagram $\D_1$ in $\D$ such that $\p \D_1 = u_1u_2$, where $u_1^{-1}$ is a subpath of
$\p \Gamma_{1}$, $|u_1| \ge 0$ and $\ph(u_1)$ is a reduced or empty word over
$\{ a_2^{\pm 1}, b_1^{\pm 1}, \dots, b_{n_0-1}^{\pm 1} \}$, while $u_2^{-1}$ is a subpath of
$\p \Gamma_{2}$, $|u_2| \ge 0$ and $\ph(u_2)$ is a reduced or empty  word over $\{ a_1^{\pm 1} \}$, see Fig.~9.6.
The equality $|u_1| + |u_2| =0$ implies that the faces $\Pi_1, \Pi_2$ form a reducible pair.
This contradiction to $\D$ being reduced  proves that  $|u_1|+ |u_2| >0$, whence $|\D_1(2)|>0$.

\begin{center}
\usetikzlibrary{arrows}
\begin{tikzpicture}[scale=.6]
\draw  plot[smooth, tension=.88] coordinates { (-5,-2)(-4,2) (-2,4) (0,2) (1,-2)};
\draw  plot[smooth, tension=.5] coordinates { (-4,-2)(-3,2) (-2,3) (-1,2) (0,-2)};
\draw  plot[smooth, tension=.5] coordinates { (-4,4)(-3,0) (-2,-1) (-1,0) (0,4)};
\draw  plot[smooth, tension=.88] coordinates { (-5,4)(-4,0) (-2,-2) (0,0) (1,4)};
\draw  (-3.3,1) [fill = black] circle (.05);
\draw  (-.7,1) [fill = black] circle (.05);
\draw  [-latex](-2.06,3) -- (-1.94,3);
\draw  [-latex] (-1.94,-1) -- (-2.06,-1);
\node at (-2,2.4) {$u_1$};
\node at (-2,-0.4) {$u_2$};
\node at (-3.8,1) {$\Pi_1$};
\node at (-0.1,1) {$\Pi_2$};
\node at (-2.8,3) {$\Gamma_1$};
\node at (-1,-1) {$\Gamma_2$};
\node at (-2,1) {$\Delta_1$};
\node at (-2,-3) {Fig.~9.6};
\end{tikzpicture}
\end{center}

We now pick a  disk subdiagram $\D'_1$ in $\D$ such that
$|\D'_1(2)|$ is minimal,
$$
\p \D'_1 = u'_1 u'_2 , \quad  |u'_1|+ |u'_2| >0 ,
$$
$\ph(u_1')$ is a reduced or empty word over the alphabet
$\{ a_2^{\pm 1}, b_1^{\pm 1}, \dots, b_{n_0-1}^{\pm 1} \}$,  and $\ph(u_2)$ is a reduced or empty  word over $\{ a_1^{\pm 1} \}$. Consider an $a_1$-band $\Gamma_3$ in  $\D'_1$
if  $|u'_2| >0$ or consider an $a_2$-band $\Gamma_4$ in  $\D'_1$ if $|u'_2| =0$ and $|u'_1| >0$ such that a connecting line of $\Gamma_j$, $j=3,4$, connects points on $u'_2$ if $j =3$ or on $u'_1$ if $j =4$, see  Fig.~9.7.  It is easy to check that, taking $\Gamma_3$, or $\Gamma_4$,  out of $\D'_1$,
we obtain two disk subdiagrams  $\D'_{1,1}, \D'_{1,2}$ one of which has
the above properties of  $\D'_1$ and a fewer number of faces, see Fig.~9.7. This contradiction to the minimality of $\D'_1$ completes the proof of Lemma~\ref{c1}.
\end{proof}

\begin{center}
\begin{tikzpicture}[scale=.44]
\draw  plot[smooth, tension=.45] coordinates {(-7,0)(-6,2) (-4,4) (0,4) (2,2) (3,0)};
\draw  plot[smooth, tension=.45] coordinates {(-7,0)(-6,-2) (-4,-4) (0,-4) (2,-2) (3,0)};
\draw  (-7,0) [fill = black] circle (.05);
\draw  (3,0) [fill = black] circle (.05);
\draw  [-latex](-2.06,4.18) -- (-1.94,4.18);
\draw  [-latex](-1.94,-4.18) -- (-2.06,-4.18);
\node at (-2,.7) {$\Gamma_3$};
\node at (-6.9,2.8) {$\Delta'_1$};
\node at (-4,2) {$\Delta'_{1,1}$};
\node at (-2,-2) {$\Delta'_{1,2}$};
\node at (-2,3.5) {$u'_1$};
\node at (-2,-3.5) {$u'_2$};
\node at (-2,-5.5) {Fig.~9.7};
\draw  plot[smooth, tension=.45] coordinates {(-5,-3.2) (-4,-2) (-2,0) (0,-2) (1,-3.2)};
\draw  plot[smooth, tension=.45] coordinates {(-6,-2) (-4.8,-.8) (-2,1.8) (0.8,-.8) (2,-2)};
\end{tikzpicture}
\end{center}

Coming back to the diagrams $\D_{0}$ and $\D_{0,b}$, see \er{d0b}, we can see from Lemma~\ref{c1} that the number of
$a_1$-bands in $\D_{0,b}$ is at most $\tfrac 12 |\p  \D_{0,b} |_{a_1} \le  \tfrac 12 |\p  \D_{0} |$. Similarly, the number  of
$a_2$-bands in $\D_{0,b}$ is at most $\tfrac 12 |\p  \D_{0,b} | = \tfrac 12 |\p  \D_{0} |$.
It follows from the definitions and Lemma~\ref{c1} that the number $|\D_{0,b}(2) |$ is equal to the number of intersections of
connecting lines of $a_1$-bands and $a_2$-bands in $\D_{0,b}$. Hence, referring to  Lemma~\ref{c1} again, we obtain that
$|\D_{0,b}(2)| \le (\tfrac 12 |\p \D_{0}| )^2 $. In view of \er{d0b},  we finally have
$$
|\D_{0}(2)| = \tfrac{1}{n_0} |\D_{0,b}(2)| \le \tfrac {1}{4n_0} |\p \D_{0}|^2 ,
$$
as desired.  The proof of Lemma~\ref{91} is complete.
\end{proof}

Suppose $\D_0$ is a reduced diagram over presentation \er{pr3b}, where $|n_1| = |n_2|$. It follows from Lemma~\ref{c1} that
$|\D_{0}(2)| \le  \tfrac {1}{4n_0} |\p \D_{0}|^2$. This inequality means that the space bound \er{sp} becomes $O((\log |W|)^3)$ as $n(W) \le  \tfrac {1}{4n_0} |\p \D_0 |^2 =  \tfrac {1}{4n_0} |W|^2$.
The space part of Theorem~\ref{thm4} is proved.
\medskip

To prove the time part  of Theorem~\ref{thm4}, we review the proof of $\mathsf P$ part of Theorem~\ref{thm2}. We observe that our arguments enable us, concurrently with computation of  numbers $\lambda(W(i,j))$, $\mu_3(W(i,j))$ for every parameterized subword $W(i,j)$ of $W$ such that $\lambda(W(i,j)) = \ell < \infty$,  to
inductively construct a minimal diagram  $\D(i,j)$ over \er{pr3b}  such that
$\ph( \p |_0  \D (i,j)) \equiv  W(i,j)a_1^{-\ell}$.  Details of this construction are straightforward in each of the subcases considered in the proof of $\mathsf P$ part of Theorem~\ref{thm2}. Note that
this modified algorithm  can still be run in time $O( |W|^4)$ for the general case.

If, in addition, $|n_1| = |n_2|$ then the inequality \eqref{bnnd} can be improved to
\begin{equation*}
| \lambda(W(i,j) ) |  \le | W(i,j) |,  \quad   \mu_3(W(i,j) ) \le  | W(i,j) |^2 .
\end{equation*}
Hence, using, as before, binary representation for numbers $\lambda(W(i',j') )$, $\mu_3(W(i',j') )$,
we can run inductive computation of numbers $\lambda(U)$, $\mu_3(U)$ for $U = W(i,j)$ and construction of
a diagram $\D(i,j)$  whenever it exists for given $i,j$,  in time
$$
O(| W(i,j)|\log | W(i,j)|) .
$$
This improves the bound for running time of our modified  algorithm,  that computes the numbers $\lambda(W ), \mu_3(W)$ and constructs a minimal  diagram  $\D(1, |W|)$ for $W$,  from  $O( |W|^4)$ to  $O( |W|^3\log |W|)$, as desired.

Theorem~\ref{thm4} is proved.
\end{proof}

\section{Polygonal Curves in the Plane and Proofs of Theorems~\ref{thm5}, \ref{thm6} and Corollary~\ref{cor3}}

Let $\TT$  denote a tessellation of the plane $\mathbb  R^2$ into unit squares whose vertices are points with integer coordinates. Let $c$ be a finite closed path in $\TT$  so that edges of $c$ are edges of $\TT$. Recall that
we have  the following two types of elementary operations over  $c$. If $e$ is an oriented edge of $c$, $e^{-1}$ is the edge with an opposite to $e$ orientation, and $ee^{-1}$ is a subpath of $c$ so that  $c = c_1 ee^{-1} c_2$, where $c_1, c_2$ are subpaths of $c$, then the operation $c \to c_1 c_2$ over $c$ is called an {\em \EH\ of type 1}.  Suppose that $c =c_1 u c_2$, where $c_1, u, c_2$ are subpaths of $c$, and a boundary path $\p s$ of a unit square $s$ of  $\TT$ is $\p s = uv$, where $u, v$ are subpaths of $\p s$ and either of $u, v$  could be of zero length, i.e., either of $u, v$  could be a single vertex of $\p s$. Then the operation $c \to c_1 v^{-1} c_2$ over $c$ is called an {\em \EH\ of type 2}.

\begin{T5} Let $c$ be a finite closed path in a tessellation $\TT$  of the plane $\mathbb  R^2$
into unit squares so that edges of $c$ are edges of $\TT$. Then a minimal number $m_2(c)$
such that there is a finite sequence of \EHs\ of type 1--2, which turns $c$ into a single point and which contains
$m_2(c)$ \EHs\ of type 2, can be computed in deterministic space $O((\log |c|)^3)$ or in deterministic time $O( |c|^3\log |c| )$, where $|c|$ is the length of $c$.

Furthermore, such a sequence of  \EHs\ of type 1--2,
which turns $c$ into a single point and which contains  $m_2(c)$ \EHs\ of type 2,  can also be computed in  deterministic space $O((\log |c|)^3)$   or in deterministic time $O( |c|^3\log |c| )$.
\end{T5}

\begin{proof} First  we assign $\ph$-labels to edges of the tessellation $\TT$. If an edge $e$ goes from a point of $\mathbb  R^2$ with coordinates  $(i,j)$  to  a point with coordinates  $(i+1,j)$, then we set  $\ph(e):= a_1$ and  $\ph(e^{-1}):= a_1^{-1}$.
If an edge $f$ goes from a point of $\mathbb  R^2$ with coordinates    $(i,j)$ to a point with coordinates  $(i,j+1)$, then we set  $\ph(f):= a_2$ and  $\ph(f^{-1}):= a_2^{-1}$.

Let $c$ be a finite closed path in $\TT$ whose edges are those of $\TT$. Without loss of generality, we may assume that $c$ starts at the origin, hence $c_-= c_+$ has coordinates $(0,0)$  (otherwise, we could apply a logspace reduction to achieve this property). Denote $c = \wtl f_1 \dots \wtl f_{|c|}$, where
 $\wtl f_1, \dots, \wtl f_{|c|}$ are edges of $c$, and let
$$
 \ph(c) := \ph(\wtl f_1) \dots \ph( \wtl f_{|c|}) ,
$$
where $ \ph(c)$ is a word over the alphabet $\A^{\pm 1} = \{ a_1^{\pm 1}, a_2^{\pm 1} \}$. Since $c$ is closed, it follows that $\ph(c) \overset{\GG_5}{=} 1$, where
\begin{equation}\label{pp1}
    \GG_5 :=  \langle \  a_1,  a_2  \ \|  \   a_2 a_1 a_2^{-1} a_1^{ -1} = 1  \,  \rangle .
\end{equation}

\begin{lem}\label{eh} Suppose that a closed path $c$ in $\TT$ can be turned into  a point by a finite sequence
$\Xi$ of \EHs\  and $m_2(\Xi)$ is the number of \EHs\ of type 2 in this sequence. Then there is a \ddd\ $\D$ such that $\ph(\p \D) \equiv \ph(c)$ and $|\D(2) | = m_2(\Xi )$.

Conversely, suppose there is a \ddd\ $\D$ with $\ph(\p \D) \equiv \ph(c)$. Then there is a finite sequence $\Xi$
of \EHs\  such that  $\Xi$  turns $c$ into a point and  $m_2(\Xi)= |\D(2) |$.
\end{lem}

\begin{proof} Given a finite sequence $\Xi$ of \EHs\ $\xi_1, \xi_2,\dots$, we can construct a \ddd\ $\D$ over \eqref{pp1} such that $\ph(\p \D) \equiv \ph(c)$ and $m_2(\Xi)= |\D(2) |$.
Indeed, starting with a simple closed path $q_c$ in the
plane $\mathbb R^2$ (without any tessellation) such that  $\ph(q_c) \equiv \ph(c)$, we can simulate  \EHs\  in the sequence  $\Xi$ so that an \EH\ of type 1 is simulated by folding a suitable pair of edges of the path $q_c$, see Fig.~10.1(a), and an \EH\ of type 2 is simulated by attachment of a face $\Pi$ over \eqref{pp1} to the bounded region of $\mathbb R^2$ whose boundary is $q_c$, see Fig.~10.1(b).

\begin{center}
\usetikzlibrary{arrows}
\begin{tikzpicture}[scale=.68]

\draw [->>]  (1.9,0) -- (2.99,0);
\node at (2.5,-2) {Fig.~10.1(a)};
\draw  (0,0) ellipse (1.5 and 1.2);
\draw  (0,1.2)[fill = black]circle (0.05);
\draw  (-0.62,1.091)[fill = black]circle (0.05);
\draw  (0.62,1.091)[fill = black]circle (0.05);
\draw [-latex]  (-.4,1.16) -- (-.25,1.19);
\draw [-latex]  (.26,1.19) -- (.37,1.18);
\draw [-latex]  (.1,-1.2) -- (-.1,-1.2);
\node at (9,0.5) {};
\node at (-0.4,1.6) {$e_1$};
\node at (0.4,1.6) {$e_2$};
\node at (0,-1.6) {$q_c$};
\draw  (5,0) ellipse (1.5 and 1.2);
\draw  (5,1.2)[fill = black]circle (0.05);
\draw  (5,1.2) -- (5,1.9);
\draw [-latex]  (5,1.4) -- (5,1.65);
\draw  (5,1.9)[fill = black]circle (0.05);
\node at (5,-1.6) {$q_{c'}$};
\draw [-latex]  (5.1,-1.2) -- (4.9,-1.2);
\node at (6.17,1.6) {$e_1 = e_2^{-1}$};
\draw [->>]  (11.9,0) -- (12.99,0);
\node at (12.5,-2) {Fig.~10.1(b)};
\draw  (10,0) ellipse (1.5 and 1.2);

\draw [-latex]  (9.9,1.2) -- (10.1,1.2);
\draw  (9.38,1.091) node (v1) {}[fill = black]circle (0.05);
\draw  (10.62,1.091) node (v2) {}[fill = black]circle (0.05);

\draw [-latex]  (10.1,-1.2) -- (9.9,-1.2);

\node at (10,1.6) {$u_1$};
\node at (10,-1.6) {$q_c$};
\draw  (15,0) ellipse (1.5 and 1.2);
\draw [-latex]  (14.9,1.2) -- (15.1,1.2);
\draw  (14.38,1.091) node (u1) {}[fill = black]circle (0.05);
\draw  (15.62,1.091) node (u2) {}[fill = black]circle (0.05);
\node at (15,1.6) {$u_1$};

\draw  plot[smooth, tension=1.1] coordinates {(u1) (15,0.3) (u2)};
\draw [-latex]  (14.9,0.32) -- (15.07,0.3);
\node at (15,-1.6) {$q_{c'}$};
\draw [-latex]  (15.1,-1.2) -- (14.9,-1.2);

\node at (15,0.75) {$\Pi$};
\node at (15,-0.15) {$u_2$};
\node at (12.9,1.4) {$\partial \Pi = u_2 u_1^{-1}$};
\end{tikzpicture}
\end{center}

\noindent
As a result, we will fill out the bounded region of $\mathbb R^2$ whose boundary path is $q_c$ with faces and obtain
a required \ddd\ $\D$ over \eqref{pp1}.

The converse can be established by a straightforward induction on $|\D(2)|$.
\end{proof}

It follows from Lemma~\ref{eh} and Theorem~\ref{thm4} that a minimal number $m_2(c)$ of \EHs\ of type 2 in a sequence of \EHs\ that turns the path $c$ into a point can be computed in deterministic space
$O((  \log  |\ph(c)| )^3) = O((\log |c|)^3)$  or in deterministic time $O(|c|^3\log |c|)$.
\medskip

It remains to show that a desired sequence of \EHs\ for $c$ can be computed in deterministic space $O((\log |c|)^3)$ or in deterministic time $O(|c|^3\log |c|)$. We remark that Lemma~\ref{eh} and its proof reveal a close connection between sequences of \EHs\ that turn $c$ into a point and \ddd s $\D$ over \eqref{pp1} such that $\ph(\p |_0 \D) \equiv \ph(c)$.
According to Theorem~\ref{thm4}, we can construct a \ddd\  $\D_c$ such that $\ph(\p |_0 \D_c) \equiv \ph(c)$ and  $|\D_c(2) | = m_2(c)$ in  space $O((\log |c|)^3)$   or in time $O(|c|^3\log |c|)$.  Since
$|\D_c(2) | \le  |\p \D_c |^2/4  = |c|^2/4 $ by Lemma~\ref{91}, it follows that we can use this diagram
$\D_c$ together with Lemma~\ref{eh}
to construct a desired sequence of \EHs\ in  deterministic time $O(|c|^2)$. Therefore, a desired sequence of \EHs\
can be constructed in  deterministic  time  $O(|c|^3\log |c|)$.
\medskip

It is also tempting to use this diagram $\D_c$ together
with Lemma~\ref{eh} to construct, in   space $O((\log |c|)^3)$,  a desired sequence of \EHs\
which turns $c$ into a point.
However, this approach does not quite work for the space part of the proof  because inductive arguments of
the proof of  Lemma~\ref{eh} use intermediate paths whose storage would not be possible in $O((\log |c|)^3)$ space.
For this reason, we have to turn back to  calculus of \pbb s.
\smallskip

Denote $W := \ph(c)$. As in Sects.~2, 6, let $P_W$ be a path such that $\ph(P_W) \equiv W$ and vertices of $P_W$ are integers $0,1, \dots, |W|=|c|$.  Denote $P_W =  f_1 \dots  f_{|c|}$, where $ f_1, \dots, f_{|c|}$ are edges of $P_W$, and let
$$
\gamma : P_W \to c
$$
be a cellular map such that $\gamma( f_i) = \wtl f_i$ for  $i=1, \dots, |W|$. Recall that
$c = \wtl  f_1 \dots  \wtl f_{|c|}$. Note that $\gamma(i) =  (\wtl f_i)_+$, where  $i=1, \dots, |W|$,  and $\gamma(0) =  (\wtl f_1)_- = (\wtl f_{|c|})_+$.

Suppose $B$ is a \PBS\ for $W$ and let $B = \{ b_1, \dots, b_k \}$, where $b_i(2) \le b_j(1)$ if $i <j$. Recall that
the arc $a(b)$ of a \pbb\ $b \in B$ is a subpath of $P_W$ such that $a(b)_- = b(1)$ and $a(b)_+ = b(2)$.

We now define a path $\Gamma(B)$ for $B$ in $\TT$ so that $\Gamma(B):= \gamma(P_W) = c$ if $k=0$, i.e., $B = \varnothing$, and for  $k >0$ we set
\begin{equation}\label{pp2}
 \Gamma(B) := c_0 \delta(a(b_1)) c_1   \ldots \delta(a(b_k)) c_k  ,
\end{equation}
where $c_i = \gamma(d_i)$  and $d_i$ is a subpath of $P_W$ defined by $(d_i)_- = b_i(2)$ and
$(d_i)_+ = b_{i+1}(1)$, except for the case when $i = 0$, in which case $(d_0)_- = 0$, and
except for the case when $i = k$, in which case $(d_k)_- = |W|$. For every $i=1, \ldots, k$, the path
$\delta(a(b_i))$ in \eqref{pp2} is defined by the equalities
$$
\delta(a(b_i))_- = \gamma(b_{i}(1)) , \qquad \delta(a(b_i))_+ = \gamma(b_{i}(2))
$$
and by the equality $\ph( \delta(a(b_i)) ) \equiv a_1^{k_1}a_2^{k_2}$ with some integers $k_1, k_2$.

Suppose that $B_0, B_1, \ldots, B_\ell$ is an operational sequence of \PBS s for $W$ that corresponds to a sequence
 $\Omega$ of \EO s that turns the empty \PBS\ $B_0$ into a final \PBS\  $B_\ell = \{ b_{\ell, 1} \}$. We also assume that $b_{\ell, 1}(4) = |\D_c(2)|$ and $\D_c$ is a minimal \ddd\ over \eqref{pp1} such that
 $\ph(\p \D_c) \equiv W \equiv \ph(c)$. It follows from Lemmas~\ref{91},~\ref{eh} that
\begin{equation}\label{inqd}
|\D_c(2)| = m_2(c)  \le |\p \D_c|^2/4 = |W|^2/4 .
\end{equation}

Hence, in view of Lemma~\ref{8bb}, see also inequalities \eqref{ineqn1}--\eqref{ineqn2},  we may assume that,  for every  \pbb\ $b \in \cup_{i=0}^\ell B_i$, it is true that
\begin{align}\label{ieq1}
0 & \le b(1), b(2) \le |W| , \\ \notag
0 & \le | b(3) |  \le (|W|_{a_1} + (|n_1| +|n_2|) |\D_c(2)|)/2 \\   \label{ieq2}
  & \le (|W|_{a_1} + |W|^2/2)/2 \le  |W|^2/2  ,  \\ \label{ieq3}
0 & \le b(4) \le  |\D_c(2)| \le  |W|^2/4  .
\end{align}
Thus the space needed to store a \pbb\ $b \in \cup_{i=0}^\ell B_i$ is $O(\log |W|)$.

For every \PBS\ $B_i$, denote $B_i = \{ b_{i, 1}, \ldots, b_{i, k_i} \}$, where, as above,
$b_{i, j}(2) \le b_{i, j'}(1)$ whenever $j < j'$.

For every \PBS\ $B_i$,  we consider a path $c(i) := \Gamma(B_i)$ in $\TT$ defined by the formula \eqref{pp2},
hence, we have
\begin{equation}\label{pp3}
c(i) := \Gamma(B_i) = c_0(i) \delta(a(b_{i,1})) c_1(i)   \ldots \delta(a(b_{i,k_i})) c_{k_i}(i)  .
\end{equation}
As before, for $B_0 = \varnothing$, we set $c(0) := \Gamma(B_0) = c$.

Note that the last path $c(\ell)$ in the sequence $c(0),  \ldots, c(\ell)$, corresponding to a final
\PBS\  $B_\ell = \{ b_{\ell, 1} \}$, where $b_{\ell, 1} = (0, |W|, 0, m_2(c))$, consists of  a single vertex and has the form \eqref{pp3} equal to $c(\ell) =  c_0(\ell) \delta(a(b_{\ell,1})) c_1(\ell)$,
where $c_0(\ell) =    c_1(\ell) = c_- = c_+$ and $\delta(a(b_{\ell,1})) =c$.

\begin{lem}\label{f2} Let $b_{i,j} \in B_i$ be a \pbb . Then $\ph(\delta(a(b_{i,j}))  ) \equiv a_1^{b_{i,j}(3)}$, i.e., the points $\gamma(b_{i,j}(1))$, $\gamma(b_{i,j}(2))$ of $\TT$ have the same $y$-coordinate.
\end{lem}

\begin{proof} Since the path $\gamma(a(b_{i,j})) \delta(a(b_{i,j})^{-1}$ in  $\TT$ is closed, it follows that
 $$
 \ph( \gamma(a(b_{i,j}))) \overset{\GG_5}{=}  \ph( \delta(a(b_{i,j})^{-1} ) ,
 $$
where $\GG_5$ is given by presentation \eqref{pp1}. On the other hand, $\ph( \gamma(a(b_{i,j} )) \equiv \ph( a(b_{i,j} ))$. Hence, by Claim~(D) of the proof of Lemma~\ref{7bb}, we get that
 \begin{equation*}
\ph( \delta(a(b_{i,j})) )  \overset{\GG_5}{=}   \ph( a(b_{i,j} ) \overset{\GG_5}{=}  a_1^{b_{i,j}(3)}  .
\end{equation*}
Since $\ph( \delta(a(b_{i,j})) ) \equiv a_1^{k_1}a_2^{k_2}$ with some integers $k_1, k_2$, it follows that
$k_1 = b_{i,j}(3)$ and $k_2 = 0$, as required.
\end{proof}

We now analyze how the path $c(i)$, defined by \eqref{pp3}, changes in comparison with the path $c(i-1)$, $i \ge 1$.

Suppose that a \PBS\ $B_i$, $i \ge 1$, is obtained from $B_{i-1}$ by an addition. Then $c(i)=c(i-1)$, hence no change over $c(i-1)$ is done.  Note that the form \eqref{pp3} does change by an insertion of a subpath consisting of
a single vertex which is the $\delta$-image of the arc of the added to $B_{i-1}$ starting \pbb .

Assume that a \PBS\ $B_i$, $i \ge 1$, is obtained from $B_{i-1}$ by an extension of type 1 on the left and
$b_{i,j} \in B_i$ is obtained from $b_{i-1,j} \in B_{i-1}$  by this \EO . Let $a(b_{i,j}), a(b_{i-1,j}) $ be the arcs of $b_{i,j}, b_{i-1,j}$, resp., and let $a(b_{i,j}) = e_1 a(b_{i-1,j})$, where $e_1$ is an edge of $P_W$ and $\ph(e_1) = a_1^{\e_1}$, $\e_1 = \pm 1$.

If $\e_1 \cdot b_{i-1,j}(3) \ge 0$, then we can see that $c(i) = c(i-1)$ because
$$
c_{j-1}(i-1)   \delta(a(b_{i-1,j}))  = c_{j-1}(i)   \delta(a(b_{i,j}))
$$
and all other syllables of the paths $c(i)$ and $ c(i-1)$, as defined in \eqref{pp3}, are identical.

On the other hand, if $\e_1 \cdot b_{i-1,j}(3) < 0$, then the subpath $c_{j-1}(i)  \delta(a(b_{i,j}))$ of
$c(i)$ differs from   the subpath $c_{j-1}(i-1)   \delta(a(b_{i-1,j}))$ of  $c(i-1)$ by cancelation of a subpath
$ee^{-1}$, where $e$ is the last edge of $c_{j-1}(i-1)$,  $e^{-1}$ is the first edge of $\delta(a(b_{i-1,j}))$,
and $\ph(e) =a_1^{- b_{i-1,j}(3) / |b_{i-1,j}(3) | } $. Since all other syllables of
the paths $c(i)$ and $ c(i-1)$, as defined in \eqref{pp3}, are identical, the change of $ c(i-1)$, resulting in $c(i)$, can be described as an \EH\ of type 1 which deletes a subpath $ee^{-1}$, where $e$ is an edge of $c$ defined by the equalities
\begin{equation}\label{eab1}
e_+ = \gamma(b_{i-1,j}(1)), \qquad  \ph(e) = a_1^{- b_{i-1,j}(3) / |b_{i-1,j}(3) | } .
\end{equation}

The case when a \PBS\ $B_i$, $i \ge 1$, is obtained from $B_{i-1}$ by an extension of type 1 on the right
is similar but we need to make some changes. Keeping most of the foregoing notation unchanged, we let
 $a(b_{i,j}) = a(b_{i-1,j})e_2 $, where $e_2$ is an edge of $P_W$ and $\ph(e_2) = a_1^{\e_2}$, $\e_2 = \pm 1$.

If $\e_2 \cdot b_{i-1,j}(3) \ge 0$, then as above we conclude that $c(i) = c(i-1)$ because
$$
\delta(a(b_{i-1,j})) c_{j}(i-1)  =   \delta(a(b_{i,j}))c_{j}(i)
$$
and all other syllables of  the paths $c(i)$ and $ c(i-1)$, as defined in \eqref{pp3}, are identical.

If $\e_1 \cdot b_{i-1,j}(3) < 0$, then we can see from the definitions and from Lemma~\ref{f2} that
the path  $c(i)$ can be obtained from  $ c(i-1)$  by an \EH\ of type 1 which deletes a subpath $e^{-1}e$, where $e$ is an edge of $c$ defined by the equalities
\begin{equation}\label{eab2}
e_- = \gamma(b_{i-1,j}(2)), \qquad  \ph(e) = a_1^{- b_{i-1,j}(3) / |b_{i-1,j}(3) | } .
\end{equation}

We  conclude that, in the case when $B(i)$ is obtained from  $B(i-1)$  by an extension of type 1, it follows from  inequalities \eqref{ieq1}--\eqref{ieq3} and from equations \eqref{eab1}--\eqref{eab2}
that the \EH\ of type 1 that produces
$c(i)$ from $ c(i-1)$   can be computed in  space $O(\log |W|)$ when the \pbb s $b_{i,j} \in B_i$ and $b_{i-1,j} \in B_{i-1}$ are known.
\medskip

Suppose that a \PBS\ $B_i$, $i \ge 1$, is obtained from $B_{i-1}$ by an extension of type 2  and
$b_{i,j} \in B_i$ is obtained from $b_{i-1,j} \in B_{i-1}$  by this \EO . Denote $a(b_{i,j}) = e_1 a(b_{i-1,j}) e_2$, where $e_1, e_2$ are edges of $P_W$ and $\ph(e_1) = \ph(e_2)^{-1} =a_2^{\e}$,   $\e = \pm 1$. According to the definition of an extension of type 2,  $b_{i-1,j}(3) \ne 0$ and, in view of the equalities $n_1=n_2 =1$, see \eqref{pr3b} and \eqref{pp1}, we can derive from Lemma~\ref{f2} that the path $ c(i-1)$ turns into the path $c(i)$ in the following fashion. Denote $c_{j-1}(i-1) = \wtl c_{j-1}(i-1) \wtl e_1$ and $c_{j}(i-1) = \wtl e_2 \wtl c_{j}(i-1)$,
where  $\wtl c_{j-1}(i-1)$,   $\wtl c_{j}(i-1)$  are subpaths of $c_{j-1}(i-1)$,  $c_{j}(i-1)$, resp., and   $\wtl e_1, \wtl e_2$ are edges such that $\gamma( e_{i'}) = \wtl e_{i'}$, $i' =1,2$. Then
$c_{j-1}(i)  = \wtl c_{j-1}(i-1) $,  $c_{j}(i)  = \wtl c_{j}(i-1)$, and the path $\delta(a(b_{i,j}))$ has the properties that
\begin{align*}
& \delta(a(b_{i,j}))_-  = (\wtl e_1)_-, \qquad \delta(a(b_{i,j}))_+ = (\wtl e_2)_+, \\
& \ph(\delta(a(b_{i-1,j})) )  \equiv a_1^{b_{i-1,j}(3)} \equiv \ph(\delta(a(b_{i,j})) ) ,
\end{align*}
see Fig.~10.2, and the other syllables of the paths $c(i)$ and $c(i-1)$, as defined in \eqref{pp3},  are identical.

\begin{center}
\begin{tikzpicture}[scale=.64]
\draw  (-2,2) rectangle (6,0);
\draw (0,0) -- (0,2);
\draw (2,0) -- (2,2);
\draw (-5,0) -- (2,0);
\draw (4,0) -- (9,0);
\draw (4,0) -- (4,2);
\draw [-latex](2.4,2) --(2.5,2);
\draw [-latex](1.4,0) --(1.6,0);
\draw [-latex](-2,0.9) --(-2,1.1);
\draw [-latex](0,0.9) --(0,1.1);
\draw [-latex](2,0.9) --(2,1.1);
\draw [-latex](4,0.9) --(4,1.1);
\draw [-latex](6,1.1) --(6,0.88);
\draw [-latex](-3.6,0) --(-3.4,0);
\draw [-latex](7.4,0) --(7.6,0);
\node at (-1.,.5) {$s_1$};
\node at (1.,.5) {$s_2$};
\node at (5,.5) {$s_k$};
\node at (2.4,2.6) {$\delta(a(b_{i-1,j}))$};
\node at (-2.99,1) {$\widetilde e_1 = g_1$};
\node at (-.45,1) {$g_2$};
\node at (1.55,1) {$g_3$};
\node at (3.35,1) {$g_{k-1}$};
\node at (7.3,1) {$\widetilde e_2 = g_k^{-1}$};

\node at (-4,-.6) {$\widetilde c_{j-1}(i-1) = c_j(i)  $};
\node at (8.2,-.6) {$\widetilde c_{j}(i-1) = c_j(i)  $};
\node at (1.5,-.6) {$\delta(a(b_{i-1,j}))$};
\node at (2,-2.) {Fig.~10.2};
\draw  (-2,0) [fill = black] circle (.055);
\draw  (-2,2) [fill = black] circle (.055);
\draw  (6,0) [fill = black] circle (.055);
\draw  (6,2) [fill = black] circle (.055);
\node at (-3.4,2.9) {$ c_{j-1}(i-1) = \widetilde c_{j-1}(i-1) \widetilde e_1  $};
\node at (7.8,2.9) {$ c_{j}(i-1) = \widetilde e_2  \widetilde c_{j}(i-1)  $};
\end{tikzpicture}
\end{center}

Denote $k := | b_{i-1,j}(3) | = | b_{i,j}(3) | >0$ and let
\begin{align*}
\delta(a(b_{i-1,j})) = d_{i-1,1} \ldots d_{i-1,k}, \qquad
\delta(a(b_{i,j})) = d_{i,1} \ldots d_{i,k} ,
\end{align*}
where $d_{i',j'}$ are edges of the paths $\delta(a(b_{i-1,j})) $, $\delta(a(b_{i,j}))$.
Also, we let $g_1, \ldots, g_k$ be the edges of $\TT$ such that
$\ph(g_{i'}) = \ph(\wtl e_1) = a_2^{\e}$ and $(g_{i'})_- =  (d_{i, i'})_-$ for $i' = 1, \ldots, k$, in particular,
$g_1 = \wtl e_1$ and $g_k = \wtl e_2^{-1}$, see Fig.~10.2.
Then there are $k$ \EHs\ of type 2 which turn $c(i-1)$ into $c(i)$
and which use $k$ squares of the region bounded by the closed path
$\wtl e_1 \delta(a(b_{i-1,j}))  \wtl e_2 \delta(a(b_{i,j}))^{-1}$, see Fig.~10.2. For example, the first \EH\ of type 2 replaces the subpath $\wtl e_1 d_{i-1,1} = g_1 d_{i-1,1}$ by $d_{i,1} g_2$. Note that
$\p s_1 = g_1 d_{i-1,1} (d_{i,1} g_2)^{-1}$ is a (negatively orientated)  boundary path of a square $s_1$ of $\TT$.
The second \EH\ of type 2 (when $k \ge 2$) replaces the subpath $g_2 d_{i-1,2}$ by $d_{i,2} g_3$, where $\p s_2 = g_2 d_{i-1,2} (d_{i,2} g_3)^{-1}$ is a (negatively orientated) boundary path of a square $s_2$  of $\TT$, and so on.

In view of inequalities \eqref{ieq1}--\eqref{ieq3},  these $k = | b_{i-1,j}(3) |$ \EHs\ of type 2 that
produce  $c(i)$ from $ c(i-1)$   can be computed in space $O(\log |W|)$ when the \pbb s $b_{i,j} \in B_i$ and $b_{i-1,j} \in B_{i-1}$ are available.
\medskip

Suppose that a \PBS\ $B_i$, $i \ge 1$, is obtained from $B_{i-1}$ by an extension of type 3  and
$b_{i,j} \in B_i$ is obtained from $b_{i-1,j} \in B_{i-1}$  by this \EO .
By the definition of an extension of type 3, $b_{i-1,j}(3) = 0$ and  $a(b_{i,j}) = e_1 a(b_{i-1,j}) e_2$, where $e_1, e_2$ are edges of $P_W$ and $\ph(e_1) = \ph(e_2)^{-1} \ne a_1^{\pm 1}$. Hence,
$\ph(e_1) = \ph(e_2)^{-1} = a_2^{\e}$,  $\e = \pm 1$. It follows from Lemma~\ref{f2} that
\begin{align*}
\delta(a(b_{i-1,j}))  =  c_{j-1}(i-1)_+  =  c_{j}(i-1)_- , \quad
\delta(a(b_{i,j}))  =  c_{j-1}(i)_+  =  c_{j}(i)_-  .
\end{align*}

As above, denote $c_{j-1}(i-1) = \wtl c_{j-1}(i-1) \wtl e_1$ and $c_{j}(i-1) = \wtl e_2 \wtl c_{j}(i-1)$,
where  $\wtl c_{j-1}(i-1)$,   $\wtl c_{j}(i-1)$  are subpaths of $c_{j-1}(i-1)$,  $c_{j}(i-1)$, resp., and   $\wtl e_1, \wtl e_2$ are edges such that $\gamma( e_{i'}) = \wtl e_{i'}$, $i' =1,2$. Then it follows from the definitions
that
$c_{j-1}(i)  = \wtl c_{j-1}(i-1) $,   $c_{j}(i)  = \wtl c_{j}(i-1)$ and that all other syllables of the paths $c(i)$ and $c(i-1)$, as defined in \eqref{pp3}, are identical.  Hence, the change of the path
$c(i-1)$ into $c(i)$ can be described as an \EH\ of type 1 that deletes the subpath $\wtl e_1 \wtl e_2$ of
$c(i-1)$. Since $(e_1)_+ = b_{i-1,j}(1)$ and $\wtl e_1 = \gamma(e_1)$, it is easy to see from  inequalities \eqref{ieq1}--\eqref{ieq3} that we can compute
this \EH\ of type 1 in space $O(\log |W|)$ when the \pbb s $b_{i,j} \in B_i$ and $b_{i-1,j} \in B_{i-1}$ are given.
\smallskip

Suppose that a \PBS\ $B_i$, $i \ge 1$, is obtained from $B_{i-1}$ by a merger operation  and
$b_{i,j} \in B_i$ is obtained from \pbb s $b_{i-1,j}, b_{i-1,j+1} \in B_{i-1}$  by this merger.

First assume that $b_{i-1,j}(3) \cdot  b_{i-1,j+1}(3) \ge 0$. It follows from  Lemma~\ref{f2} and from the definitions that
\begin{align*}
 c_{j}(i-1)  =     \delta(a(b_{i-1,j}))_+ ,  \qquad
 \delta(a(b_{i-1,j})) c_{j}(i-1) \delta(a(b_{i-1,j+1})) = \delta(a(b_{i,j}))
\end{align*}
and all  other syllables of the paths $c(i)$ and $c(i-1)$, as defined in \eqref{pp3}, are identical. Therefore, we have the equality of paths $c(i-1)= c(i)$ in this case. Note that factorizations \eqref{pp3} of  $c(i)$ and $c(i-1)$ are different.
\medskip

Now assume that $b_{i-1,j}(3) \cdot  b_{i-1,j+1}(3) < 0$. It follows from  Lemma~\ref{f2} and from the definitions that $c_{j}(i-1) =     \delta(a(b_{i-1,j}))_+$ and that the subpath $\delta(a(b_{i,j}))$ of  $c(i)$ can be obtained from the  subpath $\delta(a(b_{i-1,j})) c_{j}(i-1) \delta(a(b_{i-1,j+1})) $ of $c(i-1)$ by cancelation of
$\min(|b_{i-1,j}(3)|,  |b_{i-1,j+1}(3)| )$ pairs of edges in $c(i-1)$ so that the last edge of $\delta(a(b_{i-1,j}))$
is canceled with the first edge of $\delta(a(b_{i-1,j+1}))$ as a subpath $ee^{-1}$ in the definition
of an \EH\ of type 1 and so on until a shortest path among $\delta(a(b_{i-1,j}))$, $\delta(a(b_{i-1,j+1}))$
completely cancels. Note that all  other syllables of the paths $c(i)$ and $c(i-1)$, as defined in \eqref{pp3},
are identical. Thus the path $c(i)$ can be obtained from $c(i-1)$ by $\min(|b_{i-1,j}(3)|,  |b_{i-1,j+1}(3)| )$
\EHs\ of type 1 which, in view of   \eqref{ieq1}--\eqref{ieq3}, can be computed in space $O(\log |W|)$ when the \pbb s $b_{i,j} \in B_i$ and $b_{i-1,j}, b_{i-1,j+1} \in B_{i-1}$ are known.
\medskip

Recall that, in the proof of Theorem~\ref{thm4}, we devised an algorithm $\mathfrak{A}_{n}$ that, in deterministic space \eqref{sp}, constructs a sequence of \EO s $\Omega$ and a corresponding sequence of \PBS s
$B_0, \ldots, B_\ell$ for $W$ such that $B_0$ is empty, $B_\ell = \{ b_{\ell, 1} \}$ is final and
$b_{\ell, 1}(4) = | \D(2) | \le n$, where $\D$ is a minimal \ddd\ over \eqref{pp1} with
$\ph(\p \D) \equiv W$.  In view of inequalities \eqref{inqd}, we may assume that $n = |W|^2/4$, hence, the bound
\eqref{sp} becomes $O( (\log |W|)^3)$ and, as we saw in \eqref{ieq1}--\eqref{ieq3},
every \pbb\ $b \in \cup_{i=0}^\ell B_i$ requires space $O(\log |W|)$ to store. As was discussed above, when given
\PBS s $B_{i-1}$ and $B_{i}$, we can construct, in space $O(\log |W|)$, a sequence of \EHs\ that turns the path $c(i-1)$ into $c(i)$. Since the sequence of \PBS s
$B_0, \ldots, B_\ell$ for $W \equiv \ph(c)$ is constructible in space $O( (\log |W|)^3)$, it follows that
a sequence $\Xi$ of \EHs\ of type 1--2 that turns the path $c=c(0)$ into a vertex $c(\ell) = c_-=c_+$ can also be constructed in deterministic   space $O( (\log |W|)^3)$. This completes the proof of Theorem~\ref{thm5}.
\end{proof}

Recall that a {\em polygonal} closed curve $c$ in the plane $\mathbb R^2$, equipped with a tessellation $\TT$
into unit squares, consists of finitely many line segments $c_1, \dots, c_k$, $k >0$, whose endpoints
are vertices of $\TT$,  $c=  c_1 \dots c_k$, and $c$ is closed, i.e., $c_- = c_+$.  If $c_i \subset \TT$ then the $\TT$-length $|c_i|_{\TT}$ of $c_i$
is the number of edges of $c_i$. If
$c_i \not\subset \TT$ then the $\TT$-length $|c_i|_{\TT}$ of $c_i$
is the number of connected components in $c_i  \setminus \TT$.
We assume that $|c_i|_{\TT} >0$ for every $i$ and set $|c|_{\TT} := \sum_{i=1}^k |c_i|_{\TT}$.

\begin{T6} Suppose that $n \ge 1$ is a fixed integer and  $c$ is a  polygonal closed curve in the plane $\mathbb R^2$ with given tessellation $\TT$ into unit squares.
Then, in deterministic space $O( (\log |c|_{\TT} )^3)$  or in deterministic time $O(  |c|_{\TT}^{n+3} \log  |c|_{\TT} )$, one can compute a rational number $r_n$
such that $|A(c) - r_n | < \tfrac {1}{ |c|_{\TT}^n }$.

In particular, if the area  $A(c)$ defined by $c$ is known to be an integer multiple of $\tfrac 1 L$, where $L>0$ is a given integer and $L< |c|_{\TT}^{n}/2$,  then
$A(c)$ can be computed in deterministic space $O( (\log |c|_{\TT} )^3)$  or in deterministic time $O(  |c|_{\TT}^{n+3} \log  |c|_{\TT} )$.
\end{T6}

\begin{proof} As above, let $c = c_1 \dots c_k$, where each $c_i$ is a line segment of $c$ of
positive length that connects vertices of $\TT$. For every $c_i$, we define an approximating path
$\zeta(c_i)$ such that $\zeta(c_i)_- = (c_i)_-$, $\zeta(c_i)_+ = (c_i)_+$ and $\zeta(c_i) \subset \TT$.
If $c_i \subset \TT$ then we set $\zeta(c_i) := c_i$.

Assume that  $c_i \not\subset \TT$. let $R_i$ be a rectangle consisting of unit squares of $\TT$ so that $c_i$ is a diagonal of $R_i$. Consider the set $N(c_i) \subseteq  R_i$ of all squares $s$ of $R_i$ such that the intersection $s \cap c_i$
is not empty  and is not a single point. Assuming that the boundary path $\p N(c_i)$ is negatively, i.e., clockwise, oriented, we represent  $\p N(c_i)$ in the form
$$
\p N(c_i) = q_{1}(c_i) q_{2}(c_i)^{-1} ,
$$
where $q_{1}(c_i), q_{2}(c_i)^{-1}$ are subpath of $\p N(c_i)$ defined by the equalities
$$
q_{1}(c_i)_- =  q_{2}(c_i)_- = (c_i)_-, \qquad q_{1}(c_i)_+ =  q_{2}(c_i)_+ = (c_i)_+.
$$
It is easy to see that these  equations uniquely determine the paths $q_{1}(c_i)$, $q_{2}(c_i)$, see Fig.~10.3.

\begin{center}
\begin{tikzpicture}[scale=1.05]
\draw  (-2,2) rectangle (1.5,0);
\draw  (-2,0) -- (1.5,2);
\draw [-latex](.5,1.43) -- (.7, 1.54);
\draw [-latex](-1.27,0) -- (-1.2, 0);
\draw [-latex](1,2) -- (1.2, 2);
\node at (-0.2,-1.4) {Fig.~10.3};
\draw  (-2,0) [fill = black] circle (.05);
\draw  (1.5,2) [fill = black] circle (.05);
\node at (.7,1.27) {$c_i$};
\node at (-2.5,2) {$R_i$};
\draw (-2,0.5) -- (-1.5,0.5) -- (-1.5,1) -- (-1,1) -- (-0.5,1) --
(-0.5,1.5) -- (0,1.5) -- (0.5,1.5) -- (0.5,2) -- (1,2) -- (1.5,2);
\draw (-1,0) -- (-1,0.5) -- (0,0.5) -- (0,1) -- (1,1) -- (1,1.5) -- (1.5,1.5) --
(1.5,2);
\node at (-1.4,-.45) {$q_{2}(c_i)$};
\node at (1,2.37) {$q_{1}(c_i)$};
\draw  [dashed] (-2,1.5) -- (1.5,1.5);
\draw  [dashed] (-2,1) -- (1.5,1);
\draw  [dashed] (-2,.5) -- (1.5,.5);
\draw  [dashed] (-1.5,2) -- (-1.5,0);
\draw  [dashed] (-1.,2) -- (-1.,0);
\draw  [dashed] (-.5,2) -- (-.5,0);
\draw  [dashed] (0,2) -- (0,0);
\draw  [dashed] (0.5,2) -- (0.5,0);
\draw  [dashed] (1.5,2) -- (1.5,0);
\draw  [dashed] (1.,2) -- (1.,0);
\end{tikzpicture}
\end{center}

\begin{lem}\label{f3}  Suppose $c_i, R_i$ are defined as above,  $c_i \not\subset \TT$, and
$|R_i|_x$, $|R_i|_y$ are the lengths of horizontal, vertical, resp., sides of the rectangle $R_i$.
Then
\begin{equation}\label{qq1}
\max(|R_i|_x, |R_i|_y)  \le  |c_i|_\TT \le  2 \max(|R_i|_x, |R_i|_y) ,
\end{equation}
and the area $A(c_i q_{1}(c_i)^{-1})$ bounded by the polygonal closed curve $c_i q_{1}(c_i)^{-1}$ satisfies
\begin{equation}\label{qq2}
A(c_i q_{1}(c_i)^{-1}) \le |c_i|_\TT .
\end{equation}

In addition, the paths $q_{1}(c_i), q_{2}(c_i)$ can be constructed in deterministic space $O(\log |c|_\TT)$
or in deterministic time $O( |c_i|_{\TT}  \log |c|_\TT )$.
\end{lem}

\begin{proof} We say that all unit squares  of $R_i$ whose points have $x$-coordinates in the same range compose a {\em column} of $R_i$. Similarly, we say that all unit squares  of $R_i$ whose points have $y$-coordinates in the same range compose a {\em row} of $R_i$.  Since $c_i$ is a diagonal of $R_i$, it follows that $c_i$ has an intersection, different from a single point,  with a unit square from every row and from every column of $R_i$.
Hence, in view of the definition of the $\TT$-length $|c_i|_\TT$ of $c_i$, we have
$\max(|R_i|_x, |R_i|_y)  \le  |c_i|_\TT$.

It follows from the definitions of the region $N(R_i) \subseteq R_i$ and the $\TT$-length $|c_i|_\TT$ that
the number of unit squares in $N(R_i)$ is $|c_i|_\TT$, hence, the area $A(N(R_i)) $ is equal to  $|c_i|_\TT$.
Since the closed curve $c_i q_{1}(c_i)^{-1}$ is contained in $N(R_i)$ and it can be turned into a simple curve by an arbitrarily small deformation, it follows that $A(c_i q_{1}(c_i)^{-1}) \le A(N(R_i)) =  |c_i|_\TT$, as required in \eqref{qq2}.

Let  $\sll(c_i)$ denote the slope of the line that  goes through $c_i$. If $|\sll(c_i)| \le 1$ then we can see that $N(c_i)$ contains at most 2 squares in every column of $R_i$. On the other hand, if $|\sll(c_i)| \ge 1$ then  $N(c_i)$ contains at most 2 squares in every row of $R_i$. Therefore,
$
A(N(c_{i}) ) \le   2\max(|R_i|_x, |R_i|_y) .
$
Since $A(N(c_{i}) ) = |c_i|_\TT$, the inequalities \eqref{qq1} are proven.

Note that, moving along columns of $R_i$ if $\sll(c_i) \le 1$ or moving along rows of $R_i$ if
$\sll(c_i) > 1$, we can detect  all squares of $R_i$ that belong to the region $N(c_{i})$ in
space $ O(\log |c|_\TT )$ or in time  $ O( |c_i|_\TT \log |c|_\TT  )$   by checking which squares in current
column or row are intersected by $c_i$ in more than one point.  This implies that the  paths $q_{1}(c_i), q_{2}(c_i)$ can also be constructed in space $O(\log |c|_\TT)$ or in  time  $ O( |c_i|_\TT \log |c|_\TT  )$, as desired.
\end{proof}

For $c_i \not\subset \TT$, define $\zeta(c_i) := q_{1}(c_i)$. Recall that $\zeta(c_i) = c_i $ if
$c_i \subset \TT$. We can now define an approximating closed  path $\zeta(c)$ in $\TT$ for $c$ by setting
$$
\zeta(c) := \zeta(c_1) \dots \zeta(c_k) .
$$
By Lemma~\ref{f3},
\begin{align*}
|A(c) -  A(\zeta(c)) |     \le \sum_{i=1}^k A(c_i \zeta(c_i)^{-1}) \le \sum_{i=1}^k |c_i|_\TT = |c|_\TT     .
\end{align*}

Consider a refined tessellation $\TT_M$, where $M>1$ is an integer, so that every unit square $s$ of  $\TT$ is divided into $M^2$ congruent squares each of area $M^{-2}$. We repeat the foregoing definitions of the lengths
$|c_i|_{\TT_M}$, $|c|_{\TT_M}$, the paths $q_{j, M}(c_i)$, $j=1,2$, $i=1,\dots,k$,  rectangles
$R_{i,M}$, regions $N_M(c_i) \subseteq R_{i,M}$, and approximating paths $\zeta_M(c_i), \zeta_M(c)$ with respect to
the refined  tessellation $\TT_M$ in place of $\TT$.

\begin{lem}\label{f4}  Suppose $c_i \not\subset \TT$, and
$|R_{i,M}|_x$, $|R_{i,M}|_y$ denote the path length of horizontal, vertical, resp., sides of the rectangle $R_{i,M}$.
Then $|R_{i,M}|_x = M |R_{i}|_x $, $|R_{i,M}|_y = M |R_{i}|_y$,
\begin{equation}\label{qq3}
\max(|R_{i,M}|_x, |R_{i,M}|_y)  \le  |c_i|_{\TT_M}  \le  2 \max(|R_{i,M}|_x, |R_{i,M}|_y) ,
\end{equation}
and the area  bounded by the polygonal closed curve $c_i q_{1, M}(c_i)^{-1}$ satisfies
\begin{equation}\label{qq4}
A(c_i q_{1, M}(c_i)^{-1}) \le M^{-2} |c_i|_{\TT_M} .
\end{equation}

In addition, we have
\begin{equation}\label{qq5}
  M |c_i |_{\TT}/2      \le  |c_i|_{\TT_M}  \le  2 M |c_i |_{\TT} ,
\end{equation}
and the paths $q_{1, M}(c_i), q_{2, M}(c_i)$ can be constructed in deterministic space
$$
O(\log |c|_{\TT_M}) = O(\log (M |c|_{\TT}) )
$$
or in deterministic time
$
O(|c_i|_{\TT_M} \log |c|_{\TT_M}) = O( M |c_i|_{\TT} \log (M |c|_{\TT}) )
$.
\end{lem}

\begin{proof} The equalities $|R_{i,M}|_x = M |R_{i}|_x $, $|R_{i,M}|_y = M |R_{i}|_y$ are obvious from the definitions. Proofs of inequalities \eqref{qq3}--\eqref{qq4} are analogous to the proofs of
inequalities \eqref{qq1}--\eqref{qq2} of Lemma~\ref{f3}  with the correction that the area of a
square of the tessellation $\TT_M$ is now $M^{-2}$.

The inequalities \eqref{qq5} follow from  inequalities \eqref{qq1}, \eqref{qq3}  and imply that
the space and time bounds $O(\log |c|_{\TT_M})$, $O(|c_i|_{\TT_M} \log |c|_{\TT_M})$, that are obtained as corresponding bounds of Lemma~\ref{f3}, can be rewritten in the form
$O(\log (M |c|_{\TT}) )$, $O( M |c_i|_{\TT} \log (M |c|_{\TT}) )$, resp.
\end{proof}

It follows from the definitions and Lemma~\ref{f4} that
\begin{align}\notag
|A(c) -  A(\zeta_M(c)) | & \le \sum_{i=1}^k A(c_i \zeta_M(c_i)^{-1} ) \le M^{-2} \sum_{i=1}^k
|c_i|_{\TT_M} \\  \label{qq6}
&  \le   M^{-2} \sum_{i=1}^k
2M |c_i|_{\TT}   = 2M^{-1} |c|_{\TT}
\end{align}
and that the path $\zeta_M(c_i) \subset \TT_M$ can be computed in space $O(\log (M |c|_{\TT}) )$ or
in time $O( M |c_i|_{\TT} \log (M |c|_{\TT}) )$.

Setting $M := |c|_{\TT}^{n+2}$, where $n \ge 1$ is a fixed integer,  we  can see from Lemma~\ref{f4} that the closed path $\zeta_M(c) \subset \TT_M$ can be constructed
in  space $O(\log |c|_{\TT} ) $ or in time  $O(  |c|_{\TT}^{n+3} \log  |c|_{\TT} )$. Hence, by Theorem~\ref{thm5} and Lemma~\ref{f4}, the area $A(\zeta_M(c))$
can be computed in  space $O((\log |c|_{\TT_M})^3 ) = O((\log |c|_{\TT})^3 )$ or in time
$O(  |c|_{\TT_M}^{3} \log  |c|_{\TT_M} ) = O(  |c|_{\TT}^{n+3} \log  |c|_{\TT} )$. The inequality
\eqref{qq6}  together with $M = |c|_{\TT}^{n+2}$ imply that
\begin{align}\label{qq7}
|A(c) -  A(\zeta_M(c)) | \le 2  M^{-1} |c|_{\TT} < |c|_{\TT}^{-n} ,
\end{align}
here we may assume $|c|_{\TT} > 2$ for otherwise  $A(c) = 0$ and  $r_n = 0$.
\smallskip

Finally, suppose that the area $A(c)$ is known to be an integer  multiple of $\tfrac 1L$, where $L >0$ is an integer with
$L < |c|_{\TT}^{n}/2$.  Applying the foregoing approximation result,
in deterministic space $O((\log |c|_{\TT})^3 )$  or in deterministic time  $O(  |c|_{\TT}^{n+3} \log  |c|_{\TT} )$, we can compute a rational number
$r_n = A(\zeta_M(c))$, where $M = |c|_{\TT}^{n+2}$, such that  the inequality \eqref{qq7} holds true.
It follows from Lemma~\ref{91}, inequalities \eqref{qq5} and the definitions that both numerator and denominator of $r_n$ are nonnegative integers that do not exceed
$$
\max (  |c|_{\TT_M}^2/4,   M^2 ) \le  \max (  M^2 |c|_{\TT}^{2},  M^2) \le |c|_{\TT}^{2(n+3)} .
$$
Hence, a binary representation of $r_n$ takes space $O(\log |c|_{\TT})$.

Since it is known that $\tfrac 1 L > \tfrac{2}{|c|_{\TT}^{n}}$, it follows from \eqref{qq7}  that there is
at most one integer multiple of $\tfrac 1L $ at the distance $< \tfrac{1}{|c|_{\TT}^{n}}$ from $r_n$.
This means that an integer multiple of $\tfrac 1L $ closest to  $r_n$ is the area $A(c)$.

Since  $r_n$ and $L$ are available and their  binary representations take space $O(\log |c|_{\TT})$,
it follows that a closest to  $r_n$  integer multiple of  $\tfrac 1L$,  can be computed in space
$O(\log |c|_{\TT})$ or in time  $O((\log |c|_{\TT})^2)$ and this will be the desired area $A(c)$.

Thus   the area $A(c)$  bounded by $c$ can be computed  in  space  $O(\log |c|_{\TT}^3)$ or in time
$O(  |c|_{\TT}^{n+3} \log  |c|_{\TT} )$.
Theorem~\ref{thm6} is proved.
\end{proof}

\begin{C3} Let $K \ge 1$ be a fixed integer and let $c$ be a  polygonal closed curve in the plane
$\mathbb R^2$ with given tessellation $\TT$ into unit squares such that $c$ has one
of the following two properties $\mathrm{(a)}$--$\mathrm{(b)}$.

$\mathrm{(a)}$ If $c_i, c_j$ are two nonparallel line segments of $c$ then their intersection point, if it exists, has coordinates that are integer multiples of $\tfrac 1 K$.

$\mathrm{(b)}$ If $c_i$ is a line segment of $c$ and  $a_{i,x}, a_{i,y}$ are coprime integers such that
the line given by an equation $a_{i,x}x + a_{i,y}y = b_{i}$, where $b_i$ is an integer,
contains $c_i$, then $\max(|a_{i,x}|,|a_{i,y}|) \le K$.

Then the area $A(c)$ defined by $c$ can be computed in deterministic space $O( (\log |c|_{\TT} )^3)$  or in deterministic time  $O( |c|_{\TT}^{n+3}  \log |c|_{\TT} )$, where $n$ depends on $K$.

In particular, if $\TT_\ast$ is a tessellation  of the plane $\mathbb R^2$ into  equilateral  triangles of unit area,
or into regular hexagons of unit area, and $q$ is a finite closed path in $\TT_\ast$ whose edges are edges of $\TT_\ast$, then the area $A(q)$ defined by $q$ can be computed in deterministic space $O( (\log |q |)^3)$  or in deterministic time $O( |q|^5 \log |q | )$.
\end{C3}

\begin{proof} Assume that the property (a) holds for $c$. Let $t$ be a triangle in the plane whose vertices $v_1, v_2, v_3$ are  points of the intersection of some nonparallel segments of $c$.
Then, using the determinant formula
\begin{align*}
A(t)  = \tfrac 12 | \det [\overrightarrow{v_1v_2}, \overrightarrow{v_1v_3} ] |
\end{align*}
for the area $A(t)$ of $t$,  we can see that  $A(t)$ is an integer multiple of $\tfrac {1} {2K^2}$.
Since $A(c)$ is the sum of areas of triangles such as $t$ discussed above, it follows that $A(c)$ is also an  integer multiple of $\tfrac {1} {2K^2}$. Taking  $n$  so that
\begin{align}\label{qq8a}
|c|_{\TT}^n > 4K^2 ,
\end{align}
we see that Theorem~\ref{thm6} applies with $L = 2K^2$ and yields the desired conclusion.
\medskip

Suppose that the property (b) holds for $c$.  By the Cramer's rule applied to the system of two linear equations
\begin{align*}
a_{i,x}x + a_{i,y}y & = b_{i} \\
a_{j,x}x + a_{j,y}y & = b_{j}
\end{align*}
which, as in  property (b), define the lines that contain nonparallel  segments $c_i, c_j$,   we have
that the intersection point of  $c_{i}$ and $c_{j}$ has rational coordinates whose  denominators do not exceed
\begin{align*}
\left|
\det \begin{bmatrix}
 a_{i,x}  & a_{i,y} \\
 a_{j,x} & a_{j,y} \\
\end{bmatrix} \right|  \le 2K^2 .
\end{align*}

This means that coordinates of the intersection point of two nonparallel line segments $c_{i}$, $c_{j}$ of $c$
are integer multiples of  $\tfrac{1}{ (2K^2)!}$ and the case when the property (b) holds for $c$ is reduced to the case when the  property (a)  holds for $c$  with $K' =  (2K^2)!$.
\smallskip

It remains to show the last claim of Corollary~\ref{cor3}. The case when $q$ is a closed path in the tessellation $\TT_6$ of the plane $\mathbb R^2$ into regular
hexagons of unit area can be obviously reduced to the case  when $q$ is a closed path in the tessellation $\TT_3$ of the plane $\mathbb R^2$ into equilateral triangles of unit area. For this reason we discuss the latter case only.

Consider the standard tessellation
$ \TT = \TT_4$ of the plane $\mathbb R^2$ into unit squares and draw  a diagonal in each square $s$ of  $\TT_4$
which connects the lower left vertex and the upper right vertex of $s$. Let $\TT_{4,1}$ denote  thus obtained tessellation  of $\mathbb R^2$ into triangles of area 1/2. Note that $\TT_3$ and $\TT_{4,1}$ are isomorphic
as graphs and the areas of corresponding triangles differ by a coefficient of $2^{\pm 1}$. This isomorphism  enables us to define  the image $q'$  in $\TT_{4,1}$  of  a closed path $q$ in $\TT_3$. It is clear that the area $A( q')$ defined by $ q'$ relative to  $\TT_{4,1}$  is half of the area $A(q)$ defined by $q$ relative to  $\TT_{3}$, $A( q') = A( q)/2$. We can also consider $q'$ as a polygonal closed curve relative to
the standard tessellation $\TT = \TT_4$ of  $\mathbb R^2$ into unit squares. Note that
$|q'|_{\TT_4} = |q'|_{\TT_{4,1}}= |q|_{\TT_3}$ and that the property (a) of Corollary~\ref{cor3} holds for $ q'$ with the constant $K=1$  relative to $\TT_4$. Thus, by proven part (a), the area $A(q')$   can be computed in   space
$O((\log |q'|_{\TT_4})^3 )= O((\log |q|_{\TT_3})^3 ) =  O((\log |q|)^3 ) $ or in   time
$$
O( |q'|_{\TT_4}^{5} \log |q'|_{\TT_4} )=   O( |q|_{\TT_3}^{5} \log |q|_{\TT_3} ) =
 O( |q|^{5} \log |q| )
$$
for the reason that $K=1$ and we can use $n =2$ in  \er{qq8a}  unless $|q'|_{\TT_4} = 2$ in which case $A(q') = A(q) =0$. Since $A( q) = 2A( q')$, our proof is complete.
\end{proof}

It is tempting to try to lift the restrictions of Corollary~\ref{cor3} to be able to compute, in
polylogarithmic space, the  area $A(c)$ defined by an arbitrary polygonal closed curve $c$ in the plane equipped with a tessellation $\TT$ into unit squares. Approximation approach of Theorem~\ref{thm6} together with construction  of actual \EHs\ for the approximating path $\zeta_M(c)$ of  Theorem~\ref{thm5} that seem to indicate the sign of the pieces  $A(c_i \zeta_M(c_i)^{-1})$  of the intermediate area between $c$ and $\zeta_M(c)$, both done in polylogarithmic space, provide certain credibility to this idea.

However, in the general situation, this idea would not work because the rational number $A(c)$ might have an exponentially large denominator, hence, $A(c)$ could take polynomial space just to store (let alone the computations). An example would be provided by a polygonal closed curve $c= c(n)$, where $n \ge 2$ is an integer,  such that  $|c|_\TT < (n+1)^2$  and the denominator of $A(c)$ is greater than $2^{k_1 n -1}$, where $k_1 >0$ is a constant. Below are details of such an example.
\smallskip

Let $n \ge 2$ be an integer and let $p_1, \dots, p_\ell$ be all primes not exceeding $n$. Recall that there is a constant $k_1 >0$ such that $p_1 \dots p_\ell > 2^{k_1n}$, see \cite{HW}. We construct line segments
$c_{3i+1}$, $c_{3i+2}$, $c_{3i+3}$ for $i=0, \dots, \ell-1$ and $c_{3\ell+1}$ of a polygonal closed curve $c=c(n)$
by induction as follows. Let the initial vertex $(c_1)_-$ of $c_1$ be
the point (with coordinates) $(0,0)$,  $(c_1)_-:= (0,0)$. Proceeding by induction on  $i\ge 0$, assume that the point $(c_{3i+1})_- =(x_0, y_0)$ is already defined. Then the line segment  $c_{3i+1}$ goes from the point  $(c_{3i+1})_- $ to the point $(x_0, y_0+1) = (c_{3i+1})_- +(0,1)$, written
$c_{3i+1} = [(c_{3i+1})_-, (c_{3i+1})_- +(0,1)]$.
Next, we set
$$
c_{3i+2} :=  [(c_{3i+1})_+, (c_{3i+1})_+ + (-1,0)] , \qquad c_{3i+3} :=  [(c_{3i+2})_+, (c_{3i+2})_+ + (p_{i+1},-1)]
$$
and, completing the induction step, define   $(c_{3(i+1)+1})_- :=  (c_{3i+3})_+$,  see Fig.~10.4, where the case $\ell = 3$ with $p_1 =2, p_2=3, p_3 = 5$  is depicted.
\vskip 2mm

\begin{center}
\begin{tikzpicture}[scale=1.2]
\draw  (0,0)  [fill = black] circle (.045);
\draw (0,0) -- (0,1);
\draw (0,0) -- (0,1)  -- (-1,1) -- (1,0) -- (1,1) -- (0,1) -- (3,0) -- (3,1) -- (2,1) -- (7,0) -- (0,0);
\draw  (0,1) [fill = black] circle (.045);
\draw  (-1,1) [fill = black] circle (.045);
\draw  (1,1) [fill = black] circle (.045);
\draw  (1,0) [fill = black] circle (.045);
\draw  (3,1) [fill = black] circle (.045);
\draw  (3,0) [fill = black] circle (.045);
\draw  (2,1) [fill = black] circle (.045);
\draw  (7,0) [fill = black] circle (.045);
\draw [-latex](0,.3) -- (0, .37);
\node at (-0.35,0.3) {$c_1$};
\draw [-latex](-.5,1) -- (-.6, 1);
\node at (-.5,1.27) {$c_2$};
\draw [-latex](.4,.3) -- (.48, .26);
\node at (.5,.5) {$c_3$};
\draw [-latex](1,.3) -- (1, .37);
\node at (1.35,0.3) {$c_4$};
\draw [-latex](.5,1) -- (.4, 1);
\node at (.5,1.27) {$c_5$};
\draw [-latex](1.8,.4) -- (2.1, .3);
\node at (2.1,.55) {$c_6$};
\draw [-latex](3,.3) -- (3, .37);
\node at (3.35,0.3) {$c_7$};
\draw [-latex](2.5,1) -- (2.4, 1);
\node at (2.5,1.27) {$c_8$};
\draw [-latex](4.5,.5) -- (5, .4);
\node at (4.9,.7) {$c_9$};
\draw [-latex](5.,.0) -- (4.2, 0);
\node at (4.3,-.35) {$c_{10}$};
\node at (2.7,-.9) {Fig.~10.4};
\node at (-1.,-.4) {$(c_{1})_- =(c_{10})_+ = (0,0)$};
\end{tikzpicture}
\end{center}

\noindent
Finally, we define $(c_{3\ell+1})_- := [(c_{3\ell})_+, (c_{1})_-]$ and $c=c(n) :=
c_1 \dots c_{3\ell} c_{3\ell+1}$, see Fig.~10.4 where $c_{3\ell+1} = c_{10}$.

It is not difficult to check that
$$
|c|_\TT = \ell + 2\sum_{i=1}^\ell p_i < (n+1)^2
$$
and that the area $A(c)$ defined by $c$ is the following sum of areas of $2\ell$ triangles
\begin{align*}
A(c) & =   \sum_{i=1}^\ell \tfrac 12 \left( \tfrac {1}{p_i} +  (p_i-1)(1- \tfrac {1}{p_i}) \right) =
  \tfrac 12 \sum_{i=1}^\ell \left( \tfrac {1}{p_i} +  (p_i-1) - 1 + \tfrac {1}{p_i} \right)=     \\
& = \sum_{i=1}^\ell \tfrac {1}{p_i} + \tfrac 12 \sum_{i=1}^\ell (p_i-2) =
\tfrac  { \sum_{i=1}^\ell \tfrac { p_1 \dots p_\ell  }{p_i}  }{ p_1 \dots p_\ell }   + \tfrac 12 \sum_{i=1}^\ell (p_i-2) .
\end{align*}
Since $p_1 \dots p_\ell > 2^{k_1 n}$,  it follows that, after possible cancelation of 2, the denominator of $A(c)$ is greater than $2^{k_1 n - 1}$.
\medskip

It would be of interest to study similar problems for  tessellation of the hyperbolic plane into regular congruent  $2g$-gons, where $g \ge 2$,  which would be technically close to the precise word problem and to the minimal diagram problem for the standard group presentation
\begin{equation*}
    \langle \,  a_1, a_2 \dots, a_{2g-1},  a_{2g} \ \|  \   a_1 a_2 a_1^{-1} a_2^{-1} \dots a_{2g-1} a_{2g} a_{2g-1}^{-1} a_{2g}^{-1}    =1  \,  \rangle
\end{equation*}
of the fundamental group of an orientable closed surface of genus $g \ge 2$. Hopefully, there could be developed a version of calculus of brackets for such presentations that would  be suitable for these problems.
\medskip

It is likely that suitable versions of calculus of brackets could be developed for problems on efficient planar folding of RNA strands  which would provide polylogarithmic space algorithms for such problems. Recall that
available  polynomial time algorithms for such problems, see \cite{CB1}, \cite{CB2}, \cite{CB3}, \cite{CB4}, use polynomial space.
This approach would give a chance to do, also in polylogarithmic space, maximization of foldings  relative to various parameters similar to those discussed in Theorem~\ref{thm3}.
\medskip

{\em Acknowledgements.} The author is grateful to Tim Riley  for bringing to the author's attention folklore arguments that solve the precise word problem for  presentation
$\langle \, a, b \,  \|  \,  a=1,  b=1 \,  \rangle$
in polynomial time, the  article \cite{SL1},  and the similarity between the \PWPP\ for presentation
$\langle \, a, b \,  \|  \,  a=1,  b=1 \,  \rangle$ and the problem on efficient planar folding of
RNA strands.  The author thanks the referee for a number of useful remarks and suggestions.

\end{document}